\documentclass[reqno,10pt]{amsart}
\usepackage{amsmath,amsfonts,amssymb,amsxtra,latexsym,amscd,enumitem,amsthm,verbatim,bm,mathtools}
\usepackage{hyperref}
\hypersetup{colorlinks=true, pdfstartview=FitV, linkcolor=BrickRed,citecolor=BrickRed, urlcolor=BrickRed, linktoc=page}

\usepackage{graphicx}
\usepackage{cancel}
\usepackage[dvipsnames]{xcolor}
\usepackage{tabularx}
\usepackage{subcaption}

\mathtoolsset{showonlyrefs=true}

\usepackage[margin=1.30in]{geometry}
\numberwithin{equation}{section}

\newcommand{\R}{\mathbb{R}}

\newcommand{\HH}{\mathbb{H}}

\renewcommand{\a}{{\bf a}}
\newcommand{\x}{{\bf x}}
\renewcommand{\v}{{\bf v}}
\renewcommand{\u}{{\bf u}}

\newcommand{\y}{{\bf y}}
\newcommand{\m}{\mathfrak{m}}
\newcommand{\ww}{\mathfrak{w}}
\newcommand{\F}{\mathcal{F}}
\newcommand{\E}{\mathcal{E}}
\newcommand{\W}{\mathcal{W}}
\newcommand{\X}{{\bf X}}
\newcommand{\XX}{\widetilde{\X}}
\newcommand{\V}{{\bf V}}
\newcommand{\VV}{\widetilde{\V}}
\renewcommand{\L}{{\bf L}}
\renewcommand{\l}{{\bf l}}
\newcommand{\A}{\mathcal{A}}
\newcommand{\B}{\mathcal{B}}

\newcommand{\eps}{\varepsilon}

\providecommand{\ip}[1]{\langle#1\rangle}
\providecommand{\abs}[1]{\left\lvert#1\right\rvert}
\providecommand{\aabs}[1]{\lvert#1\rvert}
\providecommand{\norm}[1]{\left\|#1\right\|}
\providecommand{\nnorm}[1]{\Vert#1\Vert}
\providecommand{\dm}[1]{[\![#1]\!]}

\newtheorem{theorem}{Theorem}[section]
\newtheorem*{theorem*}{Theorem}
\newtheorem{lemma}[theorem]{Lemma}
\newtheorem{corollary}[theorem]{Corollary}

\newtheorem{proposition}[theorem]{Proposition}
\newtheorem*{proposition*}{Proposition}
\newtheorem{definition}[theorem]{Definition}
\newtheorem{remark}[theorem]{Remark}

\begin{document}

\title{\vspace*{-2cm}Stability of a point charge for the repulsive Vlasov-Poisson system}

\author{Benoit Pausader}
\address{Brown University, Providence, RI, USA}
\email{benoit\_pausader@brown.edu}

\author{Klaus Widmayer}
\address{University of Zurich, Switzerland}
\email{klaus.widmayer@math.uzh.ch}

\author{Jiaqi Yang}
\address{ICERM, Brown University, Providence, RI, USA}
\email{jiaqi\_yang1@brown.edu}

\begin{abstract}
We consider solutions of the repulsive Vlasov-Poisson system which are a combination of a point charge and a small gas, i.e.\ measures of the form $\delta_{(\mathcal{X}(t),\mathcal{V}(t))}+\mu^2d{\bf x}d{\bf v}$ for some $(\mathcal{X},
\mathcal{V}):\mathbb{R}\to\mathbb{R}^6$ and a small gas distribution $\mu:\mathbb{R}\to L^2_{{\bf x},{\bf v}}$, and study asymptotic dynamics in the associated initial value problem. If initially suitable moments on $\mu_0=\mu(t=0)$ are small, we obtain a global solution of the above form, and the electric field generated by the gas distribution $\mu$ decays at an almost optimal rate. Assuming in addition boundedness of suitable derivatives of $\mu_0$, the electric field decays at an optimal rate and we derive a modified scattering dynamics for the motion of the point charge and the gas distribution.

Our proof makes crucial use of the Hamiltonian structure. The linearized system is transport by the Kepler ODE, which we integrate exactly through an asymptotic action-angle transformation. Thanks to a precise understanding of the associated kinematics, moment and derivative control is achieved via a bootstrap analysis that relies on the decay of the electric field associated to $\mu$. The asymptotic behavior can then be deduced from the properties of Poisson brackets in asymptotic action coordinates.
\end{abstract}

\maketitle

\setcounter{tocdepth}{1}
\vspace*{-.75cm}\tableofcontents\vspace*{-.75cm}

\section{Introduction}

This article is devoted to the study of the time evolution and asymptotic behavior of a three dimensional collisionless gas of charged particles (i.e.\ a plasma) that interacts with a point charge. Under suitable assumptions, a statistical description of such a system is given via a measure $M$ on $\mathbb{R}^3_{\bf x}\times\mathbb{R}^3_{\bf v}$ that models the charge distribution, which is transported by the long-range electrostatic (Coulomb) force field generated by $M$ itself, resulting in the Vlasov-Poisson system 
\begin{equation}\label{eq:VP}
\begin{split}
\partial_tM+\hbox{div}_{{\bf x},{\bf v}}\left(M\mathfrak{V}\right)=0,\qquad \mathfrak{V}={\bf v}\cdot \nabla_{\bf x}+\nabla_{\bf x}\phi_M\cdot\nabla_{\bf v},\qquad \Delta_{\bf x}\phi_M=\int_{\mathbb{R}^3_{\bf v}}Md{\bf v}.
\end{split}
\end{equation}
The Dirac mass $M_{eq}=\delta_{(0,0)}({\bf x},{\bf v})$ is a formal stationary solution of \eqref{eq:VP}, and we propose to investigate its stability. 
We thus consider solutions of the form\footnote{Here the initial continuous density $f_0=\mu_0^2$ is assumed to be non-negative, a condition which is then propagated by the flow and allows us to work with functions $\mu$ in an $L^2$ framework rather than a general non-negative function $f$ in $L^1$ -- see also the previous work \cite{IPWW2020} for more on this.} $M(t):=q_c\delta_{(\mathcal{X}(t),\mathcal{V}(t))}(\x,\v)+q_g\mu^2({\bf x},{\bf v},t)d{\bf x} d{\bf v}$, representing a small, smooth charge distribution $\mu^2d{\bf x}d{\bf v}$ (with mass-per-particle $m_g>0$ and charge-per-particle $q_g>0$) coupled with a point charge located at $(\mathcal{X},\mathcal{V}):\mathbb{R}_t\to\mathbb{R}^3\times\mathbb{R}^3$ (of mass $M_c>0$ and charge $q_c>0$). The equations \eqref{eq:VP} then take the form
\begin{equation}\label{VPPC}
\begin{split}
\left(\partial_t+{\bf v}\cdot\nabla_{\bf x}+\frac{q}{2}\frac{{\bf x}-\mathcal{X}}{\vert {\bf x}-\mathcal{X}\vert^3}\cdot\nabla_{\bf v}\right)\mu+Q\nabla_{\bf x}\phi\cdot\nabla_{\bf v}\mu&=0,\qquad\Delta_{\bf x}\phi=\varrho=\int_{\mathbb{R}^3_{\bf v}}\mu^2d{\bf v},\\
\frac{d\mathcal{X}}{dt}=\mathcal{V},\qquad\frac{d\mathcal{V}}{dt}=\mathcal{Q}\nabla_{\bf x}\phi(\mathcal{X},t),
\end{split}
\end{equation}
for positive constants $q=q_cq_g/(2\pi \epsilon_0 m_g)$, $Q=q_g^2/(\epsilon_0m_g)$, $\mathcal{Q}=q_gq_c/(\epsilon_0M_c)$. This system couples a singular, nonlinear transport (Vlasov) equation for the continuous charge distribution $\mu^2$ to an equation for the trajectory of the point $(\mathcal{X},\mathcal{V})$ mass via their electrostatic Coulomb interaction through a Poisson equation.

\begin{remark}
\begin{enumerate}
 \item The physically relevant setting for these equations relates to electron dynamics in a plasma, when magnetic effects are neglected. In this context, our sign conventions correspond to the non-negative distribution function of a negatively charged gas. In this spirit, we will denote the electric field of the gas by $\mathcal{E}=\nabla_{\bf x}\phi$, a slightly unconventional choice that allows to save some minus signs in the formulas.
 \item The crucial qualitative feature of the forces in \eqref{VPPC} is the \emph{repulsive} nature of interactions between the gas and the point charge, i.e.\ the fact that $q>0$. Our analysis can also accommodate the setting where the gas-gas interactions are attractive. This corresponds to replacing $Q>0$ by $-Q<0$ in \eqref{VPPC}, so that (up to minor algebraic modifications) these two cases can be treated the same way. We shall henceforth focus on \eqref{VPPC} with $Q>0$, as above. (We refer to the discussion of future perspectives below in Section \ref{ssec:futpersp} for some comments regarding the attractive case.)
\end{enumerate}

\end{remark}

\subsection{Main result}
Our main result concerns \eqref{VPPC} with sufficiently small and localized initial charge distributions $\mu_0$. We establish the existence and uniqueness of global, strong solutions and we describe their asymptotic behavior as a modified scattering dynamic. While our full result can be most adequately stated in more adapted ``action-angle'' variables (see Theorem \ref{thm:global_asymptIntro} below on page \pageref{thm:global_asymptIntro}), for the sake of readability we begin here by giving a (weaker, slightly informal) version in standard Cartesian coordinates:
\begin{theorem}\label{thm:main_rough}
Given any $(\mathcal{X}_0,\mathcal{V}_0)\in\mathbb{R}^3_{\bf x}\times\mathbb{R}^3_{\bf v}$ and any initial data $\mu_0\in C^1_c((\mathbb{R}^3_{\bf x}\setminus\{\mathcal{X}_0\})\times\mathbb{R}^3_{\bf v})$, there exists $\varepsilon^\ast>0$ such that for any $0<\varepsilon<\varepsilon^\ast$, there exists a unique global strong solution of \eqref{VPPC} with initial data
\begin{equation*}
(\mathcal{X}(t=0),\mathcal{V}(t=0))=(\mathcal{X}_0,\mathcal{V}_0),\qquad \mu({\bf x},{\bf v},t=0)=\varepsilon\mu_0({\bf  x},{\bf v}).
\end{equation*}
Moreover, the electric field decays pointwise at optimal rate, and there exists a modified point charge trajectory, an asymptotic profile $\mu_\infty\in L^2((\mathbb{R}^3\setminus\{0\})\times\mathbb{R}^3)$ and a Lagrangian map $({\bf Y},{\bf W}):\mathbb{R}^3\times\mathbb{R}^3\times \mathbb{R}_+^\ast\to \R^3\times\R^3$ along which the particle distribution converges pointwise
\begin{equation*}
\begin{split}
\mu({\bf Y}({\bf x},{\bf v},t),{\bf W}({\bf x},{\bf v},t), t)\to \mu_\infty({\bf x},{\bf v}),\qquad t\to\infty.
\end{split}
\end{equation*}

\end{theorem}

\begin{remark}\label{rem:asymptotics}
\begin{enumerate}
\item Our main theorem is in fact much more precise and requires fewer assumptions, but is better stated in adapted ``action angle'' variables. We refer to Theorem \ref{thm:global_asymptIntro}. In particular, we allow initial particle distributions with positive measure in any ball around the charge $(\mathcal{X}_0,\mathcal{V}_0)$ and with noncompact support $\hbox{supp}(\mu_0)=\mathbb{R}^3_{\bf x}\times\mathbb{R}^3_{\bf v}$.

\item Under the weaker assumption that $\mu_0\in C^0_c$, we still obtain in Proposition \ref{prop:global_derivsIntro} a global solution with almost optimal decay of the electric field.

\item The charge trajectory and the Lagrangian map can be expressed in terms of an asymptotic ``electric field profile'' $\mathcal{E}^\infty$ and asymptotic charge velocity $\mathcal{V}_\infty$ and position shift $\mathcal{X}_\infty$: As $t\to+\infty$, we have that
\begin{equation}\label{eq:asymptotics}
\begin{split}
\mathcal{X}(t)&=\mathcal{X}_\infty+t\mathcal{V}_\infty-\mathcal{Q}\ln(t)\mathcal{E}^\infty(0)+O(t^{-1/10}),\\
\mathcal{V}(t)&=\mathcal{V}_\infty-\frac{\mathcal{Q}}{t}\mathcal{E}^\infty(0)+O(t^{-11/10}),\\
{\bf Y}({\bf x},{\bf v},t)&=\left(at-\frac{1}{2}\frac{q}{a^2}\ln(ta^3/q)+\ln(t)[Q\mathcal{E}^\infty(\a)+\mathcal{Q}\mathcal{E}^\infty(0)]\right)\cdot \frac{q^2}{4a^2L^2+q^2}\left(\frac{2}{q}{\bf R}+\frac{4a}{q^2}{\bf L}\times{\bf R}\right)\\
 &\qquad+
\mathcal{V}_\infty t+\mathcal{Q}\mathcal{E}^\infty(0)\ln(t)+O(1),\\
{\bf W}({\bf x},{\bf v},t)&=a\left(1-\frac{q}{2ta^3}\right)\cdot \frac{q^2}{4a^2L^2+q^2}\left(\frac{2}{q}{\bf R}+\frac{4a}{q^2}{\bf L}\times{\bf R}\right)+\mathcal{V}_\infty+O(\ln(t)t^{-2}),
\end{split}
\end{equation}
where we used the following abbreviations to allow for more compact formulas
\begin{equation}
 a^2=\aabs{\v}^2+\frac{q}{\aabs{\x}},\quad \L=\x\times\v,\quad L=\aabs{\L},\quad {\bf R}=\v\times\L+\frac{q}{2}\frac{\x}{\aabs{\x}},
\end{equation}
and $\a$ is defined in \eqref{eq:explicitATheta} below (these quantities are conservation laws for the linearized problem associated to \eqref{VPPC}). 

In the dynamics of the point charge, the term $\mathcal{Q}\ln(t)\mathcal{E}^\infty(0)$ (resp.\ $\frac{\mathcal{Q}}{t}\mathcal{E}^\infty(0)$) corresponds to a nonlinear modification of a free trajectory with velocity $\mathcal{V}_\infty$, and is also reflected in ${\bf Y}$. In addition, the first term in the expansion of ${\bf Y}$ (involving the factor $at$) derives from conservation of the energy along trajectories, the second term (involving a first logarithmic correction $\ln(ta^3/q)$) is a feature of the \emph{linear} trajectories. The term $t\mathcal{V}_\infty$ reflects a centering around the position of the point charge, and the remaining logarithmic terms are nonlinear corrections to the position. This can be compared with the asymptotic behavior close to vacuum in \cite{IPWW2020,Pan2020} by setting $q=Q=\mathcal{Q}=0$ and ignoring the motion of the point charge.

\end{enumerate}
\end{remark}

\subsubsection{Prior work}
In the absence of a point charge, the Vlasov-Poisson system has been extensively studied and the corresponding literature is too vast to be surveyed here appropriately. We focus instead on the case of three spatial and three velocity dimensions, which is of particular physical relevance. Here we refer to classical works \cite{BD1985,GI2020,LP1991,Pfa1992,Sch1991} for references on global wellposedness and dispersion analysis, to \cite{CK2016,FOPW2021,IPWW2020,Pan2020} for more recent results describing the asymptotic behavior, to \cite{Gla1996,Rei2007} for book references and  to \cite{BM2018} for a historical review.

The presence of a point charge introduces singular force fields and significantly complicates the analysis. Nevertheless, when the gas-point charge interaction is \emph{repulsive}, global existence and uniqueness of strong solutions when the support of the density is separated from the point charge has been established in \cite{MMP2011}, see also \cite{CM2010} and references therein. Global existence of weak solutions for more general support was then proved in \cite{DMS2015} with subsequent improvements in \cite{LZ2017,LZ2018,Mio2016}, and a construction of ``Lagrangian solutions'' in \cite{CLS2018}. For attractive interactions, strong well-posedness remains open, even locally in time, but global weak solutions have been constructed \cite{CMMP2012,CZW2015}. Concentration, creation of a point charge and subsequent lack of uniqueness were studied in a related system for ions in $1d$, see \cite{MMZ1994,ZM1994}. To the best of our knowledge, the only work concerning the asymptotic behavior of such solutions is the recent \cite{PW2020}, which studies the repulsive, radial case using a precursor to the asymptotic action method we develop here.

The existence and stability of other (smooth) equilibriums has been considered for the Vlasov-Poisson system, most notably in connection to Landau damping near a homogeneous background in the confined or screened case \cite{AW2021,BMM2018,FR2016,HNR2019,MV2011}, with recent progress also in the unconfined setting \cite{HNR2020,BMM2020,IPWW2022}. In the case of attractive interactions or in the presence of several species, there are many more equilibriums and a good final state conjecture seems beyond the scope of the current theory. However, there have been many outstanding works on the \emph{linear} and \emph{orbital} (in-)stability of nontrivial equilibria \cite{GL2017,GS1995,LMR2008,LMR2012,Mou2013,Pen1960}. We further highlight \cite{FHR2021,GL2017,HRS2021} which use action-angle coordinates to solve efficiently an elliptic equation in order to understand the spectrum of the linearized operator.

Finally, we note that the recent work \cite{HW2022} studies the interaction of a fast point charge with a {\it homogeneous} background satisfying a Penrose condition, for a variant of \eqref{VPPC} with a screened potential (see also the related \cite{AW2021} on Debye screening). We also refer to \cite{IJ2019} which addresses the stability of a Dirac mass in the context of the $2d$ Euler equation.

\subsection{The method of asymptotic action}
We describe now our approach to the study of asymptotic dynamics in \eqref{VPPC}, which is guided by their Hamiltonian structure.\footnote{We refer the reader to the recent \cite{MNP2022} for a derivation of this Hamiltonian structure from the underlying classical many-body problem.}

\medskip
\paragraph{\emph{Brief overview}}
We first study the linearized problem, i.e.\ the setting without nonlinear self-interactions of the gas (i.e.\ we ignore the contributions of $\phi$). There the point charge moves freely along a straight line, while the gas distribution is still subject to the electrostatic field generated by the point charge and thus solves a singular transport equation which can be integrated explicitly through a canonical change of coordinates to suitable action-angle variables. Upon appropriate choice of unknown $\gamma$ in these variables, we can thus reduce to the study of a purely nonlinear equation, given in terms of the electrostatic potential $\phi$. This and the derived electric field can be conveniently expressed (thanks to the canonical nature of the change of coordinates) as integrals of $\gamma$ over phase space, and we study their boundedness properties. In particular, assuming moments and derivatives on $\gamma$, we establish that electrical functions decay pointwise. With this, we show how to propagate such moments and derivatives, relying heavily on the Poisson bracket structure. Finally, this reveals the asymptotic behavior through an asymptotic shear equation that builds on a phase mixing property of asymptotic actions.

\medskip
Next we present our method in more detail. It is instructive to first consider the case where \emph{the point charge is stationary}, i.e.\ that $(\mathcal{X}(t),\mathcal{V}(t))\equiv (0,0)$ in a suitable coordinate frame. This happens naturally e.g.\ if $(\mathcal{X}(0),\mathcal{V}(0))=(0,0)$ and the initial distribution is symmetric with respect to three coordinate planes, which is already a nontrivial case. In practice, \eqref{VPPC} then reduces to an equation for the gas distribution alone, which (starting from its Liouville equation reformulation) can be recast in Hamiltonian form as 
\begin{equation}\label{eq:VPPC-H}
\begin{split}
\partial_t\mu-\{\mathbb{H},\mu\}=0,\qquad
\mathbb{ H}=\frac{1}{2}\mathbb{H}_2-\mathbb{H}_4\qquad
\mathbb{H}_2({\bf x},{\bf v}):=\vert {\bf v}\vert^2+\frac{q}{\vert {\bf x}\vert},\qquad\mathbb{H}_4({\bf x},{\bf v},t):=Q\psi({\bf x},t),
\end{split}
\end{equation}
where the Poisson bracket and phase space $\mathcal{P}_{\x,\v}$ are given by
\begin{equation}\label{PB}
\{f,g\}=\nabla_{\bf x}f\cdot\nabla_{\bf v}g-\nabla_{\bf v}f\cdot\nabla_{\bf x}g,\qquad \mathcal{P}_{{\bf x},{\bf v}}:=\{({\bf x},{\bf v})\in\mathbb{R}^3\times\mathbb{R}^3:\,\, \vert{\bf x}\vert>0\}.
\end{equation}
This simplified setting facilitates the presentation of the main aspects of the quantitative analysis of the gas distribution dynamics. We will subsequently explain the (numerous) modifications needed to incorporate the point charge motion in Section \ref{SecAddingPC}.

\subsubsection{Linearized equation and asymptotic actions}
We start by considering the linearization of \eqref{eq:VPPC-H},
\begin{equation}\label{LinearHamiltonianFlow}
\begin{split}
2\partial_t\mu-\{\mathbb{H}_2,\mu\}=0,\qquad \mathbb{H}_2({\bf x},{\bf v}):=\vert {\bf v}\vert^2+\frac{q}{\vert {\bf x}\vert}.
\end{split}
\end{equation}
This is nothing but transport along the characteristics of the well-known {\it repulsive two-body} system,
\begin{equation}\label{KeplerODE}
\dot{\bf x}={\bf v},\qquad \dot{\bf v}=\frac{q}{2}\frac{{\bf x}}{\vert{\bf x}\vert^3},
\end{equation}
which is, in particular, \emph{completely integrable}. Due to the repulsive nature of \eqref{KeplerODE}, its trajectories are hyperbolas (and thus open), with well-defined asymptotic velocities.

Our first main result is the construction of \emph{asymptotic action-angle variables} which provide adapted, canonical coordinates for the phase space. We denote the phase space in these angles $\vartheta$ and actions $\a$ by
\begin{equation}
 \mathcal{P}_{\vartheta,{\bf a}}:=\{(\vartheta,{\bf a})\in\mathbb{R}^3\times\mathbb{R}^3:\,\,\vert{\bf a}\vert>0\},
\end{equation}
and have:
\begin{proposition}\label{PropAA}
There exists a smooth diffeomorphism $\mathcal{T}:\mathcal{P}_{\x,\v}\to\mathcal{P}_{\vartheta,\a}$, $({\bf x},{\bf v})\mapsto(\Theta(\x,\v),\mathcal{A}(\x,\v))$ with inverse $(\vartheta,\a)\mapsto (\X(\vartheta,\a),\V(\vartheta,\a)$, which
\begin{enumerate}[wide]
 \item\label{it:canonical} is canonical, i.e.\
 \begin{equation}\label{eq:canonical}
  d\x\wedge d\v=d\Theta\wedge d\mathcal{A},
 \end{equation}

 \item\label{it:cons-comp} is compatible with conservation of energy and angular momentum:
 \begin{equation}
 H({\bf x},{\bf v})=\abs{\mathcal{A}}^2,\qquad{\bf x}\times{\bf v}=\Theta\times\mathcal{A},
 \end{equation}

\item\label{it:lin-diag} linearizes the flow of \eqref{KeplerODE} in the sense that $({\bf x}(t),{\bf v}(t))$ solves \eqref{KeplerODE} if and only if
\begin{equation}\label{LinearizationFlow}
\Theta({\bf x}(t),{\bf v}(t))=\Theta({\bf x}(0),{\bf v}(0))+t\mathcal{A}({\bf x}(0),{\bf v}(0)),\qquad\mathcal{A}({\bf x}(t),{\bf v}(t))=\mathcal{A}({\bf x}(0),{\bf v}(0)),
\end{equation}

\item\label{it:asymp-act} satisfies the ``asymptotic action property''
\begin{equation}\label{AsymptoticActions}
\begin{split}
\vert{\bf X}(\vartheta+t{\bf a},{\bf a})-t{\bf a}\vert=o(t),\qquad \vert{\bf V}(\vartheta+t{\bf a},{\bf a})-{\bf a}\vert=o(1)\quad\hbox{ as }t\to+\infty.
\end{split}
\end{equation}
\end{enumerate}
\end{proposition}

The \emph{asymptotic action-angle} property \eqref{it:asymp-act} will be crucial to the asymptotic analysis. In short, it ensures that {\bf a} parameterizes the trajectories that stay at a bounded distance from each other\footnote{This is a similar idea to the Gromov Boundary, see also \cite{MV2020} where this is developed for the Jacobi-Maupertuis metric.} as $t\to\infty$ and connects in an effective way the trajectories of \eqref{KeplerODE} to those of the free streaming $\ddot{\bf x}=0$. Put differently, different trajectories of \eqref{KeplerODE} asymptotically diverge linearly with time and their difference in $\a$, a property sometimes referred to as ``shearing'' or ``phase mixing''.

General dynamical systems are not, of course, completely integrable, and when they are, there are many different action-angle coordinates. Here the \emph{asymptotic-action property} fixes the actions and helps restrict the set of choices. Besides, since the actions are defined in a natural way (as asymptotic velocities, see \eqref{AsymptoticActions}), one can aim to find $\mathcal{T}$ through a \emph{generating function} $\mathcal{S}({\bf x},{\bf a})$ by solving a

\smallskip
\noindent
\begin{tabularx}{\textwidth}{lX}
 \emph{Scattering problem:} & Given an asymptotic velocity $\v_\infty=:\a\in\R^3$ and a location ${\bf x}_0\in\R^3$, find (if they exist) the trajectories $(\x(t),\v(t))$ through ${\bf x}_0$ with asymptotic velocity $\a$.
\end{tabularx}
\smallskip
\begin{figure}[h]
 \centering
 \includegraphics[width=0.5\textwidth]{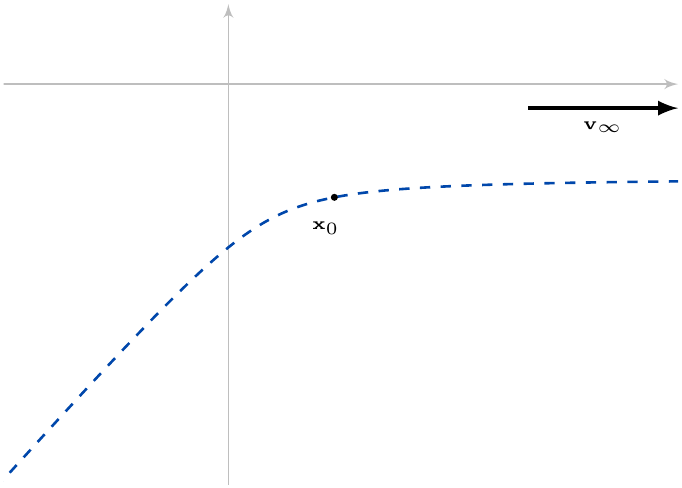}
 \caption{The scattering problem: Given a point $\x_0$ and asymptotic velocity $\v_\infty$, how to determine trajectories (one possibility dashed in blue) through $\x_0$ with asymptotic velocity $\v_\infty$?}
 \label{fig:scatter}
\end{figure}

Once such a trajectory has been found, one can define $\mathfrak{V}({\bf x},{\bf a})$ as the velocity along the trajectory at ${\bf x}$, and look for a putative $\mathcal{S}$ such that $\mathfrak{V}({\bf x},{\bf a})=\nabla_{\bf x}\mathcal{S}$. By classical arguments (see e.g.\ \cite[Chapter 8]{Meyer2017}), setting $\vartheta:=\nabla_\v\mathcal{S}$ then yields a (local) canonical change of variables.

We note that there are two related difficulties with this approach in the present context, namely $(i)$ that given ${\bf a}$, there are points ${\bf x}_0$ through which no trajectory as above passes, and $(ii)$ when the scattering problem can be solved, there are in general \textit{two different trajectories} (and thus two different ``velocity'' maps $\mathfrak{V}_\pm(\x,\a)$) through a given point ${\bf x}_0$. In fact, the set of trajectories with a given asymptotic velocity has a fold\footnote{One can think of the set of trajectories associated to a given ${\bf a}$ and angular momentum direction ${\bf L}/L$ as a (planar) sheet of paper $\mathbb{R}^2$, flatly folded over a curve $\Gamma$ (the fold), so that away from $\Gamma$ every point ${\bf x}$ corresponds to either $0$ or $2$ trajectories, depending on the side of the fold.}. Once we identify the correct projection (in phase space) of this fold, we are able to define a smooth gluing of the functions $\mathfrak{V}_\pm$ to obtain a globally smooth choice of generating function $\mathcal{S}$.

\begin{remark}
\begin{enumerate}
 \item In the present setting, we find a generating function by calculating trajectories. It is interesting to note that a reverse approach, solving Hamilton-Jacobi equations to obtain trajectories through a point with prescribed asymptotic velocity, has been used to construct families of asymptotically diverging trajectories in the more general $N$-body problem \cite{MV2020}.
 \item We note that a common way to obtain a generating function is by solving a Hamilton Jacobi equation $H({\bf x},\nabla_{\bf x}\mathcal{S}({\bf x},\a))=const$. Here we recover families of solutions, and observe that these solutions develop a singularity in finite time, so that the full generating function is obtain by gluing two such solutions along each trajectory.
\end{enumerate}

\end{remark}

The action angle property \eqref{LinearizationFlow} allows to conjugate the linearized equation \eqref{LinearHamiltonianFlow} to the free streaming for the image particle distribution
\begin{equation*}
\left(\partial_t+{\bf a}\cdot\nabla_{\vartheta}\right)\nu=0,\qquad\nu(\Theta({\bf x},{\bf v}),\mathcal{A}({\bf x},{\bf v}),t)=\mu({\bf x},{\bf v},t),
\end{equation*}
which can then be easily integrated to give $\partial_t(\nu\circ\Phi_t^{-1})=0$, where
\begin{equation}\label{DefPhiFreeStreaming}
\Phi_t({\bf x},{\bf v})=({\bf x}-t{\bf v},{\bf v}).
\end{equation}

\subsubsection{Choice of nonlinear unknown}
We next integrate explicitly the linear flow as above, and introduce the nonlinear unknown $\gamma:=\mu\circ\mathcal{T}\circ\Phi_t^{-1}$,
\begin{equation}\label{NewNLUnknown_Intro}
\begin{split}
\gamma(\vartheta,{\bf a},t)&=\mu({\bf X}(\vartheta+t{\bf a},{\bf a}),{\bf V}(\vartheta+t{\bf a},{\bf a}),t),\\
\mu({\bf x},{\bf v},t)&=\gamma(\Theta({\bf x},{\bf v})-t\mathcal{A}({\bf x},{\bf v}),\mathcal{A}({\bf x},{\bf v}),t),
\end{split}
\end{equation}
which satisfies a \emph{purely nonlinear} equation
\begin{equation}\label{NLVP}
\begin{split}
\partial_t\gamma+\{\mathbb{H}_4,\gamma\}=0,\qquad\mathbb{H}_4:=Q\psi(\widetilde{\bf X},t),\qquad\widetilde{\bf X}(\vartheta,{\bf a})={\bf X}(\vartheta+t{\bf a},{\bf a}).
\end{split}
\end{equation}
Equation \eqref{NLVP} involves the electric potential $\psi$ through the electric field $\mathcal{E}=\nabla\psi$, both of which can be expressed as integrals \emph{over phase space}, both in terms of $\mu$ and -- since $\mathcal{T}$ and $\Phi_t$ are canonical -- also conveniently in $\gamma$ as
\begin{equation}\label{PsiE}
\begin{split}
\psi({\bf y},t)&:=-\frac{1}{4\pi}\iint \frac{1}{\vert {\bf y}- \x\vert}\mu^2(\x,\v,t)\, d\x d\v=-\frac{1}{4\pi}\iint \frac{1}{\vert {\bf y}- {\bf X}(\vartheta+t\a,\a)\vert}\gamma^2(\vartheta,\a,t)\, d\vartheta d\a,\\
 \mathcal{E}_j({\bf y},t)&:=\frac{1}{4\pi}\iint \frac{{\bf y}^j-\x^j}{\vert {\bf y}- \x\vert^3}\mu^2(\x,\v,t)\, d\x d\v=\frac{1}{4\pi}\iint \frac{{\bf y}^j-{\bf X}^j(\vartheta+t\a,\a)}{\vert {\bf y}- {\bf X}(\vartheta+t\a,\a)\vert^3}\gamma^2(\vartheta,\a,t)\, d\vartheta d\a.
\end{split}
\end{equation}

\subsubsection{Analysis of the (effective) electric field and weak convergence}
The proper analysis of the electric field $\mathcal{E}$ in terms of $\gamma$ requires precise kinematic bounds on the (inverse of the) asymptotic action-angle map $({\bf X},{\bf V})$ and its derivatives. 

Using moment bounds on $\gamma$ alone, one can reduce the question of its pointwise decay to control of an \emph{effective} electric field, which captures the leading order dynamics. This in turn can be bounded in terms of moments on $\gamma$, at the cost of some logarithmic losses, which yields almost optimal decay of the electric field. More precisely, we note that as per \eqref{PsiE}, the nonlinear evolution is governed by various integrals of the measure $\gamma^2d\vartheta d\a$ on phase space, and we thus aim to prove its weak convergence. Using a variant of the continuity equation (an argument somewhat related to \cite{LP1991}), we can obtain {\it vague, scale-localized} convergence of this measure, i.e.\
\begin{equation}
\langle\varphi\rangle_R({\bf y},t):=\iint\varphi(R^{-1}({\bf y}-\widetilde{\bf X}(\vartheta,\a)))\gamma^2(\vartheta,\a,t)\, d\vartheta d\a\to\langle\varphi\rangle^\infty_R({\bf y}),\quad t\to\infty,
\end{equation}
uniformly in ${\bf y}, R$. For particle distributions $\gamma$ solving \eqref{NLVP}, this allows to obtain uniform control on the (scale-localized) effective electric field, and, after resummation of the scales, almost optimal control on the effective electric field $\mathcal{E}^{eff}$.

To obtain optimal decay bounds (and precise asymptotics) we need to control extra regularity on $\gamma$. This turns out to be significantly more involved than moment control, as described below.

\subsubsection{Moment propagation and almost optimal decay}

In order to propagate moments, we use the Poisson bracket structure \eqref{NLVP}, and the fact that for any weight function $\ww:\mathcal{P}_{{\bf x},{\bf v}}\to\mathbb{R}$ or $\ww:\mathcal{P}_{\vartheta,{\bf a}}\to\mathbb{R}$ there holds that
\begin{equation}\label{eq:mom-prop_Intro}
\begin{split}
\partial_t(\ww\gamma)+\{\mathbb{H}_4,(\ww\gamma)\}=\{\mathbb{H}_4,\ww\}\gamma=Q\mathcal{E}_j\{\widetilde{\bf X}^j,\ww\}\gamma,
\end{split}
\end{equation}
together with bounds on the electric field and some classical identities such as
\begin{equation}\label{ClassicalPB}
\begin{split}
\Big\{{\bf x},\frac{\vert{\bf v}\vert^2}{2}+\frac{q}{\vert{\bf x}\vert}\Big\}={\bf v},\qquad \{{\bf x},\vert {\bf L}\vert^2\}={\bf L}\times{\bf x},\quad{\bf L}={\bf x}\times{\bf v}.
\end{split}
\end{equation}
Choosing as weight functions the conserved quantities for the linear equation $a:=\sqrt{\mathbb{H}_2}$ resp.\ $\xi:=qa^{-1}$ and $\lambda:=\vert{\bf L}\vert$, as well as the dynamically evolving quantity $\eta:=(a/q)\vartheta\cdot{\bf a}$, this enables a bootstrap argument that leads to almost optimal moment bounds and electric field decay, assuming only initial moments control.

Our first global result for the dynamics then reads:
\begin{theorem}[see Theorem \ref{thm:global_moments}]\label{thm:global_momentsIntro}
Let $m\ge 30$, and assume that the initial density $\mu_0=\gamma_0$ satisfies
\begin{equation}\label{eq:mom_id_Intro}
\begin{split}
\Vert \langle a\rangle^{2m}\mu_0\Vert_{L^r}+\Vert \langle \xi\rangle^{2m}\mu_0\Vert_{L^r}+\Vert \langle \lambda\rangle^{2m}\mu_0\Vert_{L^r}+\Vert \langle\eta\rangle^m\mu_0\Vert_{L^\infty}&\le \varepsilon_0,\qquad r\in\{2,\infty\}.
\end{split}
\end{equation}
Then there exists a global solution $\gamma$ to \eqref{NLVP} that satisfies the bounds for $r\in\{2,\infty\}$
\begin{equation}\label{eq:gl_mom_bds_Intro}
\begin{split}
\Vert \langle a\rangle^{2m}\gamma(t)\Vert_{L^r}+\Vert \langle \xi\rangle^{2m}\gamma(t)\Vert_{L^r}&\le 2\varepsilon_0,\\
\Vert \langle \lambda\rangle^{2m}\gamma(t)\Vert_{L^r}&\le 2\varepsilon_0 (\ln(2+t))^{2m},\\
\Vert \langle\eta\rangle^m\gamma(t)\Vert_{L^\infty}&\le 2\varepsilon_0 (\ln(2+t))^{2m},
\end{split}
\end{equation}
and the associated electric field decays as
\begin{equation*}
\begin{split}
\Vert \mathcal{E}(t)\Vert_{L^\infty_x}\lesssim \varepsilon_0^2\ln(2+t)/\langle t\rangle^2.
\end{split}
\end{equation*}
\end{theorem}

\subsubsection{Derivative propagation}

In order to obtain bona fide classical solutions, we need to propagate bounds on the gradient of $\gamma$. This requires considerable care, notably because the kinematic formulas are rather involved and some derivatives produce large factors of $t$. To minimize the presence of ``bad derivatives'', we make use of the fact that the two-body problem is \emph{super-integrable}: this allows us to express all kinematically relevant quantities in terms of a set of coordinates $\textnormal{SIC}$, of which all but \emph{one, scalar} variable are constant under the flow of the two-body problem \eqref{KeplerODE} (and thus constant along the characteristics of the linearized problem \eqref{LinearHamiltonianFlow}). A natural such choice is the reduced basis $(\xi,\eta,{\bf u},{\bf L})$, where $\u=a^{-1}\a,\L\in\R^3$, $\xi,\eta\in\R$ (see also \eqref{SIC_Intro} below), and only $\eta$ evolves in the linear problem. This collection has a built-in redundancy, as ${\bf u}\cdot{\bf L}=0$. In order to work with such an \emph{overdetermined} set of coordinates, we take advantage of the symplectic structure to propagate Poisson brackets with respect to a \emph{spanning family} $\textnormal{SIC}$, i.e.\ a collection $(\{f,\cdot\})_{f\in \textnormal{SIC}}$ which spans the cotangent space. Letting $\textnormal{SIC}:=\{\xi,\eta,\u,\L\}$, we then obtain a system for such derivatives of $\gamma$ that reads
\begin{equation}\label{eq:derivIntro}
\begin{split}
\partial_t\{f,\gamma\}+\{\mathbb{H}_4,\{f,\gamma\}\}&=-\{\{f,\mathbb{H}_4\},\gamma\}=-Q\mathcal{F}_{jk}\{\widetilde{\bf X}^j,\gamma\}\{f,\widetilde{\bf X}^k\}-Q\mathcal{E}_j\{\{f,\widetilde{\bf X}^j\},\gamma\},
\end{split}
\end{equation}
where $\mathcal{F}=\nabla^2\psi$ denotes the Hessian of the electric potential. Since $\textnormal{SIC}$ is a spanning set, we can then resolve $\{\widetilde{\bf X},\gamma\}$ in terms of $(\{f,\gamma\})_{f\in \textnormal{SIC}}$, leading to a self-consistent system for bounds on the Poisson brackets:
\begin{equation*}
\begin{split}
\partial_t\vert\{f,\gamma\}\vert\lesssim \sum_{g\in \textnormal{SIC}}\mathfrak{m}_{fg}\vert\{g,\gamma\}\vert.
\end{split}
\end{equation*}
Here the coefficients $\mathfrak{m}_{fg}$ have formally enough decay, but are {\it ill-conditioned} in the sense that they do not admit bounds uniformly in the coordinates outside of a compact set. To remedy this, we introduce a set of weights $\ww_{f}$ and manage to propagate appropriate bounds on $\ww_{f}\vert\{f,\gamma\}\vert$. To account for the fact that the past ($t\to-\infty$) and future ($t\to\infty$) asymptotic velocities of a linear trajectory may differ drastically in direction, we will need to work with two different sets of spanning functions in different parts of phase space: ``past'' asymptotic action-angles in an ``incoming'' region and ``future'' asymptotic action-angles in an ``outgoing'' region. For simplicity, we thus prefer to work with pointwise bounds on the above symplectic gradients.

Altogether, a slightly simplified version of our result concerning global propagation of derivatives is the following:
\begin{proposition}[Informal Version of Proposition \ref{prop:global_derivs}]\label{prop:global_derivsIntro}
 Let $\gamma$ be a global solution of \eqref{NLVP} as in Theorem \ref{thm:global_momentsIntro}, and assume that for the selection of weights $\ww_f$ as below in \eqref{eq:def_weights} there holds that
 \begin{equation}\label{eq:deriv_init_Intro}
  \sum_{f\in \textnormal{SIC}}\norm{(\xi^{5}+\xi^{-8})\ww_f\{f,\gamma_0\}}_{L^\infty}\lesssim \eps_0.
 \end{equation}
Then we have that
 \begin{equation}\label{eq:deriv_concl_Intro}
  \sum_{f\in \textnormal{SIC}}\norm{\ww_f\{f,\gamma(t)\}}_{L^\infty}\lesssim \eps_0\ln^2(2+t).
 \end{equation}
 Moreover, the electric field $\mathcal{E}$ decays at the optimal rate
 \begin{equation}
  \norm{\mathcal{E}(t)}_{L^\infty}\lesssim \eps_0^2\ip{t}^{-2}.
 \end{equation}
\end{proposition}
We comment on two points: $(i)$ Control of one derivative of $\gamma$ allows to resum the scale-localized effective electric unknowns, which implies optimal decay of the electric field (see Proposition \ref{PropEeff} below) and lays the foundation for a quantitative understanding of the asymptotic behavior. $(ii)$ There is a ``loss'' of weights in $\xi$ when we change coordinates between past and future asymptotic actions, which is reflected by the extra factors in $\xi$ we require on the initial data in \eqref{eq:deriv_init_Intro} as compared to the propagated derivatives in \eqref{eq:deriv_concl_Intro}.

\subsubsection{Asymptotic behavior}
As mentioned above, derivative control is tied to obtaining optimal decay for the electric field, through bounds on the \emph{effective} electric field. Roughly speaking, once we control a derivative of the particle density $\gamma$ we can sum the uniform bounds obtained on the scale-localized effective field and obtain convergence of the effective electrical functions, which allows us to deduce that
\begin{equation}\label{eq:effE_Intro}
\psi(\widetilde{\bf X},t)=\frac{1}{t}\Psi^\infty({\bf a})+O(t^{-1-}),\qquad \mathcal{E}_j(\widetilde{\bf X},t)=\frac{1}{t^2}\partial_j\Psi^\infty({\bf a})+O(t^{-2-}).
\end{equation}
The additional ingredient we use is that, for large scale, we can perform an integration by parts that replaces a derivative on the Coulomb kernel with a bound on the Poisson bracket with a \emph{constant of motion for the asymptotic flow}\footnote{Recall that a {\it constant of motion} denotes an expression of the form $F({\bf x},{\bf v},t)$ which is conserved along the flow, as opposed to an {\it integral of motion} of the form $F({\bf x},{\bf v})$ which is a function on phase space alone.} $\ddot{\bf x}=0$: for a function $\varphi(\x)$ of $\x$ only, there holds that
\begin{equation*}
\begin{split}
\partial_{\bf x}\varphi=t^{-1}\{{\bf z},\varphi\},\quad{\bf z}:={\bf x}-t{\bf v}.
\end{split}
\end{equation*}
Since ${\bf z}$ is constant along free streaming, we can show that it only increases logarithmically along the flow of \eqref{KeplerODE}.

From \eqref{eq:effE_Intro} we easily obtain the asymptotic behavior: the main evolution equation for the gas distribution function $\gamma$ reads
\begin{equation}
\partial_t\gamma+\frac{Q}{t}\{\Psi^\infty({\bf a}),\gamma\}=l.o.t.
\end{equation}
To leading order, $\gamma$ is thus transported by a \emph{shear} flow, which can be integrated directly to obtain convergence of the particle distribution along \emph{modified characteristics}:
\begin{equation*}
\gamma(\vartheta+\ln(t)Q\nabla\Psi^\infty({\bf a}),{\bf a},t)\to\gamma_\infty(\vartheta,{\bf a}),\qquad t\to+\infty.
\end{equation*}

This leads us to our final result (stated here in a version without point charge dynamics):
\begin{theorem}[informal version of Theorem \ref{MainThm}]\label{thm:global_asymptIntro}
 Let $\gamma$ be a global solution of \eqref{NLVP} as in Proposition \ref{prop:global_derivsIntro}. Then there exists an asymptotic electric field profile $\mathcal{E}^\infty\in L^\infty(\R^3)$ and $\gamma_\infty\in L^\infty_{\vartheta,\a}$ such that
 \begin{equation}
  \gamma(\vartheta+\ln(t)Q\mathcal{E}^\infty(\a),\a,t)\to\gamma_\infty(\vartheta,{\bf a}),\qquad t\to+\infty.
 \end{equation}
\end{theorem}

\subsection{Accounting for the motion of the point charge}\label{SecAddingPC}
In general, the point charge will not remain stationary, as $\ddot{\mathcal{X}}\ne 0$. However, accounting for its motion brings in significant complications. 

Using the macroscopic conservation laws (akin to the introduction of the reduced mass for the classical two-body problem), at the cost of the introduction of a non-Galilean frame we can reduce the analysis to a problem for the gas distribution alone, which is then transported by a self-consistent Hamiltonian flow (see Section \ref{sec:red_gas}). Hereby it is natural to center the phase space for the gas distribution around the point charge, so that the new position coordinates are given simply by the distance from the point charge $\mathcal{X}(t)$. For the velocity we have two natural options: 
\begin{enumerate}[label=(\alph*)]
\item Center around the \emph{asymptotic velocity} $\mathcal{V}_\infty$ of the point charge: This leads to an equation of the form
\begin{equation*}
\begin{split}
\partial_t\mu-\{\mathbb{H},\mu\}=0,\qquad \mathbb{ H}=\frac{1}{2}\mathbb{H}_2-\mathbb{H}_4,\qquad\mathbb{H}_4({\bf x},{\bf v},t)=Q\psi({\bf x},t)+(\mathcal{V}(t)-\mathcal{V}_\infty)\cdot {\bf v},
\end{split}
\end{equation*}
with same linearized Hamiltonian as in \eqref{LinearHamiltonianFlow}, but adjusted perturbative term $\mathbb{H}_4$ that incorporates the asymptotic velocity through $\mathcal{V}_\infty$ of the point charge. We note that this induces a (small) uncertainty on the velocity $\mathcal{W}(t)=\mathcal{V}(t)-\mathcal{V}_\infty$, but no acceleration. This is well adapted to describing the dynamics far away from the point charge ($\aabs{\x}\gg 1$), and we thus refer to it as the \emph{``far formulation''}.

\item Center around the \emph{instantaneous velocity} $\mathcal{V}(t)$ of the point charge: This is related to a Hamiltonian of the form
\begin{equation*}
 \mathbb{ H}'=\frac{1}{2}\mathbb{H}_2-\mathbb{H}'_4,\qquad\mathbb{H}'_4({\bf x},{\bf v},t)=Q\psi({\bf x},t)-\dot{\mathcal{V}}\cdot {\bf x}.
\end{equation*}
This formulation introduces the acceleration $\dot{\mathcal{V}}$, which (although small) does not decay and becomes too influential when paired with $\vert{\bf x}\vert\gg 1$. We thus refer to this as the \emph{``close formulation''}, and will invoke it for comparatively small $\aabs{\x}\lesssim 1$.
\end{enumerate}

In practice, we will thus have to work with both formulations (and their respective sets of asymptotic action-angle variables). In particular, we will split phase space into ``close'' and ``far'' regions $\Omega_t^{cl}$ and $\Omega_t^{far}$, in which we work with the corresponding close and far formulations and propagate separately moment and derivative control, with a transition between them when a given trajectory passes from one region to the other. We emphasize that by construction each type of transition (far to close resp.\ close to far) happens at most once per trajectory. Although each such transition introduces some losses in $\xi$ weights (similarly to the case of past vs.\ future asymptotic actions), this is simple to account for.

In conclusion, analogous results to Theorem \ref{thm:global_momentsIntro} and Proposition \ref{prop:global_derivsIntro} can be established also for any initial condition on the point charge position and velocity. In particular, one sees that its acceleration is given by an electric field that decays like $t^{-2}$ to leading order, resulting in a logarithmically corrected point charge trajectory in our final result Theorem \ref{MainThm} below.

\subsection{Further remarks and perspectives}\label{ssec:prespectives}
We comment on some more points of relevance:
\begin{enumerate}[wide]
 \item \emph{The role of super-integrability:} The Kepler problem is \emph{super integrable} and as such admits $5$ independent conserved quantities. When doing computations, and especially when computing derivatives, it is useful to isolate as much as possible the conserved coordinates, for which one can hope to obtain uniform bounds, from the dynamical quantities which leads to derivatives which are large. Thus we do many computations using {\it super integrable coordinates} which are derived from the asymptotic action-angle as follows:
\begin{equation}\label{SIC_Intro}
\begin{split}
\xi:=\frac{q}{a},\quad\eta:=\frac{a}{q}\vartheta\cdot{\bf a},\qquad{\bf u}:=\frac{{\bf a}}{a},\quad{\bf L}:=\vartheta\times{\bf a},
\end{split}
\end{equation}
and we note that the linear flow is simple in these coordinates: $(\xi(t),\eta(t),{\bf u}(t),{\bf L}(t))=(\xi_0,\eta_0+tq^2\xi^{-3},{\bf u}_0,{\bf L}_0)$ and only $1$ scalar coordinate changes over time (out of $7$).

\item \emph{Types of trajectories:} It is worth distinguishing several types of trajectories in the linearized problem with respect to the above close and far regions of phase space $\Omega_t^{cl}=\{\aabs{\x}\leq 10\ip{t}\}$ and $\Omega_t^{far}=\{\aabs{\x}\geq \ip{t}\}$. For relatively small actions $a$, particles remain far from the point charge and move slowly. In particular, depending on their initial location, they may start in the far or close region, but will end up in the close region. In contrast, for large velocities, trajectories may start far away with high velocity, come close to the point charge and then speed off again, passing from $\Omega_t^{far}$ to $\Omega_t^{cl}$ and back to $\Omega_t^{far}$. Together with the distinction between incoming/outgoing dynamics, this gives four dynamically relevant, distinct regions of phase space (see also Remark \ref{rem:regions}).

\item \emph{Possible simplifications:} We emphasize that as discussed, the analysis simplifies significantly if the charge has no dynamics, i.e.\ if $(\mathcal{X}(t),\mathcal{V}(t))\equiv (0,0)$. In a similar vein, if the initial gas distribution has compact support (and the support of velocities is thus bounded from below and above), it suffices to work with either close or far formulation. This shows a clear benefit and drastic simplification if an assumption of compact support is made.
\end{enumerate}

\paragraph{\emph{Future Perspectives}}\label{ssec:futpersp}
We hope that the methods introduced in this work provide a template for the analysis of a large class of (collisionless) kinetic problems, for which the linearized equation is given as transport by a completely integrable ODE whose trajectories are open.\footnote{In the case of more complicated (e.g.\ non integrable) ODEs, already the linear analysis may be very challenging, and even for systems relatively close to the $2$-body problem one would need to account for Arnold diffusion (see e.g.\ \cite{KL2008}).} 

To be concrete we highlight some examples to this effect and further related open problems, on which we hope the analysis developed here can shed some new light:
\begin{enumerate}[leftmargin=*]
\item The gravitational case of \emph{attractive} interactions between the point charge and gas is of great importance in astrophysics. Despite this, as briefly hinted at above its mathematical investigation is in its infancy: strong solutions are not even known to exist locally in time, only global weak solutions have been constructed \cite{CMMP2012,CZW2015}. A key mathematical challenge lies in the presence of trapped trajectories, which already arise in the linearized system and drastically hinder the stabilization effect of dispersion.
In this context, we expect Proposition \ref{PropAA} to extend in the region of positive energy, while the transition to zero and negative energies (with parabolic or elliptic orbits) introduces significant new challenges. This work should also inform on the modifications needed to account for the geometry of such trajectories.

\item Along similar lines, the case of several species would also be relevant in plasma physics and may pose related challenges. On the other hand, even for repulsive interactions it would be interesting to consider a perturbation of several point charges, i.e.\ a solution to the $N$-body problem to which a small, smooth gas distribution is added. The natural starting point here is the case of two charges surrounded by a gas, which already at the linear level brings the (restricted) $3$-body problem into play. We refer to \cite{CdPD2010,KMR2009} for works in this direction for Vlasov-Poisson and related equations.

\item We believe that the Vlasov-Poisson evolution of measures which are not absolutely continuous with respect to Liouville is an interesting general problem which merits further investigation. Here and in \cite{CM2010,CMMP2012,CLS2018,DMS2015,LZ2017,LZ2018,MMP2011,PW2020}, the case of a sum of pure point and smooth density is considered, but it would be interesting to have examples where the support of the measure has (say) intermediate Hausdorff dimensions (as suggested e.g.\ by some models of star formation, see e.g.\ \cite[Sec.\ 9.6.2]{MvdBW2010}).
\end{enumerate}

\subsection{Organization of the paper}
We conclude this introduction by showing in Section \ref{sec:red_gas} how to reduce the equations \eqref{VPPC} to a problem on the gas distribution alone. Section \ref{SecLin} then studies the dynamics in the linearized problem, starting with more explanations on the method of asymptotic action (Section \ref{GenComAA}). Following this we discuss the Kepler problem (Section \ref{ssec:Kepler}) and solve the related scalar scattering problem (Section \ref{ssec:planar}), which then allows us to introduce the asymptotic action-angle variables (Section \ref{ssec:AngleAction}). Building on this, further coordinates are introduced in Section \ref{ssec:furthercoords}.

Section \ref{ssec:kinematics} establishes quantitative bounds on some of the kinematically relevant quantities, including in particular their Poisson brackets with various coordinates. 

The electric functions are studied in Section \ref{sec:efield}. We first show that they are well approximated by simpler effective functions (Proposition \ref{PropControlEF}), and obtain convergence on effective fields (Proposition \ref{PropEeff}), building on the moment and derivative control available.

Section \ref{SecBootstrap} establishes the main bootstrap arguments for the propagation of moment and derivative control. We introduce the main nonlinear unknowns in Section \ref{Sec:NLUnknowns}, and first close a bootstrap involving only moment bounds in Section \ref{sec:moments_prop}. Building on this, a second (and much more involved) bootstrap then yields control of derivatives in Section \ref{sec:derivs_prop}.

Finally, in Section \ref{sec:main-asympt} we derive the asymptotic behavior of the gas distribution function and prove our main Theorem \ref{MainThm}.

In appendix \ref{sec:appdx_trans}, we collect some auxiliary results.

\medskip
\paragraph{\textbf{Notation}} We will use the notation $A\lesssim B$ to denote the existence of a constant $C>0$ such that $A\leq C B$, when $C>0$ is independent of quantities of relevance, and write $A\lesssim_p B$ to highlight a dependence of $C$ on a parameter $p$. Moreover, to simplify some expressions we shall use the slight modification of the standard Japanese bracket as $\ip{r}:=(4+r^2)^{1/2}$, $r\in\R$, so that in particular $1\lesssim\ln\ip{0}$.

\subsection{Reduction to a problem for the gas distribution alone}\label{sec:red_gas}

The system \eqref{VPPC} can be transformed into a system that better accounts for the linear dynamics and is more easily connected to the case of radial data already investigated in \cite{PW2020}. For this, it will be convenient to recenter the phase space at the point charge.

\subsubsection{Conservation laws}

In a similar way as for the standard $2$-body problem, one can simplify the system somewhat by using the conservation laws. In order to study conserved quantities, it is convenient to observe that, when $\mu$ solves \eqref{VPPC}, for any function $\omega({\bf x},{\bf v},t)$, there holds that
\begin{equation}\label{ConservationLawsComp}
\begin{split}
\frac{d}{dt}\iint\mu^2 \omega d{\bf x}d{\bf v}=\iint\mu^2\left(\partial_t+{\bf v}\cdot\nabla_{\bf x}+\left(\frac{q}{2}\frac{{\bf x}-\mathcal{X}}{\vert {\bf x}-\mathcal{X}\vert^3}+Q\nabla_{\bf x}\phi\right)\cdot\nabla_{\bf v}\right)\omega\,\, d{\bf x}d{\bf v}.
\end{split}
\end{equation}
By testing with various choices of $\omega$, one can obtain conservation laws. We will do this for the first three moments in ${\bf v}$.

The total charge is conserved and using $\omega=1$ in \eqref{ConservationLawsComp} leads to conservation of the $0$-moment:
\begin{equation*}
\begin{split}
\frac{d}{dt}M_g(\mu)=0,\qquad M_g(\mu):=m_g\iint \mu^2\, d{\bf x}d{\bf v}.
\end{split}
\end{equation*}
The total momentum is conserved and using $\omega={\bf v}$ in \eqref{ConservationLawsComp} leads to conservation of the $1$-moment:
\begin{equation*}
\begin{split}
\frac{d}{dt}P(\mu)=0,\qquad P(\mu)&:=m_g\iint \mu^2 {\bf v}\,d{\bf x}d{\bf v}+M_c\mathcal{V}.
\end{split}
\end{equation*}
Finally, the total energy is conserved and this leads to conservation of the $2$-moment:
\begin{equation*}
\begin{split}
\frac{d}{dt}E(\mu,\mathcal{X},\mathcal{V})=0,\qquad E(\mu,\mathcal{X},\mathcal{V})&:=m_g\iint \mu^2\left(\vert {\bf v}\vert^2-2Q\phi\right)\,d{\bf x}d{\bf v}+\left(M_c\vert \mathcal{V}\vert^2-2\frac{q_cq_g}{\epsilon_0}\phi(\mathcal{X},t)\right).
\end{split}
\end{equation*}

\subsubsection{Modulation and reduced equations}

Since the Vlasov-Poisson system is invariant by Galilean transformation, we can choose a frame where the total momentum vanishes: $P(\mu)=0$. This determines the motion of the point charge in terms of the motion of the gas:
\begin{equation*}
\begin{split}
\mathcal{V}(t)&:=-\frac{m_g}{M_c}\iint\mu^2 {\bf v}\, d{\bf x}d{\bf v}.
\end{split}
\end{equation*}
As explained in Section \ref{SecAddingPC}, we will need a ``close'' and a ``far'' chart, which we introduce next.

\subsubsection*{Far formulation}
Given a solution $\mu$ on some time interval $[0,T^\ast)$, we define
\begin{equation}\label{DefVinfty}
\mathcal{V}_\infty:=\lim_{t\to T^\ast}\mathcal{V}(t),\qquad\mathcal{W}(t):=\mathcal{V}(t)-\mathcal{V}_\infty,
\end{equation}
and we introduce the new unknowns
\begin{equation}\label{NewVariables}
\begin{split}
{\bf y}&:={\bf x}-\mathcal{X},\qquad {\bf w}:={\bf v}-\mathcal{V}_\infty,\\
\nu({\bf y},{\bf w},t)&:=\mu({\bf y}+\mathcal{X},{\bf w}+\mathcal{V}_\infty,t),\qquad\mu({\bf x},{\bf v},t)=\nu({\bf x}-\mathcal{X},{\bf v}-\mathcal{V}_\infty,t).
\end{split}
\end{equation}
The new equation for terms of $\nu({\bf y},{\bf w},t)$ becomes self-consistent with a parameter $\mathcal{V}_\infty$:
\begin{equation}\label{NewVP}
\begin{split}
&\left(\partial_t+{\bf w}\cdot\nabla_{\bf y}+\frac{q}{2}\frac{{\bf y}}{\vert {\bf y}\vert^3}\cdot\nabla_{\bf w}\right)\nu+Q\nabla_{\bf y}\psi\cdot\nabla_{\bf w}\nu=\mathcal{W}(t)\cdot\nabla_{\bf y}\nu,\\
&\psi({\bf y},t)=\phi({\bf y}+\mathcal{X},t)=-\frac{1}{4\pi}\iint \frac{1}{\vert {\bf y}-{\bf r}\vert}\nu^2({\bf r},{\boldsymbol\pi},t)d{\bf r} d{\boldsymbol\pi},\\
&\mathcal{W}(t)=-\frac{M_g+M_c}{M_c}\mathcal{V}_\infty-\frac{m_g}{M_c}\iint {\bf w}\nu^2({\bf y},{\bf w},t)\,\, d{\bf y}d{\bf w}.
\end{split}
\end{equation}
We introduce the Hamiltonians:
\begin{equation}\label{Hamiltonians}
\begin{split}
\mathbb{ H}&=\frac{1}{2}\mathbb{H}_2-\mathbb{H}_4,\qquad
\mathbb{H}_2({\bf y},{\bf w}):=\vert {\bf w}\vert^2+\frac{q}{\vert {\bf y}\vert},\qquad\mathbb{H}_4({\bf y},{\bf w},t)=Q\psi({\bf y},t)+\mathcal{W}(t)\cdot {\bf w}.
%\psi&=-\frac{1}{4\pi}\iint \frac{1}{\vert q-r\vert}\nu^2(r,\pi,t)dr d\pi,\qquad\mathcal{W}=-\frac{M_g+M_c}{M_c}\mathcal{V}_\infty-\frac{m_g}{M_c}\iint p\nu^2\, dqdp
\end{split}
\end{equation}
Note that $\mathbb{H}_2$ is independent of the unknown $\nu$ and will give the linearized equation, while $\mathbb{H}_4=O(\nu^2)$ and decays in time.
The density $\nu$ is transported by the corresponding Hamiltonian vector field in \eqref{Hamiltonians} in the sense that  the first equation in \eqref{NewVP} is equivalent to
\begin{equation}
\partial_t\nu-\{\mathbb{H},\nu\}=0.
\end{equation}
\begin{lemma}
Given $C^1$ functions $(\mu,\nu,\mathcal{X},\mathcal{V})$ on a time interval $[0,T^\ast)$  and a constant $\mathcal{V}_\infty\in\mathbb{R}^3$, related by \eqref{DefVinfty}-\eqref{NewVariables}. The functions $(\mu,\mathcal{X},\mathcal{V})$ solve \eqref{VPPC} if and only if $(\nu,\mathcal{V}_\infty)$ solves
\begin{equation}\label{HamiltonianTransport}
\begin{split}
\partial_t\nu-\{\mathbb{H},\nu\}=0,\qquad\lim_{t\to T^\ast}\mathcal{W}(t)=0,
\end{split}
\end{equation}
with $\mathbb{H}$ defined in \eqref{Hamiltonians}.
\end{lemma}
\begin{proof}
We recall that $\frac{d\mathcal{X}}{dt}=\mathcal{V}$.
With notations in \eqref{NewVariables}, we have
\begin{equation}
\begin{split}
\nabla_{\bf y}\nu({\bf y},{\bf w},t)=\nabla_{\bf x}\mu({\bf x},{\bf v},t),&\qquad\nabla_{\bf w}\nu({\bf y},{\bf w},t)=\nabla_{\bf v}\mu({\bf x},{\bf v},t),\\
\partial_t\nu({\bf y},{\bf w},t)=(\partial_t+\mathcal{V}\cdot\nabla_{\bf x})\mu({\bf x},{\bf v},t),&\qquad\nabla_{\bf y}\psi({\bf y},t)=\nabla_{\bf x}\phi({\bf x},t).
\end{split}
\end{equation}
Then in terms of $\mu, {\bf x},{\bf v}$, the first equation in \eqref{NewVP} becomes 
\begin{equation}
\left(\partial_t+\mathcal{V}(t)\cdot\nabla_{\bf x}+({\bf v}-\mathcal{V}_\infty)\cdot\nabla_{\bf x}+\frac{q}{2}\frac{{\bf x}-\mathcal{X}}{\vert {\bf x}-\mathcal{X}\vert^3}\cdot\nabla_{\bf v}\right)\mu+Q\nabla_{\bf x}\phi\cdot\nabla_{\bf v}\mu=(\mathcal{V}(t)-\mathcal{V}_\infty)\cdot\nabla_{\bf x}\mu,
\end{equation}
which is exactly the first equation in \eqref{VPPC}.
\end{proof}

\subsubsection*{Close formulation}\label{ssec:pinnedframe}
Close to the point charge / for large velocities, we will prefer to center our coordinate frame around the instantaneous velocity of the point charge, and thus let
\begin{equation}\label{CenteredCoordinates}
{\bf y}^\prime={\bf y}={\bf x}-\mathcal{X}(t),\qquad {\bf w}^\prime={\bf v}-\mathcal{V}(t)={\bf w}-\mathcal{W}(t).
\end{equation}
This can be obtained by means of the generating function $S({\bf y}^\prime,{\bf w},t)={\bf y}^\prime\cdot({\bf w}-\mathcal{W}(t))$ and leads to the new Hamiltonian
\begin{equation}\label{NewHamiltonian2}
\begin{split}
\mathbb{H}^{\prime}({\bf y'},{\bf w'})=\mathbb{H}({\bf y},{\bf w})-\partial_tS=\frac{1}{2}\mathbb{H}_2({\bf y}^\prime,{\bf w}^\prime)-Q\psi({\bf y}^\prime,t)+\dot{\mathcal{W}}\cdot {\bf y}^\prime-\frac{1}{2}\vert\mathcal{W}\vert^2.
\end{split}
\end{equation}
We also note for further use that
\begin{equation}\label{DerW}
\dot{\mathcal{W}}(t)=\mathcal{Q}\nabla_{\bf y}\psi(0,t).
\end{equation}
We let
\begin{equation}
\nu'({\bf y'},{\bf w'},t):=\mu({\bf y'}+\mathcal{X},{\bf w'}+\mathcal{V}(t),t).
\end{equation}
Then equation \eqref{VPPC} is equivalent to 
\begin{equation}\label{equivVPPCNu}
\partial_t\nu'-\{\mathbb{H}^{\prime},\nu'\}=0,
\end{equation}
or
\begin{equation}\label{NewVP2}
\begin{split}
\left(\partial_t+{\bf w}'\cdot\nabla_{\bf y'}+\frac{q}{2}\frac{{\bf y'}}{\vert {\bf y'}\vert^3}\cdot\nabla_{\bf w'}\right)\nu'+Q\nabla_{\bf y'}\psi\cdot\nabla_{\bf w'}\nu'
&=\mathcal{Q}\nabla_{\bf y'}\psi(0,t)\cdot\nabla_{\bf w'}\nu',\\
\psi({\bf y'},t)&=\phi({\bf y'}+\mathcal{X},t).
\end{split}
\end{equation}

\section{Analysis of the linearized flow}\label{SecLin}

In this section, we solve the linearized equation \eqref{LinearHamiltonianFlow} using asymptotic action angle variables and introduce various other adapted coordinates.

\subsection{General comments on the asymptotic action-angle method}\label{GenComAA}

\subsubsection{Overview of the method of asymptotic actions}

There is no general way to find ``good choices'' of action-angle variable $(\Theta,\mathcal{A})$ besides trial and error; however, a few guidelines can be useful.
\begin{enumerate}

\item It is desirable that the change of variable $({\bf x},{\bf v})\mapsto(\vartheta,\a)$ be canonical, i.e. $d{\bf x}\wedge d{\bf v}=d\vartheta\wedge d\a$. This in particular ensures that the jacobian of the change of variable is $1$. A good way to enforce this is to use a {\it generating function} $S({\bf x},\a)$ such that ${\bf v}=\nabla_{\bf x} S$ and $\vartheta=\nabla_{\bf a} S$. In this case
\begin{equation*}
0=ddS=d\left(\frac{\partial S}{\partial {\bf x}^j}d{\bf x}^j+\frac{\partial S}{\partial \a^j}d\a^j\right)=d({\bf v}d{\bf x}+\vartheta d\a).
\end{equation*}
Such functions are, however often difficult to find explicitly.

\item Since we are concerned with longtime behavior, we choose $\a$ so that it captures the dispersive nature of the problem, i.e. that trajectories corresponding to different choices of $\a$ diverge as $t\delta a$. In scattering situations, one has a natural Hamiltonian at $\infty$, often $H_\infty=\vert {\bf v}_\infty\vert^2/2$ and a useful choice is $\a={\bf v}_\infty$. In this case, in the simplest situation, one needs to solve a scattering problem to define ${\bf v}=V({\bf x},{\bf v}_\infty)$ (i.e. compute the full trajectory knowing the incoming velocity and the position at one time), then integrate $V$ to recover the generating function
\begin{equation*}
\begin{split}
V({\bf x},{\bf v}_\infty)=(\nabla_{\bf x}S)({\bf x},{\bf v}_\infty),
\end{split}
\end{equation*}
and then deduce the angle $\vartheta$.

\end{enumerate}

Note that this leads to a cascade of Hamiltonians which describe various components of the dynamics: $\mathbb{H}\to \mathbb{H}_2$ to describe solutions to the perturbed problem as envelopes of solutions to the linearized problem, and $\mathbb{H}_2\mapsto H_\infty$ to describe solutions of the linearized problem via their asymptotic state.

\subsubsection{The case of radial data}

The strategy laid out above can be most easily carried out for the case of radial data, a problem already treated in \cite{PW2020}. Starting from the energy, we can express the outgoing velocity as a function of $a$ and $r$:
\begin{equation}\label{RadialV}
\begin{split}
V(r,a)=\sqrt{a^2-\frac{q}{r}}=a\sqrt{1-\frac{q}{ra^2}}
\end{split}
\end{equation}
and integrating in $r$, we find the generating function:
\begin{equation}\label{GeneratingFunctionRadial}
\begin{split}
S(r,a)=\frac{q}{a}K(\frac{ra^2}{q})=ar_{min}K(\frac{r}{r_{min}}),\qquad\frac{\partial S}{\partial r}=V,
\end{split}
\end{equation}
where $K$ is defined in \eqref{DefK} below. This gives a formula for the angle
\begin{equation*}
\begin{split}
\Theta&=\frac{\partial S}{\partial a}=-\frac{q}{a^2}K(\frac{ra^2}{q})+2r\sqrt{1-\frac{q}{ra^2}}=-\frac{q}{a^2}K(\frac{ra^2}{q})+2r\frac{V}{a}.
\end{split}
\end{equation*}
This simplifies to
\begin{equation}\label{NewDefinitionTheta}
\begin{split}
\Theta&=-r_{min}K(\frac{r}{r_{min}})+2\frac{rv}{a}=-r_{min}K(\frac{r}{r_{min}})+2\frac{rv}{\sqrt{v^2+\frac{q}{r}}},
\end{split}
\end{equation}
and since $r_{min},a$ are invariant along the trajectory, we can verify that
\begin{equation*}
\begin{split}
\dot{\Theta}&=-\dot{r}K^\prime(\frac{r}{r_{min}})+2\frac{\dot{r}v}{a}+2\frac{r\dot{v}}{a}=-\frac{v^2}{a}+2\frac{v^2}{a}+\frac{q}{ra}=a,
\end{split}
\end{equation*}
where we have used that
\begin{equation*}
\begin{split}
K^\prime(\frac{r}{r_{min}})=\sqrt{1-\frac{r_{min}}{r}}=\frac{v}{a},\qquad \dot{r}=v,\qquad\dot{v}=\frac{q}{2r^2}.
\end{split}
\end{equation*}

In the discussion above, we have used the function $K$ defined by
\begin{equation}\label{DefK}
\begin{split}
K^\prime(s)=\left[1-\frac{1}{s}\right]^\frac{1}{2},\qquad K(1)=0,\qquad K(s):=\sqrt{s(s-1)}-\ln(\sqrt{s}+\sqrt{s-1}).
\end{split}
\end{equation}
Since we can verify that
\begin{equation*}
\begin{split}
G(s)=2\sqrt{s(s-1)}-K(s),\qquad G^\prime(s)=\left[1-\frac{1}{s}\right]^{-\frac{1}{2}},\quad G(1)=0,
\end{split}
\end{equation*}
we see that, in the outgoing case $v>0$, this gives a similar choice of unknown than in \cite{PW2020}, but through a different approach. In the incoming case $v<0$, we need to change sign in \eqref{RadialV} and this leads to the new generating function
\begin{equation*}
\begin{split}
S_{in}(r,a):=-\frac{q}{a}K(\frac{ra^2}{q}),
\end{split}
\end{equation*}
and in turn we obtain
\begin{equation*}
\begin{split}
\Theta&=r_{min}K(\frac{r}{r_{min}})-2r\sqrt{1-\frac{q}{ra^2}}=\frac{v}{\vert v\vert}\left[-r_{min}K(\frac{r}{r_{min}})+2r\frac{\vert v\vert}{a}\right]=r_{min}\frac{v}{\vert v\vert}G(\frac{r}{r_{min}}).
\end{split}
\end{equation*}
We thus see that, in the radial case, the scattering problem is degenerate in the sense that there are, in general, {\it two} trajectories that pass through $r$ at time $t$ and have asymptotic velocity $a$ (one incoming one outgoing). This fold degeneracy can be resolved by introducing a function $\sigma=v/\vert v\vert=\theta/\vert\theta\vert$ and defining
\begin{equation*}
\begin{split}
S_{tot}=\sigma\frac{q}{a}K(\frac{ra^2}{q})
\end{split}
\end{equation*}
with a choice of sign that is not expressed in terms of $(r,a)$ alone.

\subsection{Angle-Action coordinates for the Kepler problem}\label{ssec:Kepler}

We want to study the Kepler system:
\begin{equation}\label{ODE}
\begin{split}
\frac{dX}{dt}=V,\qquad\frac{dV}{dt}=\frac{qX}{2\vert X\vert^3},
\end{split}
\end{equation}
in the repulsive case $q>0$. All the orbits are open and can be described using the conservation of energy $H$, angular momentum ${\bf L}$ and {\it Runge-Lenz} vector $\bf{R}$,
\begin{equation}\label{ConservedQuantitiesODE}
\begin{split}
H&:=\vert{\bf v}\vert^2+\frac{q}{\vert{\bf x}\vert},\qquad {\bf L}:={\bf x}\times{\bf v},\qquad{\bf R}:={\bf v}\times {\bf L}+\frac{q}{2}\frac{{\bf x}}{\vert{\bf x}\vert}.
\end{split}
\end{equation}
In the following formulas, it will be very useful to keep track of the homogeneity and we note that\footnote{Here $\dm{\x}$ denotes the dimension of spatial length, $\dm{t}$ the dimension of time,  $\dm{\v}=\dm{\x}\dm{t}^{-1}$ the  dimension of velocity and $\dm{q}=\dm{\x}^3\dm{t}^{-2}$ the dimension of an electric charge.}
\begin{equation}\label{eq:dims}
\begin{split}
\dm{H}=\frac{\dm{\x}^2}{\dm{t}^2}=\dm{\v}^2=\frac{\dm{q}}{\dm{\x}},\qquad \dm{q}=\dm{\mathbf{R}}=\dm{\x}\dm{\v}^2,\qquad \dm{\L}=\dm{\x}\dm{\v}.
\end{split}
\end{equation}
The conservation laws \eqref{ConservedQuantitiesODE} give $5$ functionally independent conservation laws, although of course only $3$ of these can be in involution. A classical choice is $\{H,\vert{\bf L}\vert,L_3\}$ which leads to the classical solution of the $2$-body problem by reducing it to a planar problem. Indeed, if one chooses the $z$ axis such that $L_1=L_2=0$, then they remain $0$ for all time and the motion remains in the plane $\{z=0\}$. This allows us to exhibit a convenient action-angle transformation satisfying the asymptotic action property \eqref{AsymptoticActions}.

We define the phase space to be
\begin{equation*}
\begin{split}
\mathcal{P}_{{\bf x},{\bf v}}:=\{({\bf x},{\bf v})\in\mathbb{R}^3\times\mathbb{R}^3:\,\, \vert{\bf x}\vert>0\},\qquad\mathcal{P}_{\vartheta,{\bf a}}:=\{(\vartheta,{\bf a})\in\mathbb{R}^3\times\mathbb{R}^3:\,\,\vert{\bf a}\vert>0\},
\end{split}
\end{equation*}
where the angles $\vartheta$ resp.\ actions $\a$ will have the dimension of length $\dm{\vartheta}=\dm{\x}$ resp.\ velocity $\dm{\a}=\dm{\v}$.

Our main result in this section is the construction of asymptotic action-angle variables in Proposition \ref{PropAA} which will provide adapted coordinates for the phase space.
%\begin{proposition}\label{PropAA}
%There exists a smooth diffeomorphism $\mathcal{T}:\mathcal{P}_{\x,\v}\to\mathcal{P}_{\theta,\a}$, $({\bf x},{\bf v})\mapsto(\Theta(\x,\v),\mathcal{A}(\x,\v))$ with inverse $(\vartheta,\a)\mapsto (\X(\vartheta,\a),\V(\vartheta,\a)$, which is
%\begin{enumerate}[wide]
% \item\label{it:canonical} canonical, i.e.\
% \begin{equation}\label{eq:canonical}
%  d\x\wedge d\v=d\Theta\wedge d\mathcal{A},
% \end{equation}
%
% \item\label{it:cons-comp} compatible with conservation of energy and angular momentum:
% \begin{equation}
% H({\bf x},{\bf v})=\abs{\mathcal{A}}^2,\qquad{\bf x}\times{\bf v}=\Theta\times\mathcal{A}
% \end{equation}
%
%\item\label{it:lin-diag} linearizes the flow of \eqref{ODE} in the sense that $({\bf x}(t),{\bf v}(t))$ solves \eqref{ODE} if and only if
%\begin{equation}\label{LinearizationFlow}
%\Theta({\bf x}(t),{\bf v}(t))=\Theta({\bf x}(0),{\bf v}(0))+t\mathcal{A}({\bf x}(0),{\bf v}(0)),\qquad\mathcal{A}({\bf x}(t),{\bf v}(t))=\mathcal{A}({\bf x}(0),{\bf v}(0)).
%\end{equation}
%
%\item\label{it:asymp-act} satisfies the ``asymptotic action property''
%\begin{equation}\label{AsymptoticActions}
%\begin{split}
%\vert{\bf X}(\theta+t{\bf a},{\bf a})-t{\bf a}\vert=o(t),\qquad \vert{\bf V}(\theta+t{\bf a},{\bf a})-{\bf a}\vert=o(1)\quad\hbox{ as }t\to+\infty.
%\end{split}
%\end{equation}
%\end{enumerate}
%\end{proposition}
This will be proved later in Section \ref{ssec:AngleAction} after we have solved a scattering problem. Here we highlight that by construction as a \emph{canonical transformation}, the change of variables $\mathcal{T}$ preserves the Liouville measure
 \begin{equation}
  \det\frac{\partial(\Theta,\mathcal{A})}{\partial(\x,\v)}=\det\frac{\partial(\X,\V)}{\partial(\vartheta,\a)}=1.
 \end{equation}

\begin{remark}\label{rem:virials}

Using that \eqref{ODE} is super-integrable, one can express $5$ of the 6 coordinates of a trajectory in terms of conserved quantities. The last one corresponds to the ``trace'' of time and cannot be deduced by the conservation laws alone. In our case, a good proxy for the ``trace'' of time are the quantities ${\bf x}\cdot {\bf v}$ and $\theta\cdot{\bf a}$. That this measures time lapsed is obvious in action angle coordinates from \eqref{LinearizationFlow}. For the physical variables, this follows from Virial-type computations
\begin{equation}\label{eq:virials}
\begin{split}
\frac{d}{dt}\frac{\vert {\bf x}\vert^2}{2}={\bf x}\cdot{\bf v},\qquad \frac{d}{dt}({\bf x}\cdot{\bf v})=\vert{\bf v}\vert^2+\frac{q}{2\vert{\bf x}\vert}=\frac{\vert{\bf v}\vert^2}{2}+\frac{1}{2}H.
\end{split}
\end{equation}

\end{remark}

Although not very illuminating, we can obtain explicit expressions using \eqref{ExplicitA1} and \eqref{DefTheta}:
\begin{equation}\label{eq:explicitATheta}
\begin{split}
\mathcal{A}({\bf x},{\bf v})&=\frac{2q\sqrt{H}}{4HL^2+q^2}{\bf R}+\frac{4H}{4HL^2+q^2}{\bf L}\times {\bf R},\\
\Theta({\bf x},{\bf v})&=\frac{\iota}{2}\vert{\bf x}\vert K^\prime(\rho)\left(\frac{\bf x}{\vert{\bf x}\vert}+\frac{\bf a}{\vert{\bf a}\vert}\right)+\frac{\vert{\bf x}\vert}{2}\left(\frac{\bf x}{\vert{\bf x}\vert}-\frac{\bf a}{\vert{\bf a}\vert}\right)-\sigma(\rho)\frac{\bf a}{\vert{\bf a}\vert^3},\qquad{\bf a}=\mathcal{A}({\bf x},{\bf v}),
\end{split}
\end{equation}
where $K$ is defined in \eqref{DefK}, and
\begin{equation*}
\begin{split}
\iota&=\hbox{sign}({\bf x}\cdot{\bf v}+\sqrt{H}L^2/q),\qquad\sigma=-\iota\ln(\sqrt{\rho}+\sqrt{\rho-1}),\\
\rho&=\frac{\sqrt{H}}{2q}\left(\vert{\bf x}\vert\sqrt{H}+\frac{1}{4HL^2+q^2}(4HL^2{\bf x}\cdot{\bf v}+q^2\sqrt{H}\vert{\bf x}\vert+2qL^2\sqrt{H})\right).
\end{split}
\end{equation*}
This follows combining \eqref{DefGamma1} and \eqref{DefTheta} for $\iota$, \eqref{def:sigma} for $\sigma$ and \eqref{eq:def_rho2} and \eqref{AxAXperp} for $\rho$.

In the radial case, ${\bf L}=0$, $\rho=r/r_{min}$ and we recover the formulas from the paper \cite{PW2020}. Explicit expressions for $\X(\vartheta,\a),\V(\vartheta,\a)$ will be most useful and are given in \eqref{ExpressionsX} and \eqref{ExpressionsV} below.

\subsubsection{Planar dynamics}\label{ssec:planar}
To study the geometry of trajectories in the Kepler problem \eqref{ODE}, we note that by the conservation laws \eqref{ConservedQuantitiesODE} we may choose coordinates such that the motion takes place in the $xy$-plane with angular momentum $(0,0,L)$ for some $L>0$. Switching to polar coordinates $(r,\phi)$\footnote{Note that ${\bf x}=r(\cos\phi,\sin{\phi},0)$. We recall vectors ${\bf e}_r=(\cos\phi,\sin\phi,0)$ and ${\bf e}_\phi=(-\sin\phi,\cos\phi,0)$ in the basis.}, we find that
\begin{equation}\label{ConservationLaws2dplanar}
\begin{split}
H=\dot{r}^2+r^2\dot{\phi}^2+\frac{q}{r},\qquad L=r^2\dot{\phi},\qquad{\bf R}=\left(\frac{L^2}{r}+\frac{q}{2}\right){\bf e}_r-\dot{r}L{\bf e}_\phi,
\end{split}
\end{equation}
and these can be used to integrate the equation. 
\begin{figure}[h]
 \centering
 \includegraphics[width=0.75\textwidth]{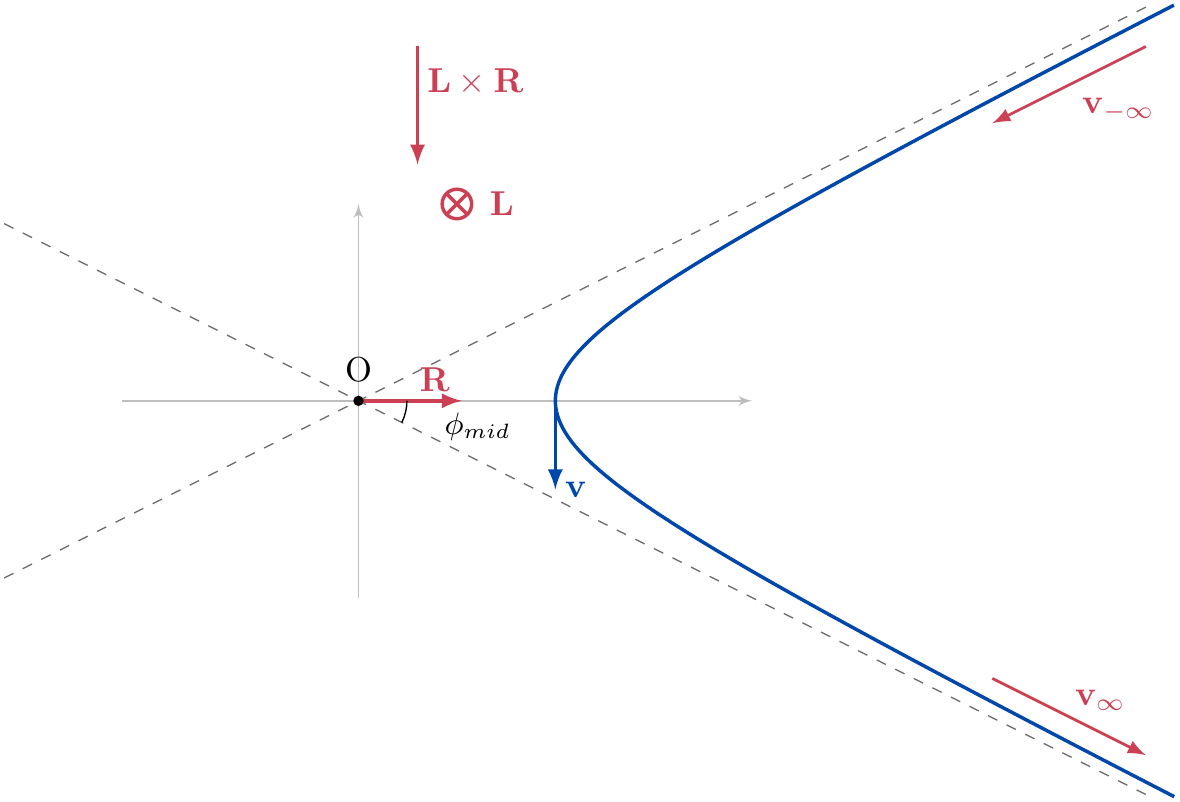}
 \caption{A sample trajectory of the Kepler problem (in blue), with conserved quantities (in red).}
 \label{fig:cons-laws}
\end{figure}
More precisely, we have that
\begin{equation}
 \frac{d\bf{v}}{dt}=\frac{q}{2r^2}(\cos\phi,\sin\phi,0), \quad L=r^2\frac{d\phi}{dt}\qquad\Rightarrow\qquad \frac{d\bf{v}}{d\phi}=\frac{q}{2L}(\cos\phi,\sin\phi,0),
\end{equation}
and hence
\begin{equation}
 {\bf v}(\phi)=\frac{q}{2L}(\sin\phi,-\cos\phi,0)+(c_1,c_2,0),
\end{equation}
where ${\bf c}=(c_1,c_2,0)$ is the constant of integration. Thus ${\bf v}$ moves along a circle (the so-called \emph{``velocity circle''} -- see e.g.\ \cite{Mil1983} and Figure \ref{fig:velocity-circle}) with center ${\bf c}$ and satisfies ${\bf v-c\perp x}$, with periapsis in direction $-{\bf c}^\perp$.\footnote{For a vector ${\bf c}=(c_1,c_2,0)$, we denote ${\bf c}^\perp=(-c_2,c_1,0)$.} As is well known, in the repulsive case ($q>0,H>0$) the trajectories are hyperbolas, and asymptotic velocities ${\bf v_{\pm\infty}}$ are thus well-defined.\footnote{In the present formulation, this can be seen for example by assuming that the velocity circle is centered on the positive $y$-axis, i.e.\ $c_1=0,c_2=\epsilon\cdot\frac{q}{2L}\geq 0$, so that from the equation $(0,0,L)={\bf x}\times{\bf v}$ it follows that $r=\frac{2L^2}{q}\frac{1}{\epsilon\cos\phi-1}$, where $\epsilon>1$ is the eccentricity of the hyperbola.}

\begin{figure}[h]
 \centering
 \includegraphics[width=0.6\textwidth]{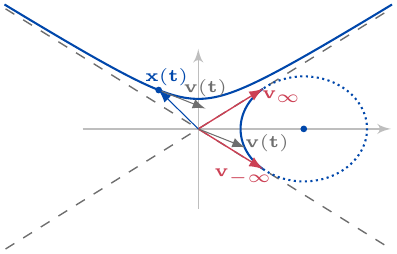}
 \caption{An illustration of a sample trajectory of the Kepler problem with associated velocity circle (in blue) and asymptotic velocities $\v_{-\infty}$ and $\v_\infty$ (in red).}
 \label{fig:velocity-circle}
\end{figure}

\subsubsection*{Scattering problem}
In order to obtain our asymptotic action, we need to understand to which extent knowledge of ${\bf x}$ and ${\bf v}_\infty$ allows to determine the full trajectory, i.e. to solve the following problem: ``find the trajectories passing through ${\bf x}$ at time $t$ and whose (forward) asymptotic velocity is ${\bf v}_\infty$''. 

Inspecting the behavior of the conservation laws at $\infty$ and at periapsis, we find that\footnote{Here and in the following, given a nonzero vector ${\bf v}$, we denote by $\hat{\bf v}={\bf v}/\vert{\bf v}\vert$ its direction vector.}
\begin{equation}\label{ConsLaws2dPlanar2}
\begin{split}
H=\frac{L^2}{r_{min}^2}+\frac{q}{r_{min}}=\vert{\bf v}_\infty\vert^2,\qquad {\bf R}=\left(\frac{L^2}{r_{min}}+\frac{q}{2}\right)\hat{\bf x}_p=\frac{q}{2}{\hat{\bf v}_\infty}-\vert {\bf v}_\infty\vert\cdot  L\hat{{\bf v}}_\infty^\perp,
\end{split}
\end{equation}
and in particular
\begin{equation*}
2R=\sqrt{4HL^2+q^2},\qquad\frac{1}{r_{min}}=\frac{\sqrt{q^2+4HL^2}-q}{2L^2},\quad r_{min}=\frac{q+\sqrt{q^2+4HL^2}}{2H},
\end{equation*}
so that only the direction of ${\bf R}$ will be important (and we recover $r_{min}=q/H$ in the radial case).

For the sake of definiteness, for a given ${\bf v_\infty}$ let us further rotate our coordinates such that ${\bf v_\infty}=(\sqrt{H},0,0)$ is parallel to the (positive) $x$-axis (see also Figure \ref{fig:scatter}). Then the possible trajectories in this setup all lie in the lower half-plane $\{y\leq 0\}$, and the center ${\bf c}$ of the velocity circle is determined by the requirement that ${\bf v}(2\pi)=(\sqrt{H},0,0)$, i.e.
\begin{equation}\label{eq:v-theta}
 {\bf v}(\phi)=\frac{q}{2L}(\sin\phi,-\cos\phi,0)+(\sqrt{H},\frac{q}{2L},0),
\end{equation}
and periapsis lies in direction $-{\bf c}^\perp=(\frac{q}{2L},-\sqrt{H},0)$.

% \begin{figure}[h]
%  \centering
%  \includegraphics[width=0.5\textwidth]{scatter.pdf}
%  \caption{The scattering problem: Given a point $\x_0$ and asymptotic velocity $\v_\infty$, how to determine trajectories (one possibility dashed in blue) through $\x_0$ with asymptotic velocity $\v_\infty$?}
%  \label{fig:scatter}
% \end{figure}

The quantity 
\begin{equation}\label{eq:def_rho}
 \rho(r,\phi):=\frac{rH}{2q}(1+\cos\phi)
\end{equation}
plays a key role for the dynamics, as illustrated by the following lemma:
\begin{lemma}\label{lem:trajectories}
 Let ${\bf x}_0=r_0(\cos\phi_0,\sin\phi_0,0)$ be given with $r_0>0$ and $\phi_0\in(\pi,2\pi)$.
\begin{enumerate}
 \item If $\rho(r_0,\phi_0)<1$, there does not exist a trajectory through ${\bf x}_0$ with asymptotic velocity ${\bf v}_\infty$.
 \item If $\rho(r_0,\phi_0)=1$, there exists exactly one trajectory through ${\bf x}_0$ with asymptotic velocity ${\bf v}_\infty$. If ${\bf v}_0$ denotes the instantaneous velocity, there holds that ${\bf x}_0\cdot{\bf v}_0=-\sqrt{H}L^2/q<0$.
 \item If $\rho(r_0,\phi_0)>1$ there are two trajectories $h_\pm$ through ${\bf x}_0$ with asymptotic velocity ${\bf v}_\infty$, corresponding to values $0<L_-< L_+$ for the angular momentum. 
\end{enumerate}

Besides, in case $(3)$, we have that, locally around $({\bf x}_0,{\bf v}_0)$, $\rho$ decreases on $h_-$ and increases on $h_+$.
\end{lemma}

\begin{proof}
 On the one hand, since $(0,0,L)={\bf x}\times {\bf v}$ we obtain directly that
\begin{equation}\label{eq:L}
 L=r\left(-\frac{q}{2L}+\frac{q}{2L}\cos\phi-\sqrt{H}\sin\phi\right)\quad\Leftrightarrow\quad L^2+Lr\sqrt{H}\sin\phi+\frac{rq}{2}(1-\cos\phi)=0.
\end{equation}
This has real solutions if and only if
\begin{equation}
 rH(1+\cos\phi)\geq 2q\quad\Leftrightarrow\quad \rho(r,\phi)=\frac{rH}{2q}(1+\cos\phi)\geq 1,
\end{equation}
and they are given by
\begin{equation}\label{eq:L-sol}
 L=-\sin\phi\frac{r\sqrt{H}}{2}\left(1\pm\sqrt{1-\frac{1}{\rho(r,\phi)}}\right).
\end{equation}
Since each choice of $L$ leads to a trajectory by \eqref{eq:v-theta}, this yields the claimed trichotomy. Along a trajectory, we see from \eqref{eq:L} that
\begin{equation}\label{eq:r-shape}
 r=-\frac{L^2}{L\sqrt{H}\sin\phi+\frac{q}{2}(1-\cos\phi)},
\end{equation}
and therefore,
\begin{equation}
 \partial_\phi r(\phi)=\frac{qL^2}{2[L\sqrt{H}\sin\phi+\frac{q}{2}(1-\cos\phi)]^2}\left(\frac{2L\sqrt{H}}{q}\cos\phi+\sin\phi\right)=\frac{r}{L}{\bf c\cdot x}=\frac{r}{L}{\bf v\cdot x}.
\end{equation}
Using \eqref{eq:def_rho} and \eqref{eq:L-sol}, we see that
\begin{equation}\label{AngleFold}
-\tan(\phi/2)=\frac{-\sin\phi}{1+\cos\phi}=\frac{L\sqrt{H}}{q}\frac{1}{\rho\pm\sqrt{\rho(\rho-1)}}=\frac{\sqrt{H}L}{q}\left[1\mp\sqrt{1-1/\rho}\right],
\end{equation}
and therefore, when $\rho=1$ we can plug in the equation above and obtain $\partial_\phi r<0$. Deriving both side of \eqref{AngleFold}, we obtain
\begin{equation*}
\begin{split}
-\frac{1}{2}(1+\tan^2(\phi/2))=\mp\frac{L\sqrt{H}}{q}\frac{1}{2\rho\sqrt{\rho(\rho-1)}}\frac{d}{d\phi}[\rho(r(\phi),\phi)],
\end{split}
\end{equation*}
which is enough since we know that, along a trajectory, $\dot{\phi}=r^{-2}L>0$.
\end{proof}

\begin{figure}[h]
 \centering
 \includegraphics[width=\textwidth]{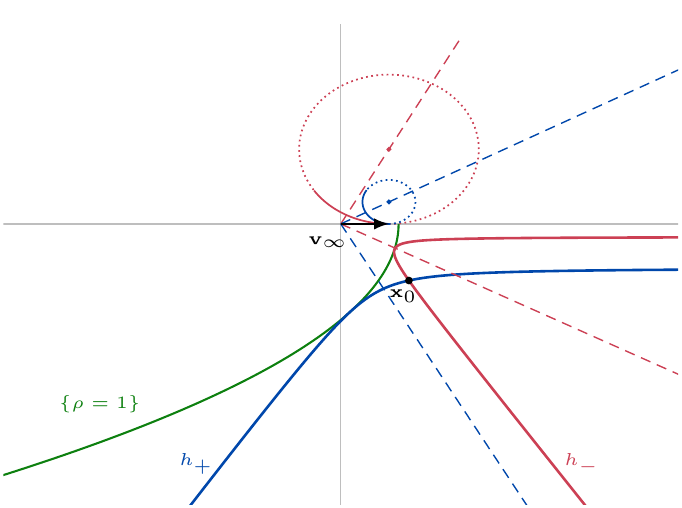}
 \caption{An illustration of Lemma \ref{lem:trajectories} with two trajectories $h_1$ (red) and $h_2$ (blue) with asymptotic velocity $\v_\infty$ through a given point $\x_0$, and their corresponding velocity circles. The green line is the level set $\{\rho=1\}$.}
 \label{fig:trajectories}
\end{figure}

From the proof we observe that, via the velocity circle, we can directly compute the angle from periapsis to asymptotic velocity as $\phi_{p}=\frac{\pi}{2}-\tilde\phi$, where $\tan\tilde\phi=\frac{q}{2L\sqrt{H}}$ (see also Figure \ref{fig:vel-circle}).
\begin{figure}[h]
 \centering
 \includegraphics[width=0.5\textwidth]{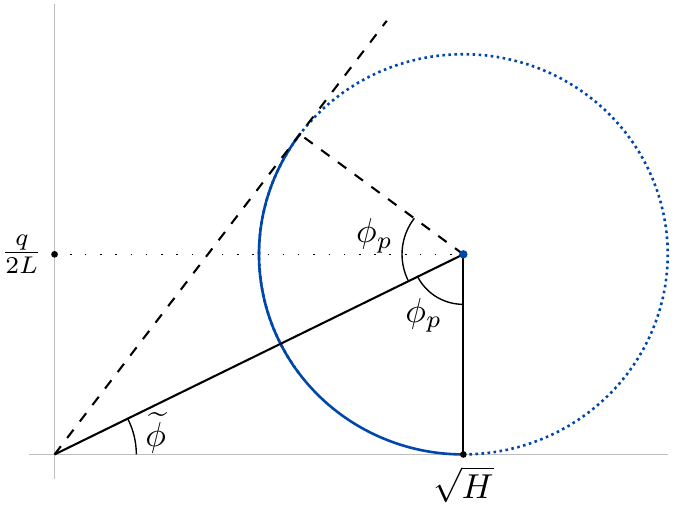}
 \caption{An illustration of the utility of the velocity circle.}
 \label{fig:vel-circle}
\end{figure}

In addition, by \eqref{eq:L} and \eqref{eq:v-theta} there holds that 
\begin{equation}\label{ExpressionVPlanarCase}
 {\bf v}(\phi)=-\frac{q}{2L}{\bf e}_\phi+(\sqrt{H},\frac{q}{2L},0)=-\left(\sqrt{H}+\frac{\sin\phi}{1-\cos\phi}\frac{L}{r}\right){\bf e}_r+\frac{L}{r}{\bf e}_\phi.
\end{equation}
This can be integrated using that, with $K$ as defined in \eqref{DefK}, and \eqref{eq:L-sol}, there holds that
\begin{equation}
\begin{aligned}
 \frac{L}{r}&=-\frac{\sqrt{H}}{2}\sin\phi\left[1\pm K'(\rho(r,\phi))\right]=\frac{1}{r}\frac{\partial}{\partial\phi}\left[\pm \frac{q}{\sqrt{H}}K(\rho(r,\phi))+\frac{r\sqrt{H}}{2}\cos\phi\right],\\
 -\frac{\sin\phi}{1-\cos\phi}\frac{L}{r}&=\frac{\sqrt{H}}{2}(1+\cos\phi)\left[1\pm K'(\rho(r,\phi))\right]=\frac{\partial}{\partial r}\left[\pm \frac{q}{\sqrt{H}}K(\rho(r,\phi))+\frac{r\sqrt{H}}{2}(1+\cos\phi)\right],
\end{aligned} 
\end{equation}
so that, using \eqref{ExpressionVPlanarCase}, we can write
\begin{equation}
 {\bf v}=\nabla_{\bf x} S(r,\phi),\qquad S(r,\phi)=\pm \frac{q}{\sqrt{H}}K(\rho(r,\phi))-\frac{r\sqrt{H}}{2}(1-\cos\phi).
\end{equation}

\subsubsection{Towards Angle-Action variables}\label{ssec:AngleAction}
In general geometry, for a given ``action'' $\a=\v_\infty$ (with $\abs{\a}^2=H$) the dimensionless (see \eqref{eq:dims}) function $\rho$ generalizes as
\begin{equation}\label{eq:def_rho2}
 \rho(\x,\a)=\frac{\abs{\a}}{2q}(\abs{\x}\abs{\a}+\a\cdot\x),
\end{equation}
and by the above we have established the following lemma:
\begin{lemma}\label{lem:gen_functions}
The functions
\begin{equation}\label{GeneratingFunction}
\begin{split}
\mathcal{S_\pm}({\bf x},{\bf a})&=\pm\frac{q}{\vert{\bf a}\vert}K(\frac{\vert{\bf a}\vert}{2q}({\bf x}\cdot{\bf a}+\vert{\bf x}\vert\vert{\bf a}\vert))-\frac{\vert{\bf x}\vert\vert{\bf a}\vert-{\bf x}\cdot{\bf a}}{2}
\end{split}
\end{equation}
with $K$ defined in \eqref{DefK} are generating functions in the sense that if $\x,\a\in\R^3$ are given with $\rho(\x,\a)\geq 1$, then
\begin{equation}
 {\bf v}_\pm=\nabla_{\bf x}\mathcal{S}_\pm
\end{equation}
define velocities corresponding to trajectories of the ODE \eqref{ODE} passing through $\x$ with asymptotic velocity $\a$. Moreover, these are the only such velocities.

In addition, the generating functions preserve the angular momentum in the sense that for the aforementioned trajectories their angular momenta are given by
\begin{equation}\label{ConservationMomentumPlane}
{\bf L}_\pm={\bf x}\times\nabla_\x\mathcal{S}_\pm=\nabla_\a\mathcal{S}_\pm\times{\bf a},
\end{equation}
and we have that $\abs{{\bf L}_+}\geq\abs{{\bf L}_-}$.
\end{lemma}

We are now ready to prove Proposition \ref{PropAA}.
\begin{proof}[Proof of Proposition \ref{PropAA}]
We give first the explicit definition of the change of variables, with some additional explicit formulas.

\subsubsection*{The fold}

Lemma \ref{lem:gen_functions} gives us two local diffeomorphisms defined through the scattering problem. We also see from Lemma \ref{lem:trajectories} that the mapping $({\bf x},{\bf a})\mapsto {\bf v}$ has a fold, but we can hope to define a nice change of variable on either side $\Omega_p$, $\Omega_f$ of the set $\Gamma$ defined by
\begin{equation}\label{DefGamma1}
\begin{split}
\Gamma:=\{{\bf x}\cdot{\bf v}=-\sqrt{H}L^2/q\},\qquad\Omega_p:=\{{\bf x}\cdot{\bf v}<-\sqrt{H}L^2/q\},\qquad\Omega_f:=\{{\bf x}\cdot{\bf v}>-\sqrt{H}L^2/q\},
\end{split}
\end{equation}
and we choose the generating function $\mathcal{S}_-$ in $\Omega_p$ and the generating function $\mathcal{S}_+$ in $\Omega_f$. From Lemma \ref{lem:trajectories}, we see that this corresponds to selecting the trajectory in the past of $\Gamma$ when $({\bf x},{\bf v})\in\Omega_p$, and the trajectory in the future of $\Gamma$ when $({\bf x},{\bf v})\in\Omega_f$. Note also that since, along trajectories, $d({\bf x}\cdot{\bf v})/dt>H/2$, each trajectory meets $\Omega_p$ and $\Omega_f$ exactly once.

\subsubsection*{Construction of $\mathcal{T}$}
To construct our angle-action variables, we note that for given $(\x,\v)\in\R^3\times\R^3$, we can use the conservation laws to find the corresponding asymptotic velocity and thus action $\mathcal{A}(\x,\v)={\bf v}_\infty$: Using \eqref{ConsLaws2dPlanar2} we see that $\vert\mathcal{A}\vert=\sqrt{H}$, that ${\bf L}\cdot\mathcal{A}=0$ and that ${\bf R}\cdot\mathcal{A}=q\sqrt{H}/2$, so that (since ${\bf L\cdot R}=0$)
\begin{equation}\label{ExplicitA1}
\begin{split}
\mathcal{A}({\bf x},{\bf v})&=\frac{2q\sqrt{H}}{4HL^2+q^2}{\bf R}+\frac{4H}{4HL^2+q^2}{\bf L}\times {\bf R}.
\end{split}
\end{equation}
By Lemma \ref{lem:trajectories}, $\rho({\bf x},\mathcal{A}({\bf x},{\bf v}))\ge 1$ and we can define the corresponding angle as 
\begin{equation}\label{DefTheta}
 \Theta(\x,\v):=\begin{cases}
 \nabla_\a\mathcal{S}_-(\x,\mathcal{A}(\x,\v))\,\,&\hbox{ for }({\bf x},{\bf v})\in\Omega_p,\\
 \nabla_\a\mathcal{S}_+(\x,\mathcal{A}(\x,\v))\,\,&\hbox{ for }({\bf x},{\bf v})\in\Omega_f.
  \end{cases}
\end{equation}
This is well-defined by Lemma \ref{lem:gen_functions} and extends by continuity to $\Gamma$ to give a mapping $(\x,\v)\mapsto(\vartheta,\a)$ continuous\footnote{In fact any choice of sign for $\mathcal{S}_\pm$ on $\Omega_f$ or $\Omega_p$ would give a continuous mapping, smooth away from the fold $\Gamma$, but our choice will give a {\it smooth gluing} at the fold.} on $\mathcal{P}_{{\bf x},{\bf v}}$.

\subsubsection*{Rescaling of the generating function}

Looking at the action of scaling on the generating function,
\begin{equation}
 \mathcal{S}_\iota(\lambda{\bf x},\lambda^{-1}{\bf a})=\iota\,\lambda\frac{q}{\vert{\bf a}\vert}K(\lambda^{-1}\frac{\vert{\bf a}\vert}{2q}({\bf x}\cdot{\bf a}+\vert{\bf x}\vert\vert{\bf a}\vert))-\frac{{\bf x}\cdot{\bf a}-\vert{\bf x}\vert\vert{\bf a}\vert}{2},\qquad \iota\in\{+,-\},
\end{equation}
and differentiating at $\lambda=1$ yields that
\begin{equation}\label{eq:xv-vs-thetaa}
{\bf x}\cdot{\bf v}-{\bf a}\cdot\vartheta=\iota\frac{q}{\vert{\bf a}\vert}\left[K(\rho)-\rho K^\prime(\rho)\right].
\end{equation}
Noting that 
\begin{equation}
 \frac{d}{dx}\left[K(x)-xK^\prime(x)\right]=-xK^{\prime\prime}(x)=-\frac{1}{2\sqrt{x(x-1)}},\quad K(1)-K'(1)=0,
\end{equation}
it follows that if $\rho(\x,\mathcal{A}(\x,\v))>1$, then
\begin{equation}\label{SignSigma}
 \begin{cases}{\bf x}\cdot{\bf v}-\vartheta\cdot{\bf a}< 0,& \iota=+,\\ {\bf x}\cdot{\bf v}-\vartheta\cdot{\bf a}> 0,& \iota=-,\end{cases}
\end{equation}
whereas $\x\cdot\v=\vartheta\cdot\a$ if and only if $\rho=1$, in which case $\mathcal{S}_+=\mathcal{S}_-$. The function $\sigma=(\vert{\bf a}\vert/q)({\bf x}\cdot{\bf v}-{\bf a}\cdot\vartheta)$ will be used to resolve the fold degeneracy corresponding to the choice of $\iota\in\{+,-\}$.

\subsubsection*{Inverse of $\mathcal{T}$}

We can now define $\mathcal{T}^{-1}$ on either side of the (image of the) fold, consistent with \eqref{DefGamma1} and \eqref{SignSigma}:
\begin{equation}\label{DefGamma}
\begin{split}
 \Gamma:=\left\{\vartheta\cdot {\bf a}=-aL^2/q\right\},\qquad  \Omega_{-}:=\left\{{\bf a}\cdot\vartheta<-aL^2/q\right\},\quad\Omega_{+}:=\left\{{\bf a}\cdot\vartheta>-aL^2/q\right\},
\end{split}
\end{equation}
so that (for points with $\rho(\x,\a)\geq 1$)
\begin{equation}\label{eq:vthetadef}
({\bf v},\vartheta)=
\begin{cases}
(\nabla_{\bf x}\mathcal{S}_-({\bf x},{\bf a}),\nabla_{\bf a}\mathcal{S}_-({\bf x},{\bf a})),&\hbox{ when }(\vartheta,{\bf a})\in \Omega_{-},\,\,({\bf x},{\bf v})\in\Omega_p,\\
(\nabla_{\bf x}\mathcal{S}_+({\bf x},{\bf a}),\nabla_{\bf a}\mathcal{S}_+({\bf x},{\bf a})),&\hbox{ when }(\vartheta,{\bf a})\in \Omega_{+},\,\,({\bf x},{\bf v})\in\Omega_f.
\end{cases}
\end{equation}

We now compute the inverse transformation $(\vartheta,\a)\mapsto (\X,\V)$. This is more challenging since both evolve over the trajectory, while the only trace of ``time'' in $(\vartheta,\a)$ variables follows from $\vartheta\cdot{\bf a}$. 
Using the conservation laws, we can already express trajectories in terms of the ``time-proxy'' ${\bf x}\cdot{\bf v}$.
We have that
\begin{equation}\label{AxAXperp}
\begin{split}
\mathcal{A}\cdot{\bf x}&=\frac{1}{4HL^2+q^2}\left[4HL^2{\bf x}\cdot{\bf v}+q^2\sqrt{H}\vert{\bf x}\vert+2qL^2\sqrt{H}\right],\\
\mathcal{A}\cdot({\bf L}\times{\bf x})&=\frac{L}{4HL^2+q^2}\left[-2q\sqrt{H}L{\bf x}\cdot{\bf v}+2qHL\vert{\bf x}\vert+4HL^3\right],
\end{split}
\end{equation}
and since
\begin{equation}\label{eq:cons_sizes}
 H=\vert{\bf v}\vert^2+\frac{q}{\vert{\bf x}\vert},\qquad L^2=\vert{\bf x}\vert^2\vert{\bf v}\vert^2-({\bf x}\cdot{\bf v})^2\quad\Rightarrow\quad L^2=\vert {\bf x}\vert^2H-q\vert{\bf x}\vert-({\bf x}\cdot{\bf v})^2
\end{equation}
we can express $\vert{\bf x}\vert$ in terms of $({\bf x}\cdot{\bf v})=\vartheta\cdot{\bf a}+\sigma$ as
\begin{equation}\label{Formula|x|}
\begin{split}
\vert{\bf x}\vert&=\frac{q}{2a^2}\left[1+\sqrt{\frac{4a^2L^2+q^2}{q^2}+4\frac{a^2}{q^2}\left({\bf x}\cdot{\bf v}\right)^2}\right],
\end{split}
\end{equation}
which gives a first formula for $\X$ as
\begin{equation}\label{ExpressionsX}
\begin{split}
{\bf X}&=\frac{q}{a^2}\left[\frac{1}{2}+\frac{q^2}{(2R)^2}\sqrt{\frac{a^2}{q^2}({\bf x}\cdot{\bf v})^2+\frac{(2R)^2}{4q^2}}+\frac{4a^2L^2}{(2R)^2}\frac{a}{q}({\bf x}\cdot{\bf v})\right]\frac{{\bf a}}{a}\\
&\quad-2\left[\frac{1}{2}+\frac{q^2}{(2R)^2}\left(\sqrt{\frac{a^2}{q^2}({\bf x}\cdot{\bf v})^2+\frac{(2R)^2}{4q^2}}-\frac{a}{q}({\bf x}\cdot{\bf v})\right)\right]\frac{{\bf L}\times{\bf a}}{a^2}.
\end{split}
\end{equation}
To give explicit formulas for $\V$, we use that by the conservation laws there holds
\begin{equation}\label{NewInnerProducts}
\begin{split}
\v\cdot\mathcal{A}&=\frac{q^2}{(2R)^2}\frac{a({\bf x}\cdot{\bf v})}{\vert {\bf x}\vert}+\frac{4a^2L^2}{(2R)^2}\left(a^2-\frac{q}{2\vert{\bf x}\vert}\right),\qquad
\v\cdot ({\bf L}\times\mathcal{A})=\frac{qaL^2}{(2R)^2}\left(2a^2-\frac{q}{\vert{\bf x}\vert}-\frac{2a({\bf x}\cdot{\bf v})}{\vert{\bf x}\vert}\right),
\end{split}
\end{equation}
so that since $\v\cdot{\bf L}=0$ we obtain
\begin{equation}
\begin{aligned}
{\bf V}&=\left[\frac{q^2}{(2R)^2}\frac{{\bf x}\cdot{\bf v}}{\vert{\bf x}\vert}+\frac{4a^2L^2}{(2R)^2}(a-\frac{q}{2a\vert{\bf x}\vert})\right]\frac{\bf a}{a}+\frac{2q^2}{(2R)^2}\left[a-\frac{q}{2a\vert{\bf x}\vert}-\frac{{\bf x}\cdot{\bf v}}{\vert{\bf x}\vert}\right]\frac{{\bf L}\times{\bf a}}{q}\\
&={\bf a}-\frac{q^2}{(2R)^2}\frac{a\vert{\bf x}\vert-{\bf x}\cdot{\bf v}}{\vert{\bf x}\vert}\left[\frac{\bf a}{a}-2\frac{{\bf L}\times{\bf a}}{q}\right]-\frac{q^2}{(2R)^2}\frac{q}{2a\vert{\bf x}\vert}\left[\frac{4a^2L^2}{q^2}\frac{\bf a}{a}+2\frac{{\bf L}\times{\bf a}}{q}\right],
\end{aligned}
\end{equation}
which together with \eqref{Formula|x|} yields
\begin{equation}\label{ExpressionsV}
\begin{split}
{\bf V}&={\bf a}-a\frac{q^2}{(2R)^2}\frac{\sqrt{\frac{a^2}{q^2}({\bf x}\cdot{\bf v})^2+\frac{a^2L^2}{q^2}+\frac{1}{4}}-\frac{a}{q}({\bf x}\cdot{\bf v})+\frac{1}{2}}{\sqrt{\frac{a^2}{q^2}({\bf x}\cdot{\bf v})^2+\frac{a^2L^2}{q^2}+\frac{1}{4}}+\frac{1}{2}}\left[\frac{\bf a}{a}-2\frac{{\bf L}\times{\bf a}}{q}\right]\\
&\quad-\frac{q^2}{(2R)^2}\frac{a}{2\sqrt{\frac{a^2}{q^2}({\bf x}\cdot{\bf v})^2+\frac{a^2L^2}{q^2}+\frac{1}{4}}+1}\left[\frac{4a^2L^2}{q^2}\frac{\bf a}{a}+2\frac{{\bf L}\times{\bf a}}{q}\right].
\end{split}
\end{equation}

\subsubsection*{Global functions}

We can now introduce the function
\begin{equation}\label{def:sigma}
\begin{split}
\sigma(\vartheta,{\bf a}):=\frac{a}{q}({\bf x}\cdot{\bf v}-\vartheta\cdot{\bf a})=\iota\left[K(\rho)-\rho K^\prime(\rho)\right]=-\iota \ln(\sqrt{\rho}+\sqrt{\rho-1}),\qquad (\vartheta,{\bf a})\in\Omega_\iota.
\end{split}
\end{equation}
In particular, $\vert\sigma\vert$ defines $\rho$ uniquely and vice-versa, which shows that $\rho$ can indeed be defined as a function of $(\vartheta,{\bf a})$ or as a function of $({\bf x},{\bf v})$.
We obtain the global function
\begin{equation}\label{ExpressionsX2}
\begin{split}
{\bf X}(\vartheta,{\bf a})&=\vartheta+\frac{q}{a^2}\left(\frac{1}{2}+\sigma\right)\frac{\bf a}{a}+ \frac{q}{a^2}\frac{q^2}{(2R)^2}D\left(\frac{a}{q}\vartheta\cdot{\bf a}+\sigma\right)\cdot \frac{2{\bf R}}{q},\\
D(y)&:=\sqrt{y^2+(2R)^2/4q^2}-y.
\end{split}
\end{equation}

\subsubsection*{Properties of $\mathcal{T}$}
Having given the detailed construction of the diffeomorphism $\mathcal{T}$, we can quickly deduce the properties listed in Proposition \ref{PropAA}: by construction via a generating function we directly have \eqref{eq:canonical} and the change of variables is canonical as in \eqref{it:canonical}. The compatibility with conservation laws \eqref{it:cons-comp} follows from the definition. For a trajectory $(\x(t),\v(t))$ of the Kepler problem \eqref{ODE} we note that by \eqref{ExplicitA1} $\mathcal{A}(\x(t),\v(t))$ is independent of time, whereas by \eqref{eq:vthetadef} we have
\begin{equation}
\begin{aligned}
 \partial_t\Theta_j(\x(t),\v(t))&=\partial_t\partial_{\a_j}S_\iota(\x(t),\mathcal{A}(\x(t),\v(t)))=\partial_{\x_k}\partial_{\a_j}S_\iota(\x(t),\mathcal{A}(\x(t),\v(t)))\partial_{\x_k}S_\iota(\x(t),\mathcal{A}(\x(t),\v(t)))\\
 &=\frac{1}{2}\partial_{\a_j}\left(\mathcal{A}(\x(t),\v(t))^2-\frac{q}{\abs{\x(t)}}\right)=\mathcal{A}_j(\x(t),\v(t)),
\end{aligned} 
\end{equation}
and \eqref{it:lin-diag} is established. Finally, the asymptotic action property \eqref{it:asymp-act} follows from \eqref{ExpressionsX} and \eqref{ExpressionsV} after some more quantitative bounds on $\sigma$ below in Lemma \ref{EstimRho}.

\end{proof}

\subsubsection{The functions $\rho,\sigma$ and ${\bf x}\cdot{\bf v}$}
As we saw in Section \ref{ssec:Kepler}, the function $\rho$ of \eqref{eq:def_rho2}
plays an important role. It is naturally defined in terms of the mixed variables $({\bf x},{\bf a})$, but we can estimate it in terms of $(\vartheta,{\bf a})$ through an implicit relation. We introduce the functions
\begin{equation}\label{eq:defGP}
\begin{split}
G(y)&:=\sqrt{y(y-1)}+\ln(\sqrt{y}+\sqrt{y-1}),\qquad P_\pm(y):=2y-1\pm2\sqrt{y(y-1)},\qquad y\geq 1,
\end{split}
\end{equation}
and note that $G$ appears naturally in the study of the radial case in \cite{PW2020}, where $\rho_{rad}=a^2r/q=r/r_{min}$ plays an important role.
\begin{lemma}\label{EqRho}
The functions
\begin{equation}\label{def:tau}
 \rho=\frac{a}{2q}({\bf x}\cdot{\bf a}+\vert{\bf x}\vert\vert{\bf a}\vert)\quad\textnormal{and}\quad \eta=\frac{a}{q}\vartheta\cdot\a
\end{equation}
are related by the equation
\begin{equation}\label{RelTauOut}
 \eta-\iota G(\rho)+\kappa^2 P_{-\iota}(\rho)=0\quad\textnormal{ in }\Omega_{\iota},\qquad \kappa:=aL/q,\qquad\iota\in\{+,-\}.
\end{equation}

\end{lemma}
All three quantities $\rho, \eta,\kappa$ are dimensionless, as can be seen from \eqref{eq:dims}.

\begin{proof}[Proof of Lemma \ref{EqRho}]

We start with an expression for $\rho$ which follows from its definition in \eqref{eq:def_rho2} via the first line in \eqref{AxAXperp} and \eqref{Formula|x|}:
\begin{equation}\label{EqForRho1}
\begin{split}
\rho
%&=\frac{2a^2L^2}{4a^2L^2+q^2}+\frac{4a^2L^2}{4a^2L^2+q^2}\cdot \frac{a}{q}({\bf x}\cdot{\bf v})+\frac{1}{2}\frac{4a^2L^2+2q^2}{4a^2L^2+q^2}\left(1+\sqrt{\frac{4a^2L^2+q^2}{q^2}+4\frac{a^2}{q^2}({\bf x}\cdot{\bf v})^2}\right)\\
&=\frac{1}{2}+\frac{2a^2L^2}{4a^2L^2+q^2}\cdot \frac{a}{q}({\bf x}\cdot{\bf v})+\frac{2a^2L^2+q^2}{4a^2L^2+q^2}\sqrt{\frac{a^2}{q^2}({\bf x}\cdot{\bf v})^2+\frac{4a^2L^2+q^2}{4q^2}}.
\end{split}
\end{equation}
In order to simplify the computations, we introduce the notations $\beta$, $k$ such that
\begin{equation*}
\begin{split}
\beta=\ln(2R/q),\qquad \sinh(-\beta+k)=\frac{2q}{2R}\frac{a}{q}({\bf x}\cdot{\bf v}),
\end{split}
\end{equation*}
so that
\begin{equation*}
\cosh(\beta)=\frac{2a^2L^2+q^2}{q\sqrt{4a^2L^2+q^2}},\qquad\sinh(\beta)=\frac{2a^2L^2}{q\sqrt{4a^2L^2+q^2}},
\end{equation*}
and \eqref{EqForRho1} gives that
\begin{equation}\label{EqForRho2}
\begin{split}
\rho-1=\frac{\cosh(k)-1}{2}=\sinh^2(k/2).
\end{split}
\end{equation}
Finally, we can rewrite the equation defining $k$ to get
\begin{equation*}
\begin{split}
\frac{a}{q}({\bf x}\cdot{\bf v})&=\frac{1}{2}e^\beta\sinh(-\beta+k)=\frac{e^{2\beta}+1}{2}\sinh(k/2)\cosh(k/2)-\frac{e^{2\beta}-1}{4}(2\cosh^2(k/2)-1).
\end{split}
\end{equation*}
Plugging in \eqref{EqForRho2} and using that $\cosh^2(k/2)=\rho$, we get that
\begin{equation*}
\begin{split}
\frac{a}{q}({\bf x}\cdot{\bf v})&=\pm\frac{e^{2\beta}+1}{2}\sqrt{\rho(\rho-1)}-\frac{e^{2\beta}-1}{4}(2\rho-1).
\end{split}
\end{equation*}
Furthermore, we recall from \eqref{eq:xv-vs-thetaa} that there holds
\begin{equation*}
\begin{split}
\frac{a}{q}\left({\bf x}\cdot{\bf v}-\vartheta\cdot{\bf a}\right)=-\iota\ln(\sqrt{\rho}+\sqrt{\rho-1})\qquad\hbox{ in }\Omega_{\iota},
\end{split}
\end{equation*}
so that combining the two equations above, we finally find that for two choices of signs $\iota_1,\iota_2\in\{+,-\}$ (both of which can a priori change depending on the region $\Omega_{-}$, $\Omega_{+}$)
\begin{equation*}
\begin{split}
\eta&=\iota_1\ln\left(\sqrt{\rho}+\sqrt{\rho-1}\right)+\iota_2\sqrt{\rho(\rho-1)}-\frac{e^{2\beta}-1}{4}\left(2\rho-1-\iota_22\sqrt{\rho(\rho-1)}\right).
\end{split}
\end{equation*}
Now we observe that when $\beta=0$, we are in the radial case and, in this case, we know that we must have $\iota_1=\iota_2$.
\end{proof}

We can now verify that $\rho$ and $\sigma$ are well-defined functions on phase space.
\begin{lemma}\label{EstimRho}
With $\kappa=aL/q$, the relation \eqref{RelTauOut} defines a $C^1$ map $\rho(\eta,\kappa)$, with $\rho\ge \rho(-\kappa^2,\kappa)=1$. For fixed $\kappa\ge0$, $\eta\mapsto\rho(\eta,\kappa)$, $\mathbb{R}\to[1,\infty)$ is $2$ to $1$, decreasing for $-\infty<\eta<-\kappa^2$ and increasing for $\eta\ge-\kappa^2$. 

Moreover, we have the following estimates: On $\Omega_{+}=\{\eta>-\kappa^2\}$ there holds
\begin{equation}\label{DerRhoOmegaOut}
\frac{1}{4}(\kappa+\eta)\le\rho\le1+\kappa^2+\eta, \quad(\rho\vert\eta\vert\le \kappa^2,\; 1\le\rho\le 1+\kappa,\text{ when } -\kappa^2\le\eta\le 0)
\end{equation}
while on $\Omega_{-}=\{\eta<-\kappa^2\}$ we have
\begin{equation}\label{DerRhoOmegaIn}
\begin{split}
\frac{1}{4}\frac{\vert\eta\vert}{1+\kappa^2}\le\rho\le\frac{1+\vert\eta\vert}{1+\kappa^2}.
\end{split}
\end{equation}

As a consequence, the function $\sigma(\eta,\kappa)=-\iota\ln(\sqrt{\rho}+\sqrt{\rho-1})$ is well defined, and for fixed $\kappa\ge 0$,  $\eta\mapsto \sigma(\eta,\kappa)$ is a bijection on $\mathbb{R}$, and we have the bounds
\begin{equation}\label{BoundsOnSigma}
\begin{split}
\vert\sigma\vert\le \ln(1+2\sqrt{\vert\eta\vert}+2\kappa),\qquad(\rho+\kappa)\vert\partial_\eta\sigma\vert+(\kappa+\kappa^{-1})\vert\partial_\kappa\sigma\vert\lesssim 1,\\
(\rho+\kappa)^2\vert\partial_\eta^2\sigma\vert+\langle\kappa\rangle(\rho+\kappa)\vert\partial_\eta\partial_\kappa\sigma\vert^2+\langle\kappa\rangle^2\vert\partial_\kappa^2\sigma\vert\lesssim 1. 
\end{split}
\end{equation}
Besides, on $\Omega_+$, we obtain slightly improved bounds:
\begin{equation}\label{ImprovedBoundsOnSigma}
\begin{split}
\vert\partial_\kappa\sigma\vert\lesssim \frac{\kappa}{\rho^2+\kappa^2},\qquad\vert\partial_\kappa^2\sigma\vert\lesssim\frac{1}{\rho^2+\kappa^2}.
\end{split}
\end{equation}
\end{lemma}

\begin{proof}[Proof of Lemma \ref{EstimRho}]
With the functions $G,P_\iota$ of \eqref{eq:defGP} we define the implicit function $F_\iota:\R\times[0,\infty)\times[1,\infty)\to\R$ by
\begin{equation*}
\begin{split}
 F_{\iota}(\eta,\kappa,y):=\eta-\iota G(y)+\kappa^2 P_{-\iota}(y),
\end{split}
\end{equation*}
so that per \eqref{RelTauOut}, $\rho$ is defined by $F_\iota(\eta,\kappa^2,\rho)=0$ on $\Omega_\iota$, $\iota\in\{-,+\}$. We note that
\begin{equation}\label{eq:derivGP}
 G(1)=0,\quad G^\prime(y)=\sqrt{\frac{y}{y-1}},\quad P_\iota(1)=1,\quad P_-(y)P_+(y)=1,\quad P'_{\iota}(y)=\iota\frac{P_\iota(y)}{\sqrt{y(y-1)}},
\end{equation}
and we have the bounds (see \cite[Lemma 2.2]{PW2020}),
\begin{equation}\label{EstimGP+P-}
\begin{aligned}
 y-1\le G(y)&\le y+\ln(2\sqrt{y})\le 2y,\\
(4y)^{-1}\le P_-(y)\le y^{-1}\le P_-(1)&=1=P_+(1)\le y\le P_+(y)\le 4y.
\end{aligned}
\end{equation}
Hence we see that when $F_\iota=0$, we have $\eta=-\kappa^2$ if and only if $y=1$. For $\eta<-\kappa^2$, we see that
\begin{equation*}
\begin{split}
F_{-}(\eta,\kappa,1)=\eta+\kappa^2<0,\qquad\lim_{y\to+\infty}F_{-}(\eta,\kappa,y)=+\infty,
\end{split}
\end{equation*}
and similarly with reversed signs for $F_{+}$ when $\eta>-\kappa^2$, so $F_{-}=0$ has at least one solution for $\eta<-\kappa^2$ and $F_{+}=0$ has at least one solution for $\eta>-\kappa^2$. We now compute the derivatives
\begin{equation*}
\begin{split}
\partial_\eta F_{\iota}&=1,\qquad\partial_\kappa F_{\iota}=2\kappa P_{-\iota}(y),
\end{split}
\end{equation*}
and by \eqref{eq:derivGP}
\begin{equation*}
\begin{split}
\partial_y F_{\iota}&=-\iota\sqrt{\frac{y}{y-1}}\left[1+\kappa^2\frac{P_{-\iota}(y)}{y}\right],
\end{split}
\end{equation*}
so that by \eqref{EstimGP+P-}
\begin{equation*}
\begin{split}
-\partial_y F_{+}\ge\left[1-1/y\right]^{-\frac{1}{2}}\left[1+(\kappa/2y)^2\right],\qquad\partial_yF_{-}\ge\left[1-1/y\right]^{-\frac{1}{2}}\left[1+\kappa^2\right].
\end{split}
\end{equation*}
This shows that $F_{\iota}(\eta,\kappa,y)=0$ has a unique solution and gives the bounds on the first derivatives in \eqref{BoundsOnSigma}. Explicitly we have that
\begin{equation}\label{eq:deriv_rho}
\begin{split}
  \partial_\eta\rho&=\frac{1}{\iota G'(\rho)-\kappa^2 P'_{-\iota}(\rho)}=\iota\frac{\sqrt{\rho(\rho-1)}}{\rho+\kappa^2 P_{-\iota}(\rho)},\qquad \partial_\kappa\rho=2\kappa P_{-\iota}(\rho)\partial_\eta\rho,\\
  \partial_\eta\sigma&=-\frac{1}{2(\rho+\kappa^2 P_{-\iota}(\rho))},\qquad\partial_\kappa\sigma=-\frac{\kappa P_{-\iota}(\rho)}{\rho+\kappa^2 P_{-\iota}(\rho)},
  \end{split}
 \end{equation}
which shows that the gradient of $\rho$ vanishes at the ``curve of surgery'' $\Gamma=\{\rho=1\}$ and that the gradient of $\sigma$ is smooth there.

The estimates for $\rho$ follow from the bounds
\begin{equation*}
\eta+\kappa^2 y+y-1\le F_{-}(\eta,\kappa,y)\le\eta+4\kappa^2 y+2y,\qquad \eta+\frac{\kappa^2}{4y}-2y\le F_{+}(\eta,\kappa,y)\le\eta+\frac{\kappa^2}{y}+1-y,
\end{equation*}
which in turn follow from the definitions of $F_{\iota}$ and the bounds \eqref{EstimGP+P-}.

Finally, we compute the second order derivatives of $\sigma$ by deriving \eqref{eq:deriv_rho} one more time. This gives
\begin{equation}
\begin{aligned}
 \partial_{\eta}^2\sigma&=\frac{-\kappa^2 P_{-\iota}(\rho)+\iota\sqrt{\rho(\rho-1)}}{2(\rho+\kappa^2 P_{-\iota}(\rho))^3},\qquad
 \partial_{\eta}\partial_\kappa\sigma
 %=\kappa P_{-\iota}(\rho)\frac{\rho+\iota\sqrt{\rho(\rho-1)}}{(\rho+\kappa^2 P_{-\iota}(\rho))^3}
 =\kappa \frac{1+P_{-\iota}(\rho)}{2(\rho+\kappa^2 P_{-\iota}(\rho))^3},\\
 \partial_{\kappa}^2\sigma
 %&=-\frac{(\rho-\kappa^2P_{-\iota})P_{-\iota}-2\kappa P_{-\iota}(\kappa P_{-\iota}+\kappa \rho P_{-\iota}^\prime)\partial_\eta\rho}{(\rho+\kappa^2 P_{-\iota})^2}\\
 &=-\frac{P_{-\iota}(\rho)}{(\rho+\kappa^2 P_{-\iota}(\rho))^3}(\rho^2-\kappa^2(1+P_{-\iota}(\rho))-\kappa^4 P_{-\iota}^2(\rho)).\\
 %&=\frac{\sigma}{2}\frac{P'_{-\sigma}(\rho)\partial_\omega\rho}{\rho+\omega P_{-\sigma}(\rho)}-\frac{\sigma}{2}(P_{-\sigma}(\rho))^2\partial_\tau^2\ell=-\frac{\sigma}{2}(P_{-\sigma}(\rho))^2\frac{\rho+\sigma\sqrt{\rho(\rho-1)}}{(\rho+\omega P_{-\sigma}(\rho))^3},
\end{aligned}
\end{equation}
and direct computations give the bounds on the second derivatives in \eqref{BoundsOnSigma}. The bounds \eqref{ImprovedBoundsOnSigma} follow by direct inspection using that $\iota=+$.
\end{proof}

\begin{remark}
It follows from \eqref{eq:deriv_rho} that the only smooth matching of action-angle variable at $\Gamma$ are the one that change the sign of $\mathcal{S}_\pm$.
\end{remark}

\subsection{Further coordinates}\label{ssec:furthercoords}
In order to parametrize the trajectories of \eqref{ODE}, further choices of coordinates will be important.

\begin{figure}
\captionsetup[subfigure]{labelfont=rm}
    \centering
    \subfloat[\centering trajectories (in red) for fixed $\sqrt{H}=1$, $\v_\infty=(1,0,0)$, varying $L$]{{\includegraphics[width=.45\textwidth]{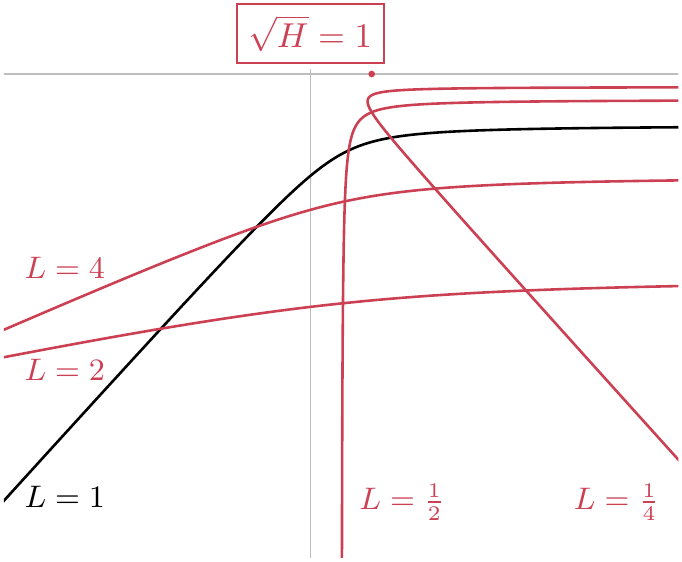} }}
    \qquad
    \subfloat[\centering trajectories (in blue) for fixed $L=1$, varying $\sqrt{H}$ (and thus $\v_\infty=(\sqrt{H},0,0)$)]{\includegraphics[width=.45\textwidth]{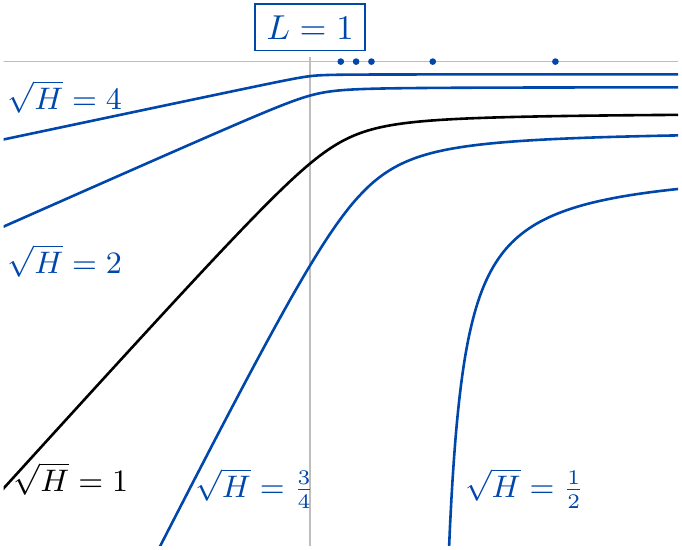}} 
    \caption{A foliation of planar dynamics of \eqref{ODE} near a given trajectory (in black) with $\sqrt{H}=1$, asymptotic velocity $\v_\infty=(1,0,0)$ and $L=1$.}
    \label{fig:sample-trajectories}
\end{figure}

\subsubsection{Super-integrable coordinates}\label{ssec:SIC}

The Kepler problem \eqref{ODE} is super integrable, and hence in an appropriate system of coordinates, only one scalar function evolves along a trajectory. When we consider derivatives, it will be crucial to use this simplification. Inspecting \eqref{LinearizationFlow}, we see that such a system can be obtained from our asymptotic action-angle coordinates by using $(\xi,\eta,\lambda,{\bf u},\L)$ where
\begin{equation}\label{SIC}
\begin{split}
\xi:=\frac{q}{a},\qquad\eta:=\frac{a}{q}\vartheta\cdot{\bf a},\qquad\lambda:=L,\qquad{\bf u}:=\frac{{\bf a}}{a},
\end{split}
\end{equation}
and $\L$ is the angular momentum, for the direction of which we write $\l=\frac{\L}{L}$. In particular, note that $\xi$, $\lambda$ and $\L$ have dimension\footnote{so that their Poisson brackets are dimensionless, e.g. $\dm{\{\xi,f\}}=\dm{f}$.} $\dm{\x}\dm{\v}$, while $\eta$ and $\u$ are dimensionless (compare \eqref{eq:dims}). Only $\eta$ evolves along a trajectory of \eqref{ODE}, namely as
\begin{equation}\label{eq:eta_evol}
\eta\mapsto\eta+t\frac{q^2}{\xi^3},
\end{equation}
and using that $\vartheta\perp\L$ (see e.g.\ \eqref{ConservationMomentumPlane}) we can recover our angle-action variables from the super-integrable coordinates via 
\begin{equation}\label{TransitionFromSIC}
\begin{split}
\vartheta=\frac{\xi^2}{q}\eta{\bf u}-\frac{\xi}{q}\L\times\u,\qquad{\bf a}=\frac{q}{\xi}{\bf u}.%\qquad \kappa=\frac{\lambda}{\xi},\qquad \frac{2{\bf R}}{q}={\bf u}-2\kappa\l\times\u.
\end{split}
\end{equation}
\begin{remark}
 Clearly the collection $(\xi,\eta,\lambda,{\bf u},\L)$ is not independent (since e.g.\ $\lambda=\abs{\L}$), and in a strict sense only gives coordinates modulo further conditions. However, the slight redundancy is convenient in that it provides relatively simple expressions for kinematic quantities, and satisfies favorable Poisson bracket properties -- see for example \eqref{PBSIC} below.
\end{remark}

Moreover, with \eqref{ConsLaws2dPlanar2} and \eqref{RelTauOut} we can write
\begin{equation}
  \kappa=\frac{\lambda}{\xi},\qquad \frac{2{\bf R}}{q}={\bf u}-2\kappa\l\times\u.
\end{equation}
With the notations
\begin{equation}\label{DefDN}
\begin{aligned}
 D(\kappa,y)&:=\sqrt{y^2+\kappa^2+\frac{1}{4}}-y, \qquad
 N(\kappa,y):=\sqrt{y^2+\kappa^2+\frac{1}{4}}+\frac{1}{2}=D(\kappa,y)+y+\frac{1}{2},
\end{aligned} 
\end{equation}
we can then use \eqref{ExpressionsX} resp.\ \eqref{ExpressionsX2} and \eqref{ExpressionsV} to see that with $\sigma=\sigma(\eta,\kappa)$ as in Lemma \ref{EstimRho} there holds
\begin{equation}\label{XVinSIC}
\begin{aligned}
 \X&=\frac{\xi^2}{q}\left(\eta+\sigma+\frac{1}{2}+\frac{D(\kappa,\eta+\sigma)}{1+4\kappa^2}\right)\u-\frac{\xi}{q}\left(1+\frac{2D(\kappa,\eta+\sigma)}{1+4\kappa^2}\right)\L\times\u,\\
 \V&=\frac{q}{\xi}\left(1-\frac{1}{2N(\kappa,\eta+\sigma)}-\frac{1}{1+4\kappa^2}\frac{D(\kappa,\eta+\sigma)}{N(\kappa,\eta+\sigma)}\right)\u+\frac{q}{\xi^2}\frac{2}{1+4\kappa^2}\frac{D(\kappa,\eta+\sigma)}{N(\kappa,\eta+\sigma)}\L\times\u.
\end{aligned} 
\end{equation}

\subsubsection{Past asymptotic action}\label{ssec:past_AA}
The angle-action variables of Section \ref{ssec:Kepler} are constructed such that the asymptotic action property \eqref{AsymptoticActions} holds for the ``future'' evolution, i.e.\ as $t\to+\infty$. However, we will also need to resolve the earlier ``past'' part of a trajectory, for which the direction can differ markedly from its evolution in the long run (see e.g.\ Figure \ref{fig:cons-laws}). For this, we will need the \emph{past} asymptotic action-angle coordinates $(\vartheta^{(-)},{\bf a}^{(-)})$, with inverse $({\bf X}^{(-)},{\bf V}^{(-)})$ defined in a similar way, except that we require instead that
\begin{equation*}
\lim_{t\to-\infty}{\bf V}^{(-)}(\vartheta^{(-)}+t{\bf a}^{(-)},{\bf a}^{(-)})={\bf a}^{(-)}.
\end{equation*}
Using that the trajectory is symmetric under reflexion from the plane spanned by ${\bf L}, {\bf R}$, we easily see that the past asymptotic velocity is given by
\begin{equation*}
{\bf a}^{(-)}=-\frac{2q\sqrt{H}}{4HL^2+q^2}{\bf R}+\frac{4H}{4HL^2+q^2}{\bf L}\times {\bf R}.
\end{equation*}
We can proceed as before using the solutions of the Hamilton-Jacobi equation $\mathcal{S}_\pm$. More precisely, we define the change of variable on $\Omega^{(-)}_-:=\{{\bf x}\cdot{\bf v}<L^2\sqrt{H}/q\}$ and $\Omega^{(-)}_+:=\{{\bf x}\cdot{\bf v}>L^2\sqrt{H}/q\}$ to be given by the generating functions $\mathcal{S}^{(-)}_{\iota}({\bf x},{\bf a}^{(-)})=-\mathcal{S}_{-\iota}({\bf x},-{\bf a}^{(-)})$, thus
\begin{equation}\label{DefPAA}
\begin{split}
\mathcal{A}^{(-)}({\bf x},{\bf v})&:=-\frac{2q\sqrt{H}}{4HL^2+q^2}{\bf R}+\frac{4H}{4HL^2+q^2}{\bf L}\times {\bf R},\\
\Theta^{(-)}({\bf x},{\bf v})&:=
\begin{cases}
\nabla_{{\bf a}}\mathcal{S}_+({\bf x},-\mathcal{A}^{(-)}({\bf x},{\bf v}))&\quad\hbox{ for }\,({\bf x},{\bf v})\in\Omega^{(-)}_-,\\
\nabla_{{\bf a}}\mathcal{S}_-({\bf x},-\mathcal{A}^{(-)}({\bf x},{\bf v}))&\quad\hbox{ for }\,({\bf x},{\bf v})\in\Omega^{(-)}_+,
\end{cases}
\end{split}
\end{equation}
and similarly define $({\bf X}^{(-)},{\bf V}^{(-)})$. We will need to understand the properties of the transition map $(\vartheta,{\bf a})\mapsto(\vartheta^{(-)},{\bf a}^{(-)})$. Using the conservation laws, we see that
\begin{equation*}
\begin{split}
a^{(-)}=a,\qquad{\bf L}^{(-)}=\vartheta^{(-)}\times{\bf a}^{(-)}={\bf L},\qquad\kappa^{(-)}=\kappa,\qquad \frac{2{\bf R}}{q}=-\frac{{\bf a}^{(-)}}{a}-\frac{2}{q}{\bf L}\times{\bf a}^{(-)},
\end{split}
\end{equation*}
and taking the dot product of the last equality with ${\bf a}$ and ${\bf L}\times{\bf a}$, we find that
\begin{equation*}
\begin{split}
\frac{{\bf a}^{(-)}\cdot{\bf a}}{a^2}=\frac{4\kappa^2-1}{4\kappa^2+1}
\end{split}
\end{equation*}
and therefore
\begin{equation*}
\begin{split}
\frac{{\bf a}^{(-)}}{a}&=\frac{4\kappa^2-1}{4\kappa^2+1}\frac{{\bf a}}{a}+\frac{4\kappa}{4\kappa^2+1}\frac{{\bf l}\times{\bf a}}{a}.
\end{split}
\end{equation*}
Once again, looking at the action of scaling, we observe that, in $\Omega^{(-)}_\iota$,
\begin{equation*}
\begin{split}
{\bf x}\cdot{\bf v}&=\vartheta^{(-)}\cdot{\bf a}^{(-)}-\iota\frac{q}{a}\ln(\sqrt{\rho^{(-)}}+\sqrt{\rho^{(-)}-1}),\\
\rho^{(-)}({\bf x},{\bf a}^{(-)})&=\rho({\bf x},-{\bf a}^{(-)})=\frac{a}{2q}(\vert{\bf x}\vert\vert{\bf a}^{(-)}\vert-{\bf a}^{(-)}\cdot{\bf x}).
\end{split}
\end{equation*}
Since we will transition between the past and future asymptotic actions at periapsis, we compute that
\begin{equation*}
\begin{split}
\frac{\bf a}{a}\cdot\frac{\bf x}{\vert{\bf x}\vert}&=-\frac{{\bf a}^{(-)}}{a}\cdot\frac{{\bf x}}{\vert{\bf x}\vert}=\cos\phi_{p}=\frac{1}{\sqrt{1+4\kappa^2}},\qquad\sin\phi_{p}=\frac{2\kappa}{\sqrt{1+4\kappa^2}}\\
\rho_p=\rho^{(-)}_p&=\frac{(r+1/r)^2}{4},\,\, r=(1+4\kappa^2)^\frac{1}{4},\qquad-\eta_p=\sigma_p=-\frac{1}{4}\ln(1+4\kappa^2)=\sigma^{(-)}_p=-\eta^{(-)}_p.
\end{split}
\end{equation*}
Together with their favorable Poisson bracket properties (see \eqref{eq:PB_SIC(-)} below), this motivates the following definition of the new coordinates as
\begin{equation}\label{eq:def_SIC(-)}
\begin{split}
\xi^{(-)}&:=-\xi,\qquad\eta^{(-)}=-\eta+\frac{1}{2}\ln(1+4\kappa^2),\qquad\lambda^{(-)}=\lambda,\\
{\bf u}^{(-)}&:=\frac{1-4\kappa^2}{4\kappa^2+1}\u-\frac{1}{\xi}\frac{4}{4\kappa^2+1}\L\times\u=\frac{-{\bf a}^{(-)}}{a},\qquad{\bf l}^{(-)}={\bf l},%\qquad{\bf w}^{(-)}={\bf l}^{(-)}\times{\bf u}^{(-)},
\end{split}
\end{equation}
and thus
\begin{equation*}
 \L^{(-)}\times\u^{(-)}=\frac{1-4\kappa^2}{4\kappa^2+1}\L\times\u+\xi\frac{4\kappa^2}{4\kappa^2+1}\u.
\end{equation*}
Moreover, we see that the evolution of \eqref{ODE} in these coordinates reads
\begin{equation*}
\eta^{(-)}(t+s)=\eta^{(-)}(t)+sq^2/(\xi^{(-)})^3,
\end{equation*}
and using that
\begin{equation*}
\begin{split}
\xi(\eta+\sigma)={\bf x}\cdot{\bf v}=\xi^{(-)}(\eta^{(-)}+\sigma^{(-)})
\end{split}
\end{equation*}
we also verify that for the formulas in \eqref{XVinSIC} we have
\begin{equation}\label{XVinSIC(-)}
\begin{split}
{\bf X}(\xi,\eta,\lambda,{\bf u},\L)&={\bf X}(\xi^{(-)},\eta^{(-)},\lambda^{(-)},{\bf u}^{(-)},\L^{(-)}),\\
{\bf V}(\xi,\eta,\lambda,{\bf u},\L)&=-{\bf V}(\xi^{(-)},\eta^{(-)},\lambda^{(-)},{\bf u}^{(-)},\L^{(-)}).
\end{split}
\end{equation}

\subsubsection{Far and close formulations}\label{ssec:farclose}
For weights that are not conserved by the linear flow and for derivatives, we will have to also work with the alternative formulation of Section \ref{ssec:pinnedframe} in terms of $\nu'$. The associated nonlinear unknown, which replaces $\gamma$ as in \eqref{NewNLUnknown_Intro}, is then
 \begin{equation}
  \gamma^\prime:=\nu^\prime\circ\mathcal{T}^{-1}\circ\Phi_t^{-1},\quad \gamma^\prime(\vartheta,{\bf a},t)=\nu({\bf X}(\vartheta+t{\bf a},{\bf a}),{\bf V}(\vartheta+t{\bf a},{\bf a})+\mathcal{W}(t),t),
 \end{equation}
 with $\mathcal{T}$ defined in Proposition \ref{PropAA} and $\Phi_t$ defined in \eqref{DefPhiFreeStreaming}, which satisfies (compare \eqref{NewHamiltonian2})
 \begin{equation}\label{NewNLEq'}
 \begin{split}
 \partial_t\gamma'+\{\mathbb{H}_4',\gamma'\}=0,\qquad \gamma'(t=0)=\nu_0',\qquad\mathbb{H}_4'=Q\psi(\widetilde{\bf X})-\dot{\mathcal{W}}\cdot \XX,
 \end{split}
 \end{equation} 
 with, as in \eqref{NewVP},
\begin{equation}
 \begin{aligned}
  \psi({\bf y},t)&=-\frac{1}{4\pi}\iint \frac{1}{\vert {\bf y}-\widetilde{\bf X}(\theta,\alpha)\vert}(\gamma')^2(\theta,\alpha,t)\, d\theta d\alpha=-\frac{1}{4\pi}\iint \frac{1}{\vert {\bf y}-\widetilde{\bf X}(\theta,\alpha)\vert}\gamma^2(\theta,\alpha,t)\, d\theta d\alpha,\\
  \dot{\mathcal{W}}&=\mathcal{Q}\nabla_x\psi(0,t),\qquad\mathcal{W}(t)\to_{t\to T^\ast}0.
 \end{aligned} 
\end{equation}
With the notation
\begin{equation}
 \Sigma_u:\mathcal{P}_{\x,\v}\to\mathcal{P}_{\x,\v},\;(\x,\v)\mapsto(\x,\v+u),\qquad u\in\R^3,
\end{equation}
by construction in \eqref{NewNLEq'} we have that
\begin{equation}\label{eq:defM_t}
 \gamma'=\gamma\circ\mathcal{M}_t,\qquad \mathcal{M}_t:=\Phi_t\circ\mathcal{T}\circ\Sigma_{\mathcal{W}(t)}\circ\mathcal{T}^{-1}\circ\Phi_t^{-1},
\end{equation}
and $\mathcal{M}_t:\mathcal{P}_{\vartheta,\a}\to\mathcal{P}_{\vartheta,\a}$ is a canonical diffeomorphism.

This distinction is relevant when $\XX=\X\circ\Phi_t^{-1}$ is relatively small resp.\ large, i.e.\ in the ``close'' and ``far'' regions
\begin{equation}
 \Omega_t^{cl}:=\{(\vartheta,\a):\vert\XX(\vartheta,\a)\vert\leq 10\ip{t}\},\qquad \Omega_t^{far}:=\{(\vartheta,\a):\vert\XX(\vartheta,\a)\vert\geq \ip{t}\},
\end{equation}
which decompose phase space $\mathcal{P}_{\vartheta,\a}=\Omega_t^{cl}\cup\Omega_t^{far}$. We note that
\begin{equation}\label{eq:XMinv}
 \XX\circ\mathcal{M}_t=\XX,
\end{equation}
so $\Omega_t^\ast$, $\ast\in\{cl,far\}$, are invariant under $\mathcal{M}_t$, and in particular
\begin{equation}\label{eq:E_clfar}
 \iint_{\Omega_t^\ast} F(\XX(\vartheta,\a))\gamma^2(\vartheta,\a)d\vartheta d\a=\iint_{\Omega_t^\ast} F(\XX(\vartheta,\a))(\gamma')^2(\vartheta,\a)d\vartheta d\a.
\end{equation}

If $\ww:\mathcal{P}_{\vartheta,\a}\to\R$ is a weight function, then we have that
\begin{equation}\label{eq:weightprimes}
 \ww\gamma'=(\ww'\gamma)\circ\mathcal{M}_t,\qquad \ww':=\ww\circ\mathcal{M}_t^{-1},
\end{equation}
and we have the following bounds between primed and unprimed weights (in the super-integrable coordinates \eqref{SIC} of Section \ref{ssec:SIC}):
\begin{lemma}\label{lem:prime_bds}
 Let $(a',\xi',\lambda',\eta'):=(a,\xi,\lambda,\eta)\circ\mathcal{M}^{-1}_t$. Then
 \begin{equation}\label{prime_bds}
 \begin{aligned}
  &\abs{a-a'}+q\abs{\xi-\xi'}/(\xi\xi^\prime)\lesssim \abs{\mathcal{W}(t)}\lesssim 1,
 \end{aligned}
 \end{equation}
 and, on $\Omega_t^{cl}$, there holds that
 \begin{equation}\label{prime_bds2}
 \begin{aligned}
  &\abs{\lambda-\lambda'}\lesssim \ip{t}\abs{\mathcal{W}(t)},\\
  &\abs{\eta-\eta'}\lesssim \ip{t}\abs{\mathcal{W}(t)}(a^2+(a')^2+a+a')+\ln(1+2\sqrt{(a+a')\ip{t}\abs{\mathcal{W}(t)}}).
 \end{aligned}
 \end{equation}
 In particular, moments in $a$ on $\gamma$, $\gamma^\prime$ are comparable in the sense that
 \begin{equation}\label{ComparableMomentNorms}
 \begin{split}
 \Vert \langle a\rangle^p\gamma_1\Vert_{L^r}&\le2\Vert \langle a\rangle^p\gamma_2\Vert_{L^r}+C_p\vert\mathcal{W}(t)\vert^p\Vert\gamma_2\Vert_{L^r}
 \end{split}
 \end{equation}
 for $\gamma_1,\gamma_2\in\{\gamma,\gamma^\prime\}$, $p\ge 0$, $r\in\{2,\infty\}$.
 
\end{lemma}

\begin{proof}
 We have that
 \begin{equation}\label{eq:diffa2}
  a^2-(a')^2=\abs{\V}^2+\frac{q}{\aabs{\X}}-\abs{\V-\mathcal{W}(t)}^2-\frac{q}{\aabs{\X}}=2\mathcal{W}(t)\cdot\V-\abs{\mathcal{W}(t)}^2,
 \end{equation}
 and thus (e.g.\ by dinstinguishing whether $a\leq \abs{\mathcal{W}(t)}$ or $a> \abs{\mathcal{W}(t)}$)
 \begin{equation}
  \abs{a-a'}\lesssim \abs{\mathcal{W}(t)}.
 \end{equation}
 The bound for $\xi=\frac{q}{a}$ follows directly, whereas we compute that
 \begin{equation}
  \abs{\lambda-\lambda'}\lesssim\vert\XX\times\VV-\XX\times(\VV-\mathcal{W}(t))\vert\lesssim\ip{t}\abs{\mathcal{W}(t)}.
 \end{equation}
 Finally we have by \eqref{def:sigma} that $\eta=\frac{a}{q}\vartheta\cdot\a=\frac{a}{q}(\X\cdot\V)-\sigma$, and thus
 \begin{equation}
 \begin{aligned}
  \eta+t\frac{a^3}{q}&=\eta\circ\Phi_t^{-1}=\frac{a}{q}(\XX\cdot\VV)-\sigma\circ\Phi_t^{-1},\\
  \eta'+t\frac{(a')^3}{q}&=\eta\circ\Phi_t^{-1}\circ\mathcal{M}^{-1}_t=\frac{a'}{q}(\XX\cdot(\VV-\mathcal{W}(t)))-\sigma\circ\Phi_t^{-1}\circ\mathcal{M}^{-1}_t,
 \end{aligned} 
 \end{equation}
 so that
 \begin{equation}\label{eq:diffeta}
  \eta-\eta^\prime=t\frac{(a')^3-a^3}{q}+\frac{a-a'}{q}(\XX\cdot\VV)+\frac{a'}{q}(\XX\cdot\mathcal{W}(t))+(\sigma\circ\Phi_t^{-1}\circ\mathcal{M}^{-1}_t-\sigma\circ\Phi_t^{-1}).
 \end{equation}
 The first three terms are bounded directly, whereas for the last one we observe that (as can be seen directly from the definition \eqref{eq:def_rho2} of $\rho$ in terms of $\x$ and $\a$) there holds that
 \begin{equation}
  \abs{\rho\circ\Phi_t^{-1}\circ\mathcal{M}^{-1}_t-\rho\circ\Phi_t^{-1}}\lesssim \abs{a-a'}(a+a')\ip{t},
 \end{equation}
 and thus from \eqref{def:sigma} we obtain
 \begin{equation}
 \begin{split}
  \abs{\sigma\circ\Phi_t^{-1}\circ\mathcal{M}^{-1}_t-\sigma\circ\Phi_t^{-1}}
   &=\left\vert\ln\left(\frac{\sqrt{\rho}+\sqrt{\rho-1}}{\sqrt{\rho^\prime}+\sqrt{\rho^\prime-1}}\right)\right\vert\lesssim\left\vert\ln(1+\frac{2(\rho-\rho^\prime)}{(\sqrt{\rho}+\sqrt{\rho^\prime})(\sqrt{\rho-1}+\sqrt{\rho^\prime-1})}\right\vert\\
%  &\lesssim \ln(1+2\sqrt{\rho\circ\Phi_t^{-1}\circ\mathcal{M}^{-1}_t-\rho\circ\Phi_t^{-1}})\\
  &\lesssim \ln(1+2\sqrt{\abs{a-a'}(a+a')\ip{t}}).
  \end{split}
 \end{equation}
Now the equivalence of norms \eqref{ComparableMomentNorms} follows from the bounds in \eqref{prime_bds} using \eqref{eq:weightprimes} as well.
\end{proof}

Furthermore, we need to detail how the linear trajectories interplay with the close/far regions. The following lemma shows that a trajectory of the linearized system \eqref{ODE} can enter or exit the close region $\Omega_t^{cl}$ at most once. 

\begin{lemma}\label{lem:traj_cl_far}
 Let $(\X(t),\V(t))$ be a trajectory of \eqref{ODE}. Assume that for $t_1<t_2$ we have that $\X(t_1)\in\Omega_{t_1}^{cl}\setminus \Omega_{t_1}^{far}$ and $\X(t_2)\in \Omega_{t_2}^{far}\setminus\Omega_{t_2}^{cl}$. Then for all $t\geq t_2$ there holds that $\X(t)\in\Omega_t^{far}$. 
\end{lemma}
\begin{proof}
 We recall from Remark \ref{rem:virials} that
 \begin{equation}\label{SecondOrderDerXVX}
  \frac{d^2}{dt^2}\abs{\X(t)}^2=\abs{\V(t)}^2+H,\qquad\frac{d}{dt}\vert\X(t)\vert^2=2\X(t)\cdot\V(t),\qquad\frac{d}{dt}\vert\V(t)\vert^2=q\frac{\X(t)\cdot\V(t)}{\vert\X(t)\vert^3},
 \end{equation}
 so that $t\mapsto\vert\X(t)\vert^2$ is convex, and in particular has at most one minimum and $\vert\X\vert$ and $\vert\V\vert$ have same monotonicity.
 
 To prove the claim, we assume that $0\le t_1< t_2$ and
 \begin{equation}
  \abs{\X(t_1)}= \langle t_1\rangle,\qquad \abs{\X(t_2)}= 10\langle t_2\rangle.
 \end{equation}
Since $\abs{\X(t_2)}>\abs{\X(t_1)}$ we have that the periapsis occurs before $t_2$, and thus $\abs{\V(t)}$ is monotone increasing on $[t_2,\infty)$.

By convexity, we see that $\vert \X(s)\vert\le\vert\X(t_2)\vert$ for $t_1\le s\le t_2$, and by conservation of energy, this implies that $\vert\V(s)\vert\le\vert\V(t_2)\vert$. If $\vert\V(t_2)\vert\le 9$, we obtain that  
\begin{equation*}
\vert \X(t_2)\vert\le (t_2-t_1)\vert\V(t_2)\vert+\vert\X(t_1)\vert\le 9(t_2-t_1)+\langle t_1\rangle< 10\langle t_2\rangle,
\end{equation*}
which is impossible. Now, since $\vert \V(t)\vert$ is increasing on $(t_2,\infty)$, integrating twice the first equation in \eqref{SecondOrderDerXVX}, we find that, for $t\ge t_2$,
 \begin{equation*}
  \abs{\X(t)}^2\geq \abs{\X(t_2)}^2+(\abs{\V(t_2)}^2+H)(t-t_2)^2/2\geq 9^2(1+t_2^2 +(t-t_2)^2/2)\geq 3(1+ t^2),
 \end{equation*}
 and thus $\X(t)\in\Omega_t^{far}$.
\end{proof}

\section{Some kinematics and Poisson brackets}\label{ssec:kinematics}
We now develop quantitative estimates on some dynamically relevant quantities. We recall from \eqref{XVinSIC} the expressions of $\X$ and $\V$ in terms of the super-integrable coordinates as
\begin{equation}\label{eq:XVgen}
 \X=X_1(\xi,\eta,\lambda)\u+X_3(\xi,\eta,\lambda)\L\times\u,\qquad \V=V_1(\xi,\eta,\lambda)\u+V_3(\xi,\eta,\lambda)\L\times\u,
\end{equation}
with 
\begin{equation}\label{eq:XX1X3}
\begin{aligned}
 X_1(\xi,\eta,\lambda)&=\frac{\xi^2}{q}\left(\eta+\sigma+\frac{1}{2}+\frac{D}{1+4\kappa^2}\right),\quad X_3(\xi,\eta,\lambda)=-\frac{\xi}{q}\left(1+\frac{2D}{1+4\kappa^2}\right),\\%=\lambda^{-1}X_2,\\
 V_1(\xi,\eta,\lambda)&=\frac{q}{\xi}\left(1-\frac{1}{2N}-\frac{1}{1+4\kappa^2}\frac{D}{N}\right),\quad V_3(\xi,\eta,\lambda)=\frac{q}{\xi^2}\frac{2}{1+4\kappa^2}\frac{D}{N},%=\lambda^{-1}V_2.
\end{aligned} 
\end{equation}
where we abbreviated $D=D(\kappa,\eta+\sigma)$ and $N=N(\kappa,\eta+\sigma)$. Since only $\eta$ evolves along a trajectory of \eqref{ODE}, with \eqref{eq:eta_evol} we directly obtain the corresponding expressions
\begin{equation}
 \XX=\X(\xi,\eta+tq^2\xi^{-3},\lambda,\u,\L),\qquad \VV=\V(\xi,\eta+tq^2\xi^{-3},\lambda,\u,\L).
\end{equation}

In the following, it will be useful to recall some simple bounds on the functions involved in the formulas defined in \eqref{DefDN}:
\begin{equation}\label{BoundsDN}
\begin{split}
0\le D(\kappa,y)\le 2\sqrt{y^2+\kappa^2+1/4},\quad \mathfrak{1}_{\{y>0\}}D(\kappa,y)\le\frac{\kappa^2+1/4}{\sqrt{y^2+\kappa^2+1/4}},\\
\partial_yD(\kappa,y)=-\frac{D(\kappa,y)}{\sqrt{y^2+\kappa^2+1/4}},\qquad\partial_\kappa D(\kappa,y)=\frac{\kappa}{\sqrt{y^2+\kappa^2+1/4}}=\partial_\kappa N(\kappa,y),\\
\max\{1,\vert y\vert,\kappa\}\le N(\kappa,y)\le 1+\vert y\vert+\kappa,\qquad -1\le \partial_y N(\kappa,y),\partial_\kappa N(\kappa,y)\le 1,
\end{split}
\end{equation}
and
\begin{equation}\label{SecondOrderDerDNFormulas}
\begin{split}
\partial_y^2D(\kappa,y)=\partial_y^2N(\kappa,y)&=\frac{D(\kappa,y)}{y^2+\kappa^2+1/4}\frac{\sqrt{y^2+\kappa^2+1/4}+y}{\sqrt{y^2+\kappa^2+1/4}},\\
\partial_\kappa\partial_yD(\kappa,y)=\partial_\kappa\partial_yN(\kappa,y)&=-\frac{y\kappa}{(y^2+\kappa^2+1/4)^\frac{3}{2}},\\
\partial_\kappa^2D(\kappa,y)=\partial_\kappa^2N(\kappa,y)&=\frac{y^2+1/4}{(y^2+\kappa^2+1/4)^\frac{3}{2}}.
\end{split}
\end{equation}

\subsubsection{Estimates on ${\bf X}$.}

Looking at \eqref{Formula|x|} at periapsis (when ${\bf x}\cdot{\bf v}=0$), we find that
\begin{equation}\label{CrudeBoundX}
\begin{split}
\vert{\bf X}\vert\ge\frac{\xi^2}{2q}\langle\kappa\rangle
\end{split}
\end{equation}
but we need more precise bounds. 

\begin{corollary}\label{CorBdsOnX}
We have the uniform bounds
\begin{equation}\label{BoundXV}
\begin{split}
\frac{1}{10}\sqrt{\eta^2+\kappa^2+\frac{1}{4}}\le \sqrt{\frac{a^2}{q^2}({\bf x}\cdot{\bf v})^2+\kappa^2+\frac{1}{4}}\le 10\sqrt{\eta^2+\kappa^2+\frac{1}{4}}
\end{split}
\end{equation}
and
\begin{equation}\label{BoundsX}
\begin{split}
\frac{ta}{100}-100\cdot\left[1+\frac{q}{a^2}\left(\abs{\eta}+\kappa\right)\right]\le \vert {\bf X}(\vartheta+t{\bf a},{\bf a})\vert\le 100 ta+100\cdot\left[1+\frac{q}{a^2}\left(\abs{\eta}+\kappa\right)\right].
\end{split}
\end{equation}

\end{corollary}

\begin{proof}[Proof of Corollary \ref{CorBdsOnX}]

The bounds \eqref{BoundXV} follow from \eqref{BoundsOnSigma} since ${\bf x}\cdot{\bf v}=\eta+\sigma$. The bound for ${\bf X}_{-1}$ follows from the estimates \eqref{BoundsDN}. The bound on $\vert{\bf X}\vert$ follows from the first formula in \eqref{ExpressionsX}. If ${\bf x}\cdot{\bf v}>0$, or if $4\kappa^2\notin [1/2,2]$,
\begin{equation*}
\begin{split}
\vert{\bf X}\vert\ge \vert {\bf X}\cdot\frac{\bf a}{a}\vert\ge \frac{q}{a^2}\left[\frac{1}{2}+\frac{q^2}{(2R)^2}\sqrt{\frac{a^2}{q^2}({\bf x}\cdot{\bf v})^2+\frac{(2R)^2}{4q^2}}+\frac{4a^2L^2}{(2R)^2}\frac{a}{q}({\bf x}\cdot{\bf v})\right]\geq\frac{q}{a^2}(1+\frac{a}{q}\abs{\x\cdot\v}),
\end{split}
\end{equation*}
while if ${\bf x}\cdot{\bf v}<0$ and $1/2\le 4\kappa^2\le 2$, we use the orthogonal direction to get
\begin{equation*}
\begin{split}
\vert{\bf X}\vert\ge \vert {\bf X}\cdot\frac{{\bf L}\times{\bf  a}}{aL}\vert\ge\frac{2aL}{a^2}\left[\frac{1}{2}+\frac{q^2}{(2R)^2}\left(\sqrt{\frac{a^2}{q^2}({\bf x}\cdot{\bf v})^2+\frac{(2R)^2}{4q^2}}-\frac{a}{q}({\bf x}\cdot{\bf v})\right)\right]\ge \frac{q}{2a^2}\vert\eta\vert-\frac{L}{a}\vert\sigma\vert.
\end{split}
\end{equation*}

\end{proof}

\subsection{Poisson brackets}\label{sec:PB_bounds}

We will make extensive use of the properties of the Poisson bracket
\begin{equation}\label{PB2}
 \{f,g\}=\nabla_\x f\cdot \nabla_\v g-\nabla_\v f\cdot \nabla_\x g.
\end{equation}
Recalling that by construction the change of variables $\mathcal{T}:(\x,\v)\mapsto(\A,\Theta)$ in Proposition \ref{PropAA} is canonical, we see that $\mathcal{T}$ leaves the Poisson bracket invariant: for any functions $f,g$ we have that 
$\{f,g\}=\{f\circ \mathcal{T},g\circ\mathcal{T}\}=\{f\circ \mathcal{T}^{-1},g\circ\mathcal{T}^{-1}\}$. In other words, we can compute Poisson brackets in either systems of physical $(\x,\v)$ or angle-action coordinates $(\vartheta,\a)$, and will slightly abuse the notation by simply writing this as
\begin{equation}
 \{f,g\}=\nabla_\x f\cdot \nabla_\v g-\nabla_\v f\cdot \nabla_\x g=\nabla_\vartheta f\cdot \nabla_\a g-\nabla_\a f\cdot \nabla_\vartheta g.
\end{equation}
Two further useful facts are its Leibniz rule and the Jacobi identity
\begin{equation}\label{Jacobi}
\begin{split}
0=\{f,\{g,h\}\}+\{g,\{h,f\}\}+\{h,\{f,g\}\},
\end{split}
\end{equation}
which can be verified by straightforward computations. Finally, the nonlinear analysis will exploit in important ways that the integral of a Poisson bracket vanishes
\begin{equation}\label{ZeroIntegralPB}
\begin{split}
\iint \{f,g\} d\x d\v=0,
\end{split}
\end{equation}
provided the derivatives of the functions have appropriate decay.

Moreover, the Poisson brackets define \emph{symplectic gradients} which we will use as vector fields to control regularity.

\begin{remark}
Not all vector fields are symplectic gradients (symplectic gradient are divergence-free), but the canonical basis is made of symplectic gradients. In addition, we have a few distinguished vector fields: 
\begin{equation*}
\begin{split}
\{\cdot,a^2\}&=\Big\{\cdot,\vert{\bf v}\vert^2+\frac{q}{\vert{\bf x}\vert}\Big\}=2{\bf v}\cdot\nabla_\x+2q\frac{{\bf x}}{\vert{\bf x}\vert^3}\cdot\nabla_\v,\\
\{\cdot,{\bf L}^j\}&=\in^{jkl}\vartheta^k\partial_{\vartheta^l}-\in^{jkl}{\bf a}^l\partial_{{\bf a}^k},\qquad
\{\cdot,\vartheta\cdot{\bf a}\}=\vartheta\cdot\nabla_{\vartheta}-{\bf a}\cdot\nabla_{\bf a},
\end{split}
\end{equation*}
and we recognize that the first vector field above is nothing but the Hamiltonian vector field associated to the linearized dynamics, while the next two are the Noether vector fields associated to the invariance of the equations under rotations $({\bf x},{\bf v})\mapsto (\mathcal{R}{\bf x},\mathcal{R}^{-1}{\bf v})$, $\mathcal{R}\in SO(3)$ and to the invariance of the equation under rescaling $({\bf x},{\bf v})\mapsto (\lambda{\bf x},\lambda^{-1}{\bf v})$, $\lambda\in\mathbb{R}_+$.
\end{remark}

A crucial first Poisson bracket identity is the one showing that the Kepler problem is integrable, i.e.\
\begin{equation}
\{a^2,{\bf L}\}=\{ H,{\bf L}\}=0.
\end{equation}
We have several nice properties in coordinates $({\bf x},{\bf v})$, in which the nonlinearity has a simple expression:
\begin{equation}\label{PBXV}
\begin{split}
\{{\bf x},a\}&=\frac{1}{2a}\Big\{{\bf x},\vert{\bf v}\vert^2+\frac{q}{\vert{\bf x}\vert}\Big\}=a^{-1}{\bf v},\qquad
\{{\bf x},L\}={\bf l}\times{\bf x},\\
\{{\bf v},a\}&=\frac{q}{2a}\frac{{\bf x}}{\vert{\bf x}\vert^3},\qquad\{{\bf v},L\}={\bf l}\times{\bf v}.
\end{split}
\end{equation}

Both classical coordinates $({\bf x},{\bf v})$ and angle-action coordinates $(\vartheta,{\bf a})$ satisfy the canonical Poisson bracket relations
\begin{equation*}
 \{\vartheta^j,\vartheta^k\}=\{{\bf a}^j,{\bf a}^k\}=0=\{\x^j,\x^k\}=\{\v^j,\v^k\},\qquad\{\vartheta^j,{\bf a}^k\}=\delta_{jk}=\{\x^j,\v^k\}.
\end{equation*}
While the super-integrable coordinates \eqref{SIC} are better adapted to the linear evolution, their Poisson bracket relations are not canonical anymore, but still relatively convenient: we have that
\begin{equation}\label{PBSIC}
\begin{split}
\{\xi,\eta\}&=1,\qquad 0=\{\xi,\lambda\}=\{\xi,{\bf u}\}=\{ \xi,\L\}=\{\eta,\lambda\}=\{\eta,{\bf u}\}=\{\eta,\L\}=\{\lambda,\L\},\\
\{\lambda,{\bf u}\}&=-\l\times\u,\quad\{{\bf u}^j,{\bf u}^k\}=0,\quad
\{\L^j,\u^k\}=\in^{jka}\u^a,\quad \{\L^j,\L^k\}=\in^{jka}\L^a,
% \{\lambda,{\bf u}\}=-{\bf w},\quad\{\lambda,{\bf w}\}={\bf u},\quad\{{\bf u}^j,{\bf w}^k\}=\lambda^{-1}{\bf l}^j{\bf l}^k\quad\{{\bf u}^j,{\bf u}^k\}=0=\{{\bf w}^j,{\bf w}^k\},\\
\end{split}
\end{equation}
and $\{(\L\times\u)^a,\L^b\}=\L^a\u^b-\L^b\u^a$.

\begin{remark}\label{PBWithLambdaRmk}
 Note that Poisson brackets with $\lambda$ follow from those with $\L$, since for any scalar function $\zeta$ there holds that
 \begin{equation*}\label{PBWithLambda}
  \{\zeta,\lambda\}=\l^j\{\zeta,\L^j\}.
 \end{equation*}
 However, for some computations it is useful to keep treating the Poisson bracket with $\lambda$ separately.
\end{remark}

\subsubsection{Estimates on derivatives of ${\bf X}$ and  ${\bf V}$}
In treating the nonlinear terms, Poisson brackets of the kinematic quantities $\XX$ and $\VV$ with the super-integrable coordinates arise. With \eqref{PBXV} we can explicitly compute some of the relevant Poisson brackets: we have that
\begin{equation}\label{PBXV2}
\begin{aligned}
 \{a,\XX\}&=-\frac{\VV}{a},\qquad \{\xi,\XX\}=\frac{\xi^2}{q}\frac{\VV}{a},\qquad \{\lambda,\XX\}=-{\bf l}\times\XX,\\
 \{a,\VV\}&=-\frac{\xi}{2}\frac{\XX}{\vert\XX\vert^3},\qquad \{\xi,\VV\}=\frac{\xi^3}{2q}\frac{\XX}{\vert\XX\vert^3}, \qquad \{\lambda,\VV\}=-{\bf l}\times\VV.
\end{aligned} 
\end{equation}
As a consequence,
\begin{equation}\label{DoublePoissonBracketsWithXi}
\begin{aligned}
 \{\{\xi,\XX^j\},\zeta\}&=\frac{\xi^2}{q^2}\left[3\VV^j\{\xi,\zeta\}+\xi\{\VV^j,\zeta\}\right],\\
 \{\{\xi,\VV^j\},\zeta\}&=3\frac{\xi^2}{2q}\frac{\XX^j}{\vert\XX\vert^3}\{\xi,\zeta\}+\frac{\xi^3}{2q}\frac{1}{\vert\XX\vert^3}\left[\{\XX^j,\zeta\}-3\frac{\XX^j\XX^k}{\vert\XX\vert^2}\{\XX^k,\zeta\}\right].
\end{aligned}
\end{equation}

More generally, for a vector of the form $\bm{U}=U_1(\xi,\eta,\lambda)\u+U_3(\xi,\eta,\lambda)\L\times\u$ (such as $\X$ and $\V$ in \eqref{eq:XVgen} and \eqref{eq:XX1X3}), and $\alpha\in\{\xi,\eta,\lambda\}$, we slightly abuse notation by writing $\partial_\alpha\bm{U}=\partial_\alpha U_1\u+\partial_\alpha U_3\L\times\u$, and can then use the chain rule and the Poisson bracket relations \eqref{PBSIC} to resolve Poisson brackets of a scalar function $\zeta$ with $\bm{U}$ as
\begin{equation}\label{eq:U_resol}
 \{\bm{U}^j,\zeta\}=\{\bm{U}^j,\eta\}\{\xi,\zeta\}-\{\bm{U}^j,\xi\}\{\eta,\zeta\}+\partial_\lambda\bm{U}^j\{\lambda,\zeta\}+U_1\{\u^j,\zeta\}+U_3\in^{jcd}\L^c\{\u^d,\zeta\}+U_3\in^{jcd}\u^d\{\L^c,\zeta\},
\end{equation}
where we have used that $\partial_\xi\bm{U}=\{\bm{U},\eta\}$ and $\partial_\eta\bm{U}=-\{\bm{U},\xi\}$.

We now collect bounds on the Poisson brackets. From \eqref{PBXV2} it directly follows that
\begin{equation}\label{PBX1}
\begin{split}
\vert\{\xi,\widetilde{\bf X}\}\vert\le\xi^2/q,\qquad\vert\{\lambda,\widetilde{\bf X}\}\vert=\vert\widetilde{\bf X}\vert,\qquad
\vert\{\xi,\widetilde{\bf V}\}\vert\le\frac{\xi^3}{2q}\frac{1}{\vert\widetilde{\bf X}\vert^2},\qquad\vert\{\lambda,\widetilde{\bf V}\}\vert=\vert\widetilde{\bf V}\vert.
\end{split}
\end{equation}
Furthermore, we have:
\begin{lemma}\label{lem:moreXVbds}
We have the following bounds for first order Poisson brackets
\begin{equation}\label{PBX1'}
\begin{alignedat}{3}
 \vert\partial_\lambda\XX\vert&\lesssim \xi^{-1}\kappa\ip{\kappa}^{-2}\vert\XX\vert,\qquad &\vert\{\XX,\eta\}\vert &\lesssim \xi^{-1}\vert\XX\vert+tq\xi^{-2},\\
 \vert\partial_\lambda\VV\vert &\lesssim q\xi^{-2}\kappa\ip{\kappa}^{-3},\qquad &\vert\{\VV,\eta\}\vert &\lesssim q\xi^{-2}(1+t\xi \vert\XX\vert^{-2}),
\end{alignedat} 
\end{equation}
and for second order Poisson brackets for $\XX$:
\begin{equation}\label{PBX1primeprime}
\begin{alignedat}{3}
\vert\{\{\XX,\xi\},\xi\}\vert&\lesssim (\xi^2/q)^3\cdot\vert\XX\vert^{-2},\qquad&
 \vert\{\{\XX,\xi\},\eta\}\vert&\lesssim \xi/q\cdot (1+t\xi\vert\XX\vert^{-2}),\\
  \vert\{\{\XX,\eta\},\eta\}\vert &\lesssim \xi^{-2}\vert\XX\vert+tq\xi^{-3}+t^2q\xi^{-2}\vert\XX\vert^{-2},\qquad&
  \vert\partial_\lambda^2\XX\vert &\lesssim \xi^{-2}\ip{\kappa}^{-2}\vert\XX\vert,\\
  \vert\partial_\lambda\{\xi,\XX\}\vert&\lesssim\xi/q\cdot\kappa\langle\kappa\rangle^{-3},\qquad
  &\vert\partial_\lambda\{\XX,\eta\}\vert &\lesssim \xi^{-2}\ip{\kappa}^{-1}(\vert\XX\vert+tq\xi^{-1}),\\
  \end{alignedat}
  \end{equation}
and for second order Poisson brackets for $\VV$:
  \begin{equation}\label{PBV1primeprime}
  \begin{alignedat}{3}
  \vert\{\{\VV,\xi\},\xi\}\vert&\lesssim \xi^5/q^2\cdot\vert\XX\vert^{-3},\qquad&
   \vert\{\{\VV,\xi\},\eta\}\vert&\lesssim \xi^{2}/q\cdot \vert\XX\vert^{-2}(1+tq/(\xi\vert\XX\vert)),\\
 \vert\{\{\VV,\eta\},\eta\}\vert &\lesssim q\xi^{-3}+tq\xi^{-2}\vert\XX\vert^{-2}+t^2q^2\xi^{-3}\vert\XX\vert^{-3},\qquad&
  \vert\partial_\lambda^2 \VV\vert&\lesssim q\ip{\kappa}^{-3}\xi^{-3},\\
  \vert\partial_\lambda\{\xi,\VV\}\vert&\lesssim\xi^2/q\kappa\langle\kappa\rangle^{-2}\vert\XX\vert^{-2}, & \vert\partial_\lambda\{\VV,\eta\}\vert &\lesssim q\xi^{-3}\langle\kappa\rangle^{-1} (\langle\kappa\rangle^{-1}+t\xi\vert\XX\vert^{-2}).
 \end{alignedat} 
 \end{equation}
\end{lemma}

\begin{proof}

We begin by noting that we can rewrite \eqref{XVinSIC} as
\begin{equation*}
\begin{split}
\X&=\frac{\xi^2}{q}\left(\frac{N}{1+4\kappa^2}+\frac{4\kappa^2}{1+4\kappa^2}(\eta+\sigma+\frac{1}{2})\right){\bf u}-\frac{\xi^2}{q}\kappa\left(1+\frac{2D}{1+4\kappa^2}\right){\bf l}\times{\bf u}
\end{split}
\end{equation*}
and therefore (e.g.\ by distinguishing whether $\kappa>1/4$ or $\kappa\leq 1/4$ and using the bound on $N$ in \eqref{BoundsDN}),
\begin{equation}\label{eq:Xprelimbd}
 \frac{\xi^2}{q}\ip{\kappa}(1+\langle\kappa\rangle^{-2}D)\lesssim  \abs{\X},
\end{equation}
and in particular $\abs{X_3}\lesssim \xi^{-1}\ip{\kappa}^{-1}\abs{\X}$.
Moreover, since $\kappa=\lambda/\xi$ and $\sigma$ essentially only depends on $\rho(\eta,\kappa)$ (see \eqref{def:sigma}), we have that 
\begin{equation}
 -\xi\partial_\xi\kappa=\lambda\partial_\lambda\kappa=\kappa,\qquad -\xi\partial_\xi\sigma=\lambda\partial_\lambda\sigma=\kappa\partial_\kappa\sigma,
\end{equation}
and thus for $D=D(\kappa,\eta+\sigma)$ we have that $-\xi\partial_\xi D=\lambda\partial_\lambda D$ and $(\xi\partial_\xi)^2 D=(\lambda\partial_\lambda)^2 D$ and
\begin{equation}\label{SecondDerDN}
\begin{split}
\xi\partial_\lambda D&=\partial_\kappa D+\partial_\kappa\sigma\partial_yD,\\
(\xi\partial_\lambda)^2D&=\partial_\kappa^2D+2\partial_\kappa\sigma\partial_\kappa\partial_yD+(\partial_\kappa\sigma)^2\partial_y^2D+\partial_\kappa^2\sigma\partial_yD,\\
\xi\partial_\lambda N&=\partial_\kappa D+\partial_\kappa\sigma\partial_yN,\\
(\xi\partial_\lambda)^2N&=\partial_\kappa^2D+2\partial_\kappa\sigma\partial_\kappa\partial_yD+(\partial_\kappa\sigma)^2\partial_y^2N+\partial_\kappa^2\sigma\partial_yN.
\end{split}
\end{equation}
Using \eqref{BoundsOnSigma} and \eqref{BoundsDN}-\eqref{SecondOrderDerDNFormulas}, we obtain that
\begin{equation}\label{eq:dlD}
 \begin{split}
 \xi\abs{\partial_\lambda D}&\lesssim \frac{\kappa}{N}(1+\ip{\kappa}^{-2}D),\qquad\langle\kappa\rangle \vert (\xi\partial_\lambda)^2D\vert\lesssim 1,
 \end{split}
\end{equation}
and similarly
\begin{equation}\label{eq:dlN}
 \xi\abs{\partial_\lambda N}\lesssim \kappa (\ip{\kappa}^{-2}+\frac{1}{N})\lesssim\frac{\kappa}{\ip{\kappa}},\qquad\langle\kappa\rangle \vert (\xi\partial_\lambda)^2N\vert\lesssim 1.
\end{equation}

\medskip

{\bf Step 1}: first order derivatives. We are now ready to prove \eqref{PBX1'}. From \eqref{eq:XX1X3} we have that
\begin{equation}\label{eq:dlX}
\begin{aligned}
 \partial_\lambda X_1&=\frac{\xi}{q}\left(\partial_\kappa\sigma+\frac{\xi\partial_\lambda D}{1+4\kappa^2}-\frac{D}{1+4\kappa^2}\frac{8\kappa}{1+4\kappa^2}\right),\\
 \partial_\lambda X_3&=-\frac{2}{q}\left(\frac{\xi\partial_\lambda D}{1+4\kappa^2}-\frac{D}{1+4\kappa^2}\frac{8\kappa}{1+4\kappa^2}\right),
\end{aligned} 
\end{equation}
and thus
\begin{equation}
 \abs{\partial_\lambda X_1}+\xi\abs{\partial_\lambda X_3} \lesssim \xi^{-1}\kappa\ip{\kappa}^{-3}\abs{\X},%\qquad \abs{\partial_\lambda X_3}\lesssim \frac{1}{\xi^2}\abs{\X},
\end{equation}
so that
\begin{equation}
 \abs{\partial_\lambda\X}\lesssim \xi^{-1}\kappa\ip{\kappa}^{-2}\abs{\X}.
\end{equation}
Since
\begin{equation}\label{eq:xidxiX}
 \xi\partial_\xi X_1=2X_1-\lambda\partial_\lambda X_1,\qquad \xi\partial_\xi X_3=X_3-\lambda\partial_\lambda X_3,
\end{equation}
it directly follows that
\begin{equation}
 \abs{\xi\partial_\xi\X}\lesssim\abs{\X}.
\end{equation}
Moreover, we have that
\begin{equation}
 \{\XX,f\}=\{\X,f\circ\Phi_t\}\circ\Phi_t^{-1},
\end{equation}
and thus in particular, using \eqref{PBXV2},
\begin{equation}\label{eq:PBXXeta}
 \{\XX,\eta\}=\{\X,\eta-tq^2\xi^{-3}\}\circ\Phi_t^{-1}=\partial_\xi\X\circ\Phi_t^{-1}-3t\xi^{-1}\VV,
\end{equation}
which gives the bound
\begin{equation}
 \vert\{\XX,\eta\}\vert\lesssim \xi^{-1}\vert\XX\vert+tq\xi^{-2}.
\end{equation}

For $\V$ we compute that
\begin{equation}\label{eq:dlV}
\begin{aligned}
 \partial_\lambda V_1&=\frac{q}{\xi^2}\left(\frac{1}{2N}\frac{\xi\partial_\lambda N}{N}-\frac{1}{1+4\kappa^2}\frac{D}{N}\left(\frac{\xi\partial_\lambda D}{D}-\frac{\xi\partial_\lambda N}{N}-\frac{8\kappa}{1+4\kappa^2}\right)\right),\\
 \partial_\lambda V_3&=\frac{q}{\xi^3}\frac{2}{1+4\kappa^2}\frac{D}{N}\left(\frac{\xi\partial_\lambda D}{D}-\frac{\xi\partial_\lambda N}{N}-\frac{8\kappa}{1+4\kappa^2}\right),
\end{aligned}
\end{equation}
and thus, using \eqref{eq:dlD} and \eqref{eq:dlN},
\begin{equation}
 \abs{\partial_\lambda\V}\lesssim q\xi^{-2}\kappa\ip{\kappa}^{-3}.
\end{equation}
With
\begin{equation}\label{DVDXi}
 \xi\partial_\xi V_1=-V_1-\lambda\partial_\lambda V_1,\qquad \xi\partial_\xi V_3=-2V_3-\lambda\partial_\lambda V_3,
\end{equation}
we obtain the bound
\begin{equation}\label{AddedDerXiV}
 \abs{\xi\partial_\xi\V}\lesssim q\xi^{-1}.
\end{equation}
Using \eqref{PBXV2},
\begin{equation}\label{eq:PBVVeta}
 \{\VV,\eta\}=\{\V,\eta-tq^2\xi^{-3}\}\circ\Phi_t^{-1}=\partial_\xi\V\circ\Phi_t^{-1}-\frac{3}{2}tq\xi^{-1}\frac{\XX}{\vert\XX\vert^3},
\end{equation}
and we deduce that
\begin{equation}
 \vert\{\VV,\eta\}\vert\lesssim q\xi^{-2}(1+t\xi \vert\XX\vert^{-2}).
\end{equation}

{\bf Step 2}: second order derivatives. We now turn to \eqref{PBX1primeprime}. The double Poisson brackets involving $\xi$ follow from \eqref{DoublePoissonBracketsWithXi} and \eqref{PBX1}-\eqref{PBX1'}. From \eqref{eq:dlX} we have that
\begin{equation}\label{eq:dl2X}
\begin{aligned}
 \partial_\lambda^2 X_1&=\frac{1}{q}\left(\partial^2_\kappa\sigma+\frac{\xi^2\partial^2_\lambda D}{1+4\kappa^2}-2\frac{\xi\partial_\lambda D}{1+4\kappa^2}\frac{8\kappa}{1+4\kappa^2}-\frac{D}{1+4\kappa^2}\frac{8}{1+4\kappa^2}\frac{1-12\kappa^2}{1+4\kappa^2}\right),\\
 \partial^2_\lambda X_3&=-\frac{2}{q\xi}\left(\frac{\xi^2\partial^2_\lambda D}{1+4\kappa^2}-2\frac{\xi\partial_\lambda D}{1+4\kappa^2}\frac{8\kappa}{1+4\kappa^2}-\frac{D}{1+4\kappa^2}\frac{8}{1+4\kappa^2}\frac{1-12\kappa^2}{1+4\kappa^2}\right),
\end{aligned} 
\end{equation}
and using \eqref{BoundsOnSigma} and \eqref{eq:dlD}-\eqref{eq:dlN}, then \eqref{eq:Xprelimbd}, we obtain
\begin{equation}
 \abs{\partial_\lambda^2\X}\lesssim \xi^{-2}\ip{\kappa}^{-2}\abs{\X},
\end{equation}
which, deriving once more the equality in \eqref{eq:xidxiX} also implies that
\begin{equation}
 \ip{\kappa}\abs{\partial_\lambda\partial_\xi\X}+\abs{\partial_\xi^2\X}\lesssim \xi^{-2}\abs{\X}.
\end{equation}
By \eqref{eq:PBXXeta} we have that
\begin{equation}
\begin{aligned}
 \{\{\XX,\eta\},\eta\}&=\{\partial_\xi\X,\eta-tq^2\xi^{-3}\}\circ\Phi_t^{-1}-\{3t\xi^{-1}\VV,\eta\}\\
%  &=\partial_\xi^2\X\circ\Phi_t^{-1}-3tq^2\xi^2\{\partial_\xi\X,\xi\}\circ\Phi_t^{-1}+3t\xi^{-1}\{\VV,\eta\}\\
 &=\partial_\xi^2\X\circ\Phi_t^{-1}-3t\xi^{-2}(2\VV+\xi\partial_\xi\V\circ\Phi_t^{-1})-3t\xi^{-1}\{\VV,\eta\},
\end{aligned} 
\end{equation}
where we have used that $\{\partial_\xi \X,\xi\}=\partial_\xi\{\X,\xi\}=-\partial_\xi(\xi^3q^{-2}\V)$. Hence by \eqref{PBX1'} there holds that
\begin{equation}
 \vert\{\{\XX,\eta\},\eta\}\vert\lesssim \xi^{-2}\vert\XX\vert+t\xi^{-3}q+t^2\xi^{-2}q\vert\XX\vert^{-2}.
\end{equation}
From \eqref{eq:PBXXeta} it also follows that
\begin{equation}\label{eq:dlPBetaX}
 \partial_\lambda\{\XX,\eta\}=\partial_\lambda\partial_\xi\X\circ\Phi_t^{-1}-3t\xi^{-1}\partial_\lambda\VV,
\end{equation}
which gives the bound
\begin{equation}
 \vert\partial_\lambda\{\XX,\eta\}\vert\lesssim \xi^{-2}\ip{\kappa}^{-1}(\vert\XX\vert+tq\xi^{-1}).
\end{equation}
The bounds on $\VV$ in \eqref{PBV1primeprime} are obtained similarly. The Poisson brackets involving $\xi$ follow from \eqref{DoublePoissonBracketsWithXi} and \eqref{PBX1}-\eqref{PBX1'}. For the other bounds, we compute from \eqref{eq:dlV} that
\begin{equation}\label{eq:dl2V}
\begin{aligned}
 \partial_\lambda^2 V_1&=\frac{q}{\xi^3}\Bigg[\frac{1}{2N}\frac{\xi^2\partial_\lambda^2 N}{N}-\frac{1}{N}\left(\frac{\xi\partial_\lambda N}{N}\right)^2-\frac{1}{1+4\kappa^2}\frac{D}{N}\left(\frac{\xi\partial_\lambda D}{D}-\frac{\xi\partial_\lambda N}{N}-\frac{8\kappa}{1+4\kappa^2}\right)^2\\
 &\qquad -\frac{1}{1+4\kappa^2}\frac{D}{N}\left(-\left(\frac{\xi\partial_\lambda D}{D}\right)^2+\frac{\xi^2\partial^2_\lambda D}{D}+\left(\frac{\xi\partial_\lambda N}{N}\right)^2-\frac{\xi^2\partial^2_\lambda N}{N}+\frac{8}{1+4\kappa^2}\frac{1-4\kappa^2}{1+4\kappa^2}\right) \Bigg],\\
 \partial_\lambda^2 V_3&=\frac{q}{\xi^4}\frac{2}{1+4\kappa^2}\frac{D}{N}\Bigg[\left(\frac{\xi\partial_\lambda D}{D}-\frac{\xi\partial_\lambda N}{N}-\frac{8\kappa}{1+4\kappa^2}\right)^2\\
 &\qquad\qquad\qquad -\left(\frac{\xi\partial_\lambda D}{D}\right)^2+\frac{\xi^2\partial_\lambda^2 D}{D}+\left(\frac{\xi\partial_\lambda N}{N}\right)^2-\frac{\xi^2\partial^2_\lambda N}{N}+\frac{8}{1+4\kappa^2}\frac{1-4\kappa^2}{1+4\kappa^2}\Bigg],
\end{aligned}
\end{equation}
and thus, observing that the terms involving $(\xi\partial_\lambda D/D)^2$ cancel in each terms, and using \eqref{eq:dlD}-\eqref{eq:dlN},
\begin{equation}
 \abs{\partial_\lambda^2 \V}\lesssim q\ip{\kappa}^{-3}\xi^{-3},
\end{equation}
which, deriving \eqref{DVDXi} also implies that
\begin{equation}
 \ip{\kappa}^2\vert\partial_\lambda\partial_\xi \V\vert+\vert\partial^2_\xi \V\vert\lesssim q\xi^{-3}.
\end{equation}
By \eqref{eq:PBVVeta} we have that
\begin{equation}
\begin{aligned}
 \{\{\VV,\eta\},\eta\}&=\{\partial_\xi\V,\eta-tq^2\xi^{-3}\}\circ\Phi_t^{-1}-\frac{3}{2}tq\Big\{\xi^{-1}\frac{\XX}{\vert\XX\vert^3},\eta\Big\}\\
%  &=\partial_\xi^2\X\circ\Phi_t^{-1}-3tq^2\xi^2\{\partial_\xi\X,\xi\}\circ\Phi_t^{-1}+3t\xi^{-1}\{\VV,\eta\}\\
 &=\partial_\xi^2\V\circ\Phi_t^{-1}-\frac{3}{2}tq\xi^{-2}\left(2\frac{\XX}{\vert\XX\vert^3}+\xi\partial_\xi\Big(\frac{\X}{\vert\X\vert^3}\Big)\circ\Phi_t^{-1}\right)-\frac{3}{2}tq\xi^{-1}\Big\{\frac{\XX}{\vert\XX\vert^3},\eta\Big\},
\end{aligned} 
\end{equation}
and hence
\begin{equation}
 \vert\{\{\VV,\eta\},\eta\}\vert\lesssim q\xi^{-3}+tq\xi^{-2}\vert\XX\vert^{-2}+t^2q^2\xi^{-3}\vert\XX\vert^{-3}.
\end{equation}
Note that \eqref{eq:PBVVeta} also gives
\begin{equation}\label{eq:dlPBetaV}
 \partial_\lambda\{\VV,\eta\}=\partial_\lambda\partial_\xi\V\circ\Phi_t^{-1}-\frac{3}{2}tq\xi^{-1}\partial_\lambda\Big(\frac{\XX}{\vert\XX\vert^3}\Big),
\end{equation}
and thus
\begin{equation}
 \vert\partial_\lambda\{\VV,\eta\}\vert\lesssim q\ip{\kappa}^{-2}\xi^{-3} +tq\xi^{-2}\ip{\kappa}^{-1}\vert\XX\vert^{-2}.
\end{equation}

\end{proof}

\subsection{Improvements in the outgoing direction}

In addition, we can obtain an improvement in the ``bulk region'':
\begin{equation}\label{eq:bulk}
 \B:=\{(\vartheta,\a)\in\mathcal{P}_{\vartheta,\a}:\;  a+\xi \le (q/\langle q\rangle) t^\frac{1}{4}/10,\quad \xi\abs{\eta}+\lambda\le 10^{-3}ta^2\},
\end{equation}
 which will be important later on.
 
\begin{lemma}\label{ImprovedBoundsOnXBulk}

In the bulk region, there holds that $\widetilde{\bf Z}:=\widetilde{\bf X}-t\widetilde{\bf V}$ and $\widetilde{\bf V}$ satisfy better bounds:
\begin{equation}\label{PrecisedAsymptotBoundsZV}
\begin{split}
\vert\widetilde{\bf Z}\vert&\lesssim\frac{\xi^2}{q}\ln\langle t\rangle+\frac{\xi^2}{q}(1+\vert\eta\vert+\kappa),\qquad\vert\widetilde{\bf V}-{\bf a}\vert\lesssim\frac{\xi^2}{q}\frac{1}{\langle t\rangle},
\end{split}
\end{equation}
and in particular
\begin{equation}\label{ApproxXBulk}
\begin{split}
\vert\widetilde{\bf X}-t\a\vert&\lesssim \frac{\xi^2}{q}\ln\langle t\rangle+\frac{\xi^2}{q}(1+\vert\eta\vert+\kappa).
\end{split}
\end{equation}

\end{lemma}

\begin{proof}[Proof of Lemma \ref{ImprovedBoundsOnXBulk}]
Using \eqref{XVinSIC}, we directly express
\begin{equation*}
\begin{split}
\widetilde{\bf Z}&=\frac{\xi^2}{q}\left(\eta+\widetilde{\sigma}+\frac{1}{2}+\frac{ta^3/q}{2N}\right){\bf u}-\frac{\xi^2}{q}\kappa({\bf l}\times{\bf u})+ \frac{\xi^2}{q}\frac{D}{1+4\kappa^2}\cdot\left(1+\frac{ta^3/q}{N}\right)\cdot \frac{2{\bf R}}{q},\\
\widetilde{\bf V}-\a&=-\frac{q}{2N}\left({\bf u}+\frac{D}{1+4\kappa^2}\frac{2{\bf R}}{q}\right)=-\frac{q}{2N}\left(\left(1+\frac{D}{1+4\kappa^2}\right){\bf u}-2\frac{\kappa D}{1+4\kappa^2}{\bf l}\times{\bf u}\right),
\end{split}
\end{equation*}
and using Corollary \ref{CorBdsOnX}, we find that, in the bulk,
\begin{equation*}
\begin{split}
\vert\widetilde{\bf X}\vert\lesssim ta,\qquad \vert\widetilde{\sigma}\vert\lesssim\ln\langle t\rangle,\qquad N\gtrsim tq^2/\xi^3,\qquad D\lesssim 1.
\end{split}
\end{equation*}
Hence \eqref{PrecisedAsymptotBoundsZV} follows by direct inspection and directly implies \eqref{ApproxXBulk}.

\end{proof}

\begin{remark}\label{rem:XVexpansion}
In fact, using the equations above, one can push the asymptotic development further
\begin{equation*}
\begin{split}
\widetilde{\bf Z}&=\frac{\xi^2}{q}(-\frac{1}{2}\ln(ta^3/q)+\eta-\ln 2+1-\frac{1}{4}\frac{\ln(ta^3/q)}{ta^3/q}){\bf u}-\frac{\xi\lambda}{q}({\bf l}\times{\bf u})+O(\frac{\xi^2}{q}\frac{\kappa^2+\vert\eta\vert+1}{ta^3/q}),
\end{split}
\end{equation*}
and in particular, we have the asymptotics of trajectories:
\begin{equation*}
\begin{split}
\widetilde{\bf X}&=\left(at-\frac{1}{2}\frac{q}{a^2}\ln(ta^3/q)\right)\cdot \frac{q^2}{4a^2L^2+q^2}\left(\frac{2}{q}{\bf R}+\frac{4a}{q^2}{\bf L}\times{\bf R}\right)+O(1),\\
\widetilde{\bf V}&=a\left(1-\frac{q}{2ta^3}\right)\cdot \frac{q^2}{4a^2L^2+q^2}\left(\frac{2}{q}{\bf R}+\frac{4a}{q^2}{\bf L}\times{\bf R}\right)+O(t^{-2}).
\end{split}
\end{equation*}

\end{remark}

Along the ``future'' part of a trajectory, i.e.\ after periapsis (roughly speaking), some important improvements for the $\lambda$ derivatives in these bounds are possible:
\begin{lemma}\label{lem:betterdl}
 In the region $\{\X\cdot\V> -\xi\ip{\kappa}\}$ we have the bounds
\begin{equation}\label{eq:betterdl}
\begin{aligned}
\abs{X_3}+\abs{\partial_\lambda \X}&\lesssim \kappa\frac{\xi}{q}\frac{\xi^2}{q\abs{\X}}\lesssim\frac{\xi}{q}\frac{\kappa}{\ip{\kappa}},\qquad& \aabs{\partial_\lambda^2\X}&\lesssim\frac{1}{q}\frac{\xi^2}{q\abs{\X}},\\
\abs{V_3}+\abs{\partial_\lambda \V}&\lesssim \frac{\xi^2}{q}\frac{1}{\abs{\X}^2}\frac{\kappa}{\ip{\kappa}},\qquad&\aabs{\partial_\lambda^2\VV}&\lesssim \frac{\xi}{q}\langle\kappa\rangle^{-1}\frac{1}{\vert\XX\vert^2}\\
 \aabs{\partial_\lambda\{\X,\xi\}}&\lesssim \frac{\xi^5}{q^3}\frac{\kappa}{\langle\kappa\rangle}\vert\X\vert^{-2},\qquad&\abs{\partial_\lambda\{\V,\xi\}}&\lesssim\frac{\xi^4}{q^2}\frac{\kappa}{\langle\kappa\rangle}\vert\X\vert^{-3}.
\end{aligned} 
\end{equation}
Moreover, for $t\geq 0$ we also have that, in the region $\{\XX\cdot\VV> -\xi\ip{\kappa}\}$,
\begin{equation}\label{eq:betterdlPBetaX}
\begin{aligned}
 \aabs{\partial_\lambda\{\XX,\eta\}}&\lesssim \frac{\xi}{q\vert\XX\vert}\left(\frac{\xi\langle\kappa\rangle}{q}+\frac{t}{\aabs{\XX}}\right),\qquad
 \aabs{\partial_\lambda\{\VV,\eta\}}\lesssim \frac{1}{\aabs{\XX}^2}\left(\frac{\xi}{q}+\frac{t}{\aabs{\XX}}\right).
 \end{aligned}
\end{equation}
\end{lemma}

\begin{remark}
 Along the ``past'' part of a trajectory, i.e.\ in the region $\{\X\cdot\V< -\xi\ip{\kappa}\}$, we can use the past asymptotic actions of Section \ref{ssec:past_AA} to obtain the same, improved bounds: this follows from the fact that the expressions of $\X$ and $\V$ are identical (see \eqref{XVinSIC(-)}), and satisfy the same Poisson bracket relations (see \eqref{eq:PB_SIC(-)}).
\end{remark}

\begin{proof}[Proof of Lemma \ref{lem:betterdl}]
We treat separately the regions $\{\abs{\X\cdot\V}\leq \ip{\kappa}\xi\}$ and $\{\X\cdot\V> \ip{\kappa}\xi\}$.
\begin{enumerate}[wide]
\item The region $\{\abs{\X\cdot\V}\leq \ip{\kappa}\xi\}$.
Since by \eqref{eq:virials} along a trajectory of \eqref{ODE} there holds
\begin{equation}
 \frac{d}{dt}\X\cdot\V\geq \frac{1}{2}a^2,
\end{equation}
a trajectory passes through the region $\{\abs{\X\cdot\V}\leq \ip{\kappa}\xi\}$ in time at most $2\xi\ip{\kappa}\cdot 2a^{-2}=4q^{-2} \xi^3\ip{\kappa}$.
Letting $t_p$ denote the time of periapsis (where by \eqref{Formula|x|} there holds $\abs{\X(t_p)}\geq \frac{\xi^2}{q}\ip{\kappa}$), we thus have by \eqref{eq:virials} that in this region
\begin{equation}
 \frac{\xi^4}{q^2}\ip{\kappa}^2\leq \abs{\X}^2\leq \abs{\X(t_p)}^2+\int_{t_p}^t\vert\frac{d}{dt}\abs{\X(s)}^2\vert ds\lesssim\frac{\xi^4}{q^2}\ip{\kappa}^2.
\end{equation}
The claimed bounds then follow from those in Lemma \ref{lem:moreXVbds}.

\item The region $\{\X\cdot\V> \ip{\kappa}\xi\}$.
Here the claim follows from the bounds on $D,N$ and their derivatives in \eqref{BoundsDN}. (Here again, we write $D=D(\kappa,\eta+\sigma)$ and $N=N(\kappa,\eta+\sigma)$ if not specified otherwise). To begin, we observe that since $\eta+\sigma=\xi^{-1}\X\cdot\V>\ip{\kappa}>0$ we can use the bound from \eqref{BoundsDN},
 \begin{equation}
  \mathfrak{1}_{y\geq 0}D(\kappa,y)=\mathfrak{1}_{y\geq 0}\frac{\kappa^2+1/4}{\sqrt{y^2+\kappa^2+1/4}+y}\leq \frac{\kappa^2+1/4}{\sqrt{y^2+\kappa^2+1/4}},
 \end{equation}
to deduce from \eqref{eq:XX1X3} the bounds
\begin{equation}\label{eq:X12_bds}
 \frac{1}{2}\frac{\xi^2}{q}N\leq \abs{X_1}\leq 2\frac{\xi^2}{q}N,\qquad \frac{\xi}{q}\leq \abs{X_3}\leq 2\frac{\xi}{q}.
\end{equation}
Since moreover
\begin{equation}
 D\lesssim\frac{\ip{\kappa}^2}{N}\lesssim\ip{\kappa},
\end{equation}
we can strengthen \eqref{eq:dlD} and \eqref{eq:dlN} to
\begin{equation}\label{eq:X12_bds3}
 \xi\abs{\partial_\lambda D}\lesssim \frac{\kappa}{N},\qquad \xi\abs{\partial_\lambda N}\lesssim \frac{\kappa}{\ip{\kappa}}.
\end{equation}
Furthermore, we note from \eqref{def:tau} that, after periapsis, $\rho\gtrsim q\xi^{-2}\abs{\X}$ and using \eqref{ImprovedBoundsOnSigma}, we see that
\begin{equation}
 \abs{\partial_\kappa\sigma}\lesssim\kappa \rho^{-1}\lesssim \kappa\frac{\xi^2}{q\abs{\X}}.
\end{equation}
From \eqref{eq:dlX} and \eqref{eq:X12_bds}-\eqref{eq:X12_bds3} it thus follows that
\begin{equation}\label{eq:evenbetterdlX}
 \abs{\partial_\lambda X_1}\lesssim \frac{\xi}{q}\kappa\frac{\xi^2}{q\abs{\X}},\qquad \abs{\partial_\lambda X_3}\lesssim \frac{1}{q}\frac{\kappa}{\ip{\kappa}^2}\frac{1}{N}\lesssim \frac{1}{q}\frac{\kappa}{\ip{\kappa}^2}\frac{\xi^2}{q\abs{\X}},
\end{equation}
and this gives the first and second bound in \eqref{eq:betterdl}.
From \eqref{eq:XX1X3} we have that
 \begin{equation}
  \abs{V_1}\lesssim\frac{q}{\xi},\qquad \abs{V_3}\lesssim \frac{q}{\xi^2}\frac{1}{N^2},
 \end{equation}
and by \eqref{eq:dlV} we obtain
\begin{equation}\label{ImprovedVDerProof}
 \vert\partial_\lambda\V\vert\leq\abs{\partial_\lambda V_1}+\lambda\abs{\partial_\lambda V_3}\lesssim\frac{q}{\xi^2}\frac{1}{N^2} \frac{\kappa}{\ip{\kappa}}\lesssim \frac{\xi^2}{q}\frac{1}{\abs{\X}^2}\frac{\kappa}{\ip{\kappa}},
\end{equation}
where we used that $\abs{\X}\lesssim \frac{\xi^2}{q}N$ by \eqref{eq:X12_bds}.

Inspecting \eqref{SecondDerDN}, \eqref{BoundsDN}-\eqref{SecondOrderDerDNFormulas} with the improved bound in \eqref{ImprovedBoundsOnSigma}, we obtain that $\abs{(\xi\partial_\lambda)^2 D}\lesssim N^{-1}$, and  we obtain from \eqref{eq:dl2X} that
\begin{equation}
 \abs{\partial_\lambda^2 X_1}\lesssim \langle\kappa\rangle^{-1}\frac{1}{q}\frac{\xi^2}{q\abs{\X}},\qquad \abs{\partial_\lambda^2 X_3}\lesssim \frac{1}{q\xi}\frac{1}{\ip{\kappa}^2}\frac{1}{N}\lesssim \frac{1}{q\xi}\frac{1}{\ip{\kappa}^2}\frac{\xi^2}{q\abs{\X}},
\end{equation}
which gives $\abs{\partial_\lambda^2\X}\lesssim\frac{1}{q}\langle\kappa\rangle^{-1}\frac{\xi^2}{q\abs{\X}}$. Similarly, deriving \eqref{eq:xidxiX}, we obtain that
\begin{equation*}
 \aabs{\partial_\lambda\partial_\xi\X}\lesssim \frac{\langle\kappa\rangle}{q}\frac{\xi^2}{q\vert\X\vert},
\end{equation*}
and using \eqref{eq:dlPBetaX} and \eqref{ImprovedVDerProof}, we obtain \eqref{eq:betterdlPBetaX}. Similarly, starting from \eqref{eq:dl2V}, we obtain the improved bound
\begin{equation*}
\begin{split}
\vert\partial_\lambda^2{\bf V}\vert&\lesssim\frac{q}{\xi^3}\langle\kappa\rangle^{-1}\frac{1}{N^2}\lesssim\frac{\xi}{q}\langle\kappa\rangle^{-1}\frac{1}{\vert\X\vert^2}
\end{split}
\end{equation*}
and deriving \eqref{DVDXi} and using \eqref{ImprovedVDerProof}, we find that
\begin{equation}\label{eq:betterdlxiV}
 \abs{\partial_\xi\partial_\lambda \V}\lesssim \frac{\xi}{q}\frac{1}{\abs{\X}^2}\frac{\kappa}{\ip{\kappa}},
\end{equation}
and thus the second bound in \eqref{eq:betterdlPBetaX}
%\eqref{eq:betterdlPBetaV}
follows with \eqref{eq:dlPBetaV}.
\end{enumerate}
\end{proof}

We can now compute some bounds which will be important later. We have some simple Poisson brackets
\begin{corollary}\label{CorPBa}

We have the general bounds
\begin{equation*}
\begin{split}
\vert\{\widetilde{\bf X},{\bf a}\}\vert&\lesssim\frac{q}{\xi^2}\vert\widetilde{\bf X}\vert,\qquad\vert\{\widetilde{\bf V},{\bf a}\}\vert\lesssim\frac{q^2}{\xi^3\langle\kappa\rangle^2},
\end{split}
\end{equation*}
while on the bulk, we have the stronger bounds
\begin{equation*}
\begin{split}
\mathfrak{1}_{\mathcal{B}}\vert\{\widetilde{\bf X},{\bf a}\}\vert&\lesssim1,\qquad\mathfrak{1}_{\mathcal{B}}\vert\{\widetilde{\bf V},{\bf a}\}\vert\lesssim\frac{\xi^3}{q^2}\langle t\rangle^{-2},\qquad\mathfrak{1}_{\mathcal{B}}\vert\{\widetilde{\bf V},\eta\}\vert\lesssim\frac{q}{\xi^2},
\end{split}
\end{equation*}
and the precised formula
\begin{equation}\label{DerXiX}
\begin{split}
\mathfrak{1}_{\mathcal{B}}\vert\partial_\xi\widetilde{\bf X}-2\widetilde{\bf X}\vert&\le \frac{\xi^2}{q}\kappa\le\vert\widetilde{\bf X}\vert.
\end{split}
\end{equation}

\end{corollary}

\begin{proof}[Proof of Corollary \ref{CorPBa}]

Using \eqref{PBXV2} and \eqref{PBSIC}, we compute that
\begin{equation*}
\begin{split}
\{\widetilde{\bf X}^j,{\bf a}^k\}&=\{\widetilde{\bf X}^j,a\}{\bf u}^k+a\widetilde{X}_3\in^{jmn}\{{\bf L}^m,{\bf u}^k\}{\bf u}^n=a^{-1}\widetilde{\bf V}^j{\bf u}^k-a\widetilde{X}_3(\delta^{jk}-{\bf u}^j{\bf u}^k),\\
\{\widetilde{\bf V}^j,{\bf a}^k\}&=\{\widetilde{\bf V}^j,a\}{\bf u}^k+a\widetilde{V}_3\in^{jmn}\{{\bf L}^m,{\bf u}^k\}{\bf u}^n=\frac{\xi}{2}\frac{\widetilde{\bf X}^j}{\vert\widetilde{\bf X}\vert^3}{\bf u}^k-a\widetilde{V}_3(\delta^{jk}-{\bf u}^j{\bf u}^k).
\end{split}
\end{equation*}
The first four bounds follow by inspection using \eqref{eq:XX1X3} and \eqref{eq:Xprelimbd}. The bound on $\{\widetilde{\bf V},\eta\}$ follows using \eqref{AddedDerXiV} and \eqref{eq:PBVVeta}. For \eqref{DerXiX}, we observe that by \eqref{eq:xidxiX} we have $\xi\partial_\xi\XX=2\XX-\lambda\partial_\lambda\widetilde{X}_1\u-(\widetilde{X}_3+\lambda\partial_\lambda\widetilde{X}_3)\L\times\u$, which we can estimate in the bulk $\mathcal{B}$ with \eqref{eq:X12_bds} and \eqref{eq:evenbetterdlX} from the proof of Lemma \ref{lem:betterdl}.

\end{proof}

We will also need some derivative bounds.

\begin{lemma}\label{ZPBFormula}
We have the general derivative bounds
\begin{equation}\label{GeneralBoundZBracket}
\begin{split}
\vert\{\widetilde{\bf X},\gamma\}\vert+\vert\{\widetilde{\bf Z},\gamma\}\vert&\lesssim_q \left[t(\xi^{-1}+\xi^{-3})+1+\xi^4+\vert\eta\vert^2+\lambda^2\right]\cdot\sum_{f\in\textnormal{SIC}}\ww_f\vert\{f,\gamma\}\vert,\\
\vert\{\a,\gamma\}\vert+\vert\{\VV,\gamma\}\vert&\lesssim_q\left[1+\xi^{-3}+(t+\xi^{-1})/\vert\XX\vert\right]\cdot\sum_{f\in\textnormal{SIC}}\ww_f\vert\{f,\gamma\}\vert,
\end{split}
\end{equation}
and in the bulk we have the more precise bounds
\begin{equation}\label{PrecisedBoundZBracketBulk}
\begin{split}
\mathfrak{1}_{\mathcal{B}}\vert \{\widetilde{\bf Z},\gamma\}\vert&\lesssim \ip{\xi}^2\ln\ip{t}\cdot\sum_{f\in\textnormal{SIC}}\ww_f\vert\{f,\gamma\}\vert,\\
% \mathfrak{1}_{\mathcal{B}}\vert \{\widetilde{\bf Z},\gamma\}\vert&\lesssim\left[\langle \xi\rangle^2\ln\langle t\rangle+\red{\lambda}\right]\cdot\sum_{f\in\hbox{SIC}}\ww_f\vert\{f,\gamma\}\vert,\\
%\mathfrak{1}_{\mathcal{B}}\vert \{\widetilde{\bf V},\gamma\}\vert&\lesssim\left[1+\xi^{-3}\right]\cdot\sum_{f\in\hbox{SIC}}\mathfrak{w}_f\vert\{f,\gamma\}\vert,\\
\mathfrak{1}_{\mathcal{B}}\vert \{\widetilde{\bf V}-\a,\gamma\}\vert&\lesssim\frac{\xi+\xi^{3}}{t}\cdot\sum_{f\in\textnormal{SIC}}\ww_f\vert\{f,\gamma\}\vert.
\end{split}
\end{equation}

\end{lemma}

\begin{proof}[Proof of Lemma \ref{ZPBFormula}]

For control of $\{\widetilde{\bf X},\gamma\}$, we use \eqref{eq:U_resol}  to get
\begin{equation*}
\begin{split}
\{\XX,\gamma\}&=\frac{\xi^2\{\XX^j,\eta\}}{1+\xi}\ww_\xi\{\xi,\gamma\}-\frac{(1+\xi)\{\XX^j,\xi\}}{\xi^2}\ww_\eta\{\eta,\gamma\}+(1+\xi^{-1})\partial_\lambda\XX^j{\bf l}^k\cdot\ww_{\bf L}\{{\bf L}^k,\gamma\}+\widetilde{X}_1\ww_{\bf u}\{{\bf u}^j,\gamma\}\\
&\quad+\widetilde{X}_3\in^{jcd}({\bf L}^c\cdot\ww_{\bf u}\{{\bf u}^d,\gamma\}+{\bf u}^d(1+\xi^{-1})\cdot\ww_{{\bf L}}\{{\bf L}^c,\gamma\})
\end{split}
\end{equation*}
and using \eqref{PBX1} and \eqref{PBX1'} with \eqref{BoundsX}, this gives the bound on $\{\XX,\gamma\}$ in \eqref{GeneralBoundZBracket}. The general bounds for $\{\widetilde{\bf V},\gamma\}$ follows similarly, while the improved bounds follow from Corollary \ref{CorPBa}. In addition,
\begin{equation}\label{eq:veca_PB}
\begin{split}
\{\a,\gamma\}&=-\frac{q}{\xi^2}\{\xi,\gamma\}{\bf u}+\frac{q}{\xi}\{\u,\gamma\}.
\end{split}
\end{equation}

The estimates of $\{\widetilde{\bf Z},\gamma\}$ follow along similar lines. Starting with
\begin{equation*}
\begin{split}
\widetilde{\bf Z}&=\widetilde{Z}_1{\bf u}+\widetilde{Z}_3{\bf L}\times{\bf u},\\
\widetilde{Z}_1&=\frac{\xi^2}{q}\left(\eta+\widetilde{\sigma}+\left(\frac{1}{2}+\frac{\widetilde{D}}{1+4\kappa^2}\right)\left(1+\frac{tq^2}{\xi^3\widetilde{N}}\right)\right),\qquad
\widetilde{Z}_3=-\frac{2\xi}{q}\left(\frac{1}{2}+\frac{\widetilde{D}}{1+4\kappa^2}\right)\left(1+\frac{tq^2}{\xi^3\widetilde{N}}\right),
\end{split}
\end{equation*}
where $\widetilde{D}=D(\kappa,\eta+tq^2/\xi^3)$ and similarly for $\widetilde{\sigma}$ and $\widetilde{N}$. We compute that
\begin{equation*}
\begin{split}
\{\widetilde{Z}_1{\bf u},\gamma\}&=\frac{\{\xi,\gamma\}}{\xi}\cdot{\bf u}\cdot\Big[2\widetilde{Z}_1-\frac{\xi^2}{q}\Big(\kappa\partial_\kappa\widetilde{\sigma}+\frac{3tq^2}{\xi^3}\partial_\eta\widetilde{\sigma}+\left(\frac{1}{2}+\frac{\widetilde{D}}{1+4\kappa^2}\right)\frac{tq^2}{\xi^3\widetilde{N}}(3-\frac{\kappa\partial_\kappa\widetilde{N}}{\widetilde{N}}-\frac{3tq^2}{\xi^3}\frac{\partial_\eta\widetilde{N}}{\widetilde{N}})\\
&\quad+\frac{1}{1+4\kappa^2}\left(1+\frac{tq^2}{\xi^3\widetilde{N}}\right)\left(\kappa\partial_\kappa\widetilde{D}+\frac{3tq^2}{\xi^3}\partial_\eta\widetilde{D}-\frac{8\kappa^2}{1+4\kappa^2}\widetilde{D}\right)\Big)\Big]\\
&\quad+\{\eta,\gamma\}{\bf u}\frac{\xi^2}{q}\left(1+\partial_\eta\widetilde{\sigma}-\left(\frac{1}{2}+\frac{\widetilde{D}}{1+4\kappa^2}\right)\partial_\eta\widetilde{N}+\frac{1}{1+4\kappa^2}\left(1+\frac{tq^2}{\xi^3\widetilde{N}}\right)\partial_\eta\widetilde{D}\right)\\
&\quad+\frac{\{\lambda,\gamma\}}{\xi}{\bf u}\frac{\xi^2}{q}\left(\partial_\kappa\widetilde{\sigma}-\left(\frac{1}{2}+\frac{\widetilde{D}}{1+4\kappa^2}\right)\frac{tq^2}{\xi^3\widetilde{N}}\frac{\partial_\kappa\widetilde{N}}{\widetilde{N}}+\frac{1}{1+4\kappa^2}\left(1+\frac{tq^2}{\xi^3\widetilde{N}}\right)(\partial_\kappa\widetilde{D}-\frac{8\kappa}{1+4\kappa^2}\widetilde{D})\right)\\
&\quad+\widetilde{Z}_1\{{\bf u},\gamma\},
\end{split}
\end{equation*}
and similarly
\begin{equation*}
\begin{split}
\{Z_3,\gamma\}&=\frac{\{\xi,\gamma\}}{\xi}\cdot\Big[-\widetilde{Z}_3+\frac{2\xi}{q}\frac{1}{1+4\kappa^2}\left(1+\frac{tq^2}{\xi^3\widetilde{N}}\right)\left(\kappa\partial_\kappa\widetilde{D}-\frac{8\kappa^2}{1+4\kappa^2}\widetilde{D}+\frac{3tq^2}{\xi^3}\partial_\eta\widetilde{D}\right)\\
&\quad-\frac{2\xi}{q}\left(\frac{1}{2}+\frac{\widetilde{D}}{1+4\kappa^2}\right)\frac{tq^2}{\xi^3\widetilde{N}}(-3+\frac{3tq^2}{\xi^3}\frac{\partial_\eta\widetilde{N}}{\widetilde{N}}+\frac{\kappa\partial_\kappa\widetilde{N}}{\widetilde{N}})\Big]\\
&\quad+\{\eta,\gamma\}\cdot\left(-\frac{2\xi}{q}\left(1+\frac{tq^2}{\xi^3\widetilde{N}}\right)\partial_\eta\widetilde{D}+\frac{2\xi}{q}\left(\frac{1}{2}+\frac{\widetilde{D}}{1+4\kappa^2}\right)\frac{tq^2}{\xi^3\widetilde{N}}\frac{\partial_\eta\widetilde{N}}{\widetilde{N}}\right)\\
&\quad+\frac{\{\lambda,\gamma\}}{\xi}\cdot\frac{2\xi}{q}\Big(\frac{1}{1+4\kappa^2}\left(1+\frac{tq^2}{\xi^3\widetilde{N}}\right)\left(-\kappa\partial_\kappa\widetilde{D}+\frac{8\kappa}{1+4\kappa^2}\widetilde{D}\right)+\left(\frac{1}{2}+\frac{\widetilde{D}}{1+4\kappa^2}\right)\frac{tq^2}{\xi^3\widetilde{N}}\frac{\partial_\kappa\widetilde{N}}{\widetilde{N}}\Big).
\end{split}
\end{equation*}
In general, we have that
\begin{equation}\label{GeneralBounds}
\begin{split}
\vert\widetilde{Z}_1\vert\lesssim\frac{\xi^2}{q}(1+\vert\eta\vert)+\frac{tq}{\xi},\qquad\vert\widetilde{Z}_3\vert&\lesssim\frac{\xi}{q}(1+\vert\eta\vert)+\frac{tq}{\xi^2},\qquad\vert\kappa\partial_\kappa\widetilde{\sigma}\vert\lesssim 1,\qquad\vert\partial_\eta\widetilde{\sigma}\vert\lesssim\langle\kappa\rangle^{-1},
\end{split}
\end{equation}
and we obtain the bound
\begin{equation*}
\begin{split}
\vert\{\widetilde{\bf Z},\gamma\}\vert&\lesssim \frac{1}{1+\xi}\left(\frac{\xi^3}{q}(\vert\eta\vert+1)+tq\right)\cdot\ww_\xi\vert\{\xi,\gamma\}\vert+\frac{1+\xi}{q}\left(1+\frac{\vert\eta\vert+tq^2\xi^{-3}}{1+\kappa^2}\right)\cdot\ww_\eta\vert\{\eta,\gamma\}\vert\\
&\quad+\left(t\frac{q}{\xi}+\frac{\xi^2}{q}(\vert\eta\vert+\kappa)\right)\cdot\ww_{\bf u}\vert\{{\bf u},\gamma\}\vert+\frac{1+\xi}{q}\left(1+\frac{\vert\eta\vert+tq^2\xi^{-3}}{1+\kappa^2}\right)\cdot\ww_{\bf L}\vert\{{\bf L},\gamma\}\vert,
\end{split}
\end{equation*}
from which we deduce \eqref{GeneralBoundZBracket}. In the bulk $\mathcal{B}$, we have that
\begin{equation*}
\begin{split}
\vert\widetilde{Z}_1\vert\lesssim\frac{\xi^2}{q}\ln (2+ t),\qquad\vert\widetilde{Z}_3\vert\lesssim \xi/q,\qquad \vert\partial_\kappa\widetilde{\sigma}\vert\lesssim\frac{\xi^6\kappa}{t^2q^4},\qquad \vert\partial_\eta\widetilde{\sigma}\vert\lesssim\frac{\xi^3}{tq^2},\\
\frac{tq^2}{\xi^3\widetilde{N}}\lesssim 1,\qquad\widetilde{D}\lesssim\frac{\xi^3(1+\kappa^2)}{tq^2},\qquad\vert\partial_\eta\widetilde{D}\vert\lesssim\frac{\xi^6(1+\kappa^2)}{t^2q^4},\qquad\vert\partial_\kappa\widetilde{N}\vert=\vert \partial_\kappa\widetilde{D}\vert\lesssim\frac{\xi^3\kappa}{tq^2},\qquad\vert\partial_\eta\widetilde{N}\vert\lesssim 1,
\end{split}
\end{equation*}
and this gives that
\begin{equation*}
\begin{split}
\mathfrak{1}_{\mathcal{B}}\vert \{\widetilde{\bf Z},\gamma\}\vert&\lesssim \frac{\xi^3}{q(1+\xi)}\ln\ip{t}\cdot\ww_\xi\vert\{\xi,\gamma\}\vert+\frac{1+\xi}{q}\cdot\ww_\eta\vert\{\eta,\gamma\}\vert+\frac{\xi^2}{q}\ln\langle t\rangle\cdot\mathfrak{m}_{\bf u}\vert\{{\bf u},\gamma\}\vert+\ww_{\bf L}\vert\{{\bf L},\gamma\}\vert,
\end{split}
\end{equation*}
from which \eqref{PrecisedBoundZBracketBulk} follows. Similarly, we compute that
\begin{equation*}
\begin{split}
\{\VV^j-\a^j,\gamma\}&=-\frac{\VV^j-\a^j}{\widetilde{N}}\{\widetilde{N},\gamma\}+(\widetilde{V}_1-1)\{\u^j,\gamma\}+\widetilde{V}_3\in^{jkl}\{\L^k,\gamma\}{\bf u}^l+\widetilde{V}_3\in^{jkl}\L^k\{\u^l,\gamma\}\\
&\quad-\frac{q}{2\widetilde{N}}\frac{1}{1+4\kappa^2}\left[\{\widetilde{D},\gamma\}-\frac{8\kappa \widetilde{D}}{1+4\kappa^2}\{\kappa,\gamma\}+\frac{2\widetilde{D}}{\xi^2}\{\xi,\gamma\}\right](\u^j-2\kappa({\bf l}\times{\bf u})^j).
\end{split}
\end{equation*}
Since in the bulk we have also
\begin{equation*}
\begin{split}
\vert\VV-\a\vert+\vert\widetilde{V}_1\vert\lesssim\frac{\xi^3}{tq},\qquad\vert\widetilde{V}_3\vert\lesssim \frac{\xi^6}{t^2q^3}\kappa,
\end{split}
\end{equation*}
we deduce that
\begin{equation*}
\begin{split}
\vert\{\VV-\a,\gamma\}\vert&\lesssim_q \frac{\xi^3}{t}\ww_\xi\vert\{\xi,\gamma\}\vert+\frac{\xi^4\langle\xi\rangle}{t^2}\ww_\eta\vert\{\eta,\gamma\}\vert+\frac{\xi^3}{t}\ww_\u\vert\{\u,\gamma\}\vert+\frac{\langle\xi\rangle\xi}{t}\ww_\L\vert\{\L,\gamma\}\vert.
\end{split}
\end{equation*}

\end{proof}

\begin{remark}
In view of the terms arising in the nonlinearity, for the sake of completeness we also compute that
\begin{equation}
 \{\xi,V_3\}=\lambda^{-1}\{\xi,\V\}\cdot(\l\times\u)=\lambda^{-1}\frac{\xi^3}{2q}\frac{\X\cdot(\l\times\u)}{\abs{\X}^3}=\frac{\xi^3}{2q}\frac{X_3}{\abs{\X}^3}.
\end{equation}
In the region $\{\X\cdot\V>-\ip{\kappa}\xi\}$, as in the above lemma we obtain with \eqref{eq:X12_bds} that
\begin{equation}
 \aabs{\{\xi,\widetilde{V}_3\}}\lesssim \frac{\xi^2}{q}\frac{1}{\aabs{\XX}^2}\cdot\frac{\xi^2}{q\aabs{\XX}},\qquad \abs{\{\eta,V_3\}}\lesssim \frac{\xi}{q\abs{\X}^{2}},
\end{equation}
and thus also
\begin{equation}\label{eq:PBxiV3}
 \aabs{\{\eta,\widetilde{V}_3\}}\lesssim\aabs{\{\eta,V_3\}
 \circ\Phi_t^{-1}}+\aabs{3t\xi^{-4}q^2\{\xi,\widetilde{V}_3\}}\lesssim \frac{\xi}{q\aabs{\XX}^{2}}+\frac{t}{\aabs{\XX}^3}.
\end{equation}
Similarly,
\begin{equation}\label{eq:PBxiX3}
 \aabs{\{\xi,X_3\}}=\aabs{\frac{\xi^3}{q^2}V_3}\lesssim \frac{\xi}{q}\frac{1}{N^2}\lesssim \frac{\xi^5}{q^3}\frac{1}{\abs{\X}^2},\qquad\vert\{\eta,X_3\}\vert\lesssim\frac{1}{q},\qquad
 \vert\{\eta,\widetilde{X}_3\}\vert\lesssim\frac{1}{q}(1+\frac{t\xi}{\abs{\X}^2}).
\end{equation}

\end{remark}

\subsection{Transition maps}\label{ssec:transitions}
Since we propagate derivative control through Poisson brackets with the super-integrable coordinates, we will need to understand the relation between Poisson brackets in past versus future asymptotic actions, and compare the close and far formulations.

\subsubsection*{Transition from past to future asymptotic actions}
By construction, the past asymptotic actions are anchored at the periapsis in \eqref{eq:def_SIC(-)} in such a way that the Poisson bracket relations \eqref{PBSIC} remain almost unchanged: 
Using that
\begin{equation*}
\begin{split}
\{\eta,\kappa\}=\frac{\kappa}{\xi},\qquad\{\kappa,{\bf u}\}=-\frac{1}{\xi}(\l\times\u),\qquad\{\kappa,\l\times\u\}=\frac{1}{\xi}{\bf u},
\end{split}
\end{equation*}
we can verify with \eqref{eq:def_SIC(-)} that
\begin{equation}\label{eq:PB_SIC(-)}
\begin{split}
\{\xi^{(-)},\eta^{(-)}\}=1,\qquad 0&=\{\xi^{(-)},\lambda^{(-)}\}=\{\eta^{(-)},\lambda^{(-)}\},\\
\{\xi^{(-)},\u^{(-)}\}=\{\eta^{(-)},\u^{(-)}\}=0&=\{\xi^{(-)},\L^{(-)}\}=\{\eta^{(-)},\L^{(-)}\}=\{\lambda^{(-)},\L^{(-)}\},\\
\{\lambda^{(-)},{\bf u}^{(-)}\}&=\l^{(-)}\times\u^{(-)},\quad\{{\bf u}^{(-),j},{\bf u}^{(-),k}\}=0,\\
\{\L^{(-),j},\u^{(-),k}\}&=\in^{jka}\u^{(-),a},\quad \{\L^{(-),j},\L^{(-),k}\}=\in^{jka}\L^{(-),a}.
\end{split}
\end{equation}
The relation between Poisson brackets in past and future asymptotic actions is now easily established: for a scalar function $\zeta$ we have
\begin{equation}\label{eq:inout_PBtrans}
\begin{split}
\{\xi^{(-)},\zeta\}&=-\{\xi,\zeta\},\qquad\{\eta^{(-)},\zeta\}=-\frac{1}{\xi}\frac{4\kappa^2}{4\kappa^2+1}\{\xi,\zeta\}-\{\eta,\zeta\}+\frac{1}{\xi}\frac{4\kappa}{4\kappa^2+1}\{\lambda,\zeta\},\\
 \{\lambda^{(-)},\zeta\}&=\{\lambda,\zeta\},\qquad  \{\L^{(-)},\zeta\}=\{\L,\zeta\},\\
\{{\bf u}^{(-)},\zeta\}&=\frac{1}{\xi}\left[-\frac{16\kappa^2}{(4\kappa^2+1)^2}\u+\frac{4\kappa(1-4\kappa^2)}{(4\kappa^2+1)^2}\l\times\u\right]\{\xi,\zeta\}+\frac{1}{\xi}\left[\frac{16\kappa^2}{(4\kappa^2+1)^2}{\bf u}+\frac{32\kappa^2}{(4\kappa^2+1)^2}\l\times\u\right]\{\lambda,\zeta\}\\
&\quad+\frac{1-4\kappa^2}{4\kappa^2+1}\{{\bf u},\zeta\}-\frac{1}{\xi}\frac{4}{4\kappa^2+1}\{\L\times\u,\zeta\}.
\end{split}
\end{equation}
In particular, we highlight that the transition from past to future asymptotic actions (at e.g.\ the periapsis) can be carried out along a given trajectory.

\subsubsection*{Transition between far and close formulations}
With the notation of Section \ref{ssec:farclose}, we recall that the transition between the close and far formulations is given by the canonical diffeomorphism $\mathcal{M}_t$ of \eqref{eq:defM_t}. In particular, for Poisson brackets with a scalar function $\ww$ on phase space we have
\begin{equation}\label{eq:PBrel'}
 \{\ww,\gamma'\}=\{\ww,\gamma\circ\mathcal{M}_t\}=\{\ww',\gamma\}\circ\mathcal{M}_t,\qquad \ww':=\ww\circ\mathcal{M}_t^{-1}.
\end{equation}

The derivative control of Section \ref{sec:derivs_prop} will be given in terms of a collection of Poisson brackets of the unknown $\gamma$ with the functions $f\in\{\xi,\eta,\L,\u\}$, weighted by $\ww_f\in\R$, collected as
\begin{equation}
 \mathcal{D}(\vartheta,\a;t):=\sum_{f\in\{\xi,\eta,\L,\u\}}\ww_f\aabs{\{f,\gamma\}}.
\end{equation}
The corresponding transition maps are collected below in Lemma \ref{lem:trans_short} resp.\ Lemma \ref{lem:trans_full}.

\section{Bounds on the electric field}\label{sec:efield}

From here on, we stop tracking the homogeneity of the quantities since the magnitude of electric quantities will be compared to powers of $\sqrt{t^2+\vert{\bf y}\vert^2}$ which is not homogeneous.

\subsection{Electric field, localized electric field and effective field}

Given a density $\gamma$, we define the electric field and its derivative as
\begin{equation}\label{DefEF}
\begin{split}
\mathcal{E}_j({\bf y},t):=\frac{1}{4\pi}\iint {\bf U}_j({\bf y}-\widetilde{\bf X}(\vartheta,\a ))\cdot \gamma^2(\vartheta,\a,t)d\vartheta d\a ,\\
\mathcal{F}_{jk}({\bf y},t):=\frac{1}{4\pi}\iint\mathcal{M}_{jk}({\bf y}-\widetilde{\bf X}(\vartheta,\a ))\cdot\gamma^2(\vartheta,\a ,t)d\vartheta d\a ,
\end{split}
\end{equation}
where
\begin{equation*}
\begin{split}
{\bf U}_j({\bf y}):=\partial_{{\bf y}^j}\frac{-1}{\vert {\bf y}\vert}=\frac{{\bf y}^j}{\vert {\bf y}\vert^3},\qquad \mathcal{M}_{jk}({\bf y}):=\partial_{y^j}\partial_{y^k}\frac{-1}{\vert {\bf y}\vert}=\frac{1}{\vert {\bf y}\vert^3}\left(\delta_{jk}-3\frac{{\bf y}^j{\bf y}^k}{\vert{\bf y}\vert^2}\right).
\end{split}
\end{equation*}
Here we show how to bound the electric field and its derivative using moments on the unknown $\gamma$. Towards this, we define the \emph{bulk zone} as in \eqref{eq:bulk}
\begin{equation*}%\label{eq:bulk}
 \B:=\{(\vartheta,\a)\in\mathcal{P}_{\vartheta,\a}:\;  a+\xi \le (q/\langle q\rangle) t^\frac{1}{4}/10,\quad \xi\abs{\eta}+\lambda\le 10^{-3}ta^2\},
\end{equation*}
and we notice from Corollary \ref{CorBdsOnX} resp.\ Lemma \ref{ImprovedBoundsOnXBulk} that if $(\vartheta,\a)\in\B$, then $(\vartheta+t\a,\a)\in\Omega_+$ for $t>0$, and we have the simple bounds
\begin{equation}\label{eq:bulk-bds}
 2ta^3/q\geq\eta(\vartheta+t\a,\a)\ge \frac{1}{2}ta^3/q,\qquad \rho(\vartheta+t\a,\a)\ge \frac{1}{8}ta^3/q,\qquad 10^{-3}ta\le \vert\widetilde{\bf X}\vert\le 10^3ta.
\end{equation}
In the complement $\B^c$ we can trade moments for decay in the sense that for all $k\geq 0$
\begin{equation}\label{eq:nonbulk-bds}
 \mathfrak{1}_{\B^c}\lesssim_k t^{-k}\left[\xi^2\abs{\eta}^2+\lambda^2+a^4+\xi^4\right]^k\mathfrak{1}_{\B^c}.
\end{equation}

In order to obtain refined bounds on the electric field, we decompose it onto scales. Let $\varphi\in C^\infty_c(\R^3)$ be a standard, radial cutoff function with $\text{supp}(\varphi)\subset\{1/2\le \vert {\bf x}\vert\le 2\}$ and $\int_{\R^3}\varphi({\bf x})/\vert{\bf x}\vert^2d{\bf x}=1$. Note that since
 \begin{equation}\label{IntroducingVarphi}
  \frac{1}{4\pi\abs{\x}}=\int_{R=0}^\infty\varphi_R(\x)\frac{dR}{R^2},\qquad\varphi_R(\x)=\varphi(R^{-1}\x),
 \end{equation}
 we can decompose
 \begin{equation}\label{eq:defER}
  \mathcal{E}_j(\y,t)=\int_{R=0}^\infty \mathcal{E}_R(\y,t)\frac{dR}{R},\qquad \mathcal{E}_{j,R}(\y,t)=-\frac{1}{R^2}\iint\partial_j\varphi_R(\y-\XX(\vartheta,\a))\gamma^2(\vartheta,\a,t)d\vartheta d\a.
 \end{equation}
We also introduce the effective electric field
\begin{equation*}
\begin{split}
 \mathcal{E}^{eff}_j({\bf y},t)=\int_{R=0}^\infty\mathcal{E}^{eff}_{j,R}({\bf y},t)\frac{dR}{R},\qquad \mathcal{E}^{eff}_{j,R}({\bf y},t)&=-R^{-2}\iint \partial_j\varphi_R({\bf y}-t\a)\gamma^2(\vartheta,\a,t)d\vartheta d\a.
 %&=\frac{1}{t^2}e^j_{R/t}(\y/t,t).
\end{split}
\end{equation*}
We proceed similarly with the derivative of the electric field
\begin{equation*}
\begin{split}
\mathcal{F}_{jk}({\bf y},t)&=\int_{R=0}^\infty \mathcal{F}_{jk,R}({\bf y},t)\frac{dR}{R},\qquad \mathcal{F}_{jk,R}({\bf y},t):=-R^{-3}\iint\partial_j\partial_k\varphi_R({\bf y}-\widetilde{\bf X}(\vartheta,\a ))\cdot\gamma^2(\vartheta,\a ,t)d\vartheta d\a,\\
\mathcal{F}_{jk}^{eff}({\bf y},t)&=\int_{R=0}^\infty \mathcal{F}^{eff}_{jk,R}({\bf y},t)\frac{dR}{R},\qquad \mathcal{F}^{eff}_{jk,R}({\bf y},t):=-R^{-3}\iint\partial_j\partial_k\varphi_R({\bf y}-t\a )\cdot\gamma^2(\vartheta,\a ,t)d\vartheta d\a .
\end{split}
\end{equation*}
When $R$ is too small, volume bounds are not enough to overcome the singularity at $R=0$ and we rewrite, using \eqref{ZeroIntegralPB}, and the constant of motion ${\bf z}$ for the asymptotic equation:
\begin{equation}\label{FjkAndPB2}
\begin{split}
\mathcal{F}_{jk,R}({\bf y},t)&=R^{-2}\iint \{\partial_j\varphi_R({\bf y}-\x),{\bf v}^k\}\cdot\mu^2(\x,\v ,t)d\x d\v,\\
&=-R^{-2}t^{-1}\iint \{\partial_j\varphi_R({\bf y}-{\bf x}),{\bf x}^k-t{\bf v}^k\}\cdot\mu^2(\x,\v ,t)d\x d\v\\
&=-R^{-2}t^{-1}\iint\{\partial_j\varphi_R({\bf y}-\widetilde{\bf X}(\vartheta,\a )),\widetilde{\bf Z}^k\}\cdot\gamma^2(\vartheta,\a ,t)d\vartheta d\a \\
&=-2R^{-2}t^{-1}\iint\partial_j\varphi_R({\bf y}-\widetilde{\bf X}(\vartheta,\a ))\{\widetilde{\bf Z}^k,\gamma\}\cdot\gamma(\vartheta,\a ,t)d\vartheta d\a,
\end{split}
\end{equation}
which has a similar structure as $\mathcal{E}_j$, except that we have replaced one copy of $\gamma$ with a derivative.

\subsection{Approximating the electric field by the effective electric field}

Assuming only bounds on the moments, we can obtain good bounds on the electric field, and assuming control on Poisson brackets leads to control on the derivatives of the electric field.
\begin{proposition}\label{PropControlEF}

The electric field is well approximated by the effective field:
\begin{equation}\label{ControlEMom}
\begin{split}
\left[t^2+\vert{\bf y}\vert^2\right]\cdot\vert \mathcal{E}({\bf y},t)-\mathcal{E}^{eff}({\bf y},t)\vert&\lesssim
t^{-\frac{1}{5}}N_1,\\
\left[t^2+\vert{\bf y}\vert^2\right]^\frac{3}{2}\cdot\vert \mathcal{F}({\bf y},t)-\mathcal{F}^{eff}({\bf y},t)\vert&\lesssim t^{-\frac{1}{5}}N_2,
\end{split}
\end{equation}
with
\begin{equation*}
\begin{split}
N_1&:=\norm{\left[\xi^2+a^3+\vert\eta\vert^2+\lambda\right]\gamma}_{L^2_{\vartheta,\a}}^2+\norm{\left[\xi^4+a^{10}+\vert\eta\vert^{10}+\lambda^5\right]\gamma}_{L^\infty_{\vartheta,\a}}^2,\\
N_2&:=N_1+\norm{\left[\xi^2+a^2+\vert\eta\vert^{2}+\lambda^2\right]^{12}\gamma}_{L^\infty_{\vartheta,\a}}^2+\sum_{f\in \textnormal{SIC}}\Vert \ww_f\{f,\gamma\}\Vert_{L^\infty}^2.
\end{split}
\end{equation*}
\end{proposition}

\begin{proof}
 We first observe that $\mathcal{E}_R$ satisfies the simple bound
 \begin{equation}\label{eq:ERsimplebd}
 R^2 \abs{\mathcal{E}_R(\y,t)}+R^2 \aabs{\mathcal{E}^{eff}_R(\y,t)}+R^3\vert\mathcal{F}_R(\y,t)\vert+R^3\vert\mathcal{F}^{eff}_R(\y,t)\vert\lesssim \norm{\gamma}_{L^2_{\vartheta,\a}}^2.
 \end{equation}
 If $t\lesssim_q1$, the bounds follow by simple estimates on the convolution kernel. In what follows, we assume that $t\gg_q1$.
 
 \medskip
 
 {\bf A}) The electric field. We prove the first bound in \eqref{ControlEMom}. {\bf (A1) Large scales}: $R\ge (t^2+\vert{\bf y}\vert^2)^\frac{3}{8}/100$. We first compare at each scale
\begin{equation*}
\begin{split}
\mathcal{E}_{j,R}({\bf y},t)-\mathcal{E}^{eff}_{j,R}({\bf y},t)&= R^{-2}\iint_{\mathcal{B}} \left[\partial_j\varphi_R({\bf y}-\widetilde{\bf X}(\vartheta,\a))-\partial_j\varphi_R({\bf y}-t\a)\right]\gamma^2(\vartheta,\a,t)d\vartheta d\a\\
&\quad+R^{-2}\iint_{\mathcal{B}^c} \left[\partial_j\varphi_R({\bf y}-\widetilde{\bf X}(\vartheta,\a))-\partial_j\varphi_R({\bf y}-t\a)\right]\gamma^2(\vartheta,\a,t)d\vartheta d\a=I_{\mathcal{B}}+I_{\mathcal{B}^c}.
\end{split}
\end{equation*}
In the bulk, we can use \eqref{ApproxXBulk} to get
\begin{equation*}
\begin{split}
\vert I_{\mathcal{B}}\vert&\lesssim R^{-3}\iint_{\mathcal{B}} \vert\widetilde{\bf X}(\vartheta,\a)-t\a\vert \gamma^2(\vartheta,\a,t)d\vartheta d\a\lesssim R^{-3}\langle \ln t\rangle\iint \left[1+\xi^4+\eta^4+\lambda^2\right]\gamma^2(\vartheta,\a,t)d\vartheta d\a.
\end{split}
\end{equation*}
Outside the bulk, we estimate each term separately $\vert I_{\mathcal{B}^c}\vert\le \vert I^1_{\mathcal{B}^c}\vert+\vert I^2_{\mathcal{B}^c}\vert$. If $\vert {\bf y}\vert\le 10t$, we can use \eqref{eq:nonbulk-bds} to deduce that 
 \begin{equation}
 \begin{aligned}
  \vert I^1_{\mathcal{B}^c}\vert+ \vert I^2_{\mathcal{B}^c}\vert&\lesssim R^{-2}\iint  t^{-k}\left[\xi^{2k}\abs{\eta}^{2k}+\lambda^{2k}+a^{4k}+\xi^{4k}\right]\gamma^2(\vartheta,\a,t)d\vartheta d\a\\
  &\lesssim (t^2+\vert\y\vert^2)^{-k/2}R^{-2}\Vert \left[\xi^{2k}+a^{2k}+\vert\eta\vert^{2k}+\lambda^k\right]\gamma\Vert_{L^2_{\vartheta,\a}}^2.
 \end{aligned} 
 \end{equation}
If $\vert{\bf y}\vert\ge 10t$, and $R\ge \vert{\bf y}\vert/4$, the same bound gives
\begin{equation}
 \begin{aligned}
  \vert I^1_{\mathcal{B}^c}\vert+ \vert I^2_{\mathcal{B}^c}\vert&\lesssim R^{-2}\langle t\rangle^{-k}\mathfrak{1}_{\{R\ge (t^2+\vert\y\vert^2)^\frac{1}{2}/10\}}\cdot\Vert \left[\xi^{2k}+a^{2k}+\vert\eta\vert^{2k}+\lambda^k\right]\gamma\Vert_{L^2_{\vartheta,\a}}^2.
 \end{aligned} 
 \end{equation}
Else, we see that, on the support of $\partial_j\varphi_R({\bf y}-t\a)$, we must have that $ a\ge \vert {\bf y}\vert/(2t)$, and we can modify the bound above to bound the effective field
 \begin{equation*}
 \begin{aligned}
\vert I_{\mathcal{B}^c}^2\vert&:= \vert \iint_{\mathcal{B}^c}\partial_j\varphi_R({\bf y}-t\a)\gamma^2(\vartheta,\a,t)d\vartheta d\a\vert\\
 &\lesssim \vert{\bf y}\vert^{-k}R^{-2}\iint  \xi^{-k}\left[\xi^{2k}\abs{\eta}^{2k}+\lambda^{2k}+a^{4k}+\xi^{4k}\right]\gamma^2(\vartheta,\a,t)d\vartheta d\a\\
  &\lesssim \left(t^2+\vert y\vert^2\right)^{-k/2}R^{-2}\Vert \left[\xi^{2k}+a^{3k}+\vert\eta\vert^{2k}+\lambda^{2k}\right]\gamma\Vert_{L^2_{\vartheta,\a}}^2.
 \end{aligned} 
 \end{equation*}
On the support of $\partial_j\varphi_R({\bf y}-\widetilde{\bf X}(\vartheta,\a))$, we have that $\vert\widetilde{X}(\vartheta,\a)\vert\ge\vert\y\vert/2$ and therefore
\begin{equation}\label{UpperBoundytEF}
\begin{split}
t^2+\vert\y\vert^2&\le4(t^2+\vert\widetilde{\bf X}\vert^2)\lesssim t^2(1+a^2)+\vert\eta\vert^4+\frac{\xi^8}{q^4}+\frac{\lambda^4}{q^2},
\end{split}
\end{equation}
and we can use variation of the previous argument:
 \begin{equation*}
 \begin{aligned}
\vert  I_{\mathcal{B}^c}^1\vert&:= \vert \iint_{\mathcal{B}^c}\partial_j\varphi_R({\bf y}-\widetilde{\bf X}(\vartheta,\a))\gamma^2(\vartheta,\a,t)d\vartheta d\a\vert\\
 %&\lesssim \vert{\bf y}\vert^{-k}R^{-2}\iint  \xi^{k}\left[\xi^{2k}\abs{\eta}^{2k}+L^{2k}+a^{4k}+\xi^{4k}\right]\gamma^2(\vartheta,\a,t)d\vartheta d\a\\
  &\lesssim \left(t^2+\vert y\vert^2\right)^{-k/2}R^{-2}\Vert \left[\xi^{2k}+a^{2k}+\vert\eta\vert^{2k}+\lambda^{k}\right]\gamma\Vert_{L^2_{\vartheta,\a}}^2.
 \end{aligned} 
 \end{equation*}

Taking $k=1$ and integrating over $R\ge(t^2+\vert\y\vert^2)^\frac{3}{8}$, we obtain an acceptable contribution to the first line in \eqref{ControlEMom} since
\begin{equation*}
\begin{split}
\left[t^2+\vert\y\vert^2\right]\int_{R=(t^2+\vert\y\vert^2)^\frac{3}{8}}^\infty\left[\left[t^2+\vert\y\vert^2\right]^{-\frac{1}{2}}+\langle t\rangle^{-1}\mathfrak{1}_{\{R\ge\sqrt{t^2+\vert\y\vert^2}\}}+\frac{\langle\ln t\rangle}{R}\right]\frac{dR}{R^3}\lesssim \langle t\rangle^{-\frac{1}{5}}.
\end{split}
\end{equation*}

{\bf (A2) Small scales}: $R\le (t^2+\vert {\bf y}\vert^2)^\frac{3}{8}/100$, contributions of the electric field. In this case, we again observe that \eqref{UpperBoundytEF} continues to hold on the support of $\partial_j\varphi_R({\bf y}-\widetilde{\bf X}(\vartheta,\a))$ (this is clear if $\vert\y\vert\le 10t$, while if $\vert \y\vert\ge 10t$, then the bound on $R$ forces $\widetilde{\bf X}$ to have a similar size). Using also that $\vert\widetilde{\bf V}\vert\le a$, we can bound the contribution outside the bulk
\begin{equation}\label{MainVolumeGainBC}
\begin{split}
&(t^2+\vert\y\vert^2)^{\frac{3}{2}}\vert\mathcal{E}_{R,\mathcal{B}^c}\vert\\
\lesssim& R^{-2}\iint \left\vert\partial_j\varphi_R(\y-\widetilde{\bf X}(\vartheta,\a))\right\vert\cdot\langle a\rangle^3\left(\eta^{12}+\lambda^6+a^{12}+\xi^{12}\right)\cdot \frac{a^4}{\langle\widetilde{\bf V}\rangle^4}\cdot\gamma^2(\vartheta,\a,t)d\vartheta d\a\\
\lesssim& R^{-2}N_1\cdot\iint \left\vert\partial_j\varphi_R(\y-\widetilde{\bf X}(\vartheta,\a))\right\vert\cdot \frac{d\vartheta d\a}{\langle\widetilde{\bf V}\rangle^4}\\
\lesssim& R^{-2}N_1\cdot\iint \left\vert\partial_j\varphi_R(\y-{\bf x})\right\vert\cdot \frac{d\x d\v}{\langle\v\rangle^4}\lesssim N_1R,
\end{split}
\end{equation}
where in the last line, we have used the fact that $(\vartheta,\a)\mapsto(\XX,\VV)$ is canonical.

In the bulk, since $ta\ge t^\frac{3}{4}/10$, and using \eqref{ApproxXBulk}, we have that
\begin{equation}\label{EstimInBulk333}
\begin{split}
\vert \y-ta\vert\le \vert\y\vert/2,\qquad t^2+\vert\y\vert^2\lesssim_q\min\{\vert\y\vert^2\langle\xi\rangle^2,t^2\langle a\rangle^2\},
%\qquad t^2+\vert\widetilde{\bf X}(\vartheta,\a)\vert^2\lesssim t^2\langle a\rangle^2
\end{split}
\end{equation}
and we can use Lemma \ref{LemVolume} to get
\begin{equation*}
\begin{split}
\left[t^2+\vert\y\vert^2\right]^\frac{3}{2}\vert\mathcal{E}_{R,\mathcal{B}}\vert&\lesssim (\vert\y\vert^3/R^2)\iint_{\mathcal{B}} \left|\partial_j\varphi_R(\y-\XX(\vartheta,\a))\right|\cdot\langle \xi\rangle^3\gamma^2(\vartheta,\a,t)d\vartheta d\a\\
&\lesssim R\norm{[\langle\xi\rangle^{6}+\ip{\eta}^3+\ip{\lambda}^3]\gamma}_{L^\infty_{\vartheta,\a}}^2,
\end{split}
\end{equation*}
and we can deduce a control over the small scales:
\begin{equation*}
\begin{split}
\left[t^2+\vert\y\vert^2\right]^\frac{3}{2}\int_{R=0}^{(t^2+\vert\y\vert^2)^\frac{3}{8}}\vert\mathcal{E}_{R,\mathcal{B}}\vert\frac{dR}{R}\lesssim (t^2+\vert\y\vert^2)^\frac{3}{8}\cdot \norm{[\langle\xi\rangle^{6}+\ip{\eta}^3+\ip{\lambda}^3]\gamma}_{L^\infty_{\vartheta,\a}}^2.
\end{split}
\end{equation*}

{\bf (A3) Small scales}: $R\le (t^2+\vert {\bf y}\vert^2)^\frac{3}{8}/100$, contributions of the effective field. Using that, on the support of integration,
\begin{equation*}
\begin{split}
\langle\vartheta\rangle&\lesssim \frac{\xi^2}{q}\vert\eta\vert+\frac{\xi\lambda}{q},\qquad \vert\y\vert/t\lesssim\langle a\rangle,
\end{split}
\end{equation*}
which follows from \eqref{TransitionFromSIC} for the first inequality and direct inspection (separating the case $\vert \y\vert\le t$ and $\vert\y\vert\ge t$) for the second. A simple rescaling gives
\begin{equation}\label{ControlEEffSmallR}
\begin{split}
\left[t^2+\vert\y\vert^2\right]^\frac{3}{2}\vert\mathcal{E}^{eff}_R\vert&\lesssim R^{-2} t^3\iint \vert\partial_j\varphi_R({\bf y}-t\a)\vert \langle a\rangle^3 \gamma^2(\vartheta,\a,t)d\vartheta d\a\lesssim R\Vert\langle\vartheta\rangle^2\langle a\rangle^{\frac{3}{2}}\gamma\Vert_{L^\infty_{\vartheta,\a}}^2\\
&\lesssim R \Vert \left[a^3+\xi^3+\vert\eta\vert^4+\lambda^4\right]\gamma\Vert_{L^\infty_{\vartheta,\a}}^2.
\end{split}
\end{equation}

We conclude that the small scales give an acceptable contribution to the first line of \eqref{ControlEMom} since
\begin{equation*}
\begin{split}
\left[t^2+\vert\y\vert^2\right]\cdot \int_{R=0}^{(t^2+\vert\y\vert^2)^\frac{3}{8}} \left[t^2+\vert \y\vert^2\right]^{-\frac{3}{2}} dR\lesssim \left[t^2+\vert \y\vert^2\right]^{-\frac{1}{10}}.
\end{split}
\end{equation*}

\medskip

{\bf B}) Derivatives of the electric field. The bound on $\mathcal{F}$ follows similar lines, with a variation on small scales, where we make use of \eqref{FjkAndPB2} to improve the summability as $R\to0$. For large scales, we compare similarly
\begin{equation*}
\begin{split}
\mathcal{F}_{jk,R}({\bf y},t)-\mathcal{F}^{eff}_{jk,R}({\bf y},t)&= R^{-3}\iint_{\mathcal{B}} \left[\partial_j\partial_k\varphi_R({\bf y}-\widetilde{\bf X}(\vartheta,\a))-\partial_j\partial_k\varphi_R({\bf y}-t\a)\right]\gamma^2(\vartheta,\a,t)d\vartheta d\a\\
&\quad+R^{-3}\iint_{\mathcal{B}^c} \left[\partial_j\partial_k\varphi_R({\bf y}-\widetilde{\bf X}(\vartheta,\a))-\partial_j\partial_k\varphi_R({\bf y}-t\a)\right]\gamma^2(\vartheta,\a,t)d\vartheta d\a\\
&=II_{\mathcal{B}}+II_{\mathcal{B}^c},
\end{split}
\end{equation*}
and again
\begin{equation*}
\begin{split}
\vert II_{\mathcal{B}}\vert&\lesssim R^{-4}\langle \ln t\rangle\Vert \left[1+\xi^4+\eta^4+\lambda^2\right]\gamma\Vert_{L^2_{\vartheta,\a}}^2,\\
\vert II^1_{\mathcal{B}^c}\vert+\vert II^2_{\mathcal{B}^c}\vert&\lesssim R^{-3}\left(\left[t^2+\vert\y\vert^2\right]^{-1}+\mathfrak{1}_{\{R\ge\sqrt{t^2+\vert\y\vert^2}\}}t^{-1}\right)\Vert\left[\xi^{2}+a^{2}+\vert\eta\vert^{2}+\lambda\right]\gamma\Vert_{L^2_{\vartheta,\a}}^2,
\end{split}
\end{equation*}
and we conclude similarly.

For small scales, we make use of \eqref{FjkAndPB2} and adapt the above computations, computing the contribution of each term separately. Outside of the bulk, we use
\eqref{UpperBoundytEF} and \eqref{eq:nonbulk-bds} together with the bounds \eqref{GeneralBoundZBracket} to get
\begin{equation}\label{SmallScaleFOutsidBulk}
\begin{split}
\left[t^2+\vert\y\vert^2\right]^2\vert\mathcal{F}_{R,\mathcal{B}}\vert
&\lesssim t^{-1}R^{-2}\iint_{\mathcal{B}^c} \left|\partial_j\varphi_R(\y-\XX(\vartheta,\a))\right|\cdot\left[\xi^{16}+ a^{10}+\vert\eta\vert^{8}+\lambda^{8}\right]\vert \gamma \{\widetilde{\bf Z},\gamma\}\vert \frac{\langle a\rangle^4}{\langle\widetilde{\bf V}\rangle^4}d\vartheta d\a\\
&\lesssim t^{-1}R^{-2}\cdot\iint  \left|\partial_j\varphi_R(\y-\XX(\vartheta,\a))\right|\cdot\frac{d\vartheta d{\bf a}}{\langle\widetilde{\bf V}\rangle^4}\cdot N_2^2\lesssim (R/t)\cdot N_2^2.
\end{split}
\end{equation}
In the bulk, we use \eqref{FjkAndPB2} and \eqref{EstimInBulk333}, we can estimate
\begin{equation*}
\begin{split}
\left[t^2+\vert\y\vert^2\right]^2\vert\mathcal{F}_{R,\mathcal{B}}\vert&\lesssim R(\vert\y\vert/R)^3\iint_{\mathcal{B}} \left|\partial_j\varphi_R(\y-\XX(\vartheta,\a))\right|\cdot\langle a\rangle\langle\xi\rangle^3\vert \gamma \{\widetilde{\bf Z},\gamma\}\vert d\vartheta d\a
\end{split}
\end{equation*}
and using \eqref{PrecisedBoundZBracketBulk} with Lemma \ref{LemVolume}, we obtain the bound
\begin{equation*}
\begin{split}
\left[t^2+\vert\y\vert^2\right]^2\vert\mathcal{F}_{R,\mathcal{B}}\vert&\lesssim R\langle\ln t\rangle N_2^2,
\end{split}
\end{equation*}
and this leads to an acceptable contribution. Finally, the bound on the effective field is treated similarly using \eqref{TransitionFromSIC} to get bounds on the Poisson bracket with $\vartheta$:
\begin{equation*}
\begin{split}
\mathcal{F}^{eff}_{jk,R}({\bf y},t)&=R^{-2}t^{-1}\iint\{\vartheta^k,\partial_j\varphi_R({\bf y}- t\a )\}\cdot\gamma^2(\vartheta,\a ,t)d\vartheta d\a\\
&=-2R^{-2}t^{-1}\iint \partial_j\varphi_R({\bf y}- t\a )\cdot \{\vartheta^k,\gamma\}\cdot\gamma(\vartheta,\a ,t)d\vartheta d\a .
\end{split}
\end{equation*}
Therefore, in the bulk, using Lemma \ref{LemVolume},
\begin{equation*}
\begin{split}
\left[t^2+\vert{\bf y}\vert^2\right]^2\vert \mathcal{F}^{eff}_{jk,R}({\bf y},t)\vert&\lesssim R^{-2}t^3\iint \vert\partial_j\varphi_R({\bf y}- t\a )\vert\cdot \langle a\rangle^4\langle\vartheta\rangle^4\vert\{\vartheta^k,\gamma\}\vert\cdot\gamma(\vartheta,\a ,t)\frac{d\vartheta d\a}{\langle\vartheta\rangle^4}\lesssim R N_2^2,
\end{split}
\end{equation*}
and we can proceed similarly outside the bulk using \eqref{SmallScaleFOutsidBulk} instead.
\end{proof}

\begin{lemma}\label{LemVolume}
Let $\tilde \varphi\in C^\infty_c(\R)$ be a radial cutoff function, $\text{supp}(\tilde\varphi)\subset [0,3]$, and $\tilde\varphi> 0$ on $[0,2]$. Then on $\mathcal{B}$ we have the following bound: \begin{equation*}
\begin{split}
\iint_{\mathcal{B}}\tilde\varphi(R^{-1}\vert\widetilde{\bf X}-{\bf y}\vert)f(\vartheta,\a,t)\, d\vartheta d{\bf a}
&\lesssim \norm{\langle\xi\rangle\langle\eta\rangle^2\langle\lambda\rangle^2f}_{L^\infty_{\vartheta,\a}}(R/\vert {\bf y}\vert)^3.
\end{split}
\end{equation*}

\end{lemma}

\begin{proof}
For any unit vector ${\bf e}$ with $\widetilde{\bf X}\cdot{\bf e}=0$, we observe that,
\begin{equation*}
\begin{split}
\vert\widetilde{\bf X}-{\bf y}\vert^2=({\bf e}\cdot{\bf y})^2+\vert \widetilde{\bf X}-{\bf y}_\perp\vert^2,\qquad{\bf y_\perp}=-{\bf e}\times({\bf e}\times{\bf y}),
\end{split}
\end{equation*}
so that
\begin{equation*}
\begin{split}
\tilde\varphi(R^{-1}\vert\widetilde{\bf X}-{\bf y}\vert)&\lesssim \tilde\varphi(R^{-1}|{\bf e}\cdot{\bf y}|)\tilde\varphi(R^{-1}\vert \widetilde{\bf X}-{\bf y}_\perp\vert).
\end{split}
\end{equation*}

We can consider the {\it symplectic disintegration} of the Liouville measure associated to the mapping $AMom:(\vartheta,{\bf a})\mapsto {\bf l}={\bf L}/L\in\mathbb{S}^2$ and we obtain accordingly
\begin{equation*}
\begin{split}
\iint d{\vartheta}d{\bf \a}=\int_{{\bf l}\in\mathbb{S}^2}d\nu({\bf l})\iint_{\mathcal{P}_{\bf l}} d\bar{\vartheta} d\bar{\a},\qquad\mathcal{P}_{{\bf l}}=\{\vartheta\cdot{\bf l}=0={\bf a}\cdot{\bf l}\}
\end{split}
\end{equation*}
Indeed, both sides are invariant under joint rotations $(\vartheta,\a)\mapsto (R\vartheta,R\a)$ and their restriction to a plane agrees\footnote{Alternatively, one may observe that
\begin{equation*}
\omega_{\mathcal{P}_{\vartheta,{\bf a}}}=AMom^\ast\omega_{\mathbb{S}^2}+\omega_{\mathcal{P}_{\bf l}},
\end{equation*}
where $\omega_Y$ stands for the natural symplectic form on $Y$.
We thank P.\ G\'erard for this observation.}.

Thus it suffices to consider the planar case. When ${\bf l}=(0,0,1)$, we choose coordinates $\xi,\eta,\lambda,\varphi$ such that
\begin{equation*}
{\bf a}=\frac{q}{\xi}(\cos\varphi,\sin\varphi,0),\qquad\varphi=\arctan({\bf a}^2/{\bf a}^1)
\end{equation*}
which are in involution:
\begin{equation*}
\begin{split}
\{\xi,\eta\}=1=\{\varphi,\lambda\},\qquad\{\xi,\lambda\}=\{\eta,\lambda\}=0=\{\xi,\varphi\}=\{\eta,\varphi\}.
\end{split}
\end{equation*}
On $\mathcal{P}_{\bf l}\cap\mathcal{B}$, the mapping ${\bf Y}:\,(\xi,\varphi)\mapsto\widetilde{\bf X}(\xi,\eta,\lambda,\varphi)$ is an embedding\footnote{This can be seen since changing $\varphi$ amounts to rotate about $0$, while $\partial_\xi\vert\widetilde{\bf X}\vert^2>0$.} $I\times\mathbb{S}^1\to\mathbb{R}^2\setminus\{0\}$ for some interval $I(\eta,\lambda,t)$, and, using \eqref{PBXV} and \eqref{DerXiX}, we see that, on the support of integration, we have the bound on the Jacobian
\begin{equation*}
\xi\vert\det\frac{\partial{\bf Y}}{\partial(\xi,\varphi)}\vert\simeq \vert \widetilde{\bf X}\vert^2\simeq \vert{\bf y}\vert^2,
\end{equation*}
and we deduce that
\begin{equation*}
\begin{split}
\iint_{\mathcal{B}}\tilde\varphi(R^{-1}\vert\widetilde{\bf X}-{\bf y}\vert)f(\vartheta,\a,t)\, d\vartheta d{\bf a}
&\lesssim \int_{\mathbb{S}^2}\tilde\varphi(R^{-1}|{\bf l}\cdot{\bf y}|)\left(\iint_{\mathcal{P}_{\bf l}\cap\mathcal{B}}\tilde\varphi(R^{-1}\vert \widetilde{\bf X}-{\bf y}_\perp\vert)f\,  d\bar{\vartheta} d\bar{\a}\right)d\nu({\bf l})\\
&\lesssim \norm{\langle\xi\rangle\ip{\eta}^2\ip{\lambda}^2f}_{L^\infty_{\vartheta,\a}}(R/\vert {\bf y}\vert)^3.
%&\lesssim \norm{[\langle\xi\rangle^{3p+\frac{3}{2}}+\ip{\eta}^3+\ip{\lambda}^3]\gamma}_{L^\infty_{\vartheta,\a}}^2(R/\vert {\bf y}\vert)^3
\end{split}
\end{equation*}

\end{proof}

\subsection{The effective fields}

In this section, we complement Proposition \ref{PropControlEF} by obtaining bounds on the effective fields for particle distributions following the evolution equation \eqref{NewNLEq}. These follow from variations on the continuity equation.

To control various terms, we introduce an appropriate weighted envelope function. For $\phi\in C^\infty_c(\mathbb{R}^3)$, define
\begin{equation*}
\begin{split}
M_R({\bf y},t;\phi)&:=\iint \phi(R^{-1}({\bf y}-{\bf a}))\gamma^2(\vartheta,{\bf a},t)\, d\vartheta d{\bf a}.
\end{split}
\end{equation*}
Bounds on the effective electric potential, field and derivative will be related to the convergence properties of $M_R$ in various norms (see \eqref{BesovBoundMF1} and \eqref{BesovBoundMF2}). In particular, they allow to obtain global bounds on the initial data. For conciseness, we introduce the near identity
\begin{equation}\label{NearID}
I_k({\bf y},R):=1+\left[1+\vert{\bf y}\vert^2\right]^\frac{k}{2}\cdot\mathfrak{1}_{\{\vert{\bf y}\vert\le 10R\}},
\end{equation}
and we can state our bounds in terms of $I_{k}({\bf y},R)$.
\begin{lemma}\label{BoundIEFLem}
We have bounds for the initial data
\begin{equation*}
\begin{split}
\left[1+\vert{\bf y}\vert^2\right]^\frac{k}{2}\cdot M_R({\bf y},0;\phi)&\lesssim I_k({\bf y},R)\cdot\min\{1,R^3\}\cdot\left[\Vert \langle a\rangle^{k/2}\gamma_0\Vert_{L^2_{\vartheta,{\bf a}}}^2+\Vert \langle\vartheta\rangle^2\langle a\rangle^{k/2}\gamma_0\Vert_{L^\infty_{\vartheta,{\bf a}}}^2\right],
\end{split}
\end{equation*}
and in particular
\begin{equation*}
\begin{split}
\left[1+\vert{\bf y}\vert^2\right]\cdot\vert\mathcal{E}^{eff}({\bf y},0)\vert&\lesssim \Vert \langle a\rangle\gamma_0\Vert_{L^2_{\vartheta,{\bf a}}}^2+\Vert \langle a\rangle\langle\vartheta\rangle^2\gamma_0\Vert_{L^\infty_{\vartheta,{\bf a}}}^2,\\
\left[1+\vert{\bf y}\vert^2\right]^\frac{3}{2}\cdot\vert\mathcal{F}^{eff}({\bf y},0)\vert&\lesssim \Vert \langle a\rangle^\frac{3}{2}\gamma_0\Vert_{L^2_{\vartheta,{\bf a}}}^2+\Vert \langle a\rangle^\frac{3}{2}\langle\vartheta\rangle^2\gamma_0\Vert_{L^\infty_{\vartheta,{\bf a}}}^2.
\end{split}
\end{equation*}

\end{lemma}

\begin{proof}

Indeed
\begin{equation*}
\begin{split}
M_R({\bf y},0;\phi)&\lesssim\Vert \gamma\Vert_{L^2_{\vartheta,{\bf a}}}^2,\qquad
\left[1+\vert{\bf y}\vert^2\right]^\frac{k}{2}\cdot M_R({\bf y},0;\phi)\mathfrak{1}_{\{\vert{\bf y}\vert\ge 10R\}}\lesssim \Vert \langle a\rangle^{k/2}\gamma\Vert_{L^2_{\vartheta,{\bf a}}}^2,
\end{split}
\end{equation*}
and when $R\lesssim 1$,
\begin{equation*}
\begin{split}
M_R({\bf y},0;\phi)&\lesssim R^3\Vert \langle\vartheta\rangle^2\gamma\Vert_{L^\infty_{\vartheta,{\bf a}}}^2,\qquad
\left[1+\vert{\bf y}\vert^2\right]^\frac{k}{2}\cdot M_R({\bf y},0;\phi)\mathfrak{1}_{\{\vert{\bf y}\vert\ge 10R\}}\lesssim R^3\Vert \langle\vartheta\rangle^2\langle a\rangle^{k/2}\gamma\Vert_{L^\infty_{\vartheta,{\bf a}}}^2.
\end{split}
\end{equation*}
The bounds on $\mathcal{E}^{eff}$ and $\mathcal{F}^{eff}$ follow by direct integrations decomposing into $\{\vert\y\vert\le 10R\}$ and $\{\vert\y\vert\ge 10R\}$.

\end{proof}

Assuming only moment bounds, we can obtain almost sharp decay for the effective electric field, and assuming moments and Poisson brackets, we can obtain sharp decay for the effective electric and almost sharp decay for its derivatives. 

\begin{proposition}\label{PropEeff}

$(i)$ Assume that $\gamma$ satisfies \eqref{NewNLEq} and the bounded moment bootstrap assumptions \eqref{eq:mom_btstrap_assump}. For any fixed $\phi\in C^\infty_c(\mathbb{R}^3)$ and any $0\le s\le t\le T$, there holds that, uniformly in $R,{\bf y}$ and $0\le k\le 3$,
\begin{equation}\label{BesovBoundMF1}
\begin{split}
\left[1+\vert{\bf y}\vert^2\right]^\frac{k}{2}\vert M_R({\bf y},t;\phi)-M_R({\bf y},s;\phi)\vert\lesssim_\phi\varepsilon_1^4I_k({\bf y},R)\cdot\min\{R^2,R^{-1}\}\int_s^t\langle u\rangle^{-\frac{3}{2}}du.
%\sup_{R,{\bf y}} \langle R\rangle^\frac{1}{2}\langle{\bf y}\rangle^2R^{-2}\vert M_R({\bf y},t;\phi)-M_R({\bf y},s;\phi)\vert\lesssim_\phi\varepsilon_1^4\int_s^t\langle u\rangle^{-\frac{3}{2}}du,
\end{split}
\end{equation}
In particular, $M_R({\bf y},t;\phi)$ converges uniformly to a limit $M_\infty({\bf y};\phi)$. This implies almost optimal decay on the effective electric field
\begin{equation}\label{BoundEff1}
\begin{split}
\left[t^2+\vert\y\vert^2\right]\cdot\vert\mathcal{E}^{eff}({\bf y},t)\vert&\lesssim \varepsilon_1^2 \ln\ip{t}.
\end{split}
\end{equation}

$(ii)$ Assume in addition that $\gamma$ satisfies the stronger bootstrap assumption \eqref{eq:deriv_bstrap_assump2}, then we can strengthen \eqref{BesovBoundMF1} to
\begin{equation}\label{BesovBoundMF2}
\begin{split}
\left[1+\vert{\bf y}\vert^2\right]^\frac{k}{2}\vert M_R({\bf y},t;\phi)-M_R({\bf y},s;\phi)\vert\lesssim_\phi\varepsilon_1^4I_k({\bf y},R)\cdot\min\{R^3,R^{-1}\}\int_s^t\langle u\rangle^{-\frac{6}{5}}du,
%\sup_{R,{\bf y}} \langle R\rangle^\frac{1}{2}\langle{\bf y}\rangle^3R^{-3}\vert M_R({\bf y},t;\phi)-M_R({\bf y},s;\phi)\vert\lesssim_\phi\varepsilon_1^4\int_s^t\langle u\rangle^{-\frac{6}{5}}du
\end{split}
\end{equation}
so that
\begin{equation}\label{BoundEeff2}
\begin{split}
\left[t^2+\vert\y\vert^2\right]\cdot\vert\mathcal{E}^{eff}({\bf y},t)\vert&\lesssim \varepsilon_1^2,\\
\left[t^2+\vert\y\vert^2\right]^\frac{3}{2}\cdot\vert\mathcal{F}^{eff}({\bf y},t)\vert&\lesssim \varepsilon_1^2\ln\ip{t}.
\end{split}
\end{equation}

$(iii)$ In addition, under the hypothesis of $(ii)$, there exists $\Psi^\infty$ and $\mathcal{E}_j^\infty=\partial_j\Psi^\infty$ such that
\begin{equation}\label{AsymptoticEFCCL}
\begin{split}
\Vert \langle{\bf a}\rangle(t\Psi^{eff}(t{\bf a},t)-\Psi^\infty({\bf a}))\Vert_{L^\infty}+\Vert \langle {\bf a}\rangle^2(t^2\mathcal{E}^{eff}_j(t{\bf a},t)-\mathcal{E}^\infty_j({\bf a}))\Vert_{L^\infty}&\lesssim\varepsilon_1^2\langle t\rangle^{-\frac{1}{10}},
\end{split}
\end{equation}
and consequently
\begin{equation}\label{AsymptoticEFCCL2}
\begin{split}
\Vert t\Psi(\widetilde{\bf X}(\vartheta,\a),t)-\Psi^\infty(\a)\Vert_{L^\infty(\mathcal{B})}+\Vert t^2\mathcal{E}_j(\widetilde{\bf X}(\vartheta,\a),t)-\mathcal{E}_j^\infty(\a)\Vert_{L^\infty(\mathcal{B})}\lesssim \varepsilon_1^2 t^{-\frac{1}{10}}.
\end{split}
\end{equation}

\end{proposition}

\begin{proof}[Proof of Proposition \ref{PropEeff}]

Let $\phi_R({\bf x}):=\phi(R^{-1}{\bf x})$. Using \eqref{NewNLEq}, we compute that
\begin{equation*}
\begin{split}
\partial_tM_R({\bf y},t;\phi)&=\iint \phi_R({\bf y}-{\bf a})\{\mathbb{H}_4,\gamma^2\}\, d\vartheta d{\bf a}=-\iint \gamma^2\{\mathbb{H}_4,\phi_R({\bf y}-{\bf a})\}\, d\vartheta d{\bf a}\\
&=R^{-1}\iint_{\mathcal{B}} \gamma^2\left[\mathcal{E}_j\{\widetilde{\bf X}^j,\a^l\}+\mathcal{W}_j\{\widetilde{\bf V}^j,{\bf a}^l\}\right](\partial_l\phi_R)({\bf y}-{\bf a})\, d\vartheta d{\bf a}\\
&\quad+R^{-1}\iint_{\mathcal{B}^c} \gamma^2\left[\mathcal{E}_j\{\widetilde{\bf X}^j,\a^l\}+\mathcal{W}_j\{\widetilde{\bf V}^j,{\bf a}^l\}\right](\partial_l\phi_R)({\bf y}-{\bf a})\, d\vartheta d{\bf a},
\end{split}
\end{equation*}
and, using Corollary \ref{CorPBa}, \eqref{TransitionFromSIC} and a crude estimate, we can get a good bound in the bulk:
\begin{equation*}
\begin{split}
\vert\partial_tM_{R,\mathcal{B}}({\bf y},t)\vert&\lesssim R^{-1}\left\vert \iint_{\mathcal{B}} \gamma^2\left[\mathcal{E}_j\{\widetilde{\bf X}^j,\a^l\}+\mathcal{W}_j\{\widetilde{\bf V}^j,{\bf a}^l\}\right](\partial_l\phi_R)({\bf y}-{\bf a})\, d\vartheta d{\bf a}\right\vert\\
&\lesssim \left[\Vert\mathcal{E}\Vert_{L^\infty}+\langle t\rangle^{-1}\vert\mathcal{W}\vert\right]\cdot\min\{R^2,R^{-1}\}\cdot\left[\Vert\gamma\Vert_{L^2_{\vartheta,\a}}^2+\Vert \langle\vartheta\rangle^2\gamma\Vert_{L^\infty_{\vartheta,\a}}^2\right],
\end{split}
\end{equation*}
and assuming that $\vert{\bf y}\vert\ge 10R$,
\begin{equation*}
\begin{split}
\left[1+\vert{\bf y}\vert^2\right]^\frac{k}{2}\vert\partial_tM_{R,\mathcal{B}}({\bf y},t)\vert \mathfrak{1}_{\{\vert{\bf y}\vert\ge 10R\}}&\lesssim R^{-1}\langle a\rangle^\frac{k}{2}\left\vert \iint_{\mathcal{B}} \gamma^2\left[\mathcal{E}_j\{\widetilde{\bf X}^j,\a^l\}+\mathcal{W}_j\{\widetilde{\bf V}^j,{\bf a}^l\}\right](\partial_l\phi_R)({\bf y}-{\bf a})\, d\vartheta d{\bf a}\right\vert\\
&\lesssim \left[\Vert\mathcal{E}\Vert_{L^\infty}+\langle t\rangle^{-1}\vert\mathcal{W}\vert\right]\cdot\min\{R^2,R^{-1}\}\cdot\left[\Vert\langle a\rangle^\frac{k}{2}\gamma\Vert_{L^2_{\vartheta,\a}}^2+\Vert\langle a\rangle^\frac{k}{2} \langle\vartheta\rangle^2\gamma\Vert_{L^\infty_{\vartheta,\a}}^2\right].
\end{split}
\end{equation*}
Outside the bulk we also use \eqref{eq:nonbulk-bds} to get
\begin{equation*}
\begin{split}
\vert\partial_tM_{R,\mathcal{B}^c}({\bf y},t)\vert&\lesssim R^{-1}\left\vert \iint_{\mathcal{B}^c} \gamma^2\left[\mathcal{E}_j\{\widetilde{\bf X}^j,\a^l\}+\mathcal{W}_j\{\widetilde{\bf V}^j,{\bf a}^l\}\right](\partial_l\phi_R)({\bf y}-{\bf a})\, d\vartheta d{\bf a}\right\vert\\
&\lesssim R^{-1}\left[\Vert\vert{\bf x}\vert\mathcal{E}\Vert_{L^\infty}+\vert\mathcal{W}\vert\right]\cdot\langle t\rangle^{-1}\iint_{\mathcal{B}^c}\frac{q}{\xi^2}\langle a\rangle\left[a^4+\xi^4+\eta^4+\lambda^2\right]\gamma^2\cdot\vert (\partial_l\phi_R)({\bf y}-{\bf a})\vert d\vartheta d\a\\
&\lesssim \varepsilon_1^2\langle t\rangle^{-\frac{7}{4}}\min\{R^2,R^{-1}\}N_1^2,
\end{split}
\end{equation*}
and similarly
\begin{equation*}
\begin{split}
&\left[1+\vert{\bf y}\vert^2\right]^\frac{k}{2}\vert\partial_tM_{R,\mathcal{B}^c}({\bf y},t)\vert\mathfrak{1}_{\{\vert{\bf y}\vert\ge 10R\}}\\
\lesssim &\,R^{-1}\left\vert \iint_{\mathcal{B}^c} \gamma^2\left[\mathcal{E}_j\{\widetilde{\bf X}^j,\a^l\}+\mathcal{W}_j\{\widetilde{\bf V}^j,{\bf a}^l\}\right](\partial_l\phi_R)({\bf y}-{\bf a})\, d\vartheta d{\bf a}\right\vert\\
\lesssim &\,R^{-1}\left[\Vert\vert{\bf x}\vert\mathcal{E}\Vert_{L^\infty}+\vert\mathcal{W}\vert\right]\cdot\langle t\rangle^{-1}\iint_{\mathcal{B}^c}\frac{q}{\xi^2}\langle a\rangle^{k+1}\left[a^4+\xi^4+\eta^4+\lambda^2\right]\gamma^2\cdot\vert (\partial_l\phi_R)({\bf y}-{\bf a})\vert d\vartheta d\a\\
\lesssim &\,\varepsilon_1^2\langle t\rangle^{-\frac{7}{4}}\min\{R^2,R^{-1}\}N_1^2.
\end{split}
\end{equation*}
Adding the two lines above, we get that, for $k\le 3$,
\begin{equation}\label{BounddMdt1}
\begin{split}
\vert \partial_tM_R({\bf y},t;\phi)\vert&\lesssim \varepsilon_1^4\langle t\rangle^{-\frac{7}{4}}N_1^2\cdot\min\{R^2,R^{-1}\},\\
\left[1+\vert{\bf y}\vert^2\right]^\frac{k}{2}\cdot \vert \partial_tM_R({\bf y},t;\phi)\vert\mathfrak{1}_{\{\vert{\bf y}\vert\ge10R\}}&\lesssim \varepsilon_1^4\langle t\rangle^{-\frac{7}{4}}N_1^2\cdot\min\{R^2,R^{-1}\},
\end{split}
\end{equation}
which gives \eqref{BesovBoundMF1}.

We can now fix $\phi\ge 0$ such that $\phi\varphi=\varphi$ (with $\varphi$ from \eqref{IntroducingVarphi}) and let $M_R^{(2)}({\bf y},t):=M_R({\bf y},t;\phi)$. Letting $R_1=\min\{t,(t^2+\vert\y\vert^2)^\frac{3}{8}\}$, we see that
\begin{equation*}
\begin{split}
\left[t^2+\vert{\bf y}\vert^2\right]\cdot\mathcal{E}^{eff}_{j}({\bf y},t)&=\left[t^2+\vert{\bf y}\vert^2\right]\cdot\int_{R=0}^{\infty}\iint\partial_j\varphi_R({\bf y}-t{\bf a})\gamma^2(\vartheta,{\bf a},t)\,d\vartheta d{\bf a}\, \frac{dR}{R^3}\\
&=\left[1+\vert\y/t\vert^2\right]\cdot \int_{r=0}^{\infty} \frac{1}{r^2}\iint(\partial_j\varphi)(r^{-1}({\bf y}/t-{\bf a}))\gamma^2(\vartheta,{\bf a},t)\,d\vartheta d{\bf a}\, \frac{dr}{r},
\end{split}
\end{equation*}
where $r=R/t$. To go further, we observe the simple bounds as in Lemma \ref{BoundIEFLem},
\begin{equation}\label{SimpleBoundsMR}
\begin{split}
\langle{\bf y}\rangle^k\cdot M_R({\bf y},t;\phi)&\lesssim I_k({\bf y},R)\cdot\min\{1,R^3\}\cdot\left[\Vert \langle a\rangle^{k/2}\gamma(t)\Vert_{L^2_{\vartheta,{\bf a}}}^2+\Vert \langle\vartheta\rangle^2\langle a\rangle^{k/2}\gamma(t)\Vert_{L^\infty_{\vartheta,{\bf a}}}^2\right],
\end{split}
\end{equation}
and therefore we can estimate the case of low scales $0\le r\le r_1:=R_1/t\le 1$:
\begin{equation*}
\begin{split}
I_{low}&:=\left[1+\vert\y/t\vert^2\right]\cdot \left\vert \int_{r=0}^{r_1} \frac{1}{r^2}\iint(\partial_j\varphi)(r^{-1}({\bf y}/t-{\bf a}))\gamma^2(\vartheta,{\bf a},t)\,d\vartheta d{\bf a}\, \frac{dr}{r}\right\vert\\
&\lesssim\langle \y/t\rangle^2  \int_{r=0}^{r_1} M^{(2)}_r({\bf y}/t,t)\, \frac{dr}{r^3}\lesssim\varepsilon_1^2 N_1^2r_1,
\end{split}
\end{equation*}
while for the higher scale, we integrate in time using \eqref{SimpleBoundsMR} at time $t=0$ and \eqref{BounddMdt1}:
\begin{equation*}
\begin{split}
I_{high}&:=\left[1+\vert\y/t\vert^2\right]\cdot \left\vert \int_{r=r_1}^{\infty} \frac{1}{r^2}\iint(\partial_j\varphi)(r^{-1}({\bf y}/t-{\bf a}))\gamma^2(\vartheta,{\bf a},t)\,d\vartheta d{\bf a}\, \frac{dr}{r}\right\vert\\
%&\lesssim \langle\y/t\rangle^2 \left\vert \int_{r=r_1}^{\infty} \iint(\partial_j\varphi)(r^{-1}({\bf y}/t-{\bf a}))\left[\gamma^2(\vartheta,{\bf a},t)-\gamma^2(\vartheta,{\bf a},0)\right]\,d\vartheta d{\bf a}\, \frac{dr}{r^3}\right\vert\\
%&\qquad +\langle\y/t\rangle^2\cdot \left\vert \int_{r=r_1}^{\infty}\iint(\partial_j\varphi)(r^{-1}({\bf y}/t-{\bf a}))\gamma^2(\vartheta,{\bf a},0)\,d\vartheta d{\bf a}\, \frac{dr}{r^3}\right\vert\\
&\lesssim\langle \y/t\rangle^2  \int_{r=r_1}^\infty\vert M^{(2)}_r({\bf y}/t,t)-M^{(2)}_r({\bf y}/t,0)\vert \, \frac{dr}{r^3}+\langle \y/t\rangle^2  \int_{r=r_1}^\infty M^{(2)}_r({\bf y}/t,0)\, \frac{dr}{r^3}\\
&\lesssim \varepsilon_0^2 N_1^2+\varepsilon_1^4N_1^2\int_{r=r_1}^\infty\frac{dr}{r\langle r\rangle^\frac{1}{2}}\lesssim\varepsilon_0^2+\varepsilon_1^4\langle \ln r_1\rangle,
\end{split}
\end{equation*}
which gives \eqref{BoundEff1}.

\medskip

In case we also have control of some derivative, we can alter the first step to get that, for $k\le 3$,
\begin{equation}\label{BounddMdt1New}
\begin{split}
\left[1+\vert{\bf y}\vert^2\right]^\frac{k}{2}\cdot R^{-3}\vert \partial_tM_R({\bf y},t;\phi)\vert&\lesssim \varepsilon_1^4I_k({\bf y},R)\cdot\langle t\rangle^{-\frac{5}{4}}N_2^2.
\end{split}
\end{equation}
For $R\ge t^{-\frac{1}{2}}$, we can still use \eqref{BounddMdt1}. In case $R\le t^{-\frac{1}{2}}$, we compute that
\begin{equation*}
\begin{split}
\partial_tM_R({\bf y},t;\phi)&=\iint \phi_R({\bf y}-{\bf a})\{\mathbb{H}_4,\gamma^2\}\, d\vartheta d\a=2\iint \phi_R({\bf y}-{\bf a})\left[\mathcal{E}_j\{\widetilde{\bf X},\gamma\}+\mathcal{W}_j\{\widetilde{\bf V},\gamma\}\right]\gamma\, d\vartheta d\a,
\end{split}
\end{equation*}
and therefore, using Lemma \ref{ZPBFormula} and Lemma \ref{LemVolume} inside the bulk, and Lemma \ref{ZPBFormula}, \eqref{eq:nonbulk-bds} and \eqref{MainVolumeGainBC} outside the bulk, we obtain that
\begin{equation*}
\begin{split}
R^{-3}\left[1+\vert{\bf y}\vert^2\right]^\frac{k}{2}\vert \partial_tM_R\vert&\lesssim I_k({\bf y},R)\cdot\left[t\Vert \mathcal{E}\Vert_{L^\infty}+t\vert\mathcal{W}\vert\right]\cdot N_2^2,
\end{split}
\end{equation*}
which is enough to get \eqref{BesovBoundMF2}. Proceeding as above, but using \eqref{BesovBoundMF2} instead of \eqref{BesovBoundMF1}, we also get \eqref{BoundEeff2}. This also implies \eqref{AsymptoticEFCCL} since
\begin{equation*}
\begin{split}
\langle a\rangle^2\vert t^2\mathcal{E}^{eff}_j(t{\bf a},t)-s^2\mathcal{E}^{eff}_j(s{\bf a},s)\vert&\lesssim \langle a\rangle^2\int_{r=0}^{\infty}\left\vert\iint\partial_j\varphi_{r}({\bf a}-\alpha)\left[\gamma^2(\theta,\alpha,t)-\gamma^2(\theta,\alpha,s)\right]\, d\theta d\alpha\right\vert \frac{dr}{r^3}\\
&\lesssim \int_{r=0}^\infty \langle a\rangle^2\vert M^{(2)}_r({\bf a},t)-M^{(2)}_r({\bf a},s)\vert \frac{dr}{r^3},
\end{split}
\end{equation*}
which gives \eqref{AsymptoticEFCCL}. Finally \eqref{AsymptoticEFCCL2} follows from \eqref{BoundEeff2}, \eqref{AsymptoticEFCCL}, Proposition \ref{PropControlEF} and \eqref{ApproxXBulk} once we observe that, for $(\vartheta,{\bf a})\in\mathcal{B}$,
\begin{equation*}
\begin{split}
\vert \mathcal{E}_j(\widetilde{\bf X}(\vartheta,{\bf a}),t)-\mathcal{E}_j(t{\bf a},t)\vert&\lesssim \Vert\mathcal{F}\Vert_{L^\infty}\cdot\vert\widetilde{\bf X}(\vartheta,{\bf a})-t{\bf a}\vert\lesssim\varepsilon_1^2\langle t\rangle^{-5/2}.
\end{split}
\end{equation*}

\end{proof}

\section{Nonlinear analysis: bootstrap propagation}\label{SecBootstrap}
In this section we will establish moment and derivative control in the nonlinear dynamics. In our terminology, ``linear'' will henceforth refer to features of the linearized equations. In particular, the linear characteristics are the solutions of the linearized equations, i.e.\ of the Kepler problem \eqref{ODE}, and thus by no means straight lines.

\subsection{Nonlinear unknowns}\label{Sec:NLUnknowns}
We now switch to a new unknown adapted to the study of nonlinear asymptotic dynamics. We fix the (forward) asymptotic action map $\mathcal{T}:({\bf x},{\bf v})\mapsto (\vartheta,{\bf a})$ of Proposition \ref{PropAA} and we define
\begin{equation}
 \gamma:=\nu\circ\mathcal{T}^{-1}\circ\Phi_t^{-1},
\end{equation}
where $\Phi_t$ is the flow of the linear characteristics of $\nu$, i.e.\ of the Kepler problem \eqref{ODE}. More explicitly, we have 
\begin{equation}\label{NewNLUnknown}
\begin{split}
\gamma(\vartheta,{\bf a},t)&=\nu({\bf X}(\vartheta+t{\bf a},{\bf a}),{\bf V}(\vartheta+t{\bf a},{\bf a}),t),\\
\nu(q,p,t)&=\gamma(\Theta(q,p)-t\mathcal{A}(q,p),\mathcal{A}(q,p),t),
\end{split}
\end{equation}
and since $\Phi_t^{-1}\circ\mathcal{T}^{-1}$ is a canonical transformation (Proposition \ref{PropAA}) that filters out the linear flow, we observe that $\gamma$ evolves under the purely nonlinear Hamiltonian $\mathbb{H}_4\circ\Phi_t^{-1}\circ\mathcal{T}^{-1}$ of \eqref{Hamiltonians}: with a slight abuse of notation we have
\begin{equation}\label{NewNLEq}
\begin{split}
\partial_t\gamma+\{\mathbb{H}_4,\gamma\}=0,\qquad \gamma(t=0)=\nu_0,\qquad\mathbb{H}_4=Q\psi(\widetilde{\bf X})+\mathcal{W}\cdot \widetilde{\bf V},\\
\psi({\bf y},t)=-\frac{1}{4\pi}\iint \frac{1}{\vert {\bf y}-\widetilde{\bf X}(\theta,\alpha)\vert}\gamma^2(\theta,\alpha,t)\, d\theta d\alpha,\qquad\dot{\mathcal{W}}=\mathcal{Q}\nabla_x\psi(0,t),\qquad\mathcal{W}(t)\to_{t\to T^\ast}0.
\end{split}
\end{equation}
This is the equation we will focus on.

\subsection{Moment propagation}\label{sec:moments_prop}
In this section we will show that control of moments and (almost sharp) decay of the electric field can be obtained independently of derivative bounds.
\begin{theorem}\label{thm:global_moments}
Let $m\ge 30$, and assume that the initial density $\nu_0$ satisfies
\begin{equation}\label{eq:mom_id}
\begin{split}
\Vert \langle a\rangle^{2m}\nu_0\Vert_{L^r}+\Vert \langle \xi\rangle^{2m}\nu_0\Vert_{L^r}+\Vert \langle \lambda\rangle^{2m}\nu_0\Vert_{L^r}+\Vert \langle\eta\rangle^m\nu_0\Vert_{L^\infty}&\le \varepsilon_0,\qquad r\in\{2,\infty\}.
\end{split}
\end{equation}
Then there exists a global solution $\gamma$ to \eqref{NewNLEq} that satisfies the bounds for $r\in\{2,\infty\}$
\begin{equation}\label{eq:gl_mom_bds}
\begin{split}
\Vert \langle a\rangle^{2m}\gamma(t)\Vert_{L^r}+\Vert \langle \xi\rangle^{2m}\gamma(t)\Vert_{L^r}&\le 2\varepsilon_0,\\
\Vert \langle \lambda\rangle^{2m}\gamma(t)\Vert_{L^r}&\le 2\varepsilon_0\ln^{2m}\ip{t},\\
\Vert \langle\eta\rangle^m\gamma(t)\Vert_{L^\infty}&\le 2\varepsilon_0 \ln^{2m}\ip{t}.
\end{split}
\end{equation}
In particular, $T^\ast=\infty$, $\mathcal{W}(t)\to 0$ as $t\to\infty$, and the electric field decays as
\begin{equation*}
\begin{split}
\Vert \mathcal{E}(t)\Vert_{L^\infty_x}\lesssim \varepsilon_0^2\ln\ip{t}/\langle t\rangle^2.
\end{split}
\end{equation*}
\end{theorem}

By standard local existence theory, in order to establish Theorem \ref{thm:global_moments}, it suffices to prove the following result concerning the propagation of moments:
\begin{proposition}\label{prop:moments_prop}
 Let $m\ge 30$, and assume that the initial density $\nu_0$ satisfies \eqref{eq:mom_id}. If $(\gamma,\mathcal{W})$ is a solution of \eqref{NewNLEq} on $[0,T^\ast]$ satisfying
 \begin{equation}\label{eq:mom_btstrap_assump}
\begin{split}
\Vert \langle a\rangle^{2m}\gamma(t)\Vert_{L^r}+\Vert \langle \xi\rangle^{2m}\gamma(t)\Vert_{L^r}+\Vert \langle \lambda\rangle^{2m}\gamma(t)\Vert_{L^r}+\Vert \langle\eta\rangle^m\gamma(t)\Vert_{L^\infty}&\leq \eps_1\langle \ln(2+t)\rangle^{4m},\quad r\in\{2,\infty\},\\
\sqrt{1+t^2+\vert\y\vert^2}\vert\mathcal{E}(\y,t)\vert&\leq\varepsilon_1^2\langle 2+t\rangle^{-\frac{1}{2}},
\end{split}
\end{equation}
then we have the almost optimal bound for $0\le t\le T^\ast$:
\begin{equation}\label{AlmostSharpDecayEF}
\left[1+t^2+\vert\y\vert^2\right]\vert \mathcal{E}(\y,t)\vert\lesssim \eps_1^2\ip{\ln(2+t)},
\end{equation}
and in fact we have the improved moment bounds for $r\in\{2,\infty\}$: there exists $C>0$ such that, for $0\le t\le T^\ast$:
\begin{align}%\label{eq:mom_btstrap_claim}
 \Vert \ip{a}^{2m}\gamma(t)\Vert_{L^r}+\Vert \ip{\xi}^{2m}\gamma(t)\Vert_{L^r}&\le \varepsilon_0 +C\varepsilon_1^3,\label{eq:mom_btstrap_claim1}\\
 \norm{\lambda^n\gamma(t)}_{L^r}&\le \varepsilon_0+C\varepsilon_1^3\langle \ln(2+t)\rangle^{2n},\qquad 1\leq n\leq 2m,\label{eq:mom_btstrap_claim2}\\
 \norm{\eta^k\gamma(t)}_{L^\infty}&\le \varepsilon_0+C\varepsilon_1^3\langle \ln(2+t)\rangle^{2k},\qquad 1\leq k\leq m.\label{eq:mom_btstrap_claim3}
\end{align}

\end{proposition}

\begin{remark}\label{RemarkPropositionMoments}
\begin{enumerate}
 \item Moments in $a$ are relevant when $a$ is large, and thus when the trajectories of the linearized problem closely approach the point charge. In this setting, it is more favorable to work with the close formulation of Section \ref{ssec:pinnedframe}, where the microscopic velocities are centered around that of the point charge.
 \item Moments in $a$ and $\xi$ can be propagated by themselves, and we only make use of one moment in $L$ resp.\ $\eta$ to obtain a uniform in time (rather than logarithmically growing) bound for $\xi$ moments.
 \item Moments in $L$ and $\eta$ lead to a slow logarithmic growth. Here $\eta$ is not conserved under the linear flow, and we are only able to propagate fewer of the associated moments (and only in $L^\infty$) -- see \eqref{eq:mom_btstrap_claim3}.
\end{enumerate}
\end{remark}

The proof of this Proposition relies on the observation that for any weight function $\ww:\mathcal{P}_{\x,\v}\to\R$ or $\ww:\mathcal{P}_{\vartheta,\a}\to\R$ there holds that
\begin{equation}\label{eq:mom-prop}
\begin{split}
\partial_t(\ww\nu)-\{\mathbb{H},(\ww\nu)\}=-\nu\{\mathbb{H},\ww\},\qquad \partial_t(\ww\gamma)+\{\mathbb{H}_4,(\ww\gamma)\}=\gamma\{\mathbb{H}_4,\ww\}.
\end{split}
\end{equation}

\begin{proof}[Proof of Proposition \ref{prop:moments_prop}]

Since $m\ge 30$, we get from \eqref{eq:mom_btstrap_assump} that
\begin{equation*}
N_1\lesssim \varepsilon_1\langle\ln(2+t)\rangle^{100}.
\end{equation*}
The electric field decay \eqref{AlmostSharpDecayEF} follows by combining Proposition \ref{PropControlEF} and Proposition \ref{PropEeff}. From this and \eqref{NewNLEq} it directly follows that
\begin{equation}\label{eq:Wdecay}
 \abs{\mathcal{W}(t)}\lesssim \int_t^{T^\ast}\abs{\mathcal{E}(s)}ds\lesssim \eps_1^2 \ip{\ln(2+t)}\ip{t}^{-1}.
\end{equation}

\paragraph{\textbf{Proof of \eqref{eq:mom_btstrap_claim1}}}

\paragraph{Moments in $a$}
As explained in Remark \ref{RemarkPropositionMoments}, we use the close formulation. Using \eqref{ComparableMomentNorms}, it suffices to prove the bound for $\gamma^\prime$ for which we can use \eqref{NewNLEq'}. Using \eqref{eq:mom-prop}, we obtain that
\begin{equation*}
\begin{split}
\partial_t(\langle a\rangle^n\gamma^\prime)+\{\mathbb{H}_4^\prime,\langle a\rangle^n\gamma^\prime\}&=\gamma^\prime\{\mathbb{H}_4^\prime,\langle a\rangle^n\}=\gamma^\prime\cdot na\langle a\rangle^{n-2}\left[Q\mathcal{E}_j(\widetilde{\bf X})-\dot{\mathcal{W}}_j\right]\{\widetilde{\bf X}^j,a\},
\end{split}
\end{equation*}
and using \eqref{PBX1}, \eqref{AlmostSharpDecayEF} and \eqref{eq:Wdecay}, we deduce that
\begin{equation}
 \frac{d}{dt}\norm{\ip{a}^n\gamma'}_{L^r}\lesssim \eps_1^2 \ip{t}^{-2}\ip{\ln(2+t)}\cdot \norm{\ip{a}^{n-1}\gamma'}_{L^r},
\end{equation}
which gives the uniform bounds for moments in $a$ in \eqref{eq:mom_btstrap_claim1}.

\paragraph{Moments in $\xi$}

Letting $1\leq n\leq 2m$ we have that
\begin{equation}\label{PBMomXi1}
 \{\HH_4,\ip{\xi}^{n}\}=n\ip{\xi}^{n-2}\xi\left[\mathcal{E}_j(\widetilde{\bf X},t)\cdot\{\XX^j,\xi\}+\mathcal{W}_j\{\VV^j,\xi\}\right],
\end{equation}
and using \eqref{PBX1} followed by \eqref{CrudeBoundX} then \eqref{AlmostSharpDecayEF}, we get that
\begin{equation}\label{PBMomXi2}
\begin{split}
\langle\xi\rangle^{-2}\xi\vert\mathcal{E}_j(\widetilde{\bf X},t)\cdot\{\XX^j,\xi\}\vert\lesssim \xi\vert\mathcal{E}(\widetilde{\bf X},t)\vert\lesssim \vert\widetilde{\bf X}\vert^\frac{1}{2}\vert\mathcal{E}(\widetilde{\bf X},t)\vert\lesssim \varepsilon_1^2\langle t\rangle^{-9/8}.
\end{split}
\end{equation}
For the Poisson bracket with $\VV$ we split into bulk and non-bulk regions. In the bulk, using \eqref{eq:bulk-bds}, we get
\begin{equation*}
 \vert\{\VV,\xi\}\vert\mathfrak{1}_{\B}\lesssim \frac{\xi^3}{q\vert\XX\vert^2}\mathfrak{1}_{\B}\lesssim \frac{\xi^5}{t^2q^3}\mathfrak{1}_{\B}\lesssim\ip{t}^{-1/2},
\end{equation*}
while outside the bulk, we use \eqref{CrudeBoundX} to get
\begin{equation*}
 \vert\{\VV,\xi\}\vert\mathfrak{1}_{\B^c}\lesssim \frac{\xi^3}{q\vert\XX\vert^2}\mathfrak{1}_{\B^c}\lesssim \frac{q\xi}{\xi^2+\lambda^2}\cdot \ip{t}^{-1/4n}(\eta^4+\lambda^2+a^4+\xi^4 )^{1/4n},
\end{equation*}
and thus
\begin{equation}\label{PBMomXi3}
\langle\xi\rangle^{n-2}\xi\vert\mathcal{W}\cdot\{\VV,\xi\}\vert\lesssim \abs{\mathcal{W}(t)}\left[\ip{t}^{-1/2}\langle\xi\rangle^{n-1}+\ip{t}^{-1/4n}(\ip{\xi}^{n-1}+\eta+\lambda)\right].
\end{equation}
Using \eqref{eq:mom-prop}, \eqref{PBMomXi1}, \eqref{PBMomXi2} and \eqref{PBMomXi3}, we deduce that
\begin{equation*}
\begin{split}
\frac{d}{dt}\Vert\langle\xi\rangle^n\gamma\Vert_{L^r}&\lesssim\varepsilon_1^2\langle t\rangle^{-9/8}\Vert\langle\xi\rangle^n\gamma\Vert_{L^r}+\langle t\rangle^{-1/4n}\vert\mathcal{W}(t)\vert\cdot\left[\Vert\langle\xi\rangle^{n-1}\gamma\Vert_{L^r}+\Vert \eta\gamma\Vert_{L^r}+\Vert \lambda\gamma\Vert_{L^r}\right],
\end{split}
\end{equation*}
and using \eqref{eq:Wdecay} and the bootstrap hypothesis \eqref{eq:mom_btstrap_assump}, we obtain a uniform bound.

\paragraph{\textbf{Proof of \eqref{eq:mom_btstrap_claim2}}}
For $1\leq n\leq 2 m$ we compute with \eqref{PBXV} that
\begin{equation}
 \{\HH_4,\lambda^{n}\}=n\lambda^{n-1}\{Q\psi(\XX,t)+\mathcal{W}(t)\cdot\VV,\lambda\}=n\lambda^{n-1}\left(Q\mathcal{E}(\XX,t)\cdot ({\bf l}\times\XX)-\mathcal{W}(t)\cdot({\bf l}\times\VV)\right),
\end{equation}
and thus via \eqref{eq:mom-prop}, the decay estimates \eqref{eq:Wdecay} and \eqref{AlmostSharpDecayEF} as well as the bootstrap assumptions \eqref{eq:mom_btstrap_assump} and the uniform bounds \eqref{eq:mom_btstrap_claim1} for moments in $a$ just established, there holds that
\begin{equation}
\begin{aligned}
 \frac{d}{dt}\norm{\lambda\gamma}_{L^r}&\lesssim \nnorm{\vert\XX\vert\mathcal{E}(\widetilde{\bf X},t)}_{L^\infty}\norm{\gamma}_{L^r}+\abs{\mathcal{W}(t)}\norm{a\gamma}_{L^r}\lesssim\eps_1^3\ip{t}^{-1}\ip{\ln(2+t)},
\end{aligned} 
\end{equation}
and for $n\geq 2$
\begin{equation}
\begin{aligned}
 \frac{d}{dt}\norm{\lambda^n\gamma}_{L^r}&\lesssim \nnorm{\vert\XX\vert\mathcal{E}(\widetilde{\bf X},t)}_{L^\infty}\norm{\lambda^{n-1}\gamma}_{L^r}+\abs{\mathcal{W}(t)}\norm{\lambda^{n-1}a\gamma}_{L^r}\\
 &\lesssim \nnorm{\vert\XX\vert\mathcal{E}(\widetilde{\bf X},t)}_{L^\infty}\norm{\lambda^{n-1}\gamma}_{L^r}+\abs{\mathcal{W}(t)}\cdot\norm{\lambda^{n}\gamma}_{L^r}^\frac{n-1}{n}\Vert a^n\gamma\Vert_{L^r}^\frac{1}{n}\\
 &\lesssim \varepsilon_1^3\ip{t}^{-1}\ip{\ln(2+t)}\cdot\ip{\ln(2+t)}^{2n-2}.
\end{aligned} 
\end{equation}

\paragraph{\textbf{Proof of \eqref{eq:mom_btstrap_claim3}}}
Note that here we only propagate $L^\infty$ bounds. This is because the Poisson bracket bounds of \eqref{PBX1'} favor the formulation in terms of $\gamma$ on $\Omega_t^{far}$, while on $\Omega_t^{cl}$ the alternative formulation in terms of $\gamma'$ is advantageous.

To this end, we note that
\begin{equation}
\begin{aligned}
 \norm{\eta^k\gamma}_{L^\infty}&\leq \norm{\eta^k\gamma}_{L^\infty(\Omega_t^{far})}+\norm{\eta^k\gamma}_{L^\infty(\Omega_t^{cl})},
\end{aligned} 
\end{equation}
where by Lemma \ref{lem:prime_bds} and \eqref{eq:Wdecay}
\begin{equation}
 \abs{\eta^k-(\eta')^k}\lesssim \eps_1^2\cdot \left[\ip{\ln(2+t)}^k(1+a^2)^k+\min\{\eta,\eta'\}^k\right],
\end{equation}
and thus there holds that
\begin{equation}
\begin{aligned}
 \norm{\eta^k\gamma}_{L^\infty(\Omega_t^{cl})}&\leq \norm{[\eta^k-(\eta')^k]\gamma}_{L^\infty(\Omega_t^{cl})}+\norm{(\eta')^k\gamma}_{L^\infty(\Omega_t^{cl})}\\
 &\lesssim \eps_1^2\ip{\ln(2+t)}^k\norm{(1+a^2)^k\gamma}_{L^\infty(\Omega_t^{cl})}+2\norm{(\eta')^k\gamma}_{L^\infty(\Omega_t^{cl})}.
\end{aligned} 
\end{equation}
Furthermore, by invariance of $\Omega_t^{cl}$ under $\mathcal{M}_t$ (see \eqref{eq:XMinv}) and \eqref{eq:weightprimes} we have that
\begin{equation}\label{eq:eta'gamma'}
 \norm{(\eta')^k\gamma}_{L^\infty(\Omega_t^{cl})}=\norm{((\eta')^k\gamma)\circ\mathcal{M}_t}_{L^\infty(\Omega_t^{cl})}=\norm{\eta^k\gamma'}_{L^\infty(\Omega_t^{cl})}.
\end{equation}

Similarly, the bootstrap assumptions \eqref{eq:mom_btstrap_assump} imply $\eta$ moment bounds also on $\gamma'$, namely
\begin{equation}\label{eq:eta'gamma}
 \norm{(\eta')^k\gamma}_{L^\infty(\Omega_t^{cl})}\lesssim \eps_1^2\ip{\ln(2+t)}^k\norm{(1+a^2)^k\gamma}_{L^\infty(\Omega_t^{cl})}+2\norm{\eta^k\gamma}_{L^\infty(\Omega_t^{cl})}.
\end{equation}
In conclusion, since trajectories can enter/exit the close/far regions at most once (Lemma \ref{lem:traj_cl_far}) and $\Omega_t^{far}\cap\Omega_t^{cl}=\{\ip{t}\leq \vert\XX\vert\leq 10\ip{t}\}$, to establish the claim it suffices to propagate $\eta$ moments on $\gamma'$ for trajectories in $\Omega_t^{cl}$, while for trajectories in $\Omega_t^{far}$ we work with $\gamma$ directly.

\subsubsection*{On $\Omega_t^{far}$}
With the Poisson bracket bounds \eqref{PBX1'}, we have for $1\leq k\leq m$
\begin{equation*}
\begin{aligned}
 \abs{\{\HH_4,\eta^{k}\}}&=k\eta^{k-1}\vert\{Q\psi(\XX,t)+\mathcal{W}(t)\cdot\VV,\eta\}\vert\\
 &\lesssim \eta^{k-1}\big[\vert\mathcal{E}(\widetilde{\bf X},t)\vert(\xi^{-1}\vert\XX\vert+tq\xi^{-2})+\abs{\mathcal{W}(t)}(q\xi^{-2}+tq/(\xi\vert\XX\vert^2))\big]\\
 &\lesssim \eta^{k-1}[a+a^2]\cdot\left[(t+\vert\widetilde{\bf X}\vert)\vert\mathcal{E}(\widetilde{\bf X},t)\vert+\vert\mathcal{W}(t)\vert\right],
\end{aligned} 
\end{equation*}
and using \eqref{AlmostSharpDecayEF}, \eqref{eq:Wdecay} and \eqref{eq:mom_btstrap_claim1} with the bootstrap assumption \eqref{eq:mom_btstrap_assump}, we see that on $\Omega^{far}_t$
\begin{equation}\label{eq:eta-f2}
\begin{split}
\frac{d}{dt}\vert \eta^k\gamma(t)\vert\lesssim \varepsilon_1^2\langle t\rangle^{-1}\langle\ln(2+t)\rangle\cdot \Vert\eta^k\gamma(t)\Vert_{L^\infty}^\frac{k-1}{k}\cdot\Vert \langle a\rangle^{2k}\gamma(t)\Vert_{L^\infty}^\frac{1}{k}&\lesssim\varepsilon_1^3\langle t\rangle^{-1}\langle\ln(2+t)\rangle^{2k-1}.
\end{split}
\end{equation}

\subsubsection*{On $\Omega_t^{cl}$}
On the other hand, we have that
\begin{equation*}
\begin{aligned}
 \abs{\{\HH_4',\eta^{k}\}}&=k\eta^{k-1}\vert\{Q\psi(\widetilde{\bf X},t)-\dot{\mathcal{W}}(t)\cdot \XX,\eta\}\vert \lesssim \eta^{k-1}(\abs{\mathcal{E}(t)}+\vert\dot{\mathcal{W}}(t)\vert)(\xi^{-1}\vert\XX\vert+tq\xi^{-2}),
\end{aligned} 
\end{equation*}
which will be used when $\vert\XX\vert\lesssim \ip{t}$. This gives as before
\begin{equation}\label{eq:eta-c2}
\begin{aligned}
 \frac{d}{dt}\abs{\eta^k\gamma'}&\lesssim \ip{t}[\nnorm{\mathcal{E}(t)}_{L^\infty}+\vert\dot{\mathcal{W}}(t)\vert]\cdot\Vert \eta^k\gamma^\prime(t)\Vert_{L^\infty(\Omega^{cl}_t)}^\frac{k-1}{k}\norm{\langle a\rangle^{2k}\gamma'(t)}^\frac{1}{k}_{L^\infty(\Omega^{cl}_t)}\\
 &\lesssim \varepsilon_1^3\ip{t}^{-1}\ip{\ln(2+t)}^{2k-1},
\end{aligned} 
\end{equation}
where we have also used \eqref{ComparableMomentNorms} and \eqref{eq:eta'gamma} to bound the norms of $\gamma^\prime$ in terms of bounds on $\gamma$. Finally, combining \eqref{eq:eta-f2} and \eqref{eq:eta-c2} gives the claim.
\end{proof}

\subsection{Derivative Propagation}\label{sec:derivs_prop}
In this section we will propagate derivative control on our nonlinear unknowns. To exploit the symplectic structure, we do this by means of Poisson brackets with a suitable ``spanning'' set of functions $S$:
\begin{definition}
 Let $S\subset C^1(\mathcal{P}_{\vartheta,\a},\R)$ be a set of smooth, scalar functions on phase space. Then $S$ is a \emph{spanning} set, if at all points of $\mathcal{P}_{\vartheta,\a}$ the collection of associated symplectic gradients $\{f,\cdot\}$, $f\in S$, spans the cotangent space of $\mathcal{P}_{\vartheta,\a}$.
\end{definition}
We note that the functions in a spanning set $S$ need not be independent. By composition with the canonical diffeomorphism $\mathcal{T}^{-1}$ of Proposition \ref{PropAA}, we obtain an equivalent description in terms of the ``physical'' phase space $\mathcal{P}_{\x,\v}$.

Guided by the super-integrable coordinates (Section \ref{ssec:SIC}, see also equations \eqref{eq:deriv} and \eqref{eq:deriv'} below), we choose the collection $\textnormal{SIC}:=\{\xi,\eta,\u,\L\}$ as spanning set. For a suitable choice of weight functions $\ww_f>0$, $f\in \textnormal{SIC}$, defined in \eqref{eq:def_weights} below, we will then propagate derivative control through a functional of the form
\begin{equation}\label{eq:defP}
 \mathcal{D}[\zeta](\vartheta,\a;t):=\sum_{f\in \textnormal{SIC}}\ww_f\aabs{\{f,\zeta\}}(\vartheta,\a),\qquad \zeta\in\{\gamma,\gamma'\}.
\end{equation}
Due to the motion of the point charge, as in Section \ref{sec:moments_prop} we will work with both ``close'' and ``far'' formulations of our equations, and thus consider
\begin{equation}
 \mathcal{D}(\vartheta,\a;t):=
 \begin{cases}
  \mathcal{D}[\gamma](\vartheta,\a;t),&  (\vartheta,\a)\in\Omega_t^{far},\\
  \mathcal{D}[\gamma'](\vartheta,\a;t),&  (\vartheta,\a)\in\Omega_t^{cl}.
 \end{cases}
\end{equation}

The remainder of this section then establishes the following result:
\begin{proposition}\label{prop:global_derivs}
 There exists $\eps_0>0$ such that the following holds: Let $\gamma(0)=\nu_0$ be given, satisfying the hypothesis \eqref{eq:mom_id} of Theorem \ref{thm:global_moments}, and let
 \begin{equation}\label{eq:def_weights}
  \ww_\xi=\frac{1+\xi}{\xi^2},\qquad \ww_\eta=\frac{\xi^2}{1+\xi}, \qquad  \ww_\u=1,\qquad \ww_\L=\frac{\xi}{1+\xi}.
 \end{equation}
 Then there holds that if
 \begin{equation}\label{eq:deriv_btstrap_id}
  \norm{(\xi^{5}+\xi^{-8})\mathcal{D}(\vartheta,\a;0)}_{L^\infty_{\vartheta,\a}}\lesssim \eps_0,
 \end{equation}
 the global solution $\gamma$ of Theorem \ref{thm:global_moments} satisfies
 \begin{equation}\label{eq:global_derivs}
  \norm{\mathcal{D}(\vartheta,\a;t)}_{L^\infty_{\vartheta,\a}}\lesssim \eps_0\ln^2\ip{t}.
 \end{equation}
\end{proposition}
\begin{remark}
The proof proceeds through control of $\mathcal{D}(\vartheta,\a;t)$ along trajectories of \eqref{ODE}.
\begin{enumerate}
 \item We note that all weights in \eqref{eq:def_weights} are functions exclusively of $\xi$, which is conserved along such trajectories. Consequently, additional weights in $\xi$ on $\mathcal{D}(\vartheta,\a;t)$ can be propagated as well.
 \item The proof shows that a slightly weaker assumption than \eqref{eq:deriv_btstrap_id} (expressed in terms of another $\xi$-dependent weight function \eqref{eq:defUps}) suffices, and in fact gives a slightly stronger conclusion. 
\end{enumerate}
\end{remark}

\subsubsection{Proof setup}
To understand the evolution of symplectic gradients on our nonlinear unknowns, we begin by observing that by the chain rule one has that for any functions $g_j:\mathcal{P}\to\R$, $1\leq j\leq 3$ and $G:\R^2\to\R$ there holds that
\begin{equation}
 \{G(g_1,g_2),g_3\}=\partial_1 G\cdot\{g_1,g_3\}+\partial_2 G\cdot\{g_2,g_3\}.
\end{equation}
Using the Jacobi identity \eqref{Jacobi} and the equations \eqref{NewNLEq}, we thus find
\begin{equation}\label{eq:deriv}
\begin{split}
\partial_t\{f,\gamma\}+\{\mathbb{H}_4,\{f,\gamma\}\}&=-\{\{f,\mathbb{H}_4\},\gamma\}=-Q\{\{f,\psi(\XX)\},\gamma\}-\mathcal{W}_j\{\{f,\widetilde{\bf V}^j\},\gamma\}\\
&=-Q\mathcal{F}_{jk}\{\widetilde{\bf X}^j,\gamma\}\{f,\widetilde{\bf X}^k\}-Q\mathcal{E}_j\{\{f,\widetilde{\bf X}^j\},\gamma\}-\mathcal{W}_j\{\{f,\widetilde{\bf V}^j\},\gamma\}.
\end{split}
\end{equation}
Alternatively, if we choose the formulation \eqref{NewNLEq'} with velocities centered on the point charge, we obtain
\begin{equation}\label{eq:deriv'}
\begin{split}
\partial_t\{f,\gamma'\}+\{\mathbb{H}_4',\{f,\gamma'\}\}&=-\{\{f,\mathbb{H}_4'\},\gamma'\}=-Q\{\{f,\psi(\XX)\},\gamma'\}+\dot{\mathcal{W}}_j\{\{f,\widetilde{\bf X}^j\},\gamma'\}\\
&=-Q\mathcal{F}_{jk}\{f,\widetilde{\bf X}^k\}\{\widetilde{\bf X}^j,\gamma'\}-\left(Q\mathcal{E}_j-\dot{\mathcal{W}}_j\right)\{\{f,\widetilde{\bf X}^j\},\gamma'\}.
\end{split}
\end{equation}
Using \eqref{eq:U_resol}, we can then express \eqref{eq:deriv} as
\begin{equation}\label{eq:deriv2}
 \partial_t\{f,\gamma\}+\{\mathbb{H}_4,\{f,\gamma\}\}=\sum_{g\in \textnormal{SIC}}\m_{fg}\{g,\gamma\},
\end{equation}
for a suitable coefficient matrix $(\m_{fg})_{f,g\in \textnormal{SIC}}$, and similarly for \eqref{eq:deriv'}:
\begin{equation}\label{eq:deriv'2}
 \partial_t\{f,\gamma'\}+\{\mathbb{H}'_4,\{f,\gamma'\}\}=\sum_{g\in \textnormal{SIC}}\m'_{fg}\{g,\gamma'\}.
\end{equation}

\medskip

In addition to the distinction between ``far'' and ``close'' dynamics and the corresponding formulations \eqref{eq:deriv2} and \eqref{eq:deriv'2}, we will separate ``incoming'' from ``outgoing'' dynamics through the decomposition $\mathcal{P}_{\vartheta,\a}=\mathcal{I}_t\cup\mathcal{O}_t$ with
\begin{equation}\label{eq:inoutdomains}
 \mathcal{I}_t:=\{(\vartheta,\a):\XX\cdot\VV\le (L^2+\frac{q^2}{a^2})^{1/2}\},\qquad \mathcal{O}_t:=\{(\vartheta,\a):\XX\cdot\VV\ge -(L^2+\frac{q^2}{a^2})^{1/2}\}.
\end{equation}
This reflects the fact that along a trajectory of \eqref{ODE} the asymptotic velocities at $\pm\infty$ may differ drastically in direction. As a consequence, on $\mathcal{I}_t$ their location $\XX$ is better approximated by the \emph{past} asymptotic action, and there we will thus work with the \emph{past} angle-action coordinates of Section \ref{ssec:past_AA}. Moreover, the splitting is chosen such that periapsis occurs in $\mathcal{I}_t\cap\mathcal{O}_t$: recall that the time of periapsis $t_p(\vartheta_0,\a_0)\in\R$ of a trajectory of \eqref{ODE} starting at a given $(\vartheta_0,\a_0)$ is such that
\begin{equation}
 \XX\cdot\VV\,(\vartheta_0,\a_0)=0\quad\hbox{ when }t=t_p(\vartheta_0,\a_0).
\end{equation}

\medskip
We thus have four dynamically relevant regions of phase space $\mathcal{P}_{\vartheta,\a}=\mathcal{I}_t\cup\mathcal{O}_t=\Omega_t^{far}\cup\Omega_t^{cl}$:
\begin{equation}\label{eq:regions}
 \mathcal{I}_t^{far}:=\mathcal{I}_t\cap\Omega_t^{far},\quad \mathcal{O}_t^{far}:=\mathcal{O}_t\cap\Omega_t^{far},\quad \mathcal{I}_t^{cl}:=\mathcal{I}_t\cap\Omega_t^{cl},\quad \mathcal{O}_t^{cl}:=\mathcal{O}_t\cap\Omega_t^{cl}.
\end{equation}
\begin{remark}\label{rem:regions}
 \begin{enumerate}
 \item The four regions of \eqref{eq:regions} have overlap.

 \item General trajectories of \eqref{ODE} can visit all four dynamically distinct regions: starting far away from periapsis and incoming in $\mathcal{I}_t^{far}$, a trajectory with sufficiently large velocity $a\gg 1$ (i.e.\ $\xi\ll 1$) will proceed to $\mathcal{I}_t^{cl}$ and through $\mathcal{O}_t^{cl}$ to end up in $\mathcal{O}_t^{far}$.
 
 \item If at initial time our unknown has compact support in phase space, we can drastically simplify the proof (at the cost of making our assumptions dependent on the radius of the initial support) and work with only one formulation: For the forward in time evolution it then suffices to consider future asymptotic actions, and if we redefine the close resp.\ far regions to account for the largest and  smallest asymptotic actions of the initial data, we can reduce to the case where all trajectories lie in either $\mathcal{O}_t^{far}$ or $\mathcal{O}_t^{cl}$. 

\end{enumerate}
\end{remark}

\begin{proof}[Proof of Proposition \ref{prop:global_derivs}]
To establish Proposition \ref{prop:global_derivs}, we control $\mathcal{D}(\vartheta,\a;t)$ pointwise along trajectories of \eqref{ODE}. For a given trajectory, we do this separately in the four phase space regions defined in \eqref{eq:regions}, and combine with bounds for the transition between the coordinates/formulations:
\begin{itemize}
 \item Incoming to outgoing (from $\mathcal{I}_t$ to $\mathcal{O}_t$): We work with incoming asymptotic actions on $\mathcal{I}_t=\mathcal{I}_t^{cl}\cup\mathcal{I}_t^{far}$, and outgoing asymptotic actions on $\mathcal{O}_t=\mathcal{O}_t^{cl}\cup\mathcal{O}_t^{far}$. If a trajectory passes through periapsis (which it can do at most once), we transition between coordinates there. Hereby, we recall from \eqref{eq:def_SIC(-)} that $\xi^{(-)}=-\xi$ and $\lambda^{(-)}=\lambda$, so that in particular the splitting in \eqref{eq:inoutdomains} is well-defined. By Lemma \ref{lem:trans_inout} below we have that 
 \begin{equation}
  \mathcal{D}[\zeta](\vartheta,\a;t_p(\vartheta,\a))\lesssim (\xi^2+\xi^{-2}) \mathcal{D}[\zeta](\vartheta^{(-)},\a^{(-)};t_p(\vartheta,\a)),\qquad \zeta\in\{\gamma,\gamma'\}.
 \end{equation}

 \item Far to close, or close to far (between $\Omega_t^{far}$ and $\Omega_t^{cl}$): We recall that since $\gamma'=\gamma\circ\mathcal{M}_t$ are related by the canonical diffeomorphism $\mathcal{M}_t$ of \eqref{eq:defM_t}, their Poisson brackets are related by \eqref{eq:PBrel'}, and we have in particular that
\begin{equation}\label{eq:defP'}
 \mathcal{D}[\gamma']=\mathcal{D}'[\gamma]\circ\mathcal{M}_t,\qquad \mathcal{D}'[\gamma]:=\sum_{f\in \textnormal{SIC}}\ww_{f'}\aabs{\{f',\gamma\}},
\end{equation}
where we used the notation of \eqref{eq:weightprimes} to denote $f':=f\circ\mathcal{M}_t^{-1}$. We recall further that $\Omega_t^{cl}$, $\Omega_t^{far}$ and the transition region $\Omega_t^{cl}\cap\Omega_t^{far}$ are left invariant by $\mathcal{M}_t$, and by the below Lemma \ref{lem:trans_short} the weights $\xi$ and $\xi'$ are comparable on $\Omega_t^{cl}$, and there holds that
\begin{equation}
 \mathcal{D}[\gamma](t)\lesssim (\xi +\xi^{-3})\mathcal{D}'[\gamma](t)\lesssim (\xi +\xi^{-3})^2\mathcal{D}[\gamma](t) \quad \textnormal{on }\Omega_t^{cl}\cap\Omega_t^{far}.
\end{equation}
We recall also that by Lemma \ref{lem:traj_cl_far}, per trajectory at most two transitions between close and far regions are possible.
\end{itemize}

We highlight that the weight functions $\ww_f$, $f\in \textnormal{SIC}$ depend solely on $\xi$, which is constant along a trajectory of \eqref{ODE}. The above ``losses'' in $\xi$ are thus only incurred at the transitions, and propagated along the trajectories otherwise. To properly account for this, we introduce the following function that tracks the $\xi$ weights that will be used for transition:
\begin{equation}\label{eq:defUps}
 \Upsilon:\mathcal{P}_{\vartheta,\a}\to\R,\quad (\vartheta,\a)\mapsto \ip{\xi}\cdot
 \begin{cases}
 (\xi^2+\xi^{-2})(\xi+\xi^{-3})^2,\qquad &\textnormal{on }\mathcal{I}_t^{far},\\
 (\xi^2+\xi^{-2})(\xi+\xi^{-3}),\qquad &\textnormal{on }\mathcal{I}_t^{cl},\\
 (\xi+\xi^{-3}),\qquad &\textnormal{on }\mathcal{O}_t^{cl},\\
 1,\qquad &\textnormal{on }\mathcal{O}_t^{far}.
 \end{cases}
\end{equation}
Since \eqref{eq:deriv_btstrap_id} implies that $\nnorm{\Upsilon\mathcal{D}(0)}_{L^\infty}\lesssim\eps_0$, to establish Proposition \ref{prop:global_derivs} it then suffices to show the following bootstrap argument: If for $0\leq T_1<T_2$ we have that for some $0<\delta\ll 1$ there holds that
\begin{equation}\label{eq:deriv_btstrap_assump}
 \norm{\Upsilon\mathcal{D}(t)}_{L^\infty_{\vartheta,\a}}\lesssim \eps\ip{t}^\delta,
\end{equation}
then in fact
\begin{equation}\label{eq:deriv_btstrap_concl}
 \norm{\Upsilon\mathcal{D}(t)}_{L^\infty_{\vartheta,\a}}\lesssim \eps+\eps^2\ln(t/T_1).
\end{equation}
We note that while by construction \eqref{eq:deriv_btstrap_assump} only gives control of $\mathcal{D}[\gamma]$ on $\Omega_t^{far}$ resp.\ $\mathcal{D}[\gamma']$ on $\Omega_t^{cl}$, under the assumption \eqref{eq:deriv_btstrap_assump} we obtain from Lemma \ref{lem:trans_full} via \eqref{eq:PBcl_byfar} resp.\ \eqref{eq:PBcl_byfar_bulk} bounds for $\mathcal{D}[\gamma]$ across \emph{both} far and close regions:
\begin{equation}\label{eq:deriv_bstrap_assump2}
 \mathfrak{1}_{\mathcal{B}}\mathcal{D}[\gamma]\lesssim \eps\ip{t}^\delta,\qquad \mathfrak{1}_{\mathcal{B}^c}\mathcal{D}[\gamma]\lesssim \eps (1+\ip{t}q\xi^{-2})\ip{t}^\delta.
\end{equation}
In particular, by Proposition \ref{PropEeff} $(ii)$ it then holds that
\begin{equation}\label{eq:EFdecay0}
 \sqrt{t^2+\abs{\y}^2}\abs{\E(\y,t)}\lesssim \eps^2\ip{t}^{-1},\quad  \sup_{\y\in\R^3}\left[t^2+\abs{\y}^2\right]\abs{\F(\y,t)}\lesssim \eps^2\ln\ip{t}\cdot\ip{t}^{-1}.
\end{equation}

We note that the expressions for the kinematically relevant quantities $\XX,\VV$ (and thus also for their derivatives) agree up to signs in both incoming and outgoing asymptotic actions (see \eqref{XVinSIC(-)}), and can thus be treated analogously. We will henceforth focus on the case of outgoing asymptotic actions; the diligent reader can easily trace the necessary adaptations for the incoming asymptotic actions. To complete the proof it thus suffices to provide a bootstrap argument showing that if a trajectory stays in $\mathcal{O}_t^{far}$ resp.\ $\mathcal{O}_t^{cl}$ on a time interval $t\in[T_1,T_2]$, its weighted symplectic gradients $\mathcal{D}[\gamma]$ resp.\ $\mathcal{D}'[\gamma]$ grow by at most a factor of $1+\ln^2(t/T_1)$. This is done in Lemmas \ref{lem:farprop} resp.\ \ref{lem:clprop} below. Once these are established (see Sections \ref{sec:prop_far} and \ref{sec:prop_cl} below), the proof of Proposition \ref{prop:global_derivs} is thus complete.
\end{proof}

As a result of direct computation (see Section \ref{ssec:transitions}), we have that
\begin{lemma}[Transition from incoming to outgoing]\label{lem:trans_inout}
 Consider a trajectory of \eqref{ODE} starting at $(\vartheta,\a)$, with periapsis $t_p(\vartheta,\a)$. Then
 \begin{equation}
  \mathcal{D}[\zeta](\vartheta,\a;t_p(\vartheta,\a))\lesssim (\xi^2+\xi^{-2})\mathcal{D}[\zeta](\vartheta^{(-)},\a^{(-)};t_p(\vartheta,\a)) ,\qquad \zeta\in\{\gamma,\gamma'\}.
 \end{equation}
\end{lemma}
\begin{proof}
 This follows directly by collecting the terms in \eqref{eq:inout_PBtrans}.
\end{proof}

\begin{lemma}[Transition between close and far]\label{lem:trans_short}
Under the assumptions of Proposition \ref{prop:global_derivs}, consider outgoing (resp.\ incoming) asymptotic actions on $\mathcal{O}_t$ (resp.\ $\mathcal{I}_t$). Then on $\Omega_t^{cl}$ there holds that for some $C>0$
\begin{equation}\label{eq:xi-xi'-comp}
 C^{-1}\xi'\leq\xi\leq C\xi',
\end{equation}
and on $\Omega_t^{cl}\cap\Omega_t^{far}$ we have that
\begin{equation}\label{eq:deriv-deriv-PB}
 \mathcal{D}'[\gamma](t)\lesssim (\xi +\xi^{-3})\mathcal{D}[\gamma](t),\qquad \mathcal{D}[\gamma](t)\lesssim (\xi +\xi^{-3})\mathcal{D}'[\gamma](t).
\end{equation}

\end{lemma}
\begin{proof}
 We note that by construction the expressions for $\gamma$ in terms of $\gamma'$ and vice versa are analogous (see \eqref{eq:defM_t} and observe also that $\gamma=\gamma'\circ\mathcal{M}_t^{-1}$, where $\mathcal{M}_t^{-1}=\mathcal{T}\circ\Phi_t\circ\Sigma_{-\mathcal{W}(t)}\circ\Phi_t^{-1}\circ\mathcal{T}^{-1}$) and thus once \eqref{eq:xi-xi'-comp} is established, by symmetry it suffices to prove the first estimate of \eqref{eq:deriv-deriv-PB}. This follows from Lemma \ref{lem:trans_full} of Appendix \ref{sec:appdx_trans} (see also $(1)$ of Remark \ref{rem:transition_bds}), where we prove more precise estimates for the relation between Poisson brackets in the close and far formulation.
\end{proof}

\bigskip
\subsubsection{The coefficients $\m_{fg}$}\label{ssec:mfg}
Since the kinematic quantities $\XX$ and $\VV$ are naturally expressed in terms of the slightly larger spanning set $\textnormal{SIC}\cup\{\lambda\}=:\textnormal{SIC}_+$, which we also used in the computation of Poisson bracket bounds in Section \ref{sec:PB_bounds}, it is convenient to carry on using $\lambda$ explicitly for computations. We note from Remark \ref{PBWithLambdaRmk} that Poisson brackets with $\lambda$ follow directly from those with $\L$.

Using \eqref{eq:U_resol} we express the right hand side of \eqref{eq:deriv} in the form \eqref{eq:deriv2}. This yields that for $f\in\textnormal{SIC}$ we have that
\begin{equation}
\begin{aligned}
 -\m_{f\xi}&=Q\F_{jk}\{\XX^j,\eta\}\{f,\XX^k\}+Q\E_j\{\{f,\XX^j\},\eta\}+\W_j\{\{f,\VV^j\},\eta\},\\
 \m_{f\eta}&=Q\F_{jk}\{\XX^j,\xi\}\{f,\XX^k\}+Q\E_j\{\{f,\XX^j\},\xi\}+\W_j\{\{f,\VV^j\},\xi\},\\
 -\m_{f\lambda}&=Q\F_{jk}(\partial_\lambda\XX^j)\{f,\XX^k\}+Q\E_j\partial_\lambda(\{f,\XX^j\})+\W_j\partial_\lambda(\{f,\VV^j\}),\\
 -\m_{f\u^a}&=Q\F_{jk}(\mathbb{P}_{\bf u}^a\widetilde{\bf X}^j)\{f,\XX^k\}+Q\E_j\mathbb{P}_{\u}^a\{f,\XX^j\}+\W_j\mathbb{P}_{\u}^a\{f,\VV^j\},\\
 -\m_{f\L^a}&=Q\F_{jk}(\mathbb{P}^a_{\bf L}\widetilde{\bf X}^j)\{f,\XX^k\}+Q\E_j\cdot\mathbb{P}_{\L}^a\{f,\XX^j\}+\W_j\mathbb{P}_{\L}^a\{f,\VV^j\},
\end{aligned}
\end{equation}
where we used the notation $\mathbb{P}_{\bm{U}}^a$ to denote the $a$-th component of the projection onto a vector\footnote{With the understanding that $\mathbb{P}_{\bf u}^a(X_3\in^{jcd}{\bf L}^c{\bf u}^d)=X_3\in^{jcd}{\bf L}^c\delta^{ad}$ and $\mathbb{P}_{\bf L}^a(X_3\in^{jcd}{\bf L}^c{\bf u}^d)=X_3\in^{jcd}{\bf u}^d\delta^{ac}$.} $\bm{U}\in\R^3$. We thus obtain the following expressions for the multipliers:
\begin{itemize}[wide]
 \item In $\xi$: 
 \begin{equation}\label{eq:mult_xi}
 \begin{aligned}
  -\m_{\xi\xi}&=Q\mathcal{F}_{jk}\{\XX^j,\eta\}\{\xi,\XX^k\}+Q\mathcal{E}_j\{\{\xi,\XX^j\},\eta\}+\mathcal{W}_j\{\{\xi,\VV^j\},\eta\},\\
 \m_{\xi\eta}&=Q\mathcal{F}_{jk}\{\XX^j,\xi\}\{\xi,\XX^k\}+Q\mathcal{E}_j\{\{\xi,\XX^j\},\xi\}+\mathcal{W}_j\{\{\xi,\VV^j\},\xi\},\\
 -\m_{\xi\lambda}&=Q\mathcal{F}_{jk}(\partial_\lambda\XX^j)\{\xi,\XX^k\}+Q\mathcal{E}_j\partial_\lambda(\{\xi,\XX^j\})+\mathcal{W}_j\partial_\lambda(\{\xi,\VV^j\}),\\
  -\m_{\xi\u^a}&=Q[\mathcal{F}_{ak}\widetilde{X}_1+\mathcal{F}_{jk}\widetilde{X}_3\in^{jca}\L^c]\{\xi,\XX^k\}+Q[\mathcal{E}_a\{\xi,\widetilde{X}_1\}+\mathcal{E}_j\{\xi,\widetilde{X}_3\}\in^{jca}\L^c]\\
  &\qquad+[\mathcal{W}_a\{\xi,\widetilde{V}_1\}+\mathcal{W}_j\{\xi,\widetilde{V}_3\}\in^{jca}\L^c],\\
 -\m_{\xi\L^a}&=Q\mathcal{F}_{jk}\widetilde{X}_3\in^{jad}\u^d\{\xi,\XX^k\}+Q\mathcal{E}_j\{\xi,\widetilde{X}_3\}\in^{jad}\u^d+\mathcal{W}_j\{\xi,\widetilde{V}_3\}\in^{jad}\u^d.
 \end{aligned}
 \end{equation}
 
 \item In $\eta$:
 \begin{equation}\label{eq:mult_eta}
\begin{aligned}
 - \m_{\eta\xi}&=Q\mathcal{F}_{jk}\{\XX^j,\eta\}\{\eta,\XX^k\}+Q\mathcal{E}_j\{\{\eta,\XX^j\},\eta\}+\mathcal{W}_j\{\{\eta,\VV^j\},\eta\},\\
 \m_{\eta\eta}&=Q\mathcal{F}_{jk}\{\XX^j,\xi\}\{\eta,\XX^k\}+Q\mathcal{E}_j\{\{\eta,\XX^j\},\xi\}+\mathcal{W}_j\{\{\eta,\VV^j\},\xi\}=-\m_{\xi\xi},\\
 -\m_{\eta\lambda}&=Q\mathcal{F}_{jk}(\partial_\lambda\XX^j)\{\eta,\XX^k\}+Q\mathcal{E}_j\partial_\lambda(\{\eta,\XX^j\})+\mathcal{W}_j\partial_\lambda(\{\eta,\VV^j\}),\\
  -\m_{\eta\u^a}&=Q[\mathcal{F}_{ak}\widetilde{X}_1+\mathcal{F}_{jk}\widetilde{X}_3\in^{jca}\L^c]\{\eta,\XX^k\}+Q[\mathcal{E}_a\{\eta,\widetilde{X}_1\}+\mathcal{E}_j\{\eta,\widetilde{X}_3\}\in^{jca}\L^c]\\
  &\qquad+[\mathcal{W}_a\{\eta,\widetilde{V}_1\}+\mathcal{W}_j\{\eta,\widetilde{V}_3\}\in^{jca}\L^c],\\
 -\m_{\eta\L^a}&=Q\mathcal{F}_{jk}\widetilde{X}_3\in^{jad}\u^d\{\eta,\XX^k\}+Q\mathcal{E}_j\{\eta,\widetilde{X}_3\}\in^{jad}\u^d+\mathcal{W}_j\{\eta,\widetilde{V}_3\}\in^{jad}\u^d.
\end{aligned}
\end{equation}

 \item In $\u$: Noting that by \eqref{eq:U_resol} and \eqref{PBSIC}, we have that 
\begin{equation}\label{eq:PBuX}
\begin{aligned}
 \{\u^k,\XX^j\}&=\partial_\lambda\XX^j\in^{krs}\l^r\u^s+\widetilde{X}_3(\delta^{jk}-\u^j\u^k)\\
 &=\partial_\lambda\widetilde{X}_1\u^j\in^{krs}\l^r\u^s+\lambda\partial_\lambda\widetilde{X}_3\in^{jpq}\in^{krs}\l^p\u^q\l^r\u^s+\widetilde{X}_3(\delta^{jk}-\u^j\u^k),
\end{aligned}
\end{equation}
there holds that
\begin{equation}\label{eq:mult_u}
\begin{aligned} 
 -\m_{\u^a\xi}&=Q\mathcal{F}_{jk}\{\XX^j,\eta\}\{\u^a,\XX^k\}+Q\mathcal{E}_j\{\{\u^a,\XX^j\},\eta\}+\mathcal{W}_j\{\{\u^a,\VV^j\},\eta\},\\
 \m_{\u^a\eta}&=Q\mathcal{F}_{jk}\{\XX^j,\xi\}\{\u^a,\XX^k\}+Q\mathcal{E}_j\{\{\u^a,\XX^j\},\xi\}+\mathcal{W}_j\{\{\u^a,\VV^j\},\xi\},\\
 -\m_{\u^a\lambda}&=Q\mathcal{F}_{jk}(\partial_\lambda\XX^j)\{\u^a,\XX^k\}+Q\mathcal{E}_j\partial_\lambda(\{\u^a,\XX^j\})+\mathcal{W}_j\partial_\lambda(\{\u^a,\VV^j\}),\\
 -\m_{\u^{a}\u^b}&=Q[\mathcal{F}_{bk}\widetilde{X}_1+\mathcal{F}_{jk}\widetilde{X}_3\in^{jcb}\L^c]\{\u^a,\XX^k\}\\
 &\qquad+Q[\mathcal{E}_b\partial_\lambda\widetilde{X}_1\in^{ars}\l^r\u^s+\mathcal{E}_j\partial_\lambda\widetilde{X}_1\u^j\in^{arb}\l^r]\\
 &\qquad+Q\mathcal{E}_j\lambda\partial_\lambda\widetilde{X}_3\l^p\l^r(\in^{jpb}\in^{ars}\u^s+\in^{jpq}\in^{arb}\u^q)-Q\widetilde{X}_3(\mathcal{E}_b\u^a+\mathcal{E}_j\u^j\delta^{ab})\\
 &\qquad+[\mathcal{W}_b\partial_\lambda\widetilde{V}_1\in^{ars}\l^r\u^s+\mathcal{W}_j\partial_\lambda\widetilde{V}_1\u^j\in^{arb}\l^r]\\
 &\qquad+\mathcal{W}_j\lambda\partial_\lambda\widetilde{V}_3\l^p\l^r(\in^{jpb}\in^{ars}\u^s+\in^{jpq}\in^{arb}\u^q)-\widetilde{V}_3(\mathcal{W}_b\u^a+\mathcal{W}_j\u^j\delta^{ab}),\\
 -\m_{\u^{a}\L^b}&=Q\mathcal{F}_{jk}\widetilde{X}_3\in^{jbd}\u^d\{\u^a,\XX^k\}+Q\mathcal{E}_j\lambda^{-1}\partial_\lambda\widetilde{X}_1\u^j\in^{abs}\u^s\\
&\qquad +Q\mathcal{E}_j\partial_\lambda\widetilde{X}_3\u^q\u^s(\in^{jbq}\in^{ars}\l^r+\in^{jpq}\in^{abs}\l^p)\\
 &\qquad +\mathcal{W}_j\lambda^{-1}\partial_\lambda\widetilde{V}_1\u^j\in^{abs}\u^s\\
&\qquad +\mathcal{W}_j\partial_\lambda\widetilde{V}_3\u^q\u^s(\in^{jbq}\in^{ars}\l^r+\in^{jpq}\in^{abs}\l^p).
\end{aligned}
\end{equation}

\item In $\L$: Using that $\{\L^k,\XX^j\}=\in^{kjp}\XX^p$ and $\{\L^k,\VV^j\}=\in^{kjp}\VV^p$ we have that
\begin{equation}\label{eq:mult_L}
\begin{aligned} 
- \m_{\L^a\xi}&=Q\mathcal{F}_{jk}\{\XX^j,\eta\}\in^{akp}\XX^p+Q\mathcal{E}_j\in^{ajp}\{\XX^p,\eta\}+\mathcal{W}_j\in^{ajp}\{\VV^p,\eta\}, \\
 \m_{\L^a\eta}&=Q\mathcal{F}_{jk}\{\XX^j,\xi\}\in^{akp}\XX^p+Q\mathcal{E}_j\in^{ajp}\{\XX^p,\xi\}+\mathcal{W}_j\in^{ajp}\{\VV^p,\xi\}, \\
- \m_{\L^a\lambda}&=Q\mathcal{F}_{jk}\partial_\lambda\XX^j\in^{akp}\XX^p+Q\mathcal{E}_j\in^{ajp}\partial_\lambda\XX^p+\mathcal{W}_j\in^{ajp}\partial_\lambda\VV^p,\\
 -\m_{\L^a\u^b}&=Q[\mathcal{F}_{bk}\widetilde{X}_1+\mathcal{F}_{jk}\widetilde{X}_3\in^{jcb}\L^c]\in^{akp}\XX^p\\
 &\qquad + Q\mathcal{E}_j[\in^{ajb}\widetilde{X}_1+\in^{ajp}\widetilde{X}_3\in^{pcb}\L^c]+ \mathcal{W}_j[\in^{ajb}\widetilde{V}_1+\in^{ajp}\widetilde{V}_3\in^{pcb}\L^c],\\
%  \m_{\L^a\L^b}&=F_{jk}\widetilde{X}_3\in^{jbd}\u^d\in^{akp}\XX^p+E_j\in^{ajp}\widetilde{X}_3\in^{pbd}\u^d+W_j\in^{ajp}\widetilde{V}_3\in^{pbd}\u^d\\
 -\m_{\L^a\L^b}&=Q\mathcal{F}_{jk}\widetilde{X}_3\in^{jbd}\u^d\in^{akp}\XX^p+Q\widetilde{X}_3(\delta^{ab}\mathcal{E}_j\u^j-\mathcal{E}_b\u^a)+\widetilde{V}_3(\delta^{ab}\mathcal{W}_j\u^j-\mathcal{W}_b\u^a).
\end{aligned}
\end{equation}
\end{itemize}

\medskip

\paragraph{The coefficients $\m'_{fg}$}
Inspecting \eqref{eq:deriv'} and comparing with \eqref{eq:deriv}, we see that the coefficients $\m'_{fg}$ can be read off from $\m_{fg}$ by ignoring the terms in $\W$ and replacing $Q\E_j$ by $Q\E_j-\dot{\W}_j$.

\subsubsection{Propagation in the far formulation}\label{sec:prop_far}
We now demonstrate how derivative control as in \eqref{eq:defP} can be propagated in the far region $\Omega_t^{far}=\{\aabs{\XX}\geq \ip{t}\}$. As it turns out, Poisson brackets with respect to $\xi, \u$ have more favorable bounds than those with respect to $\eta,\L$. We thus define
\begin{equation}\label{SplittingGUDer}
\begin{split}
 \mathcal{G}(\vartheta,\a;t)&:=\ww_\xi\aabs{\{\xi,\gamma\}}(\vartheta,\a;t)+\ww_\u\aabs{\{\u,\gamma\}}(\vartheta,\a;t),\\
  \mathcal{U}(\vartheta,\a;t)&:=\ww_\eta\aabs{\{\eta,\gamma\}}(\vartheta,\a;t)+\ww_\L\aabs{\{\L,\gamma\}}(\vartheta,\a;t),
  \end{split}
\end{equation}
for the weight functions $\ww_f$, $f\in \textnormal{SIC}$, defined in \eqref{eq:def_weights}. We then have:
\begin{lemma}\label{lem:farprop}
 Assume that $(\vartheta,\a)$ is such that for some $0\leq T_1< T_2$ we have $(\vartheta,\a)\in\mathcal{O}_t^{far}$ for $T_1\leq t\leq T_2$. Assume also that $\gamma$ is a solution of \eqref{eq:deriv2} defined on the time interval $[T_1,T_2]$ satisfying for some $0<\delta\ll 1$
 \begin{equation}
  \mathcal{D}[\gamma](\vartheta,\a;t)\lesssim\eps\ip{t}^\delta,
 \end{equation}
 and the decay estimate \eqref{eq:EFdecay0} holds, i.e.
 \begin{equation}\label{eq:EFdecay}
  \sqrt{t^2+\abs{\y}^2}\abs{\E(\y,t)}\lesssim \eps^2\ip{t}^{-1},\quad  \sup_{\y\in\R^3}\left[t^2+\abs{\y}^2\right]\abs{\F(\y,t)}\lesssim \eps^2\ln\ip{t}\cdot\ip{t}^{-1}.
 \end{equation}
 Then, with $\mathcal{G}$, $\mathcal{U}$ defined in \eqref{SplittingGUDer},
 \begin{equation}\label{eq:dtBG_far}
 \begin{aligned}
  &\frac{d}{dt} \mathcal{G}(\vartheta,\a;t)\lesssim \eps^2 \ip{t}^{-5/4}[\mathcal{G}(\vartheta,\a;t)+\mathcal{U}(\vartheta,\a;t)],\\
  &\frac{d}{dt} \mathcal{U}(\vartheta,\a;t)\lesssim \eps^2 \ln\ip{t}\cdot \ip{t}^{-1}\mathcal{G}(\vartheta,\a;t)+\eps^2 \ip{t}^{-5/4}\mathcal{U}(\vartheta,\a;t),
 \end{aligned} 
 \end{equation}
 and thus
 \begin{equation}
  \mathcal{G}(\vartheta,\a;t)\lesssim\eps,\qquad  \mathcal{U}(\vartheta,\a;t)\lesssim\eps+\eps^2 \ln^2(t/T_1).
 \end{equation}

\begin{proof}
 By construction, we have that
 \begin{equation}
  \frac{d}{dt}\ww_f\{f,\gamma\}=\sum_{f,g\in\textnormal{SIC}}\m_{fg}\frac{\ww_f}{\ww_g}\cdot\ww_g\{g,\gamma\},
 \end{equation}
 and to prove the claim it thus suffices to suitably bound the expressions $\m_{fg}\cdot (\ww_f/\ww_g)$. For this we will use the Poisson bracket bounds on $\XX$ and $\VV$ established in Section \ref{sec:PB_bounds}, in particular also the refinements in outgoing asymptotic actions (see Lemma \ref{lem:betterdl}), the decay estimates \eqref{eq:EFdecay} and the fact that $\aabs{\XX}\gtrsim\ip{t}$.
 
 From \eqref{eq:mult_xi}, \eqref{PBX1primeprime} and \eqref{eq:betterdl}, we have the bounds
 \begin{equation}
\begin{aligned}
 \abs{\m_{\xi\xi}}&\lesssim \abs{\F}\cdot\xi/q[\vert\XX\vert+tq\xi^{-1}]+ \abs{\E}\xi/q\cdot[1+t\xi\vert\XX\vert^{-2}]+\vert \W\vert\vert\XX\vert^{-2} \xi^2/q[1+tq\xi^{-1}\aabs{\XX}^{-1}]\\
 &\lesssim \left[\abs{\F}\vert\XX\vert+\abs{\E}\right]\left[\aabs{\XX}^{1/2}q^{-1/2}+t\aabs{\XX}^{-1}\right]+\abs{\W}\aabs{\XX}^{-1}\left[1+tq^{1/2}\aabs{\XX}^{-3/2}\right],\\
 \abs{\m_{\xi\eta}}&\lesssim \xi^4/q^2\left[\abs{\F}+\abs{\E}\xi^2/q\vert\XX\vert^{-2}\right]+\abs{\W}\xi^5/q^2\vert\XX\vert^{-3},\\
\abs{\m_{\xi\u^a}}&\lesssim \frac{\xi^2}{q}\left[\abs{\F}\vert\XX\vert+\abs{\E}+\xi\abs{\W}\vert\XX\vert^{-2}\right],\\
 \abs{\m_{\xi\L^a}}+\abs{\m_{\xi\lambda}}&\lesssim \frac{\xi^2}{q}\abs{\F}(\aabs{\widetilde{X}_3}+\aabs{\partial_\lambda\XX})+\abs{\E}\left(\frac{\xi^5}{q^3}\frac{1}{\aabs{\XX}^2}+\frac{\xi^3}{q^2}\aabs{\partial_\lambda\VV}\right)+\abs{\W}\left(\frac{\xi^4}{q^2}\frac{1}{\aabs{\XX}^3}+\frac{\xi^3}{q}\aabs{\partial_\lambda\XX}\aabs{\XX}^{-3}\right)\\\
 &\lesssim \frac{\xi^3}{q^2}\left(\abs{\F}+\abs{\E}\frac{\xi^2/q}{\aabs{\XX}^2}\right)+\abs{\W}\frac{\xi^4}{q^2}\aabs{\XX}^{-3},\\
\end{aligned}
\end{equation}
 so that by \eqref{eq:def_weights}, \eqref{eq:EFdecay} and $\aabs{\XX}\gtrsim\ip{t}$ there holds that
 \begin{equation}\label{eq:mxi_bds_far}
 \begin{alignedat}{2}
  \abs{\m_{\xi\xi}}&\lesssim \eps^2\ip{t}^{-\frac{5}{4}},\qquad &\abs{\m_{\xi\eta}}\frac{\ww_\xi}{\ww_\eta}&= \abs{\m_{\xi\eta}}\frac{(1+\xi)^2}{\xi^4}\lesssim\eps^2\ip{t}^{-\frac{5}{4}},\\
  \abs{\m_{\xi\L}}\frac{\ww_\xi}{\ww_\L}=\abs{\m_{\xi\L}}\frac{(1+\xi)^2}{\xi^3}&\lesssim\eps^2\ip{t}^{-\frac{5}{4}},\qquad &\abs{\m_{\xi\u}}\frac{\ww_\xi}{\ww_\u}&= \abs{\m_{\xi\u}}\frac{(1+\xi)}{\xi^2}\lesssim\eps^2\ip{t}^{-\frac54}.
 \end{alignedat} 
 \end{equation}
 
Similarly, from \eqref{eq:mult_eta} and using \eqref{eq:betterdlPBetaX} and \eqref{eq:PBxiX3}, we obtain that
\begin{equation}
\begin{aligned}
 \abs{\m_{\eta\xi}}%&\lesssim\abs{\F}\xi^{-2}(\vert\XX\vert+tq\xi^{-1})^2+\abs{\E}\aabs{\{\eta,\XX^j\},\eta\}}+\abs{\W}\vert\{\{\eta,\VV\},\eta\}\vert\\
 &\lesssim \abs{\F}\xi^{-2}[\aabs{\XX}^2+t^2q^2\xi^{-2}]+\abs{\E}\xi^{-2}[\vert\XX\vert+tq\xi^{-1}+t^2q\vert\XX\vert^{-2}]\\
 &\qquad +\abs{\W}q\xi^{-3}[1+t\xi\vert\XX\vert^{-2}+t^2q\vert\XX\vert^{-3}],\\
 \abs{\m_{\eta\eta}}&=\abs{\m_{\xi\xi}},\\
 \abs{\m_{\eta\u^a}}&\lesssim [\abs{\F}\aabs{\XX}+\abs{\E}]\xi^{-1}[\aabs{\XX}+tq\xi^{-1}]+\abs{\W}q\xi^{-2} (1+t\xi\vert\XX\vert^{-2}),\\
 \abs{\m_{\eta\L^a}}+\abs{\m_{\eta\lambda}}&\lesssim \abs{\F}[\aabs{\widetilde{X}_3}+\aabs{\partial_\lambda\XX}]\xi^{-1}[\aabs{\XX}+tq\xi^{-1}]+\abs{\E}\frac{1}{q}[1+\frac{t\xi}{\aabs{\XX}^2}]\\
 &\qquad +\abs{\W}\frac{1}{\aabs{\XX}^2}[\frac{t}{\aabs{\XX}}+\frac{\xi}{q}],
\end{aligned}
\end{equation}
implying that
 \begin{equation}\label{eq:meta_bds_far}
 \begin{alignedat}{2}
  \abs{\m_{\eta\xi}}\frac{\ww_\eta}{\ww_\xi}=\abs{\m_{\eta\xi}}\frac{\xi^4}{(1+\xi)^2}&\lesssim [\aabs{\XX}^2\aabs{\F(\XX)}+\aabs{\XX}\aabs{\E(\XX)}+\eps^2\ip{t}^{-1}],\qquad &\abs{\m_{\eta\eta}}&= \abs{\m_{\xi\xi}},\\
  \abs{\m_{\eta\u}}\frac{\ww_\eta}{\ww_\u}=\abs{\m_{\eta\u}}\frac{\xi^2}{1+\xi}&\lesssim [\aabs{\XX}^2\aabs{\F(\XX)}+\aabs{\XX}\aabs{\E(\XX)}+\eps^2\ip{t}^{-1}],\qquad &\abs{\m_{\eta\L}}\frac{\ww_\eta}{\ww_\L}&=\abs{\m_{\eta\L}}\xi&\lesssim\eps^2\ip{t}^{-\frac32},
 \end{alignedat} 
 \end{equation}
which suffices by \eqref{eq:EFdecay}.

For Poisson brackets with $\u$, we note that by \eqref{eq:PBuX} and \eqref{eq:betterdl}, there holds that 
\begin{equation}
 \begin{split}
 \aabs{\{\u^k,\XX^j\}}\lesssim [\aabs{\partial_\lambda\XX}+\aabs{\widetilde{X}_3}]\lesssim \frac{\xi}{q}&,\qquad \aabs{\{\u^k,\VV^j\}}\lesssim \frac{\xi^2}{q}\frac{1}{\aabs{\XX}^2},\\
 \aabs{\partial_\lambda\{\u^k,\XX^j\}}\lesssim\frac{1}{q}\frac{\xi^2/q}{\vert\XX\vert}&,\qquad\aabs{\partial_\lambda\{\u^k,\VV^j\}}\lesssim\frac{\xi/q}{\vert\XX\vert^2},
 \end{split}
\end{equation}
and with \eqref{eq:betterdl} and \eqref{eq:PBxiX3} it follows that
\begin{equation}
 \aabs{\{\{\u^k,\XX^j\},\xi\}}\lesssim\aabs{\partial_\lambda\{\XX,\xi\}}+\aabs{\{\widetilde{X}_3,\xi\}}\lesssim\frac{\xi^5}{q^3}\frac{1}{\aabs{\XX}^2},\qquad \aabs{\{\{\u^k,\VV^j\},\xi\}}\lesssim \frac{\xi^2}{q}\frac{1}{\aabs{\XX}^2}.
\end{equation}
Furthermore, using \eqref{eq:PBuX}, \eqref{eq:betterdlPBetaX} and \eqref{eq:PBxiV3}-\eqref{eq:PBxiX3},
\begin{equation}
\begin{split}
 \aabs{\{\{\u^a,\XX^j\},\eta\}}\lesssim q^{-1}(1+\frac{t\xi}{\aabs{\XX}^2}),\qquad  \aabs{\{\{\u^a,\VV^j\},\eta\}}\lesssim \frac{1}{\vert\XX\vert^2}\left(\frac{\xi}{q}+\frac{t}{\vert\XX\vert}\right).
 \end{split}
\end{equation}
We deduce from \eqref{eq:mult_u} that
\begin{equation}
\begin{aligned}
 \abs{\m_{\u^a\xi}}&\lesssim q^{-1}\left[\abs{\F}\vert\XX\vert+\vert\mathcal{E}\vert\right][1+tq/(\xi\aabs{\XX})]+\abs{\W}\frac{1}{\aabs{\XX}^2}\left(\frac{\xi}{q}+\frac{t}{\aabs{\XX}}\right),\\
 \abs{\m_{\u^a\eta}}&\lesssim \frac{\xi^3}{q}\frac{1}{\vert\XX\vert}\left[\abs{\F}\vert\XX\vert+\abs{\E}\right]+\abs{\W}\frac{\xi^2}{q}\frac{1}{\aabs{\XX}^2},\\
 \abs{\m_{\u^a\u^b}}&\lesssim \frac{\xi}{q}[\abs{\F}\aabs{\XX}+\abs{\E}]+\abs{\W}\frac{\xi^2/q}{\aabs{\XX}^2},\\
 \abs{\m_{\u^a\lambda}}+ \abs{\m_{\u^a\L^b}}&\lesssim \frac{\xi^2}{q^2}\frac{1}{\vert\XX\vert}\left[\abs{\F}\vert\XX\vert+\abs{\E}\right]+\abs{\W}\frac{\xi}{q}\frac{1}{\aabs{\XX}^2},
\end{aligned}
\end{equation}
and hence
\begin{equation}\label{eq:mu_bds_far}
\begin{alignedat}{2}
 \abs{\m_{\u\xi}}\frac{\ww_\u}{\ww_\xi}=\abs{\m_{\u\xi}}\frac{\xi^2}{1+\xi}&\lesssim \eps^2\ip{t}^{-\frac{5}{4}},\qquad &\abs{\m_{\u\eta}}\frac{\ww_\u}{\ww_\eta}=\abs{\m_{\u\eta}}\frac{1+\xi}{\xi^2}&\lesssim \eps^2\ip{t}^{-\frac{5}{4}},\\
 \abs{\m_{\u\u}}&\lesssim \eps^2\ip{t}^{-\frac{5}{4}},\qquad &\abs{\m_{\u\L}}\frac{\ww_\u}{\ww_\L}=\abs{\m_{\u\L}}\frac{1+\xi}{\xi}&\lesssim\eps^2\ip{t}^{-\frac32}.
\end{alignedat} 
\end{equation}

Finally, from \eqref{eq:mult_L} we have the bounds
\begin{equation}
\begin{aligned}
 \abs{\m_{\L^a\xi}}&\lesssim \xi^{-1}[\abs{\F}\aabs{\XX}+\abs{\E}][\aabs{\XX}+tq\xi^{-1}]+\abs{\W}q\xi^{-2}[1+t\xi\aabs{\XX}^{-2}],\\
 \abs{\m_{\L^a\eta}}&\lesssim \frac{\xi^2}{q}[\abs{\F}\aabs{\XX}+\abs{\E}]+\abs{\W}\frac{\xi^3}{q}\frac{1}{\aabs{\XX}^2},\\
 \abs{\m_{\L^a\u^b}}&\lesssim [\abs{\F}\aabs{\XX}+\abs{\E}]\aabs{\XX}+\abs{\W}\frac{q}{\xi},\\
 \abs{\m_{\L^a\lambda}}+ \abs{\m_{\L^a\L^b}}&\lesssim [\abs{F}\aabs{\XX}+\abs{\E}]\frac{\xi}{q}+\abs{\W}\frac{\xi^2}{q}\frac{1}{\aabs{\XX}^2},\\
\end{aligned}
\end{equation}
from which we see that
\begin{equation}\label{eq:mL_bds_far}
\begin{alignedat}{2}
 \abs{\m_{\L\xi}}\frac{\ww_\L}{\ww_\xi}=\abs{\m_{\L\xi}}\frac{\xi^3}{(1+\xi)^2}&\lesssim [\aabs{\XX}^2\aabs{\F(\XX)}+\aabs{\XX}\aabs{\E(\XX)}+\eps^2\ip{t}^{-1}],\qquad &\abs{\m_{\L\eta}}\frac{\ww_\L}{\ww_\eta}= \abs{\m_{\L\eta}}\frac{1}{\xi}&\lesssim\eps^2\ip{t}^{-\frac54},\\
 \abs{\m_{\L\u}}\frac{\ww_\L}{\ww_\u}=\abs{\m_{\L\u}}\frac{\xi}{1+\xi}&\lesssim [\aabs{\XX}^2\aabs{\F(\XX)}+\aabs{\XX}\aabs{\E(\XX)}+\eps^2\ip{t}^{-1}],\qquad &\abs{\m_{\L\L}}&\lesssim\eps^2\ip{t}^{-\frac54}.
\end{alignedat} 
\end{equation}

In conclusion, the first line of \eqref{eq:dtBG_far} follows by combining \eqref{eq:mxi_bds_far} and \eqref{eq:mu_bds_far}, while the second line of \eqref{eq:dtBG_far} follows from \eqref{eq:meta_bds_far} and \eqref{eq:mL_bds_far}.
\end{proof}
 
\end{lemma}

\subsubsection{Propagation in the close formulation}\label{sec:prop_cl}
In contrast to the far region, in the close region $\Omega_t^{cl}=\{\aabs{\XX}\leq  10\ip{t}\}$ the spatial location of a trajectory $\X(t)$ may be comparatively small. As a consequence, for the Gr\"onwall arguments for derivative propagation we need some refined estimates on the dynamics:
\begin{lemma}\label{lem:Xintegral}
 Let $\X(t)$ be a trajectory of the ODE \eqref{ODE}. Then for any $t\geq0$ we have that
 \begin{equation}\label{eq:Xintegral}
  \int_{-\infty}^{\infty} \frac{d\tau}{\aabs{\X(\tau)}^2}\lesssim \xi^{-1}.
 \end{equation}
\end{lemma}
\begin{proof}
 We recall that by \eqref{eq:virials} there holds that
 \begin{equation}
  \frac{d^2}{dt^2}\aabs{\X(t)}^2\gtrsim a^2\geq0.
 \end{equation}
 Denoting by $t_p$ the time of periapsis, we have with $\aabs{\X(t_p)}\gtrsim \frac{q}{a^2}$ that
 \begin{equation}
  \aabs{\X(t)}^2\gtrsim \frac{q^2}{a^4}+a^2(t-t_p)^2=a^{-4}(q^2+a^6(t-t_p)^2).
 \end{equation}
 Since the trajectories of \eqref{ODE} (and thus the integral in \eqref{eq:Xintegral}) are symmetric with respect to periapsis, we may assume without loss of generality that $t_p=0$. We then have that
 \begin{equation}
  \int_{-\infty}^{\infty} \frac{d\tau}{\aabs{\X(\tau)}^2}=2\int_{0}^{\infty} \frac{d\tau}{\aabs{\X(\tau)}^2}\lesssim a^4\int_0^{q/a^{3}}\frac{d\tau}{q^2}+a^4\int_{q/a^{3}}^\infty\frac{d\tau}{a^6(t-t_p)^2}\lesssim \frac{a}{q}.
 \end{equation}
\end{proof}

Also in the close region Poisson brackets with respect to $\xi, \u$ have more favorable bounds than those with respect to $\eta,\L$, and, as in \eqref{SplittingGUDer}, we define
\begin{equation*}
\begin{split}
 \mathcal{G}'(\vartheta,\a;t)&:=\ww_\xi\aabs{\{\xi,\gamma'\}}(\vartheta,\a;t)+\ww_\u\aabs{\{\u,\gamma'\}}(\vartheta,\a;t),\\
 \mathcal{U}'(\vartheta,\a;t)&:=\ww_\eta\aabs{\{\eta,\gamma'\}}(\vartheta,\a;t)+\ww_\L\aabs{\{\L,\gamma'\}}(\vartheta,\a;t),
 \end{split}
\end{equation*}
for the weight functions $\ww_f$, $f\in \textnormal{SIC}$, defined in \eqref{eq:def_weights}. We then have
\begin{lemma}\label{lem:clprop}
 Assume that $(\vartheta,\a)$ is such that for some $T_2>T_1\geq 0$ we have $(\vartheta,\a)\in\Omega_t^{cl}$ for $T_1\leq t\leq T_2$. Assume also that $\gamma'$ is a solution of \eqref{eq:deriv'2} defined on the time interval $[T_1,T_2]$ satisfying for some $0<\delta\ll 1$
 \begin{equation}
  \mathcal{D}[\gamma'](\vartheta,\a;t)\lesssim\eps\ip{t}^\delta,
 \end{equation}
 and we have the decay estimates \eqref{eq:EFdecay}.
 Then
 \begin{equation}\label{eq:B'G'_bds}
  \mathcal{G}'(\vartheta,\a;t)\lesssim\eps,\qquad  \mathcal{U}'(\vartheta,\a;t)\lesssim\eps+\eps^2 \ln^2(t/T_1).
 \end{equation}
\end{lemma}
\begin{proof}

 As note at the end of Section \ref{ssec:mfg}, the coefficients $\m_{fg}'$ can be deduced easily from $\m_{fg}$, and thus from the proof of Lemma \ref{lem:farprop} we directly have the bounds
\begin{equation}
\begin{aligned}
 \abs{\m'_{\xi\xi}}&\lesssim \abs{\F}\cdot\xi/q[\vert\XX\vert+tq\xi^{-1}]+ (\abs{\E}+\aabs{\dot\W})\xi/q\cdot[1+t\xi\vert\XX\vert^{-2}],\\
 \abs{\m'_{\xi\eta}}&=\abs{\F}\xi^4/q^2+(\abs{\E}+\aabs{\dot\W})\xi^2/q^2(\xi^4/q\vert\XX\vert^{-2}),\\
  \abs{\m'_{\xi\u^a}}&\lesssim \frac{\xi^2}{q}\left[\abs{\F}\vert\XX\vert+\abs{\E}+\aabs{\dot\W}\right],\\
 \abs{\m'_{\xi\lambda}}+ \abs{\m'_{\xi\L^a}} &\lesssim \frac{\xi^3}{q^2}\abs{\F}+(\abs{\E}+\aabs{\dot\W})\frac{\xi^5}{q^3}\frac{1}{\aabs{\XX}^2},\\
\end{aligned}
\end{equation}
as well as
\begin{equation}
\begin{aligned}
 \abs{\m'_{\eta\xi}} &\lesssim \abs{\F}\xi^{-2}[\aabs{\XX}^2+t^2q^2\xi^{-2}]+(\abs{\E}+\aabs{\dot\W})\xi^{-2}[\vert\XX\vert+tq\xi^{-1}+t^2q\vert\XX\vert^{-2}],\\
 \abs{\m'_{\eta\eta}}&=\abs{\m'_{\xi\xi}},\\
 \abs{\m'_{\eta\u^a}}&\lesssim [\abs{\F}\aabs{\XX}+\abs{\E}+\aabs{\dot\W}]\xi^{-1}[\aabs{\XX}+tq\xi^{-1}],\\
 \abs{\m'_{\eta\lambda}}+ \abs{\m'_{\eta\L^a}}&\lesssim \abs{\F}[\aabs{\widetilde{X}_3}+\aabs{\partial_\lambda\XX}]\xi^{-1}[\aabs{\XX}+tq\xi^{-1}]+(\abs{\E}+\aabs{\dot\W})\frac{1}{q}[1+t\xi\aabs{\XX}^{-2}],
\end{aligned}
\end{equation}
and
\begin{equation}
\begin{aligned}
 \abs{\m'_{\u^a\xi}}&\lesssim q^{-1}\left[\abs{\F}\aabs{\XX}+\vert\mathcal{E}\vert+\vert\dot{\mathcal{W}}\vert\right]\cdot \left[1+tq/(\xi\vert\XX\vert)\right],\\
 \abs{\m'_{\u^a\eta}}&\lesssim \frac{\xi^3}{q^2}\abs{\F}+(\abs{\E}+\aabs{\dot\W})\frac{\xi^5}{q^3}\frac{1}{\aabs{\XX}^2},\\
 \abs{\m'_{\u^a\u^b}}&\lesssim \frac{\xi}{q}[\abs{\F}\aabs{\XX}+\abs{\E}+\aabs{\dot\W}],\\
 \abs{\m'_{\u^a\lambda}}+ \abs{\m'_{\u^a\L^b}}&\lesssim \frac{\xi^2}{q^2}\left[\abs{\F}+[\abs{\E}+\aabs{\dot\W}]\cdot\aabs{\XX}^{-1}\right],
\end{aligned}
\end{equation}
and also
\begin{equation}
\begin{aligned}
 \abs{\m'_{\L^a\xi}}&\lesssim \xi^{-1}[\abs{\F}\aabs{\XX}+\abs{\E}+\aabs{\dot\W}][\aabs{\XX}+tq\xi^{-1}],\\
 \abs{\m'_{\L^a\eta}}&\lesssim \frac{\xi^2}{q}[\abs{\F}\aabs{\XX}+\abs{\E}+\aabs{\dot\W}],\\
 \abs{\m'_{\L^a\u^b}}&\lesssim [\abs{\F}\aabs{\XX}+\abs{\E}+\aabs{\dot\W}]\aabs{\XX},\\
 \abs{\m'_{\L^a\lambda}}+\abs{\m'_{\L^a\L^b}}&\lesssim \frac{\xi}{q}[\abs{F}\aabs{\XX}+\abs{\E}+\aabs{\dot\W}].\\
\end{aligned}
\end{equation}

Observing that in the close region we have $\xi^2/q\lesssim\aabs{\XX}\lesssim \ip{t}$, it follows that
 \begin{equation}\label{eq:mxi_bds_cl}
 \begin{aligned}
  \abs{\m'_{\xi\xi}}&\lesssim \abs{\F}\ip{t}^{\frac{3}{2}}+(\abs{\E}+\aabs{\dot\W})\bigg(q^{-\frac{1}{2}}\ip{t}^{\frac{1}{2}}+\frac{t\xi^2/q}{\aabs{\XX}^2}\bigg)\lesssim \eps^2\ip{t}^{-\frac{5}{4}}+\eps^2\ip{t}^{-1}\frac{\xi^2/q}{\aabs{\XX}^2},\\
  \abs{\m'_{\xi\eta}}\frac{\ww_\xi}{\ww_\eta}&= \abs{\m'_{\xi\eta}}\frac{(1+\xi)^2}{\xi^4}\lesssim \abs{\F}(1+\xi)^2+(\abs{\E}+\aabs{\dot\W})\frac{(1+\xi)^2\xi^2}{\aabs{\XX}^2}\lesssim\eps^2\ip{t}^{-\frac32}+\eps^2\ip{t}^{-2}\frac{\xi^2/q}{\aabs{\XX}^2},\\
  \abs{\m'_{\xi\L}}\frac{\ww_\xi}{\ww_\L}&=\abs{\m'_{\xi\L}}\frac{(1+\xi)^2}{\xi^3}\lesssim \abs{\F}(1+\aabs{\XX})+\frac{(1+\xi)^2\xi^2/q}{\vert\XX\vert^2}(\abs{\E}+\aabs{\dot\W})\lesssim\eps^2\ip{t}^{-\frac32}+\eps^2\ip{t}^{-2}\frac{\xi^2/q}{\aabs{\XX}^2},\\
  \abs{\m'_{\xi\u}}\frac{\ww_\xi}{\ww_\u}&= \abs{\m'_{\xi\u}}\frac{(1+\xi)}{\xi^2}\lesssim(1+\xi)[\abs{\F}\aabs{\XX}+\abs{\E}+\aabs{\dot\W}]\lesssim\eps^2\ip{t}^{-\frac54},
 \end{aligned} 
 \end{equation}
 and this suffices, since by Lemma \ref{lem:Xintegral} and $\xi^2/q\lesssim\aabs{\XX}\lesssim\ip{t}$ we have that
 \begin{equation}
  \int_0^t\ip{\tau}^{-1}\frac{\xi^2/q}{\aabs{\X(\vartheta+\tau\a,\a)}^2}d\tau\lesssim q^{-1}\min\{\xi,\xi^{-\frac{1}{2}}\}\lesssim_q 1.
 \end{equation}
Similarly, there holds that
\begin{equation}\label{eq:mu_bds_cl}
\begin{alignedat}{2}
 \abs{\m'_{\u\xi}}\frac{\ww_\u}{\ww_\xi}=\abs{\m'_{\u\xi}}\frac{\xi^2}{1+\xi}&\lesssim \eps^2\langle t\rangle^{-1}\Big(\ip{t}^{-\frac{1}{4}}+\frac{\xi^\frac{1}{2}}{\vert\XX\vert}\Big),\\
  \abs{\m'_{\u\eta}}\frac{\ww_\u}{\ww_\eta}=\abs{\m'_{\u\eta}}\frac{1+\xi}{\xi^2}&\lesssim \eps^2\ip{t}^{-\frac{5}{4}}\Big(1+\frac{\xi^\frac{1}{2}}{\vert\XX\vert}\Big),\\
 \abs{\m'_{\u\u}}&\lesssim \eps^2\ip{t}^{-\frac{5}{4}},\\
 \abs{\m'_{\u\L}}\frac{\ww_\u}{\ww_\L}=\abs{\m'_{\u\L}}\frac{1+\xi}{\xi}&\lesssim\eps^2\ip{t}^{-\frac{5}{4}}
 \Big(1+\frac{\xi^\frac{1}{2}}{\vert\XX\vert}\Big),
\end{alignedat} 
\end{equation}
and hence the bounds for $\mathcal{G}'$ in \eqref{eq:B'G'_bds} are established.

Turning to $\mathcal{U}'$, we have that
\begin{equation}\label{eq:meta_bds_cl}
\begin{aligned}
 \abs{\m'_{\eta\xi}}\frac{\ww_\eta}{\ww_\xi}&=\abs{\m'_{\eta\xi}}\frac{\xi^4}{(1+\xi)^2}
 %\lesssim \aabs{\XX}^2\aabs{\F(\XX)}+(\aabs{\E(\XX)}+\aabs{\dot\W})(t+\frac{\xi^2}{(1+\xi)^2}\frac{t^2}{\aabs{\XX}^2})\\
 \lesssim \eps^2\Big(\ip{t}^{-1}\ln\ip{t}+\frac{\xi}{\aabs{\XX}^2}\Big),\\
 \abs{\m'_{\eta\eta}}&= \abs{\m'_{\xi\xi}},\\
 \abs{\m'_{\eta\u}}\frac{\ww_\eta}{\ww_\u}&=\abs{\m'_{\eta\u}}\frac{\xi^2}{1+\xi}\lesssim \left[\aabs{\F(\XX)}\aabs{\XX}+\aabs{\E(\XX)}+\aabs{\dot\W}\right](\aabs{\XX}+t)\lesssim \eps^2\ip{t}^{-1}\ln\ip{t},\\
 \abs{\m'_{\eta\L}}\frac{\ww_\eta}{\ww_\L}&=\abs{\m'_{\eta\L}}\xi\lesssim\varepsilon^2\langle t\rangle^{-1}\Big(\langle t\rangle^{-\frac{1}{4}}+\frac{q}{\vert\XX\vert^2}\Big),
\end{aligned} 
\end{equation}
and since in the $\Omega_t^{cl}$ we have $\frac{\xi}{q}\lesssim\ip{t}^{\frac12}$, there holds that
\begin{equation}\label{eq:mL_bds_cl}
\begin{aligned}
 \abs{\m'_{\L\xi}}\frac{\ww_\L}{\ww_\xi}&=\abs{\m'_{\L\xi}}\frac{\xi^3}{(1+\xi)^2}\lesssim \left[\aabs{\F(\XX)}\aabs{\XX}+\aabs{\E(\XX)}+\aabs{\dot\W}\right](\aabs{\XX}+t)\lesssim \eps^2\ip{t}^{-1}\ln\ip{t},\\
 \abs{\m'_{\L\eta}}\frac{\ww_\L}{\ww_\eta}&= \abs{\m'_{\L\eta}}\frac{1}{\xi}\lesssim\eps^2\ip{t}^{-\frac54},\\
 \abs{\m'_{\L\u}}\frac{\ww_\L}{\ww_\u}&=\abs{\m'_{\L\u}}\frac{\xi}{1+\xi}\lesssim \left[\aabs{\XX}\aabs{\F(\XX)}+\aabs{\E(\XX)}+\aabs{\dot\W}\right]\aabs{\XX}\lesssim\eps^2\ip{t}^{-1}\ln\ip{t},\\
 \abs{\m'_{\L\L}}&\lesssim\left[\aabs{\F(\XX)}\aabs{\XX}+\aabs{\E(\XX)}+\aabs{\dot\W}\right]\ip{t}^{\frac12}\lesssim\eps^2\ip{t}^{-\frac54}.
\end{aligned} 
\end{equation}
As above, the terms involving $\aabs{\XX}^{-2}$ will give bounded contributions by Lemma \ref{lem:Xintegral}, and we thus obtain the bounds for $\mathcal{U}'$ in \eqref{eq:B'G'_bds}.
\end{proof}

\section{Main theorem and asymptotic behavior}\label{sec:main-asympt}

We are now ready to state and prove our main theorem which asserts global existence of small perturbations, optimal decay of the electric field and convergence of the particle distribution along modified linear trajectories. In particular, this implies Theorem \ref{thm:main_rough}.

\begin{theorem}\label{MainThm}

There exists $\varepsilon_\ast>0$ such that the following holds for all $0<\varepsilon_0<\varepsilon_\ast$. 
\begin{enumerate}[wide]
\item Given $\gamma_0\in L^2\cap L^\infty$ satisfying the smallness assumptions \eqref{eq:mom_id} for moments as well as the derivative bounds
\begin{equation}\label{eq:efield_optimal}
%  \nnorm{(\xi^4+\xi^{-6})(\xi^4+\xi^{-4})\mathcal{D}[\gamma_0]}_{L^\infty_{\vartheta,\a}}=
 \nnorm{(\xi^{5}+\xi^{-8})\sum_{f\in \textnormal{SIC}}\ww_f\aabs{\{f,\gamma_0\}}}_{L^\infty_{\vartheta,\a}}\lesssim\eps_0,
\end{equation}
there exists a unique, global $C^1$ solution $\gamma$ of \eqref{NewNLEq} satisfying the moment bounds \eqref{eq:gl_mom_bds} of Theorem \ref{thm:global_moments} as well as the slow growth of derivatives
\begin{equation}\label{eq:derivs_MainThm}
 \nnorm{\mathcal{D}(t)}_{L^\infty_{\vartheta,\a}}\lesssim \eps_0\ln^2\ip{t}
\end{equation}
with $\mathcal{D}(t)$ as in Proposition \eqref{prop:global_derivs}. The associated electric field decays at an optimal rate,
\begin{equation}\label{OptimalDecayEFMainThm}
\left[t^2+\vert\y\vert^2\right]\vert\mathcal{E}({\bf y},t)\vert\lesssim\varepsilon_0^2,
\end{equation}
and there exist an asymptotic field $\mathcal{E}^\infty$ that converges along linearized trajectories
\begin{equation}\label{ConvergenceEFMainThm}
\begin{split}
\Vert \langle \a\rangle^2(t^2\mathcal{E}(t\a,t)-\mathcal{E}^\infty({\bf a}))\Vert_{L^\infty(\mathcal{P}_{\vartheta,\a})}&\lesssim\varepsilon_0^4\langle t\rangle^{-1/10},
\end{split}
\end{equation}
and a limit gas distribution $\gamma_\infty$ that converges along a modified scattering dynamic
\begin{equation}\label{PDFConverges}
\Vert \gamma(\vartheta+\ln(t)\left[Q\mathcal{E}^\infty(\a)+\mathcal{Q}\mathcal{E}^\infty(0)\right],\a,t)-\gamma_\infty(\vartheta,\a)\Vert_{L^\infty_{\vartheta,\a}\cap L^2_{\vartheta,\a}}\lesssim \varepsilon_0\ip{t}^{-1/30}.
\end{equation}

\item As a consequence, given any $(\mathcal{X}_0,\mathcal{V}_0)\in\mathbb{R}^3_{\bf x}\times\mathbb{R}^3_{\bf v}$ and $\mu_0\in L^2\cap L^\infty$ such that the function
\begin{equation}
 \gamma_0(\vartheta,\a):=\mu_0(\X(\vartheta,\a)+\mathcal{X}_0,\V(\vartheta,\a)+\mathcal{V}_0)
\end{equation}
satisfies the assumptions of part $(1)$, there exists a unique, global $C^1$ solution of \eqref{VPPC} with $(\mathcal{X}(0),\mathcal{V}(0))=(\mathcal{X}_0,\mathcal{V}_0)$. The associated electric field decays at the optimal rate \eqref{eq:efield_optimal} and the point charge has an almost free asymptotic dynamic: There exist $(\mathcal{X_\infty},\mathcal{V}_\infty)\in\R^3_\x\times\R^3_\v$ such that
\begin{equation}\label{AsymptoticsXV}
\begin{split}
\mathcal{X}(t)=\mathcal{X}_\infty+\mathcal{V}_\infty t-\mathcal{Q}\mathcal{E}^\infty(0)\ln(t)+O(t^{-1/10}),\qquad\mathcal{V}(t)=\mathcal{V}_\infty-\frac{\mathcal{Q}}{t}\mathcal{E}^\infty(0)+O(t^{-1/10}).
\end{split}
\end{equation}
Moreover, $\mu(t)$ converges to $\gamma_\infty$ as $t\to +\infty$
\begin{equation}\label{eq:mu_mod_scat}
 \mu({\bf Y}({\bf x},{\bf v},t),{\bf W}({\bf x},{\bf v},t), t)\to \gamma_\infty({\bf x},{\bf v}),\qquad t\to\infty,
\end{equation}
along the modified trajectories
\begin{equation}\label{eq:mu_mod_traj}
\begin{aligned}
 {\bf Y}({\bf x},{\bf v},t)&=\left(at-\frac{1}{2}\frac{q}{a^2}\ln(ta^3/q)+\ln(t)[Q\mathcal{E}^\infty(\a)+\mathcal{Q}\mathcal{E}^\infty(0)]\right)\cdot \frac{q^2}{4a^2L^2+q^2}\left(\frac{2}{q}{\bf R}+\frac{4a}{q^2}{\bf L}\times{\bf R}\right)\\
 &\qquad+
 \mathcal{V}_\infty t+\mathcal{Q}\mathcal{E}^\infty(0)\ln(t)+O(1),\\
{\bf W}({\bf x},{\bf v},t)&=a\left(1-\frac{q}{2ta^3}\right)\cdot \frac{q^2}{4a^2L^2+q^2}\left(\frac{2}{q}{\bf R}+\frac{4a}{q^2}{\bf L}\times{\bf R}\right)+\mathcal{V}_\infty+O(\ln(t)t^{-2}).
\end{aligned}
\end{equation}

\end{enumerate}

\end{theorem}

\subsection{Proof of Theorem \ref{MainThm}}

\begin{proof}[Proof of Part $(1)$]
% \subsubsection{Bootstrap assumptions}
In view of the hypothesis, from Theorem \ref{thm:global_moments}  we obtain a global solution $\gamma(t)$ with the required moment bounds, and Propostion \ref{prop:global_derivs} provides the claimed derivative control \eqref{eq:derivs_MainThm}.
% \subsubsection{Convergence of the effective electric field and potential} 
Optimal decay \eqref{eq:efield_optimal} and the convergence \eqref{ConvergenceEFMainThm} for the electric field then follow from $(ii)$ and $(iii)$ of Proposition \ref{PropEeff}, respectively.

\medskip
We deduce next the motion of the point charge. Starting from the fact that
\begin{equation*}
\begin{split}
\dot{\mathcal{W}}(t)=\mathcal{Q}\mathcal{E}(0,t),\qquad\mathcal{W}(t)\to_{t\to\infty}0,
\end{split}
\end{equation*}
we obtain  $\mathcal{V}_\infty\in\R^3$ for which
\begin{equation}
 \dot{\mathcal{X}}=\mathcal{V}_\infty+\mathcal{W}(t),
\end{equation}
and from \eqref{ConvergenceEFMainThm} and integration we deduce that there exists $\mathcal{X}_\infty$ such that
\begin{equation}\label{AsymptoticXVCCL}
\begin{split}
\vert\mathcal{W}(t)+\frac{\mathcal{Q}}{t}\mathcal{E}^\infty(0)\vert&\le \varepsilon_0^4\ip{t}^{-11/10},\qquad\vert\mathcal{X}(t)-\mathcal{V}_\infty t+\mathcal{Q}\mathcal{E}^\infty(0)\ln(t)-\mathcal{X}_\infty\vert\le\varepsilon_0^4\ip{t}^{-1/10}.
\end{split}
\end{equation}
This gives \eqref{AsymptoticsXV}, and informs the derivation of the modified scattering dynamic for the gas distribution, which we give next.

Since we have convergence of the electric field, we can now integrate the leading order evolution of the gas distribution. We do this in the far formulation of the equations. Motivated by the above convergence and the expression for the Hamiltonian \eqref{NewNLEq}, we introduce the asymptotic Hamiltonian
\begin{equation*}
\begin{split}
\mathbb{H}_4^\infty:=\frac{1}{t}\left[Q\Psi^\infty(\a)-\mathcal{Q}\mathcal{E}^\infty(0)\cdot\a\right].
\end{split}
\end{equation*}
The flow of $\mathbb{H}_4^\infty$ can be integrated explicitly: its characteristics simply \emph{shear} with the diffeomorphism
\begin{equation}
 s_\infty:(\vartheta,\a)\mapsto(\vartheta-\ln(t)\left[Q\mathcal{E}^\infty(\a)+\mathcal{Q}\mathcal{E}^\infty(0)\right],\a).
\end{equation}
To account for this dynamic, we thus introduce
\begin{equation*}
\begin{split}
\sigma:=\gamma\circ s_\infty^{-1},\qquad \sigma(\vartheta,\a,t):=\gamma(\vartheta+\ln(t)\left[Q\mathcal{E}^\infty(\a)+\mathcal{Q}\mathcal{E}^\infty(0)\right],\a,t),
\end{split}
\end{equation*}
which satisfies
\begin{equation}\label{eq:dtsigma}
\begin{aligned}
\partial_t\sigma&=(\partial_t\gamma+\{\mathbb{H}_4^\infty,\gamma\})\circ s_\infty^{-1}\\
&=\left(-\{\mathbb{H}_4,\gamma\}+\{\mathbb{H}_4^\infty,\gamma\}\right)\circ s_\infty^{-1}\\
&=\left(-Q\mathcal{E}_j(\XX,t)\{\XX^j,\gamma\}-\W_j(t)\{\VV^j,\gamma\}+t^{-1}(Q\mathcal{E}_j^\infty(a)-\mathcal{Q}\mathcal{E}_j^\infty(0))\{\a^j,\gamma\}\right)\circ s_\infty^{-1} \\
%&=-Q\mathcal{E}_j\{\widetilde{\bf X}^j-t\a^j,\gamma\}-\mathcal{W}_j\{\widetilde{V}^j-\a^j,\gamma\}+\frac{1}{t}\left[Q(\mathcal{E}^\infty_j(\a)-t^2\mathcal{E}_j(\widetilde{\bf X},t))+\mathcal{Q}(\mathcal{E}^\infty(0)-t\mathcal{W}_j(t))\right]\{\a^j,\gamma\}\\
&=-Q\mathcal{E}_j(\XX,t)\{\widetilde{\bf Z}^j,\gamma\}\circ s_\infty^{-1}-(tQ\mathcal{E}_j(\XX,t)+\mathcal{W}_j(t))\{\VV^j-\a^j,\gamma\}\circ s_\infty^{-1}\\
&\quad +\left(t^{-1}Q(\mathcal{E}^\infty_j(\a)-t^2\mathcal{E}_j(\widetilde{\bf X},t))+(\mathcal{Q}t^{-1}\mathcal{E}^\infty(0)+\mathcal{W}_j(t))\right)\{\a^j,\gamma\}\circ s_\infty^{-1}.
\end{aligned}
\end{equation}
In the bulk, these terms can be bounded as follows: By the bounds \eqref{PrecisedBoundZBracketBulk} in Lemma \ref{ZPBFormula} and \eqref{OptimalDecayEFMainThm}, we have that
\begin{equation}
\begin{aligned}
 \aabs{\mathfrak{1}_{\mathcal{B}}Q\mathcal{E}_j(\XX,t)\{\widetilde{\bf Z}^j,\gamma\}}+\aabs{\mathfrak{1}_{\mathcal{B}}(tQ\mathcal{E}_j(\XX,t)+\mathcal{W}_j(t))\{\VV^j-\a^j,\gamma\}}&\lesssim \eps_0^2\ip{t}^{-2}\cdot \ip{\xi}^3\ln\ip{t}\mathfrak{1}_{\mathcal{B}}\aabs{\mathcal{D}[\gamma](t)}\\
%  &\lesssim \eps_0^3\ip{t}^{-2}(\xi^4+\xi^{-2})(1+\ip{t}\aabs{\XX}^{-1})\ip{\ln\ip{t}}^3,
 &\lesssim \eps_0^3\ip{t}^{-2}\ln^3\ip{t},
\end{aligned} 
\end{equation}
where in the last inequality we have used \eqref{eq:PBcl_byfar_bulk} to control the contribution from the derivatives of $\gamma$ in the close region. Similarly, using also \eqref{AsymptoticEFCCL} and \eqref{AsymptoticXVCCL} we obtain that
\begin{equation}
\begin{aligned}
 &\mathfrak{1}_{\mathcal{B}}\abs{t^{-1}Q(\mathcal{E}^\infty_j(\a)-t^2\mathcal{E}_j(\widetilde{\bf X},t))+(\mathcal{Q}t^{-1}\mathcal{E}^\infty(0)+\mathcal{W}_j(t))}\aabs{\{\a^j,\gamma\}}\lesssim \eps_0^4\ip{t}^{-11/10}\cdot\eps_0\ln^2\ip{t},
\end{aligned}
\end{equation}
where we used \eqref{eq:PBcl_byfar_bulk} to bound $\mathfrak{1}_{\mathcal{B}\cap\Omega_t^{cl}}\aabs{\{\a^j,\gamma\}}\lesssim (1+\xi^{-1})(\xi^2+\xi^{-3})\aabs{\mathcal{D}[\gamma'](t)}$ in the close region. Altogether we thus see from the last two lines of \eqref{eq:dtsigma} that 
\begin{equation*}
\begin{split}
\Vert \mathfrak{1}_{\mathcal{B}}\partial_t\sigma\Vert_{L^\infty_{\vartheta,\a}}&\lesssim\varepsilon_0^3 t^{-21/20}.
\end{split}
\end{equation*}
Upon mollification it follows that $\mathfrak{1}_{\mathcal{B}}\sigma$ converges in $L^\infty$ to a limit $\gamma_\infty$ as $t\to\infty$. 
We thus have that
\begin{equation}
 \norm{\gamma(\vartheta+\ln(t)\left[Q\mathcal{E}^\infty(\a)+\mathcal{Q}\mathcal{E}^\infty(0)\right],\a,t)-\gamma_\infty(\vartheta,\a)}_{L^\infty_{\vartheta,\a}}\leq \norm{\sigma\mathfrak{1}_{\mathcal{B}^c}}_{L^\infty_{\vartheta,\a}}+\norm{\sigma\mathfrak{1}_{\mathcal{B}}-\gamma_\infty}_{L^\infty_{\vartheta,\a}}\lesssim \eps_0\ip{t}^{-21/20},
\end{equation}
where outside of the bulk we used \eqref{eq:nonbulk-bds} and \eqref{eq:gl_mom_bds}. Convergence in $L^2$ then follows from interpolation with the conserved $L^1$ norm of $\gamma$.

\end{proof}

Part $(2)$ now follows with relative ease:
\begin{proof}[Proof of Part $(2)$]
 It remains to establish the expression for the modified trajectories \eqref{eq:mu_mod_traj} and the convergence of $\mu$ \eqref{eq:mu_mod_scat} along them, the rest of Part $(2)$ follows directly from Part $(1)$. 
 
 By construction, from \eqref{NewNLUnknown} and \eqref{NewVariables} we have that
 \begin{equation}
  \gamma(\vartheta,\a,t)=\nu(\X(\vartheta+t\a,\a),\V(\vartheta+t\a,\a),t)=\mu(\X(\vartheta+t\a,\a)+\mathcal{X}(t),\V(\vartheta+t\a,\a)+\mathcal{V_\infty},t).
 \end{equation}
Hence $\mu(t)$ converges along the trajectories
\begin{equation}
\begin{aligned}
 {\bf Y}(\x,\v,t)&=\X(\vartheta+\ln(t)[Q\mathcal{E}^\infty(\a)+\mathcal{Q}\mathcal{E}^\infty(0)]+t\a,\a)+\mathcal{X}(t)\\
 &=\left(at-\frac{1}{2}\frac{q}{a^2}\ln(ta^3/q)+\ln(t)[Q\mathcal{E}^\infty(\a)+\mathcal{Q}\mathcal{E}^\infty(0)]\right)\cdot \frac{q^2}{4a^2L^2+q^2}\left(\frac{2}{q}{\bf R}+\frac{4a}{q^2}{\bf L}\times{\bf R}\right)\\
 &\qquad+
\mathcal{V}_\infty t+\mathcal{Q}\mathcal{E}^\infty(0)\ln(t)+O(1),\\
 {\bf W}(\x,\v,t)&=\V(\vartheta+\ln(t)[Q\mathcal{E}^\infty(\a)+\mathcal{Q}\mathcal{E}^\infty(0)]+t\a,\a)+\mathcal{V_\infty}\\
 &=a\left(1-\frac{q}{2ta^3}\right)\cdot \frac{q^2}{4a^2L^2+q^2}\left(\frac{2}{q}{\bf R}+\frac{4a}{q^2}{\bf L}\times{\bf R}\right)+\mathcal{V}_\infty+O(\ln(t)t^{-2}),
\end{aligned} 
\end{equation}
where we have used Remark \ref{rem:XVexpansion}.
\end{proof}

\subsection*{Acknowledgments}\vfill
B.\ Pausader is supported in part by NSF grant DMS-1700282 and by a Simons fellowship. K.\ Widmayer gratefully acknowledges support of the SNSF through grant PCEFP2\_203059. J.\ Yang was supported by NSF grant DMS-1929284 while in residence 
at ICERM (Fall 2021-Spring 2022). This work was partly carried out while the authors where participating at the ICERM program on ``Hamiltonian methods in dispersive and wave equations'' and they gratefully acknowledge the hospitality of the institute. B.\ Pausader  also thanks P.\ G\'erard, P.\ Rapha\"{e}l and M.\ Rosenzweig for stimulating and informative conversations on the Vlasov-Poisson equation and its symplectic structure.

% \newpage
\appendix
\section{Auxiliary results}\label{sec:appdx_trans}
\begin{lemma}\label{lem:trans_full}
Under the assumptions of Proposition \ref{prop:global_derivs} (where in particular $\ip{t}\abs{\mathcal{W}(t)}\lesssim \eps^2$) and with outgoing (resp.\ incoming) asymptotic actions, on the close region $\Omega_t^{cl}$ there holds that for some $C>0$
\begin{equation}\label{EquivXiXi'App}
 C^{-1}\xi'\leq\xi\leq C\xi',
\end{equation}
and for any scalar function $\zeta$ we have
 \begin{equation}\label{TransitionMapBoundsApp}
 \begin{aligned}
  \ww_{\xi'}\aabs{\{\xi',\zeta\}}&\lesssim \ww_{\xi}\aabs{\{\xi,\zeta\}}+\eps^2(\langle t\rangle^{-1/2}+\aabs{\XX}^{-1})\mathcal{D}[\zeta](t),\\
  \ww_{\L'}\aabs{\{\L',\zeta\}}&\lesssim \ww_{\L}\aabs{\{\L,\zeta\}}+\eps^2\mathcal{D}[\zeta](t),\\
  \aabs{\{\u',\zeta\}}&\lesssim \xi^{-3}\ww_{\bf L}\aabs{\{\L,\zeta\}}+(1+\eps^2\xi^{-1}+\eps^2\aabs{\XX}^{-1})\mathcal{D}[\zeta](t),\\
  \ww_{\eta'}\aabs{\{\eta',\zeta\}}&\lesssim\ww_\eta\aabs{\{\eta,\zeta\}}+\ww_{\L}\aabs{\{\L,\zeta\}}+(1+\xi)\ip{t}\aabs{\XX}^{-1}\ww_\xi\aabs{\{\xi,\zeta\}}+\eps^2(1+\ip{t}\aabs{\XX}^{-1})\mathcal{D}[\zeta](t).
 \end{aligned} 
 \end{equation} 
The analogous estimates hold with the roles of the primed and unprimed variables interchanged. 
\end{lemma}

\begin{remark}\label{rem:transition_bds}
We highlight two consequences of these bounds:
\begin{enumerate}
 \item For the transition from close to far (resp.\ far to close), on $\Omega_t^{cl}\cap\Omega_t^{far}$ we have that $\ip{t}\leq\aabs{\XX}\leq 10\ip{t}$ and thus the bounds \eqref{eq:deriv-deriv-PB} from Lemma \ref{lem:trans_short} follow, i.e.\ we have that
 \begin{equation}
  \mathcal{D}[\gamma'](t)\lesssim (\xi +\xi^{-3})\mathcal{D}[\gamma](t),\qquad \mathcal{D}[\gamma](t)\lesssim (\xi +\xi^{-3})\mathcal{D}[\gamma'](t).
 \end{equation}
 \item On $\Omega_t^{cl}$ we have the following control of Poisson brackets of $\gamma$ in terms of those of $\gamma'$:
 \begin{equation}\label{eq:PBcl_byfar}
  \mathfrak{1}_{\Omega_t^{cl}}\mathcal{D}[\gamma](t)\lesssim (\xi+\xi^{-3})(1+\ip{t}\aabs{\XX}^{-1})\mathcal{D}[\gamma'](t).
 \end{equation} 
By \eqref{eq:bulk-bds}, in the bulk this improves to
\begin{equation}\label{eq:PBcl_byfar_bulk}
 \mathfrak{1}_{\Omega_t^{cl}\cap\mathcal{B}}\mathcal{D}[\gamma](t)\lesssim (\xi^2+\xi^{-3})\mathcal{D}[\gamma'](t).
\end{equation}
\end{enumerate}
\end{remark}

\begin{proof}[Proof of Lemma \ref{lem:trans_full}]
We will use repeatedly the following bounds which follow from \eqref{eq:XX1X3}, \eqref{PBX1}-\eqref{PBX1'} and \eqref{eq:betterdl}, together with the bound $\vert\XX\vert\lesssim\langle t\rangle$ on $\Omega_t^{cl}$ (with outgoing resp.\ incoming asymptotic actions):
\begin{equation}\label{PBXZetaApp}
\begin{aligned}
 \aabs{\{\VV,\zeta\}}&\lesssim\frac{q}{\xi^2}[1+\frac{t\xi}{\aabs{\XX}^2}]\aabs{\{\xi,\zeta\}}+\frac{\xi^3}{2q\aabs{\XX}^2}\aabs{\{\eta,\zeta\}}+\aabs{\VV}\aabs{\{\u,\zeta\}}+\frac{\xi^2}{q\aabs{\XX}^2}\aabs{\{\L,\zeta\}}\\
 &\lesssim q\ip{\xi}^{-1}(1+\frac{t\xi}{\aabs{\XX}^2})\ww_\xi\aabs{\{\xi,\zeta\}}+\frac{\xi\ip{\xi}}{q\aabs{\XX}^2}\ww_\eta\aabs{\{\eta,\zeta\}}+\frac{q}{\xi}\aabs{\{\u,\zeta\}}+\frac{\xi\ip{\xi}}{q\aabs{\XX}^2}\ww_\L\aabs{\{\L,\zeta\}},\\
 \aabs{\{\XX,\zeta\}}&\lesssim \xi^{-1}[\aabs{\XX}+\frac{tq}{\xi}]\aabs{\{\xi,\zeta\}}+\frac{\xi^2}{q}\aabs{\{\eta,\zeta\}}+\aabs{\XX}\aabs{\{\u,\zeta\}}+\frac{\xi}{q}\aabs{\{\L,\zeta\}}\\
 %&\lesssim \ip{\xi}^{-1}(\xi\aabs{\XX}+tq)\m_\xi\aabs{\{\xi,\zeta\}}+\frac{\ip{\xi}}{q}\m_\eta\aabs{\{\eta,\zeta\}}+\aabs{\XX}\aabs{\{\u,\zeta\}}+\frac{\ip{\xi}}{q}\m_\L\aabs{\{\L,\zeta\}}\\
 &\lesssim \ip{t}\ww_\xi\aabs{\{\xi,\zeta\}}+\frac{\ip{\xi}}{q}\ww_\eta\aabs{\{\eta,\zeta\}}+\aabs{\XX}\aabs{\{\u,\zeta\}}+\frac{\ip{\xi}}{q}\ww_\L\aabs{\{\L,\zeta\}}.
\end{aligned} 
\end{equation}
On $\Omega_t^{cl}$ there holds that $\aabs{\XX}\lesssim\ip{t}$, and thus $\xi\lesssim\ip{t}^{1/2}$. It then follows for $\xi'\neq 0$ that
\begin{equation}\label{eq:xiquot}
 \frac{\xi}{\xi'}=\frac{\xi}{q}a'=\frac{\xi}{q}(a'-a)+1,\qquad \frac{\xi}{q}\abs{a'-a}\leq 4\frac{\xi}{q}\abs{\W(t)}\lesssim \frac{\eps}{q\ip{t}^{1/2}}\qquad \Rightarrow \qquad \frac{1}{2}\xi'\leq \xi\leq 2\xi',
\end{equation}
which gives \eqref{EquivXiXi'App}.

\bigskip
\begin{enumerate}[wide]
 \item In $\xi$: Note that since $\{a,\zeta\}=\frac{1}{2a}\{a^2,\zeta\}$, we have that
\begin{equation}
 \{\xi,\zeta\}=-\frac{\xi^2}{q}\{a,\zeta\}=-\frac{\xi^3}{2q^2}\{a^2,\zeta\}
\end{equation}
and thus by \eqref{eq:diffa2}
\begin{equation*}
 \{\xi',\zeta\}=\frac{(\xi')^3}{2q^2}\{a^2-(a')^2,\zeta\}-\frac{(\xi')^3}{2q^2}\{a^2,\zeta\}=\frac{(\xi')^3}{q^2}\W(t)\cdot\{\VV,\zeta\}+(\frac{\xi'}{\xi})^3\{\xi,\zeta\}.
\end{equation*}
Hence
\begin{equation}\label{eq:xixi'PB}
\begin{aligned}
 \ww_{\xi'}\aabs{\{\xi',\zeta\}}&\lesssim\frac{\xi'\ip{\xi'}}{q^2}\abs{\W(t)}\aabs{\{\VV,\zeta\}}+\frac{\xi'\ip{\xi'}}{\xi\ip{\xi}}\ww_\xi\aabs{\{\xi,\zeta\}}\\
 &\lesssim \left(\xi'\frac{\ip{\xi'}}{\ip{\xi}}q^{-1}\Big(1+\frac{t\xi}{\aabs{\XX}^2}\Big)\abs{\W(t)}+\frac{\xi'\ip{\xi'}}{\xi\ip{\xi}}\right)\ww_\xi\aabs{\{\xi,\zeta\}}\\
 &\qquad +\frac{\xi'\ip{\xi'}}{q^2}\abs{\W(t)}\left(\frac{\xi\ip{\xi}}{q\aabs{\XX}^2}\ww_\eta\aabs{\{\eta,\zeta\}}+\frac{q}{\xi}\aabs{\{\u,\zeta\}}+\frac{\xi\ip{\xi}}{q\aabs{\XX}^2}\ww_\L\aabs{\{\L,\zeta\}}\right).
\end{aligned} 
\end{equation}
By \eqref{eq:xiquot} we then have
\begin{equation}\label{PBXIprimeZetaApp}
 \ww_{\xi'}\aabs{\{\xi',\zeta\}}\lesssim (1+\frac{\eps^2}{\aabs{\XX}})\ww_\xi\aabs{\{\xi,\zeta\}}+\eps^2\frac{1}{\aabs{\XX}}\left(\ww_\eta\aabs{\{\eta,\zeta\}}+\ww_\L\aabs{\{\L,\zeta\}}\right)+\eps^2\ip{t}^{-1/2}\aabs{\{\u,\zeta\}}.
\end{equation}
This gives the first inequality in \eqref{TransitionMapBoundsApp}.

\item In $\L$: Since $\L'-\L=-\XX\times\W=:-\delta\L$, we have that
\begin{equation}
 \{\L',\zeta\}-\{\L,\zeta\}=\W(t)\times\{\XX,\zeta\}.
\end{equation}
Hence
\begin{equation}
 \aabs{\{\L',\zeta\}}\lesssim \aabs{\{\L,\zeta\}}+\abs{\W(t)}\aabs{\{\XX,\zeta\}}
\end{equation}
and thus
\begin{equation}
\begin{aligned}
 \ww_{\L'}\aabs{\{\L',\zeta\}}&\lesssim \left(\frac{\ww_{\L'}}{\ww_\L}+\abs{\W(t)}\frac{\ip{\xi}\ww_{\L'}}{q}\right)\ww_{\L}\aabs{\{\L,\zeta\}}\\
 &\qquad +\abs{\W(t)}\ww_{\L'}\left(\langle t\rangle\ww_\xi\aabs{\{\xi,\zeta\}}+\frac{\ip{\xi}}{q}\ww_\eta\aabs{\{\eta,\zeta\}}+\aabs{\XX}\aabs{\{\u,\zeta\}}\right),
\end{aligned}
\end{equation}
and hence
\begin{equation}
 \ww_{\L'}\aabs{\{\L',\zeta\}}\lesssim \ww_{\L}\aabs{\{\L,\zeta\}}+\eps^2\left(\ww_\xi\aabs{\{\xi,\zeta\}}+\ww_\eta\aabs{\{\eta,\zeta\}}+\aabs{\{\u,\zeta\}}\right),
\end{equation}
where we used that $\m_\L\m_{\L'}^{-1}\lesssim 1$ by \eqref{eq:xiquot}. This gives the second inequality in \eqref{TransitionMapBoundsApp}.

\item In $\u$: Recall that ${\bf R}=\frac{q}{2}\frac{\X}{\abs{\X}}+\V\times\L=\frac{q}{2}\u+a\u\times\L$ and thus
\begin{equation}
 \{{\bf R},\zeta\}=\frac{q}{2}\{\u,\zeta\}-q\{\xi^{-1}\L\times\u,\zeta\}=\frac{q}{2}\{\u,\zeta\}+q\xi^{-2}(\L\times\u)\{\xi,\zeta\}-q\xi^{-1}[\L\times\{\u,\zeta\}-\u\times\{\L,\zeta\}].
\end{equation}
It follows that
\begin{equation}
 \L\times\{{\bf R},\zeta\}=\frac{q}{2}\L\times\{\u,\zeta\}-q\lambda^2\xi^{-2}\u\{\xi,\zeta\}+q\xi^{-1}[\lambda^2\{\u,\zeta\}+\L(\u\cdot\{\L,\zeta\})+\u(\L\cdot\{\L,\zeta\})],
\end{equation}
and hence
\begin{equation}
\begin{aligned}
 \{{\bf R},\zeta\}+2\xi^{-1}\L\times\{{\bf R},\zeta\}&=q\xi^{-2}(\L\times\u)\{\xi,\zeta\}+q\xi^{-1}\u\times\{\L,\zeta\}-2\xi^{-1}q\lambda^2\xi^{-2}\u\{\xi,\zeta\}\\
 &\quad +2q\xi^{-2}[\L(\u\cdot\{\L,\zeta\})+\u(\L\cdot\{\L,\zeta\})]+\Big(\frac{q}{2}+2q\xi^{-2}\lambda^2\Big)\{\u,\zeta\},
\end{aligned} 
\end{equation}
i.e.\
\begin{equation}
\begin{aligned}
 \frac{q}{2}(1+4\kappa^2)\{\u,\zeta\}&=\{{\bf R},\zeta\}+2\xi^{-1}\L\times\{{\bf R},\zeta\}-q\xi^{-2}(\L\times\u)\{\xi,\zeta\}-q\xi^{-1}\u\times\{\L,\zeta\}\\
 &\quad +2q\kappa^2\u\xi^{-1}\{\xi,\zeta\}-2q\xi^{-2}[\L(\u\cdot\{\L,\zeta\})+\u(\L\cdot\{\L,\zeta\})]\\
 &=\{{\bf R},\zeta\}+2\kappa\l\times\{{\bf R},\zeta\}+q\xi^{-1}[2\kappa^2\u-\kappa(\l\times\u)]\{\xi,\zeta\}\\
 &\quad -q\xi^{-1}[\u\times\{\L,\zeta\}+2\kappa(\l(\u\cdot\{\L,\zeta\})+\u(\l\cdot\{\L,\zeta\}))].
\end{aligned}
\end{equation}
Applying this to $({\bf u}^\prime,{\bf R}^\prime,\xi^\prime,\kappa^\prime,{\bf L}^\prime)$, we get
\begin{equation}
\begin{aligned}
 \aabs{\{\u',\zeta\}}&\lesssim \ip{\kappa'}^{-1}\aabs{\{{\bf R'},\zeta\}}+(\xi')^{-1}\aabs{\{\xi',\zeta\}}+\ip{\kappa'}^{-1}(\xi')^{-1}\aabs{\{\L',\zeta\}}.
\end{aligned} 
\end{equation}
\begin{itemize}
\item For the third term we obtain
\begin{equation}
 \ip{\kappa'}^{-1}(\xi')^{-1}\aabs{\{\L',\zeta\}}\lesssim \ip{\kappa'}^{-1}(\xi')^{-1}\ww_{\L}^{-1}\cdot\ww_\L\aabs{\{\L,\zeta\}}+\ip{\kappa'}^{-1}(\xi')^{-1}\abs{\W(t)}\aabs{\{\XX,\zeta\}}
\end{equation}
with, using \eqref{EquivXiXi'App},
\begin{equation}
 (\xi')^{-1}\ww_{\L}^{-1}\cdot\ww_\L\aabs{\{\L,\zeta\}}\lesssim \ip{\xi}\xi^{-2}\cdot\ww_\L\aabs{\{\L,\zeta\}}
\end{equation}
and, using \eqref{PBXZetaApp},
\begin{equation}
\begin{aligned}
 (\xi')^{-1}\abs{\W(t)}\aabs{\{\XX,\zeta\}}&\lesssim\abs{\W(t)}\xi^{-1}\left(\ip{t}\ww_\xi\aabs{\{\xi,\zeta\}}+\frac{\ip{\xi}}{q}\ww_\eta\aabs{\{\eta,\zeta\}}+\ip{t}\aabs{\{\u,\zeta\}}+\frac{\ip{\xi}}{q}\ww_\L\aabs{\{\L,\zeta\}}\right)\\
 &\lesssim \eps^2\cdot\xi^{-1}\cdot\left(\ww_\xi\aabs{\{\xi,\zeta\}}+\ip{t}^{-1/2}\ww_\eta\aabs{\{\eta,\zeta\}}+\aabs{\{\u,\zeta\}}+\ip{t}^{-1/2}\ww_\L\aabs{\{\L,\zeta\}}\right).
\end{aligned} 
\end{equation}

\item For the second term we obtain using \eqref{PBXIprimeZetaApp},
\begin{equation}
 (\xi')^{-1}\aabs{\{\xi',\zeta\}}\lesssim \ww_{\xi'}\aabs{\{\xi',\zeta\}}\lesssim \ww_\xi\aabs{\{\xi,\zeta\}}+\eps^2\langle t\rangle^{-\frac{1}{2}}\left(\ww_\eta\aabs{\{\eta,\zeta\}}+\aabs{\{\u,\zeta\}}+\ww_\L\aabs{\{\L,\zeta\}}\right).
\end{equation}

\item The first term is more involved. With ${\bf R}'={\bf R}-\W\times\L-\V\times\delta\L+\W\times\delta\L$ we obtain that
\begin{equation}
\begin{aligned}
 \ip{\kappa'}^{-1}\aabs{\{{\bf R'},\zeta\}}&\lesssim \ip{\kappa'}^{-1}\aabs{\{\W\times\L+\V\times\delta\L-\W\times\delta\L,\zeta\}}+\frac{\ip{\kappa}}{\ip{\kappa'}}\ip{\kappa}^{-1}\aabs{\{{\bf R},\zeta\}}\\
 &\lesssim \abs{\W(t)}\ip{\kappa'}^{-1}\left(\aabs{\{\L,\zeta\}}+(q\xi^{-1}+\abs{\W(t)})\aabs{\{\XX,\zeta\}}+\aabs{\XX}\aabs{\{\VV,\zeta\}}\right)+\frac{\ip{\kappa}}{\ip{\kappa'}}\ip{\kappa}^{-1}\aabs{\{{\bf R},\zeta\}}.
\end{aligned} 
\end{equation}
By Lemma \ref{lem:prime_bds}
\begin{equation}
 \abs{\kappa-\kappa'}\leq \xi^{-1}\abs{\lambda-\lambda'}+\lambda'\frac{\abs{\xi'-\xi}}{\xi\xi^\prime}\lesssim \abs{\W(t)}\left[\xi^{-1}\ip{t}+\xi^\prime\kappa'\right],
\end{equation}
and thus, using also that $\xi\lesssim_q \langle t\rangle$,
\begin{equation}
 \frac{\ip{\kappa}}{\ip{\kappa'}}\lesssim 1+\xi^{-1}\ip{t}\aabs{\W(t)}\lesssim 1+\eps^2\xi^{-1}.
\end{equation}
Next we use that
\begin{equation}
 \aabs{\{{\bf R},\zeta\}}\lesssim q\left[\ip{\kappa}\aabs{\{\u,\zeta\}}+\xi^{-1}\kappa\aabs{\{\xi,\zeta\}}+\xi^{-1}\aabs{\{\L,\zeta\}}\right]
\end{equation}
to conclude that
\begin{equation}
\begin{aligned}
 \ip{\kappa'}^{-1}\aabs{\{{\bf R'},\zeta\}}&\lesssim_q \left(1+\eps^2\xi^{-1}+\frac{\eps^2}{\aabs{\XX}}\right)\ww_\xi\aabs{\{\xi,\zeta\}}+\varepsilon^2\left(\langle t\rangle^{-\frac{1}{2}}(1+\xi^{-1})+\frac{1}{\aabs{\XX}}\right)\ww_\eta\aabs{\{\eta,\zeta\}}\\
 &\quad +\left(1+\eps^2\xi^{-1}\right)\aabs{\{\u,\zeta\}} +\left(1+\xi^{-3}+\frac{\eps^2}{\aabs{\XX}}\right)\ww_\L\aabs{\{\L,\zeta\}}.
\end{aligned}
\end{equation}

\end{itemize}
Therefore,
\begin{equation}
\begin{aligned}
 \aabs{\{\u',\zeta\}}&\lesssim (1+\eps^2\xi^{-1}+\eps^2\aabs{\XX}^{-1})\ww_\xi\aabs{\{\xi,\zeta\}}+\eps^2\left((1+\xi^{-1})\langle t\rangle^{-\frac{1}{2}}+\aabs{\XX}^{-1}\right)\ww_\eta\aabs{\{\eta,\zeta\}}\\
 &\quad+(1+\eps^2\xi^{-1})\aabs{\{\u,\zeta\}}+(1+\xi^{-3}+\eps^2\aabs{\XX}^{-1})\ww_\L\aabs{\{\L,\zeta\}},
 \end{aligned}
\end{equation}
which gives the third bound in \eqref{TransitionMapBoundsApp}.

\item In $\eta$: By \eqref{eq:diffeta} there holds that
\begin{equation}
\begin{aligned}
 \eta'+\sigma\circ\Phi_t^{-1}\circ\mathcal{M}_t^{-1}&=(\eta+\sigma\circ\Phi_t^{-1})-\delta A\\
 \delta A&:=t\frac{(a')^3-a^3}{q}+\frac{a-a'}{q}(\XX\cdot\VV)+\frac{a'}{q}(\XX\cdot\mathcal{W}(t)).
\end{aligned} 
\end{equation}
Note that since $\sigma\circ\Phi_t^{-1}\circ\mathcal{M}_t^{-1}=\sigma(\eta'+tq^2(\xi')^{-3},\kappa')$ we have that
\begin{equation}
\begin{aligned}
 \{\sigma\circ\Phi_t^{-1}\circ\mathcal{M}_t^{-1},\zeta\}&=\partial_\eta\sigma(\eta'+tq^2(\xi')^{-3},\kappa')(\{\eta',\zeta\}-3tq^2(\xi')^{-4}\{\xi',\zeta\})\\
 &\quad+\partial_\kappa\sigma(\eta'+tq^2(\xi')^{-3},\kappa')\{\kappa',\zeta\},
\end{aligned} 
\end{equation}
and similarly for $\sigma\circ\Phi_t^{-1}$, where by \eqref{eq:deriv_rho}, \eqref{BoundsOnSigma}-\eqref{ImprovedBoundsOnSigma} and the bound $\rho\gtrsim \vert\XX\vert\cdot q\xi^{-2}$, we have that 
\begin{equation}
 \abs{\partial_\eta\sigma}\leq \frac{1}{2\rho}\leq\frac{1}{2},\qquad 
 \abs{\partial_\eta\sigma(\eta,\kappa)}\lesssim\frac{q}{a^2\aabs{\X}},\qquad\abs{\partial_\kappa\sigma}\lesssim \min\{\kappa\rho^{-2},\ip{\kappa}^{-1}\}.
\end{equation}
Hence
\begin{equation}
\begin{aligned}
 \frac{1}{2}\aabs{\{\eta',\zeta\}}&\leq [1+\partial_\eta\sigma(\eta'+tq^2(\xi')^{-3},\kappa')]\cdot \aabs{\{\eta',\zeta\}}\\
 &\lesssim \aabs{\partial_\eta\sigma\circ\Phi_t^{-1}\circ\mathcal{M}_t^{-1}}\cdot tq^2(\xi')^{-4}\aabs{\{\xi',\zeta\}}+\aabs{\partial_\kappa\sigma\circ\Phi_t^{-1}\circ\mathcal{M}_t^{-1}}\aabs{\{\kappa',\zeta\}}\\
 &\quad +\aabs{\{\eta+\sigma\circ\Phi_t^{-1},\zeta\}}+\aabs{\{\delta A,\zeta\}}.
\end{aligned} 
\end{equation}
Hence with $\xi\sim\xi'$ we have the bounds
\begin{equation}
 \ww_\eta\cdot\aabs{\partial_\eta\sigma\circ\Phi_t^{-1}\circ\mathcal{M}_t^{-1}}\cdot tq^2(\xi')^{-4}\aabs{\{\xi',\zeta\}}\lesssim \frac{q}{(1+\xi)^2}t\aabs{\partial_\eta\widetilde{\sigma}}\cdot \ww_{\xi'}\aabs{\{\xi',\zeta\}}\lesssim \frac{t}{\aabs{\XX}}\left(\ww_\xi\aabs{\{\xi',\zeta\}}+\eps^2\mathcal{D}[\zeta]\right),
\end{equation}
and
\begin{equation}
\begin{aligned}
 \ww_\eta\cdot\aabs{\partial_\kappa\sigma\circ\Phi_t^{-1}\circ\mathcal{M}_t^{-1}}\aabs{\{\kappa',\zeta\}}&\lesssim\frac{\xi^2}{1+\xi}\cdot\ip{\kappa'}^{-1}\xi^{-1}\left(\aabs{\{\lambda',\zeta\}}+\kappa'\aabs{\{\xi',\zeta\}}\right)\\
 &\lesssim \ww_{\L}\aabs{\{\L',\zeta\}}+\frac{\xi^3}{(1+\xi)^2}\ww_\xi\aabs{\{\xi',\zeta\}}\\
 &\lesssim \xi\ww_\xi\aabs{\{\xi,\zeta\}}+\ww_{\L}\aabs{\{\L,\zeta\}}+\eps^2(1+\xi\aabs{\XX}^{-1})\mathcal{D}[\zeta](t),
\end{aligned} 
\end{equation}
where we used \eqref{eq:xixi'PB} to absorb one extra factor of $\xi$.
Similarly,
\begin{equation}
 \ww_\eta\cdot\aabs{\{\eta+\sigma\circ\Phi_t^{-1},\zeta\}}\lesssim \ww_\eta\aabs{\{\eta,\zeta\}}+\ww_{\L}\aabs{\{\L,\zeta\}}+(\xi+t\aabs{\XX}^{-1})\ww_\xi\aabs{\{\xi,\zeta\}}.
\end{equation}
Finally, we observe that with $\{a^3,\zeta\}=\frac{3}{2}a\{a^2,\zeta\}$ and using \eqref{eq:diffa2}, it follows that
\begin{equation}
\begin{aligned}
 \{(a')^3-a^3,\zeta\}&=\frac{3}{2}a'\{(a')^2-a^2,\zeta\}+\frac{3}{2}(a'-a)\{a^2,\zeta\}\\
 &=-3a'\W(t)\cdot\{\VV,\zeta\}-3(a'-a)\frac{q^2}{\xi^3}\{\xi,\zeta\},
\end{aligned} 
\end{equation}
and
\begin{equation}
\begin{aligned}
 \{a'-a,\zeta\}&=\frac{\xi'}{2q}\{(a')^2-a^2,\zeta\}+\frac{\xi'-\xi}{2q}\{a^2,\zeta\}\\
 &=-\frac{\xi'}{q}\W(t)\cdot\{\VV,\zeta\}-(\xi'-\xi)\frac{q}{\xi^3}\{\xi,\zeta\},
\end{aligned} 
\end{equation}
so that, using \eqref{prime_bds},
\begin{equation}
\begin{aligned}
 \aabs{\{\delta A,\zeta\}}&\lesssim_q t\abs{\W(t)}\left(\xi^{-1}\aabs{\{\VV,\zeta\}}+\xi^{-3}\aabs{\{\xi,\zeta\}}\right)+\xi\abs{\W(t)}\aabs{\XX}\aabs{\VV}\left(\aabs{\{\VV,\zeta\}}+\xi^{-2}\aabs{\{\xi,\zeta\}}\right)\\
 &\quad +\abs{\W(t)}\left(\aabs{\XX}\aabs{\{\VV,\zeta\}}+\aabs{\VV}\aabs{\{\XX,\zeta\}}\right)+\abs{\W(t)}\left(\aabs{\XX}(\xi')^{-2}\aabs{\{\xi',\zeta\}}+(\xi')^{-1}\aabs{\{\XX,\zeta\}}\right).
\end{aligned} 
\end{equation}
From this and \eqref{PBXZetaApp}, we deduce that
\begin{equation}
\begin{aligned}
 \ww_\eta\aabs{\{\delta A,\zeta\}}&\lesssim \eps^2 \ww_\xi\aabs{\{\xi,\zeta\}}+\xi^2\ip{\xi}^{-2}\aabs{\W(t)}\aabs{\XX}\ww_\xi\aabs{\{\xi',\zeta\}}\\
 &\quad+\abs{\W(t)}(\aabs{\XX}\xi^{2}\ip{\xi}^{-1}+t\xi\ip{\xi}^{-1})\aabs{\{\VV,\zeta\}}+\eps^2\ip{t}^{-1}\frac{\xi}{\ip{\xi}}\aabs{\{\XX,\zeta\}}\\
 &\lesssim \eps^2\left[(1+t\aabs{\XX}^{-1})\ww_\xi\aabs{\{\xi,\zeta\}}+\ww_{\bf u}\aabs{\{{\bf u},\zeta\}}+(1+\aabs{\XX}^{-1})\left(\ww_\eta\aabs{\{\eta,\zeta\}}+\ww_{{\bf L}}\aabs{\{{\bf L},\zeta\}}\right)\right].
\end{aligned} 
\end{equation}
Altogether we obtain that
\begin{equation}
 \ww_{\eta'}\aabs{\{\eta',\zeta\}}\lesssim\ww_\eta\aabs{\{\eta,\zeta\}}+\ww_{\L}\aabs{\{\L,\zeta\}}+(1+\xi)\ip{t}\aabs{\XX}^{-1}\ww_\xi\aabs{\{\xi,\zeta\}}+\eps^2(1+\ip{t}\aabs{\XX}^{-1})\mathcal{D}[\zeta](t),
\end{equation}
which gives the last inequality in \eqref{TransitionMapBoundsApp}.
\end{enumerate}

\end{proof}

\bibliographystyle{abbrv}

\bibliography{vp-lib}

\begin{thebibliography}{10}

\bibitem{AW2021}
A.~Arroyo-Rabasa and R.~Winter.
\newblock Debye screening for the stationary {V}lasov-{P}oisson equation in
  interaction with a point charge.
\newblock {\em Communications in Partial Differential Equations},
  46(8):1569--1584, 2021.

\bibitem{BD1985}
C.~Bardos and P.~Degond.
\newblock Global existence for the {V}lasov-{P}oisson equation in {$3$} space
  variables with small initial data.
\newblock {\em Annales de l'Institut Henri Poincar\'{e}. Analyse Non
  Lin\'{e}aire}, 2(2):101--118, 1985.

\bibitem{BM2018}
C.~Bardos and N.~J. Mauser.
\newblock Kinetic equations: a {F}rench history.
\newblock {\em European Mathematical Society. Newsletter}, (109):10--18, 2018.
\newblock Translation of the French original [ MR3752406].

\bibitem{BMM2018}
J.~Bedrossian, N.~Masmoudi, and C.~Mouhot.
\newblock Landau damping in finite regularity for unconfined systems with
  screened interactions.
\newblock {\em Communications on Pure and Applied Mathematics}, 71(3):537--576,
  2018.

\bibitem{BMM2020}
J.~Bedrossian, N.~Masmoudi, and C.~Mouhot.
\newblock Linearized wave-damping structure of {V}lasov-{P}oisson in
  $\mathbb{R}^3$.
\newblock {\em arXiv preprint 2007.08580}, 2020.

\bibitem{CdPD2010}
J.~Campos, M.~del Pino, and J.~Dolbeault.
\newblock Relative equilibria in continuous stellar dynamics.
\newblock {\em Communications in Mathematical Physics}, 300(3):765--788, 2010.

\bibitem{CM2010}
S.~Caprino and C.~Marchioro.
\newblock On the plasma-charge model.
\newblock {\em Kinetic and Related Models}, 3(2):241--254, 2010.

\bibitem{CMMP2012}
S.~Caprino, C.~Marchioro, E.~Miot, and M.~Pulvirenti.
\newblock On the attractive plasma-charge system in 2-d.
\newblock {\em Communications in Partial Differential Equations},
  37(7):1237--1272, 2012.

\bibitem{CZW2015}
J.~Chen, X.~Zhang, and J.~Wei.
\newblock Global weak solutions for the {V}lasov-{P}oisson system with a point
  charge.
\newblock {\em Mathematical Methods in the Applied Sciences},
  38(17):3776--3791, 2015.

\bibitem{CK2016}
S.-H. Choi and S.~Kwon.
\newblock Modified scattering for the {V}lasov-{P}oisson system.
\newblock {\em Nonlinearity}, 29(9):2755--2774, 2016.

\bibitem{CLS2018}
G.~Crippa, S.~Ligabue, and C.~Saffirio.
\newblock Lagrangian solutions to the {V}lasov-{P}oisson system with a point
  charge.
\newblock {\em Kinetic and Related Models}, 11(6):1277--1299, 2018.

\bibitem{DMS2015}
L.~Desvillettes, E.~Miot, and C.~Saffirio.
\newblock Polynomial propagation of moments and global existence for a
  {V}lasov-{P}oisson system with a point charge.
\newblock {\em Annales de l'Institut Henri Poincar\'{e}. Analyse Non
  Lin\'{e}aire}, 32(2):373--400, 2015.

\bibitem{FHR2021}
E.~Faou, R.~Horsin, and F.~Rousset.
\newblock On linear damping around inhomogeneous stationary states of the
  {V}lasov-{HMF} model.
\newblock {\em Journal of Dynamics and Differential Equations},
  33(3):1531--1577, 2021.

\bibitem{FR2016}
E.~Faou and F.~Rousset.
\newblock Landau damping in sobolev spaces for the vlasov-hmf model.
\newblock {\em Archive for Rational Mechanics and Analysis}, (219):887--902,
  2016.

\bibitem{FOPW2021}
P.~Flynn, Z.~Ouyang, B.~Pausader, and K.~Widmayer.
\newblock Scattering map for the {V}lasov-{P}oisson system.
\newblock {\em Peking Mathematical Journal, to appear (arXiv preprint
  2101.01390)}, 2021.

\bibitem{Gla1996}
R.~T. Glassey.
\newblock {\em The {C}auchy problem in kinetic theory}.
\newblock Society for Industrial and Applied Mathematics (SIAM), Philadelphia,
  PA, 1996.

\bibitem{GI2020}
M.~Griffin-Pickering and M.~Iacobelli.
\newblock Recent developments on the well-posedness theory for vlasov-type
  equations.
\newblock {\em arXiv preprint 2004.01094}, 2020.

\bibitem{GL2017}
Y.~Guo and Z.~Lin.
\newblock The existence of stable {BGK} waves.
\newblock {\em Communications in Mathematical Physics}, 352(3):1121--1152,
  2017.

\bibitem{GS1995}
Y.~Guo and W.~A. Strauss.
\newblock Nonlinear instability of double-humped equilibria.
\newblock {\em Annales de l'Institut Henri Poincar\'{e}. Analyse Non
  Lin\'{e}aire}, 12(3):339--352, 1995.

\bibitem{HRS2021}
M.~Hadzic, G.~Rein, and C.~Straub.
\newblock On the existence of linearly oscillating galaxies.
\newblock {\em Arch. Rational Mech. Anal. 243, 611--696 (2022)}, Feb. 2021.

\bibitem{HNR2019}
D.~Han-Kwan, T.~T. Nguyen, and F.~Rousset.
\newblock Asymptotic stability of equilibria for screened {V}lasov-{P}oisson
  systems via pointwise dispersive estimates.
\newblock {\em Annals of PDE}, 7(2):Paper No. 18, 37, 2021.

\bibitem{HNR2020}
D.~Han-Kwan, T.~T. Nguyen, and F.~Rousset.
\newblock On the linearized {V}lasov-{P}oisson system on the whole space around
  stable homogeneous equilibria.
\newblock {\em Communications in Mathematical Physics}, 387(3):1405--1440,
  2021.

\bibitem{HW2022}
R.~M. H\"{o}fer and R.~Winter.
\newblock A fast point charge interacting with the screened {V}lasov-{P}oisson
  system.
\newblock {\em arXiv preprint 2205.00035}, Apr. 2022.

\bibitem{IJ2019}
A.~Ionescu and H.~Jia.
\newblock Axi-symmetrization near point vortex solutions for the 2d {E}uler
  equation.
\newblock {\em Comm. Pure Appl. Math., to appear (arXiv preprint 1904.09170)},
  2021.

\bibitem{IPWW2022}
A.~Ionescu, B.~Pausader, X.~Wang, and K.~Widmayer.
\newblock Nonlinear {L}andau damping for the {V}lasov-{P}oisson system in
  $\mathbb{R}^3$: the {P}oisson equilibrium.
\newblock {\em arXiv preprint 2205.04540}, May 2022.

\bibitem{IPWW2020}
A.~D. Ionescu, B.~Pausader, X.~Wang, and K.~Widmayer.
\newblock On the {A}symptotic {B}ehavior of {S}olutions to the
  {V}lasov--{P}oisson {S}ystem.
\newblock {\em International Mathematics Research Notices. IMRN},
  (12):8865--8889, 2022.

\bibitem{KL2008}
V.~Kaloshin and M.~Levi.
\newblock Geometry of {A}rnold diffusion.
\newblock {\em SIAM Review}, 50(4):702--720, 2008.

\bibitem{KMR2009}
J.~Krieger, Y.~Martel, and P.~Rapha\"{e}l.
\newblock Two-soliton solutions to the three-dimensional gravitational
  {H}artree equation.
\newblock {\em Communications on Pure and Applied Mathematics},
  62(11):1501--1550, 2009.

\bibitem{LMR2008}
M.~Lemou, F.~M\'{e}hats, and P.~Raphael.
\newblock The orbital stability of the ground states and the singularity
  formation for the gravitational {V}lasov {P}oisson system.
\newblock {\em Archive for Rational Mechanics and Analysis}, 189(3):425--468,
  2008.

\bibitem{LMR2012}
M.~Lemou, F.~M\'{e}hats, and P.~Rapha\"{e}l.
\newblock Orbital stability of spherical galactic models.
\newblock {\em Inventiones Mathematicae}, 187(1):145--194, 2012.

\bibitem{LZ2017}
D.~Li and X.~Zhang.
\newblock On the 3-{D} {V}lasov-{P}oisson system with point charges: global
  solutions with unbounded supports and propagation of velocity-spatial
  moments.
\newblock {\em Journal of Differential Equations}, 263(10):6231--6283, 2017.

\bibitem{LZ2018}
D.~Li and X.~Zhang.
\newblock Asymptotic growth bounds for the 3-{D} {V}lasov-{P}oisson system with
  point charges.
\newblock {\em Mathematical Methods in the Applied Sciences}, 41(9):3294--3306,
  2018.

\bibitem{LP1991}
P.-L. Lions and B.~Perthame.
\newblock Propagation of moments and regularity for the {$3$}-dimensional
  {V}lasov-{P}oisson system.
\newblock {\em Inventiones Mathematicae}, 105(2):415--430, 1991.

\bibitem{MV2020}
E.~Maderna and A.~Venturelli.
\newblock Viscosity solutions and hyperbolic motions: a new {PDE} method for
  the {$N$}-body problem.
\newblock {\em Annals of Mathematics. Second Series}, 192(2):499--550, 2020.

\bibitem{MMZ1994}
A.~J. Majda, G.~Majda, and Y.~X. Zheng.
\newblock Concentrations in the one-dimensional {V}lasov-{P}oisson equations.
  {I}. {T}emporal development and non-unique weak solutions in the single
  component case.
\newblock {\em Physica D. Nonlinear Phenomena}, 74(3-4):268--300, 1994.

\bibitem{MMP2011}
C.~Marchioro, E.~Miot, and M.~Pulvirenti.
\newblock The {C}auchy problem for the 3-{D} {V}lasov-{P}oisson system with
  point charges.
\newblock {\em Archive for Rational Mechanics and Analysis}, 201(1):1--26,
  2011.

\bibitem{Meyer2017}
K.~R. Meyer and D.~C. Offin.
\newblock {\em Introduction to {H}amiltonian dynamical systems and the {N}-body
  problem}, volume~90 of {\em Applied Mathematical Sciences}.
\newblock Springer, Cham, third edition, 2017.

\bibitem{MNP2022}
J.~K. Miller, A.~R. Nahmod, N.~Pavlović, M.~Rosenzweig, and G.~Staffilani.
\newblock A rigorous derivation of the {H}amiltonian structure for the {V}lasov
  equation.
\newblock {\em arXiv preprint 2206.07589}, June 2022.

\bibitem{Mil1983}
J.~Milnor.
\newblock On the geometry of the {K}epler problem.
\newblock {\em American Mathematical Monthly}, 90(6):353--365, 1983.

\bibitem{Mio2016}
E.~Miot.
\newblock A uniqueness criterion for unbounded solutions to the
  {V}lasov-{P}oisson system.
\newblock {\em Communications in Mathematical Physics}, 346(2):469--482, 2016.

\bibitem{MvdBW2010}
H.~Mo, F.~C. van~den Bosch, and S.~White.
\newblock {\em Galaxy {F}ormation and {E}volution}.
\newblock Cambridge University Press, 2010.

\bibitem{Mou2013}
C.~Mouhot.
\newblock Stabilit\'{e} orbitale pour le syst\`eme de {V}lasov-{P}oisson
  gravitationnel (d'apr\`es {L}emou-{M}\'{e}hats-{R}apha\"{e}l, {G}uo, {L}in,
  {R}ein et al.).
\newblock Number 352, pages Exp. No. 1044, vii, 35--82. 2013.
\newblock S\'{e}minaire Bourbaki. Vol. 2011/2012. Expos\'{e}s 1043--1058.

\bibitem{MV2011}
C.~Mouhot and C.~Villani.
\newblock On {L}andau damping.
\newblock {\em Acta Math.}, 207(1):29--201, 2011.

\bibitem{Pan2020}
S.~Pankavich.
\newblock Exact large time behavior of spherically symmetric plasmas.
\newblock {\em SIAM Journal on Mathematical Analysis}, 53(4):4474--4512, 2021.

\bibitem{PW2020}
B.~Pausader and K.~Widmayer.
\newblock Stability of a point charge for the {V}lasov-{P}oisson system: the
  radial case.
\newblock {\em Communications in Mathematical Physics}, 385(3):1741--1769,
  2021.

\bibitem{Pen1960}
O.~Penrose.
\newblock Electrostatic instabilities of a uniform non-maxwellian plasma.
\newblock {\em The Physics of Fluids}, 3(2):258--265, 1960.

\bibitem{Pfa1992}
K.~Pfaffelmoser.
\newblock Global classical solutions of the {V}lasov-{P}oisson system in three
  dimensions for general initial data.
\newblock {\em Journal of Differential Equations}, 95(2):281--303, 1992.

\bibitem{Rei2007}
G.~Rein.
\newblock Collisionless kinetic equations from astrophysics---the
  {V}lasov-{P}oisson system.
\newblock In {\em Handbook of differential equations: evolutionary equations.
  {V}ol. {III}}, Handb. Differ. Equ., pages 383--476. Elsevier/North-Holland,
  Amsterdam, 2007.

\bibitem{Sch1991}
J.~Schaeffer.
\newblock Global existence of smooth solutions to the {V}lasov-{P}oisson system
  in three dimensions.
\newblock {\em Communications in Partial Differential Equations},
  16(8-9):1313--1335, 1991.

\bibitem{ZM1994}
Y.~X. Zheng and A.~Majda.
\newblock Existence of global weak solutions to one-component
  {V}lasov-{P}oisson and {F}okker-{P}lanck-{P}oisson systems in one space
  dimension with measures as initial data.
\newblock {\em Communications on Pure and Applied Mathematics},
  47(10):1365--1401, 1994.

\end{thebibliography}

\end{document}